\documentclass[a4paper,10pt]{book}

\usepackage[american,french]{babel}
\usepackage[utf8]{inputenc}
\usepackage[T1]{fontenc}

\usepackage{amsmath, amsthm, amssymb}
\usepackage{amscd}
\usepackage[dvips]{graphicx}
\usepackage[all]{xy}
\usepackage{enumerate}
\usepackage{dsfont}
\usepackage{appendix}
\usepackage{psfrag}
\usepackage{float}
\usepackage{fancyhdr}

\usepackage[usenames,dvipsnames]{pstricks}
\usepackage{epsfig}
\usepackage{pst-grad} 
\usepackage{pst-plot} 
\usepackage{pst-solides3d}
\usepackage{subfig}
\usepackage[rightcaption]{sidecap}
\usepackage{pstricks-add}

\usepackage{color}

\makeatletter
\newcommand*{\@old@slash}{}\let\@old@slash\slash
\def\slash{\relax\ifmmode\delimiter"502F30E\mathopen{}\else\@old@slash\fi}
\makeatother

\def\backslash{\delimiter"526E30F\mathopen{}}

\newtheorem{thmintro}{Theorem}

\newtheorem*{corspecial}{Corollary \ref{cor:lambda1}}

\newtheorem{thmintrofr}{Théorème}
\newtheorem*{definfr*}{D\'efinition}

\newtheorem{thm}{Theorem}[section]
\newtheorem{cor}[thm]{Corollary}
\newtheorem{prop}[thm]{Proposition}
\newtheorem{defin}[thm]{Definition}
\newtheorem{lem}[thm]{Lemma}
\newtheorem{claim}[thm]{Claim}
\newtheorem*{thm*}{Theorem}
\newtheorem*{prop*}{Proposition}
\newtheorem*{defin*}{Definition}
\newtheorem*{lem*}{Lemma}
\newtheorem*{claim*}{Claim}

\theoremstyle{remark}
\newtheorem{rem}[thm]{Remark}
\newtheorem*{rem*}{Remark}

\newcommand{\leftexp}[2]{{\vphantom{#2}}^{#1}{#2}}

\newcommand\M{\widetilde{M}}
\newcommand\F{\widetilde{F}}
\newcommand{\wt}[1]{\widetilde{#1}}

\newcommand\R{\mathbb R}
\newcommand\Z{\mathbb Z}
\newcommand\N{\mathbb N}
\newcommand\C{\mathbb C}
\newcommand\Hyp{\mathbb H}

\def\S{\mathbb S}
\newcommand\T{\mathbb T}

\newcommand\ada{A \wedge dA^{n-1}}

\newcommand\bdb{B \wedge dB^{n-1}}

\newcommand\eps{\varepsilon}

\newcommand\cost{\cos(\theta)}
\newcommand\sint{\sin(\theta)}

\newcommand\cosp{\cos(\phi)}
\newcommand\cospp{\cos^2(\phi)}
\newcommand\sinp{\sin(\phi)}
\newcommand\sinpp{\sin^2(\phi)}

\newcommand\cospsi{\cos(\psi)}

\newcommand\sinpsi{\sin (\psi)}

\newcommand\parx{\frac{\partial}{\partial x}}
\newcommand\pary{\frac{\partial}{\partial y}}
\newcommand\partheta{\frac{\partial}{\partial \theta}}

\newcommand\parxx{\frac{\partial^2}{\partial x^2}}
\newcommand\paryy{\frac{\partial^2}{\partial y^2}}
\newcommand\parxy{\frac{\partial^2}{\partial x \partial y}}

\newcommand\parphi{\frac{\partial }{\partial \phi}}
\newcommand\parpsi{\frac{\partial }{\partial \psi}}

\newcommand\parttheta{\frac{\partial^2}{\partial \theta^2}}
\newcommand\parpphi{\frac{\partial^2 }{\partial \phi^2}}

\newcommand\partialxi{\frac{\partial}{\partial x_i}}

\newcommand\ee{\eps^2\sin^2(\phi)}
\newcommand\e{\eps\sin(\phi)}

\newcommand\parfunpsi{\frac{\partial f_1 }{\partial \psi}}
\newcommand\parfdepsi{\frac{\partial f_2 }{\partial \psi}}

\newcommand\xpsi{X_{\psi}}
\newcommand\xphi{X_{\phi}}
\newcommand\xtheta{X_{\theta}}

\newcommand\fundeux{f_1 \parfdepsi - f_2 \parfunpsi}

\newcommand\plm{\tilde{P}_{l}^{m}}
\newcommand\ylm{Y_{l}^{m}}

\newcommand\MM{\mathcal{M}\widetilde{M}}

\newcommand{\hmu}{\widetilde{\mu} }

\def\L{\mathcal{L}}
\newcommand\Le{\mathcal{L}_{\eps}}

\newcommand{\Tzero}{\mathring{T}}

\DeclareMathOperator{\id}{Id}

\DeclareMathOperator{\voleucl}{\mathrm{vol}_{\mathrm{Eucl}}}

\newcommand\dx{\frac{\partial}{\partial x}}
\newcommand\dy{\frac{\partial}{\partial y}}

\newcommand\dfxx{\frac{\partial f}{\partial^2 x}}
\newcommand\dfyy{\frac{\partial f}{\partial^2 y}}
\newcommand\dfxy{\frac{\partial f}{\partial x \partial y}}

\DeclareMathOperator{\res}{Res}

\newcommand\flot{ \phi^{t} }
\newcommand\hflot{ \tilde{ \phi}^t }

\newcommand\orb{ \mathcal O }
\newcommand\leafs{ \mathcal L ^{s} }
\newcommand\leafu{ \mathcal L ^{u} }

\newcommand\fs{\mathcal F^{s} }
\newcommand\hfs{\widetilde{\mathcal F}^{s} }
\newcommand\fu{\mathcal F^{u} }
\newcommand\hfu{\widetilde{\mathcal F}^{u} }

\newcommand{\al}[1]{\widetilde{\alpha_{#1} } }

\newcommand\univ{S^1_{\text{univ}}}

\author{Thomas BARTHELM\'E}
\title{A new Laplace operator in Finsler geometry and periodic orbits of Anosov flows}
\date{January 24, 2012}

\begin{document}

\maketitle


\frontmatter


\selectlanguage{french}{
\thispagestyle{empty}
\null \vspace{\stretch{1}}
\begin{flushright}
{\it
 Je rêve d'un jour où l'égoïsme ne régnera plus dans les sciences, où on s'associera pour étudier, au lieu d'envoyer aux académiciens des plis cachetés, on s'empressera de publier ses moindres observations pour peu qu'elles soient nouvelles, et on ajoutera ``je ne sais pas le reste''.\\}
 \'Evariste Galois
\end{flushright}
\vspace{\stretch{2}} \null
}

\selectlanguage{american}
\tableofcontents

\chapter{Introduction}

This dissertation is made up of two disjoint parts. In the first part, we introduce a generalization of the Laplace operator to Finsler geometry and start its study. In the second part we concentrate on the study of periodic orbits of Anosov flows in $3$-manifolds.

\section*{Finsler Geometry and Laplacian}

Finsler geometry has always been in the shadow of Riemannian geometry. Yet it was introduced by Riemann himself in his famous 1854 ``Habilitationsschrift'' in G\"ottingen. Riemann realized that the minimal condition to obtain an integral notion of length on a manifold was to equip the tangent bundle with a family of norms. Unfortunately, Minkowski's work on convex geometry was still 40 years away, and the norms not obtained from a scalar product were not well understood. Riemann abandoned the general case with these words:
\begin{quotation}
 Die Untersuchung dieser allgemeinern Gattung w\"urde zwar keine wesentlich andere Principien erfordern, aber ziemlich zeitraubend sein und verh\"altnissm\"assig auf die Lehre vom Raume wenig neues Licht werfen, zumal da sich die Resultate nicht geometrisch ausdr\"ucken lassen; \cite{Riemann}\\
{\it  The investigation of this more general kind would require no really different principles, but would take considerable time and throw little new light on the theory of space, especially as the results cannot be geometrically expressed \footnote{Translated by William Kingdon Clifford \cite{translationRiemann}};}
\end{quotation}
It was only in 1918 that, under the direction of Carath\'eodory, Paul Finsler studied the general case during his Ph.D. \cite{Finsler} and laid the basis of this theory.

Nowadays, Finsler geometry is relatively well understood but many Riemannian problems still await their Finsler equivalent. One problem that came up frequently in the past years was the development of global analytical tools.

In Riemannian geometry, the Laplace--Beltrami operator on a Riemannian manifold has long held its place as one of the most important objects in geometric analysis. Among the reasons is that, its spectrum, while being physically motivated, shows an intriguing and intimate connection with the global geometry of the manifold.

 In the first part of this dissertation, I introduce a dynamical generalization of the Laplace operator to Finsler geometry. Note that it is not the first generalization; Bao and Lackey \cite{BaoLackey}, Shen \cite{Shen:non-linear_Laplacian} and Centore \cite{Centore:FinslerLaplacians} all gave different definitions. But as the Laplace--Beltrami operator admits several equivalent definitions, it is not very surprising that generalizations of different definitions give different operators. We will see that ours produces a linear, elliptic, symmetric, second-order differential operator and is sufficiently simple to allow explicit computations of spectra. I also hope that the reader will be convinced that our definition is natural.

\subsection*{Geodesic flow and Hilbert form}
Let us start by introducing the main notions used here.

A Finsler metric on a manifold $M$ is a \emph{smooth} collection of \emph{strongly convex}, not necessarily symmetric, norms $F(x,\cdot)$ on each tangent space $T_xM$; here, strongly convex means that the Hessian of $F^2(x,\cdot)$ is positive definite. Note that, as a norm is uniquely defined by its unit ball, we can also think of a Finsler structure has a collection of open convex sets on each tangent space containing, but not centered at, $0$ and such that the boundary of each convex set has a definite positive Hessian. This strong convexity assumption might seem a bit odd, but it is often necessary, for instance to obtain the equation of geodesics.

This notion of Finsler metric is probably the most common one, but unfortunately also the most restrictive. In particular Hilbert geometries, which have known a renewed interest in the past few years, are not included in this definition: Hilbert metrics are only $C^0$ out of the zero section. But in this dissertation as often, a regularity at least $C^2$ is necessary. However, the smooth Finslerian world is still much wider than the Riemannian one and rich in new phenomena (see for instance \cite{Katok:KZ_metric} and \cite{CNV} or the survey \cite{Alvarez:Survey}).

There are many classical examples of Finsler metrics. A class that will play a central role in this dissertation are the Randers metrics. These metrics have many interests: they are physically motivated \cite{Randers} and come up in many different domains \cite{AntonelliIngardenMatsumoto}. Moreover, they constitute one of the simplest classes of Finsler metrics: they are just obtained by adding a differential $1$-form to a Riemannian metric, rendering their study tractable.

Our approach to Finsler geometry is due to Patrick Foulon: He introduced \cite{Fou:EquaDiff} several tools that allow the study of the geometry \emph{via its dynamics}, instead of using connections and local coordinates, hence eliminating one of the most common criticisms of Finsler geometry, i.e., too many indices.

The cornerstone of our study will therefore be the geodesic flow. Geodesic flows traditionally live on the unit tangent bundle of the manifold; however, in order to compare geodesic flows of different Finsler metrics, we will always see them in the \emph{homogenized tangent bundle}:
\begin{equation*}
 HM := \left( TM \smallsetminus \{0\} \right)/ \R^+.
\end{equation*}
Another central object of our study is the \emph{Hilbert form} $A$ associated with a Finsler metric; It is a differential $1$-form on $HM$, obtained by taking the \emph{vertical derivative} of the Finsler metric. Thanks to the strong convexity hypothesis, it is \emph{contact}, i.e., if $n$ is the dimension of $M$, then $\ada$ is a volume on $HM$. The link between $A$ and the Finsler metric that will be most helpful to us is the following:

The Hilbert form $A$ entirely determines the dynamics of the Finsler metric: If we write $X$ for the vector field generating the geodesic flow, then $X$ is the \emph{Reeb field} of $A$, i.e., it is defined by the following equations:
\begin{equation*}
\left\{ 
\begin{aligned}
  A(X) &= 1 \\
 i_X dA &= 0  \, .
 \end{aligned}
\right. 
\end{equation*}

Foulon's philosophy is to study Finsler geometry only via $X$, $A$ and objects built from them.

\subsection*{A Finsler--Laplacian}

As we already mentioned, there are several equivalent definitions for the Laplace--Beltrami operator (see for instance \cite{GHL}). Historically, it was defined on $C^2$ functions on $\R^n$ as:
\begin{equation*}
 \Delta f (p) = \sum_i \frac{\partial^2 f}{\partial x_i^2}(p),
\end{equation*}
and this definition can be extended from $\R^n$ to a Riemannian manifold by taking the $x_i$ to be normal coordinates at $p$ for the Riemannian metric. An intrinsic definition is given by taking the divergence of the gradient. And finally, the Hodge Laplacian, acting on differential forms, is defined by $d\delta + \delta d$ where $\delta$ is the co-differential.

We can see right now why different generalizations of this operator to the Finslerian context could give very different operators. Consider the definition given by the divergence of the gradient: The divergence of a vector field is defined by taking the Lie derivative of a given volume; so it is not specifically Riemannian. The gradient is not either: it is just the derivative of a function seen as an element of the tangent, instead of the cotangent bundle. So an identification between $TM$ and its dual $T^{\ast}M$ gives a gradient. And there is such a canonical identification given by a Finsler metric: the Legendre transform (see \ref{sec:cotangent} for the definition). However, the difference with the Riemannian case is that the Legendre transform is in general \emph{not linear}. Hence, generalizing the Laplace operator via this method gives a non-linear operator. It was done by Shen \cite{Shen:non-linear_Laplacian} but, as natural and interesting as it is, it is not what we were looking for in a Laplacian.

Bao and Lackey \cite{BaoLackey} gave a generalization of the Hodge Laplacian and Centore \cite{Cen:mean-value_laplacian} built an operator such that harmonic functions verify the mean value property.

 All these definitions rely on the choice of a volume on the manifold, which can be a problem as there are several candidates in Finsler geometry (see for instance \cite{AlvarezThompson,BuragoBuragoIvanov}).

In this dissertation, we generalize the historical definition. The problem is that there is no good notion of orthogonality in Finsler geometry, so we cannot directly apply the Riemannian definition. Before explaining our definition, let us start with a remark: \\
If $f$ is a function on $\R^2$ and $c_{\theta}(t)$ the ray from $0$ making an angle $\theta$ with the $x$-axis, then, a direct and easy computation shows that
\begin{equation*}
 \int_{\theta = 0}^{2\pi} \left. \frac{d^2 }{dt^2} f \left(c_{\theta} (t) \right)\right|_{t=0} d\theta = \pi \left( \frac{\partial^2 f}{\partial x^2}(0) + \frac{\partial^2 f}{\partial y^2}(0) \right) .
\end{equation*}

This simple remark gives the idea for our generalization: instead of taking the sum of derivatives along orthonormal directions, we can consider an average over all directions! But to be able to consider an average, we need to find a measure for it, i.e., a solid angle naturally associated with a Finsler metric.

As the Hilbert form $A$ is a contact form, there \emph{is} a truly canonical volume form associated with a Finsler metric, but only on $HM$, and not on $M$ as in the Riemannian case.

If we denote by $\pi \colon HM \rightarrow M$ the canonical projection and by $VHM:= \ker d\pi$ the set of vectors tangent to the fibers $H_xM$, then a (solid) angle is given by a nowhere-vanishing form on $VHM$. To obtain a natural notion of angle, we can just split the canonical volume $\ada$ into an angle and a volume form on the base manifold, normalizing in order to get back the euclidean angle:
\begin{prop*}
 Let $(M,F)$ be a Finsler metric. There exists a unique volume form $\Omega^F$\ on $M$\ and an $(n-1)$-form $\alpha^F$\ on $HM$ that is nowhere zero on $VHM$ and such that:
\begin{equation*}
  \alpha^{F} \wedge \pi^{\ast}\Omega^F =  A\wedge dA^{n-1}, 
\end{equation*}
and, for all $x\in M$, 
\begin{equation*}
 \int_{H_xM} \alpha^F =  \voleucl(\S^{n-1}) .
\end{equation*}
\end{prop*}

And we can define our Finsler--Laplace operator:
\begin{defin*}
 Let $(M,F)$ be a Finsler metric. For $f \in C^{2}(M)$ and $x \in M$, we set
\begin{equation*}
  \Delta^F f (x) := \frac{n}{\voleucl \left(\mathbb{S}^{n-1}\right)} \int_{\xi \in H_xM} \left. \frac{d^2 f}{dt^2}\left(c_\xi (t) \right)\right|_{t=0}   \alpha^F_x(\xi),
\end{equation*}
where $c_{\xi}$ is the geodesic leaving $x$ in the direction $\xi \in H_xM$.
\end{defin*}

And we obtain exactly what we hoped for:
\begin{thmintro} \label{thmintro:Laplacian}
 Let $(M,F)$ be a Finsler metric. Then $\Delta^F$\ is a second-order differential operator. Furthermore:
\begin{enumerate}
 \item[(i)] $\Delta^F$\ is elliptic, 
 \item[(ii)] $\Delta^F$\ is symmetric, i.e., for any $f,g \in C^{\infty}(M)$, 
 \begin{equation*}
  \int_M  \left( f\Delta^F g - g\Delta^F f \right)\; \Omega^F = 0.
 \end{equation*}
 \item[(iii)] $\Delta^F$\ coincides with the Laplace--Beltrami operator when $F$\ is Riemannian.
\end{enumerate}
\end{thmintro}

As we will see (in section \ref{sec:Finsler-Laplacian}), the proof of this result is made very simple thanks to our definition.

Let us also point out the following consequence of {\it (i)} and {\it (ii)}: $\Delta^F$ \emph{is unitarily equivalent to a Schr\"odinger operator}.

The proof of this last point is based on the following (known) result, which is of particular interest for us:
\begin{prop*}
 Let $(M,g)$\ be a closed Riemannian manifold and $\omega$\ a volume form on $M$. There exists a unique second-order differential operator $\Delta_{g,\omega}$\ on $M$ with real coefficients such that its symbol is the dual metric $g^{\star}$, it is symmetric with respect to $\omega$ and zero on constants.\\
If $a\in C^{\infty}(M)$ is such that $\omega = a^2 v_g$, where $v_g$ is the Riemannian volume, then for $\varphi \in C^{\infty}(M)$:
\begin{equation*}
 \Delta_{g,\omega} \varphi = \Delta^{g} \varphi - \frac{1}{a^2} \langle \nabla \varphi , \nabla a^2 \rangle.
\end{equation*}
\end{prop*}

The symbol of an elliptic second-order differential operator always gives a Riemannian (co)-metric. So the above proposition tells us that our Finsler--Laplace operator is uniquely determined by its symbol and the volume $\Omega^F$. This leads to a few remarks:

First, there are far more Finsler metrics than pairs (Riemannian metric, volume), so there are many Finsler metrics sharing the same Laplacian. So our Finsler--Laplace operator cannot carry as much information as the metric it comes from. But this can be seen as an interesting source of questions: For instance, does sharing the same Laplacian imply that the metrics share some geometrical or dynamical properties?

The operators $\Delta_{g,\omega}$ are called \emph{weighted Laplacians} and were introduced by Chavel et Feldman \cite{ChavelFeldman:Isoperimetric_constants} and by Davies \cite{Davies:Heat_kernel_bounds}. They have been under much study since then (see for instance \cite{Grigoryan:heat_kernels_on_weighted_manifolds} or \cite{ColboisSavo}) and it seems natural to ask whether we could get a better understanding of them via the study of our Finsler--Laplacian or vice-versa. 

Which takes us to our last remark: Can we obtain every weighted Laplacian as a Finsler--Laplacian? As there are more Finsler metrics than pairs (Riemannian metric, volume), we conjecture that the answer is positive.

On surfaces, we prove that it is indeed the case, and that it is enough to consider Randers metrics to obtain every pair. Unfortunately, our proof is based on the local expression we obtained for Randers surfaces and it cannot be generalized as such.

\subsection*{Energy and spectrum}

As expected from a Laplacian, on compact manifolds our operator admits a discrete spectrum. We give the two classical proofs; the first is just the application of the general theory of \emph{unbounded} elliptic symmetric operators. The second, via the Min-Max method, is more interesting for a future study of the spectrum. It relies on the introduction of the following energy associated with $\Delta^F$:
\begin{equation*}
 E(u) := \frac{n}{\voleucl \left(\S^{n-1}\right) } \int_{HM} \left|L_X\left(\pi^{\ast}u \right)\right|^2 \ada
\end{equation*}

With this energy, we generalize several classical Riemannian results. First, we prove that harmonic functions are always obtained as minima of the energy. Second, we study how the energy varies inside a conformal class and show that it is invariant when $n=2$, which allows us to show:
\begin{thmintro}
\label{thmintro:inv_conforme}
Let $(\Sigma,F)$\ be a Finsler surface. If $f \colon \Sigma \xrightarrow{C^{\infty}} \R$\ and $F_f = e^f F$. Then,
\begin{equation*}
 \Delta^{F_f} = e^{-2f} \Delta^F.
\end{equation*}

\end{thmintro}

\subsection*{Explicit representation and computation of spectrum}

In order to prove, even to ourself, that our operator was worth studying, we felt that it was essential to give examples. Indeed, the fact that we manage to obtain explicit spectral data for Finslerian metrics is for us an asset of our operator.

But computing spectra is a daunting task, even in the Riemannian case. Indeed, the full spectra of the Laplace--Beltrami operator is known only for the model spaces $\Hyp^n$, $\R^n$ and $\S^n$ and some of their quotients. In order to have any chance of computing a spectrum, we looked for Finsler metrics with constant flag curvature (the flag curvature is the generalization to Finsler geometry of the sectional curvature, see \cite{Egloff:thesis,BaoChernShen}). But note that model spaces do not exist in Finsler geometry; for any $R$, there is an infinite number of non-isometric Finsler metrics of constant flag curvature $R$. Another problem is having actual examples of constant curvature metrics on closed manifolds. Among known examples, we chose to study the ones that seem to us to be most interesting and manageable.

If the flag curvature is negative and the manifold is compact, then a theorem of Akbar-Zadeh \cite{AkZ} implies that the Finsler structure is in fact Riemannian (this is still true for compact locally symmetric Finsler metrics \cite{Fou:Curvature_and_global_rigidity}).

In the same article, Akbar-Zadeh showed that a simply connected compact manifold endowed with a metric of positive constant flag curvature is a sphere (see also \cite{Rademacher:Sphere_theorem} for the pinched curvature case). But from a metric standpoint, things get much more exciting: Bryant \cite{Bryant:FS2S,Bryant:PFF2S} constructed a two-parameter family of (projectively flat) metrics of constant curvature $1$ on the $2$-sphere. Previously, Katok \cite{Katok:KZ_metric} had constructed a family of one-parameter deformations of the standard metric on $S^2$\ in order to obtain examples of metrics with only a finite number of closed geodesics. This example was later generalized and studied by Ziller \cite{Ziller:GKE}. Rademacher \cite{Rademacher:Sphere_theorem} proved that the Katok--Ziller metrics on the $2$-sphere have constant flag curvature. Note that Foulon \cite{Fou:perso} has a proof that Katok--Ziller metrics on any space have constant flag curvature.

These metrics were the perfect candidate for us. They are dynamically interesting and they admit adequate explicit formulas (see \cite{Rademacher:Sphere_theorem} for the sphere and Proposition \ref{prop:expression_KZ} in general) making them somewhat easier to study. In the case of the $2$-sphere, we obtain an approximation of the spectrum as well as the following:
\begin{thmintro} \label{thmintro:lambda1}
 For a family of Katok--Ziller metrics $F_{\eps}$ on the $2$-sphere, if $\lambda_1(\eps)$ is the smallest non-zero eigenvalue of $-\Delta^{F_{\eps}}$, then
\begin{equation}
 \lambda_1(\eps) = 2 -2 \eps^2 = \frac{8 \pi}{\rm{vol}_{\Omega^{F_{\eps}}}\left(\S^2\right)}.
\end{equation}
\end{thmintro}
Note that this result exhibits a family of Finslerian metrics realizing what is known to be the maximum for the first eigenvalue of the Laplace--Beltrami operator on $\S^2$ \cite{Hersch}. We unfortunately do not yet know whether this is also a maximum in the Finsler setting.

 Finally as Katok--Ziller metrics also exists on tori, we studied them and obtained their spectrum. Note that the flat case does not lead to new operators. Indeed, for any locally Minkowski structure on a torus, we show that the Finsler--Laplace operator is the same as the Laplace--Beltrami operator associated with the symbol metric (see Remark \ref{rem:Minkowsky_randers}). It is nonetheless interesting to do the computations as this gives some insight, shows some limits of what can be expected from this operator and proves once again that computations are feasible. For instance, we will see that there is no Poisson formula linking the length spectrum of the Finsler metric and the spectrum of the Finsler--Laplacian. Finally, if we one day want to obtain topological bounds on the spectrum, we will first have to understand these easy examples.

\subsection*{Negative curvature, spectrum and geometry at infinity}

The world of negative curvature is immensely rich in terms of the interactions between geometry, dynamics, ergodic theory and the spectrum of the Laplacian.

Let $M$ be a closed Riemannian manifold of strictly negative curvature and $\M$ its universal cover. $\M$ admits a visual boundary $\M(\infty)$ carrying (at least) three natural class of measures:
\begin{itemize}
 \item The Liouville measure class, obtained by projecting Lebesgue measures on unit spheres to the boundary along geodesic rays;
 \item The Patterson-Sullivan measure class $\nu_x$;
 \item The harmonic measures $\mu_x$, which can be obtained as the measures solving the Dirichlet problem at infinity, i.e., if $f \in C^0(\M(\infty))$, then 
\begin{equation*}
 u(x) := \int_{\xi \in \M(\infty)} f(\xi) d\mu_x(\xi)
\end{equation*}
is the unique function that verifies $\Delta^{\F} u = 0$ on $\M$ and $u(x) \rightarrow f(\xi)$ when $x \rightarrow \xi \in \M(\infty)$.
\end{itemize}

Kaimanovich \cite{Kaimanovich} showed that there exists a convex isomorphism between the cone of Radon measures on $\partial^2 \M := \M(\infty) \times \M(\infty) \smallsetminus \{\textrm{diag}\}$ invariant by $\pi_1(M)$ and the cone of Radon measures on $HM$ invariant by the geodesic flow. Via this correspondence, we can obtain the Patterson--Sullivan measures from the Bowen--Margulis measure.

It is also known that all these measure classes are ergodic with respect to the action of the fundamental group on $\M(\infty)$ (each measure in a measure class is $\pi_1(M)$-quasi-invariant, so ergodicity can be defined as usual).

In the case of constant curvature, these three classes are equal. Katok \cite{Katok:4_appl} and Ledrappier \cite{Ledrappier:Harm_measures} showed that, on surfaces, if any two of these classes coincides, then the metric has constant curvature. In higher dimensions, if the harmonic and Patterson--Sullivan measures coincides, then the universal cover of the manifold is a symmetric space \cite{BessonCourtoisGallot}.

Ledrappier obtained his regularity theorem as a corollary of a more general result, valid in any dimension:  $\nu_x = \mu_x$ if and only if $\lambda_1 = h^2/4$, where $\lambda_1$ is the bottom of the essential spectrum of $-\Delta^{\F}$ and $h$ the topological entropy.

As we introduced a Finsler--Laplacian, it was tempting to try to adapt the above results to our case. Indeed, if $F$ is negatively curved, then the geodesic flow is again contact Anosov (\cite{Fou:EESLSPC}) and, if $F$ is reversible, then $\M$ is Gromov-hyperbolic. So $\M$ also admits a visual boundary, which naturally carries the Liouville measure class and the Patterson--Sullivan measure class (see \cite{Coornaert:mesures_PS}). We studied existence of harmonic measures.

One general method to solve the Dirichlet problem at infinity and obtain harmonic measures is via potential theory. Martin \cite{Martin} constructed a boundary associated to a pair $(\M, \Delta^{\F})$. If the Martin boundary is reduced to its minimal part and if it is homeomorphic to the visual boundary, then we can deduce the existence of harmonic measures. Ancona \cite{Ancona:theorie_potentiel} identified the Martin and the visual boundaries for a very general class of elliptic operators on Gromov-hyperbolic spaces.

We prove that our operator satisfies the conditions for Ancona's theorem. Furthermore, we show that the identification between the Martin and the visual boundary is H\"older-continuous as in the Riemannian case. This allows us to copy the work of Ledrappier \cite{Ledrappier:Ergodic_properties} to show ergodic properties of the harmonic measures:
\begin{thmintro}
Let $(M,F)$ be a closed, negatively curved, reversible Finsler manifold, $\M$ its universal cover and $\mu_x$ the family of harmonic measures on $\M(\infty)$; we have the following properties:
\begin{itemize}
 \item[(i)] The harmonic measure class $\lbrace \mu_x \rbrace$ is ergodic for the action of $\pi_1(M)$ on $\M(\infty)$.\\
 \item[(ii)] For any $x\in \M$, there exists a weight $f \colon \partial^2 \M \rightarrow \R$, such that the measure $ f \mu_x \otimes \mu_x$ is ergodic for the action of $\pi_1(M)$ on $\partial^2\M$.\\
 \item[(iii)] There exists a unique measure $\mu$ on $HM$, ergodic with respect to the geodesic flow, such that its image under Kaimanovitch correspondence is $ f \mu_x \otimes \mu_x$.
\end{itemize}
\end{thmintro}

Rigidity results however are still out of reach.

Finally, note also that we do obtain a dynamical upper bound for the bottom of the spectrum of $-\Delta^{\F}$:
\begin{thmintro} \label{thmintro:dynamical_bound}
 Let $(M,F)$ be a closed, negatively curved, (not necessarily reversible) Finsler $n$-manifold and $(\M, \F)$ its universal cover. If $\lambda_1$ is the bottom of the essential spectrum of $-\Delta^{\F}$ and $h$ the topological entropy, then
\begin{equation*}
 \lambda_1 \leq \frac{nh^2}{4}.
\end{equation*}
\end{thmintro}
And we have a topological lower bound, without curvature assumptions, thanks to a generalization of a theorem of Brooks \cite{Brooks:pi1_and_spectrum}:
\begin{thmintro}\label{thmintro:Brooks}
 Let $(M,F)$ be a closed (not necessarily reversible) Finsler manifold and $(\M, \F)$ its universal cover. If $\lambda_1$ is the bottom of the essential spectrum of $-\Delta^{\F}$, then
\begin{equation*}
 \lambda_1 = 0 \quad \text{ if and only if} \quad \pi_1(M) \: \text{ is amenable}.
\end{equation*}
\end{thmintro}

\subsection*{Angle and co-angle: Finsler geometry and its Hamiltonian counter-part}

As our generalization of the Laplace operator relies essentially on the angle $\alpha^F$, we briefly studied it. For instance, we show that we can recognize when a Finsler surface is Riemannian just by looking at the angle.

 Now, recall that $\alpha^F$ is obtained by splitting the canonical volume form $\ada$. But looking on the cotangent side, we can see that there are other canonical volume forms: Indeed, a Finsler metric $F \colon TM \rightarrow \R^+$ uniquely determines an Hamiltonian $F^{\ast} \colon T^{\ast}M \rightarrow \R^+$ via the Legendre transform. $T^{\ast}M$ is a symplectic manifold and hence admits a canonical volume $d\lambda^n$, where $\lambda$ is the Liouville $1$-form. Finally, in the same manner as we obtained the Hilbert form $A$, we can obtain a one-form $B$ on $H^{\ast}M$, the homogenized cotangent bundle, naturally associated with $F^{\ast}$. The one-form $B$ is also a contact form and therefore $\bdb$ is a canonical volume form on $H^{\ast}M$.

Fortunately, all these volumes are pretty simply linked together: $\ada$ is the pull-back of $\bdb$ by the Legendre transform of $F$. And, if $\hat{r}\colon \mathring{T}^{\ast} M \rightarrow H^{\ast}M$ is the canonical projection, then $\hat{r}^{\ast}\bdb = \lambda \wedge d\lambda^{n-1} / (F^{\ast})^n$ (see Section \ref{sec:cotangent}).

We can carry out the construction of our angle form on the homogenized cotangent bundle, by splitting $\bdb$ into a volume form on $M$ and a co-angle $\beta^F$. We show that $\beta^F$ is the push-forward of $\alpha^F$ via the Legendre transform and that the volume on $M$ is the same. This proves that carrying the construction of our Finsler--Laplace operator on the cotangent bundle gives the same operator.

The study of the symplectic side is often very interesting. For instance, we can remark that the volume $\Omega^F$ that we obtain is the Holmes-Thompson volume (see section \ref{sec:holmes_thompson}). More surprisingly, we show that a Finsler metric is uniquely determined by its volume and its co-angle $\beta^F$ (Corollary \ref{cor:Finsler_uniquely_determined}).

\section*{Skewed $\R$-covered Anosov flows}

Anosov (\cite{Anosov}) managed to extract from geodesic flows in negatively curved manifolds the minimal condition giving their hyperbolicity. Since then, Anosov flows became an immense source of wonder in dynamical systems. In the second part of this dissertation, we concentrate on a ``topologically nice'' kind of Anosov flows on $3$-manifolds and study their periodic orbits.

Thierry Barbot and Sergio Fenley started studying Anosov flows via their transverse geometry and we follow their lead. The main objects under study here are the \emph{orbit space} and the \emph{leaf spaces}: If $\flot$ is an Anosov flow on a $3$-manifold $M$, $\M$ its universal cover and $\hflot$ the lifted flow, then the orbit space of $\flot$ is defined as $\M$ quotiented out by the relation ``being on the same orbit of $\hflot$'', and the stable (resp. unstable) leaf space is $\M$ quotiented out by the relation ``being on the same weak stable (resp. weak unstable) leaf of $\hflot$.''

For Anosov flows on $3$-manifolds, the orbit space is always homeomorphic to $\R^2$ \cite{Bar:these,Fen:AFM}, but in general the leaf spaces are non-Hausdorff. An Anosov flow is called \emph{$\R$-covered} if one (and hence both, see \cite{Bar:these,Fen:AFM}) of its leaf spaces is homeomorphic to $\R$. Fenley and Barbot proved that, if a stable leaf of $\hflot$ intersects every unstable leaves, then $\flot$ is a suspension of an Anosov diffeomorphism. So the interesting case is the other one, and Fenley called these flows \emph{skewed}.

It is fairly easy to see that the geodesic flow of a negatively curved (Finsler or Riemannian) surface is a skewed $\R$-covered Anosov flow. A recent result by (again) Barbot and Fenley \cite{BarbotFenley} in fact shows that it is (topologically) the only one on Seifert-fibered space. But if you stop controlling the topology, many non-algebraic examples exist (geodesic flows and suspension of Anosov diffeomorphisms are called algebraic Anosov flows because they are topologically conjugate to the action of a one-parameter group on a quotient $\Gamma \backslash G \slash K$, where $G$ is a Lie group, $K$ a compact Lie subgroup and $\Gamma$ a discrete subgroup acting co-compactly on $G \slash K$, see \cite{Tomter}).

In \cite{Fen:AFM}, Fenley produced a wealth of skewed $\R$-covered Anosov flows on atoroidal, not Seifert-fibered $3$-manifolds. Remark also that the construction of Foulon and Hasselblatt \cite{FouHassel:contact_anosov} leads to non-algebraic flows which are contact, hence skewed $\R$-covered (see \cite{Bar:PAG}).

\subsection*{Hyperbolic manifolds and isotopy class}

When the manifold is atoroidal and not Seifert-fibered, so for instance hyperbolic, skewed $\R$-covered Anosov flows have a surprising quality: every periodic orbit is freely homotopic to infinitely many other ones. This is in sharp contrast with the geodesic flow case where free homotopy classes are trivial.

We got interested in the following question: {\it given a periodic orbit, what can we say about its isotopy class?} In order to underline its interest, let us rephrase this question. Given a periodic orbit, its free homotopy class gives us a collection of topologically equivalent embeddings of $S^1$, i.e., knots, in a three-manifold. So {\it are these knots different?} Here we understand ``different'' in the traditional sense of knot theory.

Note that the questions about the type of knots that one can obtain from periodic orbits of flows are not new, and very interesting (the reader can consult Ghys' article on this subject \cite{Ghys:knots_and_dynamics}).

In collaboration with Sergio Fenley, we showed that all orbits in a free homotopy class are isotopic. If the answer is a bit disappointing, the way we obtain this result is quite interesting and opens some new questions. Thurston \cite{Thurston:3MFC}, Calegari \cite{Calegari:book} and Fenley \cite{Fen:Foliations_TG3M} proved that there exist pseudo-Anosov flows constructed from the geometry of some $\R$-covered foliations. We obtain the isotopy between the periodic orbits of $\flot$ by pushing them via a well-chosen pseudo-Anosov flow.

\subsection*{Embedded cylinders and periodic orbits}

An isotopy between periodic orbits creates an immersed cylinder, so we specialized our study to when we can obtain an \emph{embedded} cylinder between two periodic orbits.

By adapting results of Barbot \cite{BarbotFenley}, we showed that the existence of embedded annuli between orbits is essentially linked to the action of the fundamental group on the orbit space. Moreover, we can once again use the tools given by Thurston's work on $\R$-covered foliations.

We show that there exist some periodic orbits that cannot be joined to any other by an embedded cylinder. In a forthcoming paper with S. Fenley, we will show that the ``co-cylindrical classes'' are always finite.

\section*{Structure of this dissertation}

In Chapter 1, we introduce the necessary notions for the study of Finsler geometry via its dynamic and recall a number of results. We also present the symplectic side and the main properties of the Legendre transform. Readers familiar with these topics should start in Section \ref{sec:angle_and_volume_in_Finsler_geometry} where we introduce the angle $\alpha^F$ and the co-angle $\beta^{F^{\ast}}$ (Propositions \ref{prop:construction} and \ref{prop:coangle}). We then study some properties of these angles and remark that the volume $\Omega^F$ is the Holmes-Thompson volume.\\

Chapter 2 is the core of this work, it introduces the Finsler--Laplace operator (Definition \ref{def:delta}) and proves Theorem \ref{thmintro:Laplacian}. It then introduces the energy (Definition \ref{def:energy}) and proves that the spectrum can be obtained via the Min-Max principle (Theorems \ref{thm:Min-Max} and \ref{thm:Min-Max_bis}). The Chapter closes with the proof that, in dimension $2$, the Finsler--Laplace operator is still (almost) a conformal invariant (Theorem \ref{thm:inv_conforme}).\\

In Chapter 3 we study examples. We show that we can obtain a local coordinates expression of our operator for Randers surfaces (Proposition \ref{prop:symbol_laplacian_randers}). We then use that expression to show that every weighted Laplacian on surfaces can be obtained from a Randers metric (Proposition \ref{prop:Randers_donne_couple}). We then specialize to Katok--Ziller metrics on the sphere and the torus, giving explicit spectral data (Theorems \ref{thm:KZ_torus}, \ref{thm:spectre_S2} and Corollary \ref{cor:lambda1}).\\

In Chapter 4, we finally concentrate on the link between the Laplacian and the geodesic flow in negative curvature. We start by showing Theorem \ref{thmintro:dynamical_bound} (Proposition \ref{prop:upper_bound}) and Theorem \ref{thmintro:Brooks} (Theorem \ref{thm:Brooks}). We then introduce some of the potential theory needed for Ancona's theorem (Section \ref{sec:harmonic_measures}) and prove that the homeomorphism between the Martin and the visual boundaries is H\"older (Theorem \ref{thm:C_alpha_extension}). We show that Ancona's theorem applies to our operator and deduce that the Dirichlet problem at infinity still admits a unique solution (Corollary \ref{cor:harmonic_exists}). We finish the first part by generalizing Ledrappier's proof that harmonic measures are ergodic (Theorem \ref{thm:harmonic_measures_are_ergodic}).\\

Chapter 5 constitutes the second part of this dissertation. We start by recalling a number of results of Barbot and Fenley on Anosov flows, then the work of Thurston, Calegari and Fenley on $\R$-covered foliations. We then show that homotopic orbits are isotopic (Theorem \ref{thm:homo_implies_isotope}). We start the study of ``co-cylindrical'' classes and relate it to the action of the fundamental group on the universal circle.

\selectlanguage{french}{
\chapter{Introduction en français}

Cette thèse comporte deux parties distinctes, la première traite d'un nouvel opérateur de Laplace en géométrie de Finsler et la seconde de l'étude de certains types de flots d'Anosov dans des variétés atoroidales.

\section*{Géométrie de Finsler et laplacien}

La géométrie de Finsler a toujours été le parent pauvre de la géométrie riemannienne. Elle fut pourtant introduite par Riemann lui-même lors de sa fameuse ``Habilitationsschrift'' de 1854. En effet, Riemann réalisa que la condition minimale pour introduire une notion de longueur sur une variété était d'avoir une famille continue de normes sur chaque espace tangent. Malheureusement, la géométrie convexe de Minkowski était encore distante d'au moins 40 ans et les normes ne provenant pas d'un produit scalaire n'étaient pas encore bien comprises. Riemann s'éloigna donc du cadre général, disant:
\begin{quotation}
\og Die Untersuchung dieser allgemeinern Gattung w\"urde zwar keine wesentlich andere Principien erfordern, aber ziemlich zeitraubend sein und verh\"altnissm\"assig auf die Lehre vom Raume wenig neues Licht werfen, zumal da sich die Resultate nicht geometrisch ausdr\"ucken lassen\fg \cite{Riemann}\\
{\it L'étude de ce cadre plus général ne nécessite pas de principes réellement différents, mais prendrait un temps considérable et n'apporterait peu ou pas d'éclairage nouveau sur la théorie de l'espace, d'autant plus que les résultats ne peuvent être exprimés géométriquement\footnote{Traduction par l'auteur}}
\end{quotation}
Il fallut donc attendre 1918 pour que, sous l'impulsion de Carathéodory, Paul Finsler étudie dans sa thèse \cite{Finsler} ce cas général et jette ainsi les bases de la théorie qui porte aujourd'hui son nom.

La géométrie de Finsler est désormais assez bien connue, mais nombre de questions résolues en géométrie riemannienne attendent encore un \'equivalent finslérien. Depuis plusieurs années, l'une des questions récurrentes en géométrie de Finsler est le développement d'outils d'analyse globale généralisant ceux de géométrie riemannienne.

L'opérateur de Laplace--Beltrami est sans doute le plus important de ces objets, et ce pour de multiples raisons. Son spectre en particulier joue un rôle essentiel; il est la preuve d'une connexion complexe et intrigante entre le laplacien et la géométrie de la variété qui le porte. Dans la première partie de cette thèse, j'introduis une généralisation de l'opérateur de Laplace--Beltrami aux métriques de Finsler. Ce n'est pas la première fois que ceci est fait: Bao et Lackey \cite{BaoLackey}, Centore \cite{Centore:FinslerLaplacians} et Shen \cite{Shen:non-linear_Laplacian} ont chacun proposé des généralisations. Il n'est pas très étonnant qu'il en soit ainsi: l'opérateur de Laplace--Beltrami admet plusieurs définitions équivalentes, mais les extensions de chaque définition peuvent donner des résultats bien différents dans le cadre finslérien. Nous espérons convaincre le lecteur que notre construction est naturelle et suffisamment simple pour pouvoir fournir des données spectrales.

\subsection*{Flot géodésique et forme de Hilbert}
Avant de continuer, il semble important de rappeler les notions utilisées ici.

%

Une métrique de Finsler sur $M$ sera pour nous la donnée d'une famille de normes \emph{fortement convexes} $F(x, \cdot)$ sur chaque espace tangent $T_x M$, variant de manière \emph{lisse}; par fortement convexe, on entend que le hessien de $F^2(x, \cdot)$ est défini positif. Nous ne nous restreignons cependant pas au cas des normes sym\'etriques, c'est-à-dire celles vérifiant $F(x,-v) = F(x,v)$. De manière équivalente, une métrique de Finsler est obtenue par la donnée d'une famille d'ouverts convexes (contenants mais pas n\'ecessairement centrés en $0$) sur chaque espace tangent, variant régulièrement et telle que le bord de chaque convexe soit à hessien défini positif: ce convexe-l\`a est la boule unit\'e de la norme. Notons que nous utilisons cette condition de forte convexité car elle permet par exemple d'obtenir une équation des géodésiques.

Cette notion de métrique de Finsler est sans doute la plus répandue, mais aussi la plus restrictive. Par exemple, les géométries de Hilbert, qui ont connu un regain d'int\'er\^et ces dernières années, ne sont en général que $C^0$ hors de la section nulle. Malheureusement, une régularité au moins $C^2$ est souvent nécessaire pour le calcul variationnel et sera indispensable dans cette thèse. Malgr\'e tout, le monde finslérien lisse reste bien plus vaste que le monde riemannien et riche en phénomènes nouveaux (voir \cite{Katok:KZ_metric} ou \cite{CNV} ou \cite{Alvarez:Survey} par exemple).

Parmi les métriques de Finsler, les métriques de Randers sont un cas particulièrement intéressant pour de multiples raisons. D'abord, elles sont obtenues en ajoutant une $1$-forme différentielle à une métrique riemannienne, et sont donc parmi les plus simples des métriques non-riemanniennes. Mais surtout, elles apparaissent naturellement dans de nombreux domaines des mathématiques ainsi qu'en physique ou en biologie \cite{AntonelliIngardenMatsumoto}. Elles constitueront donc une famille d'exemples privilégiés au cours de cette thèse.

L'approche que nous utilisons pour étudier la géométrie de Finsler est due à Patrick Foulon \cite{Fou:EquaDiff}. Il développa un certain nombre d'outils permettant l'étude de la géométrie de manière intrinsèque \emph{via sa dynamique} au lieu de l'étudier via des connexions et les calculs en coordonnées locales qui l'accompagnent.

L'objet essentiel ici sera donc le flot géodésique associé à $F$, qui vit naturellement sur le fibré en sphère unité $T^1 M$. Cependant, pour pouvoir étudier différents flots sans changer d'espace, il sera plus intéressant de voir les flots géodésiques dans le \emph{fibré homogène}
\begin{equation*}
 HM := \left( TM \smallsetminus \{0\} \right)/ \R^+.
\end{equation*}
Un autre élément incontournable pour notre étude est la \emph{forme de Hilbert} $A$, une $1$-forme différentielle sur $HM$ canoniquement associée à $F$.  Grâce à la forte convexité de $F$, $A$ est une \emph{forme de contact}: si $n$ est la dimension de $M$, $\ada$ est une forme volume sur $HM$.

Si l'on note $X$ le champ de vecteurs engendrant le flot géodésique de $F$, alors $X$ est le \emph{champ de Reeb} de $A$, c'est-à-dire qu'il est uniquement déterminé par les équations suivantes:
\begin{equation*}
\left\{ 
\begin{aligned}
  A(X) &= 1 \\
 i_X dA &= 0  \,.
 \end{aligned}
\right.
\end{equation*}

Toute notre étude est uniquement basée sur la forme de Hilbert $A$, le flot géodésique $X$ et des objets obtenus à partir de $A$ et $X$.

\subsection*{Le laplacien en géométrie de Finsler}

En géométrie riemannienne, il y a plusieurs définitions équivalentes pour l'opérateur de Laplace (voir \cite{GHL} par exemple): la plus évoluée est sans doute l'opérateur de Hodge--Laplace agissant sur les formes différentielles; pour l'opérateur de Laplace--Beltrami sur les fonctions, la plus synthétique est donnée par la divergence du gradient; vient enfin la plus ``basique'', la définition historique, qui exprime le laplacien d'une fonction en un point $p$ comme
\begin{equation*}
 \Delta f (p) = \sum_i \frac{\partial^2 f}{\partial x_i^2}(p),
\end{equation*}
où $x_i$ sont des coordonnées normales au point $p$ pour la métrique riemannienne considérée.\\

Notons que Bao et Lackey \cite{BaoLackey} développèrent une généralisation de la première définition et Shen \cite{Shen:non-linear_Laplacian} une généralisation de la seconde. Nous nous attachons à étendre la définition historique. Centore \cite{Centore:FinslerLaplacians} ne généralise pas vraiment la définition du laplacien, mais plutôt la propriété que les fonctions harmoniques vérifient la propriété de la moyenne. Nous montrerons au cours de cette thèse que tout ces opérateurs sont différents, mais on peut remarquer tout de suite que l'opérateur de Shen ne se situe pas dans la même catégorie. En effet, son opérateur est non-linéaire car la généralisation du gradient à la géométrie de Finsler n'est pas linéaire.

Le problème pour généraliser le laplacien historique est que la notion d'ortho\-go\-nalité n'est pas adaptée à la géométrie de Finsler. L'idée pour obtenir un opérateur de Laplace est alors de faire une moyenne des dérivées secondes dans \emph{toutes} les directions, au lieu de ne considérer que des directions orthogonales. Reste alors à définir ce que veut dire moyenne dans ce cas: nous devons introduire une mesure d'angle $\alpha^F$ en chaque point, associée naturellement à notre métrique de Finsler.

L'existence d'un angle en géométrie de Finsler n'est pas évident. Cependant, comme $A$ est une forme de contact, il existe un volume sur $HM$ canoniquement associé à une métrique de Finsler. De là, il n'y a qu'un pas pour définir un angle naturel, il suffit de considérer la famille de mesures conditionnelles sur les fibres $H_xM$ et de normaliser de telle manière que l'aire des fibres soit l'aire d'une sphère unité euclidienne:
\begin{prop*}
  Soient $(M,F)$ une variété de Finsler, $\pi \colon HM \rightarrow M$ la projection canonique et $VHM := \ker d\pi$ le fibré vertical. Il existe une unique forme volume $\Omega^F$\ sur $M$\ et une $(n-1)$-forme différentielle $\alpha^F$\ sur $HM$, ne s'annulant pas sur $VHM$, telle que
\begin{equation*}
  \alpha^{F} \wedge \pi^{\ast}\Omega^F =  A\wedge dA^{n-1}, 
\end{equation*}
et, pour tout $x\in M$, 
\begin{equation*}
 \int_{H_xM} \alpha^F =  \voleucl(\S^{n-1}).
\end{equation*}
\end{prop*}
 Nous pouvons alors définir notre opérateur de Finsler--Laplace: 
\begin{definfr*}
 Soit $(M,F)$ une variété de Finsler. Pour $f \in C^{2}(M)$ et $x \in M$, on pose
\begin{equation*}
  \Delta^F f (x) := \frac{n}{\voleucl \left(\mathbb{S}^{n-1}\right)} \int_{\xi \in H_xM} \left. \frac{d^2 f}{dt^2}\left(c_\xi (t) \right)\right|_{t=0}   \alpha^F_x(\xi),
\end{equation*}
où $c_{\xi}$ est la géodésique partant de $x$ dans la direction $\xi \in H_xM$.
\end{definfr*}
Et nous obtenons exactement ce que l'on voulait:
\begin{thmintrofr} \label{thmintrofr:Laplacian}
 Soit $(M,F)$ une variété de Finsler. L'op\'erateur $\Delta^F$\ est un opérateur différentiel du second ordre qui, de plus, v\'erifie:
\begin{enumerate}
 \item[(i)] $\Delta^F$\ est elliptique;
 \item[(ii)] $\Delta^F$\ est symétrique, i.e., pour $f,g \in C^{\infty}(M)$, 
 \begin{equation*}
  \int_M \left( f\Delta^F g - g\Delta^F f \right) \; \Omega^F = 0 \; ;
 \end{equation*}
 \item[(iii)] lorsque $F$ est une métrique riemannienne, $\Delta^F$\ co\"incide avec l'opérateur de Laplace--Beltrami.
\end{enumerate}
\end{thmintrofr}
Il est à noter que la preuve de ce résultat est très facile grâce à la définition de $\Delta^F$.\\
Nous rappelons aussi la conséquence suivante de {\it (i)} et {\it (ii)}: $\Delta^F$\ \emph{est unitairement équivalent à un opérateur de Schr\"odinger.}

La preuve de ce dernier point est essentiellement basée sur le résultat (connu) suivant, qui est particulièrement intéressant pour nous:
\begin{prop*}
Soient $(M,g)$ une variété riemannienne fermée et $\omega$ une forme volume sur $M$. Il existe alors un unique opérateur différentiel du second ordre à coefficients réels $\Delta_{g,\omega}$ sur $M$ qui est symétrique par rapport à $\omega$, nul sur les constantes et dont le symbole est donné par la métrique duale $g^{\ast}$.\\
De plus, si $a\in C^{\infty}(M)$ est telle que $\omega = a^2 v_g$, où $v_g$ est le volume riemannien, alors, pour $\varphi \in C^{\infty}(M)$,
\begin{equation*}
 \Delta_{g,\omega} \varphi = \Delta^{g} \varphi - \frac{1}{a^2} \langle \nabla \varphi , \nabla a^2 \rangle.
\end{equation*}
\end{prop*}

Rappelons que le symbole d'un opérateur elliptique du second ordre donne toujours une (co)-métrique riemannienne, donc en appliquant le résultat ci-dessus, nous pouvons remarquer que l'opérateur que nous introduisons est uniquement déterminé par son symbole et le volume $\Omega^F$. On se doit ici de faire quelques remarques.

 Premièrement, il y a beaucoup plus de métriques de Finsler que de couples (métrique riemannienne, forme volume), donc il y a forcément beaucoup de métriques de Finsler ayant le même laplacien. Cela peut-être pris soit comme un inconv\'enient, car notre opérateur n'est pas aussi fin que l'opérateur de Laplace--Beltrami, soit comme une belle source de questions: par exemple, les métriques ayant le même laplacien partagent-elles des propriétés géométriques ou dynamiques? Nous espérons que le lecteur préférera, comme nous, cette seconde interprétation.

Les opérateurs $\Delta_{g,\omega}$ sont appelés laplaciens à poids et furent introduits par Chavel et Feldman \cite{ChavelFeldman:Isoperimetric_constants} et par Davies \cite{Davies:Heat_kernel_bounds}. Ils ont depuis été beaucoup étudiés (voir par exemple \cite{Grigoryan:heat_kernels_on_weighted_manifolds} ou \cite{ColboisSavo}) et il est naturel de se demander s'il serait parfois intéressant de passer par la géométrie de Finsler pour les étudier ou inversement d'étudier notre opérateur à travers eux.

Ceci nous amène à notre dernière remarque: peut-on obtenir tous les laplaciens à poids via une métrique de Finsler? Comme nous avons déjà vu qu'il y a beaucoup plus de métriques de Finsler que de couples (métrique riemannienne, forme volume), nous sommes tentés de conjecturer que la réponse est positive.

 Nous montrons d'ailleurs que c'est le cas pour les surfaces et qu'il suffit alors de considérer des métriques de Randers. Malheureusement, notre preuve est basée sur l'expression locale que nous obtenons pour les surfaces de Randers, et il est donc impossible de la généraliser en l'état.

\subsection*{Spectre et \'Energie}

Comme attendu d'un laplacien qui se respecte, l'opérateur de Finsler--Laplace sur les variétés compactes admet un spectre discret. Nous donnons les deux preuves classiques de ce résultat. La première est juste l'application de la théorie générale des opérateurs linéaires \emph{non bornés}, elliptiques et symétriques. La seconde approche, par la méthode du Min-Max, est plus intéressante pour une future étude du spectre; elle repose sur l'introduction d'une énergie associée à l'opérateur, soit
\begin{equation*}
 E(u) := \frac{n}{\voleucl \left(\S^{n-1}\right) } \int_{HM} \left|L_X\left(\pi^{\ast}u \right)\right|^2 \ada.
\end{equation*}
Avec cette énergie, nous généralisons plusieurs résultats riemanniens. Premièrement, nous montrons que les fonctions harmoniques sont obtenues comme minima de l'énergie (ce résultat est d'ailleurs une étape essentielle dans la méthode du Min-Max). Deuxièmement, nous étudions comment varie l'énergie lorsque l'on considère une classe conforme de métriques de Finsler. Nous montrons:
\begin{prop*}
 Soient $(M,F)$\ une variété de Finsler de dimension $n$, ${f \colon M \xrightarrow{C^{\infty}} \R}$\ et $F_f = e^f F$. Soit $E_f$\ l'énergie associée à $F_f$. Alors, pour tout $u\in H^1\left(M\right)$,
\begin{equation*}
 E_f(u) = \frac{n}{\voleucl \left(\S^{n-1}\right) } \int_{HM} e^{(n-2)f} \left(L_X \pi^{\ast}u \right)^2 \, \ada.
\end{equation*}
En particulier, lorsque $n=2$, l'énergie est un invariant conforme.
\end{prop*}
Et en déduisons une généralisation du résultat riemannien classique:
\begin{thmintrofr}
\label{thmintrofr:inv_conforme}
Soient $(\Sigma,F)$\ une surface de Finsler, $f \colon \Sigma \xrightarrow{C^{\infty}} \R$\ et $F_f = e^f F$. Alors
\begin{equation*}
 \Delta^{F_f} = e^{-2f} \Delta^F.
\end{equation*}

\end{thmintrofr}

\subsection*{Représentation locale et calcul de spectres}

Un nouvel opérateur, aussi naturel soit-il, n'est intéressant que s'il est étu\-diable. Pour nous, cela veut dire que l'on doit être en mesure de produire des exemples purement finslériens où l'on arrive à d\'eterminer un spectre. La méthode du Min-Max mentionnée plus haut est bien utile d'un point de vue théorique, ou pour obtenir des majorations, mais pour obtenir une formule explicite du spectre, et en l'absence de beaucoup de symétries,la méthode éprouvée reste le calcul en coordonnées. Mais avant de se lancer tête baissée dans des exemples, il faut d'abord déterminer lesquels ont une chance d'aboutir.

Même pour l'opérateur de Laplace--Beltrami, calculer le spectre est une tâche extrêmement ardue; les seuls exemples complètement connus sont les espaces modèles $\Hyp^n$, $\R^n$ et $\S^n$ ainsi que certains de leurs quotients. Or, il n'y a pas d'espaces modèles en géométrie de Finsler: en effet, pour tout $R$, il existe une infinité de métriques finslériennes \emph{non isométriques} à courbure constante $R$. Nous avons donc étudié, parmi ces métriques à courbure constante, celles qui nous apparaissent intéressantes et abordables.

En courbure négative, un théorème d'Akbar-Zadeh \cite{AkZ} (voir aussi \cite{Fou:Curvature_and_global_rigidity} pour le cas localement symétrique) montre que, sur les variétés fermées, les seules structures ayant une courbure constante sont en fait riemanniennes.

En courbure positive, Akbar-Zadeh \cite{AkZ} prouva que le revêtement universel était toujours une sphère (voir aussi \cite{Rademacher:Sphere_theorem} pour la preuve en courbure pincée). D'un point de vue métrique, les choses deviennent plus intéressantes: Bryant \cite{Bryant:FS2S,Bryant:PFF2S} construisit une famille à deux paramètres de métriques (projectivement plates) à courbure constante égale à 1 sur $\S^2$. Bien avant cela, Katok \cite{Katok:KZ_metric} avait construit une famille à un paramètre de déformation de la métrique standard sur la sphère. Ziller \cite{Ziller:GKE} étudia la géométrie de ces exemples et les généralisa à d'autres espaces. L'intérêt majeur de ces exemples, et la raison même de leur invention, est que l'on obtient des métriques (de Finsler) sur $\S^n$ aussi proches de la métrique standard que l'on veut, mais n'ayant qu'un nombre \emph{fini} d'orbites périodiques. Rademacher \cite{Rademacher:Sphere_theorem} prouva que les métriques de Katok--Ziller sur la sphère sont aussi à courbure constante égale à 1 en remarquant qu'elles coïncidaient avec des exemples qu'avait indépendamment construit Shen \cite{Shen:Finsler_metric_constant_curvature}. Notons que P. Foulon \cite{Fou:perso} a une preuve du fait que \emph{toutes} les constructions de Katok--Ziller sont à courbure constante, quel que soit l'espace.

Ces métriques sont pour nous parfaites: elles admettent une écriture explicite assez facile (voir \cite{Rademacher:Sphere_theorem} pour la sphère et la proposition \ref{prop:expression_KZ} en général), ce qui est bien utile pour nos calculs, et sont extrêmement intéressantes d'un point de vue dynamique. Nous obtenons une approximation du spectre ainsi que:
\begin{thmintrofr} \label{thmintrofr:lambda1}
 Soit $F_{\eps}$ la famille de métriques de Katok--Ziller sur $\S^2$. La plus petite valeur propre non nulle de $-\Delta^{F_{\eps}}$ est
\begin{equation*}
 \lambda_1(\eps) = 2 -2 \eps^2 = \frac{8 \pi}{\rm{vol}_{\Omega^{F_{\eps}}}\left(\S^2\right)}.
\end{equation*}
\end{thmintrofr}
Ce résultat, en plus de prouver qu'il est possible d'obtenir explicitement des valeurs propres pour des métriques de Finsler, exhibe un phénomène nouveau. Pour les métriques riemanniennes sur la sphère, Hersch \cite{Hersch} démontra que le bas du spectre était majoré par $8\pi/\text{vol}(\S^2)$, l'égalité n'étant réalisée que pour la métrique standard. Ici, nous avons un exemple d'une famille à un paramètre réalisant le maximum riemannien. Nous ne savons malheureusement pas encore si c'est aussi un maximum dans le cas des métriques de Finsler. Notons que la preuve de Hersch repose sur le fait qu'il n'y a qu'une seule classe conforme de métriques riemanniennes sur la sphère et ne peut donc malheureusement pas se généraliser telle quelle.

Les métriques de Katok--Ziller peuvent aussi être construites sur les tores et nous calculons le spectre pour ces exemples. Les opérateurs obtenus ne sont toutefois pas nouveaux dans ces cas-là. En effet, nous montrons que pour toute métrique de Finsler \emph{localement minkowskienne} sur un tore, c'est-à-dire telle que la norme de Finsler ne dépend (localement) pas du point sur la variété, l'opérateur de Finsler--Laplace associé est donné par le laplacien de son symbole. Le calcul du symbole suffit donc pour déterminer le spectre. Il est cependant intéressant d'étudier ces exemples car ils permettent de comprendre certaines limites de notre opérateur; par exemple, nous ne pourrons pas espérer avoir un lien aussi direct qu'en géométrie riemannienne entre le spectre des longueurs et le spectre du laplacien. En outre, s'il on veut un jour obtenir des bornes topologiques pour le spectre, il serait bon de commencer par ces exemples.

\subsection*{Courbure négative, spectre et géométrie à l'infini}

La courbure négative est sans doute la plus intéressante d'un point de vue de l'interaction entre géométrie, dynamique, théorie ergodique et spectre du laplacien.

Soit $M$ une variété riemannienne fermée à courbure négative et $\M$ son revêtement universel. La variété $\M$ admet une frontière à l'infini $\M(\infty)$, supportant naturellement trois classes de mesures:
\begin{itemize}
 \item les mesures de Liouville, qui sont obtenues en projetant la mesure de Lebesgue à l'infini le long de rayons géodésiques;
 \item les mesures de Patterson--Sullivan $\nu_x$;
 \item les mesures harmoniques $\mu_x$, qui sont les mesures solutions du problème de Dirichlet à l'infini. Plus précisément, si $f \in C^0(\M(\infty))$, alors
\begin{equation*}
 u(x) := \int_{\xi \in \M(\infty)} f(\xi) d\mu_x(\xi)
\end{equation*}
est l'unique fonction vérifiant $\Delta^{\F} u = 0$ sur $\M$ et $u(x) \rightarrow f(\xi)$ quand $x \rightarrow \xi \in \M(\infty)$.
\end{itemize}

Kaimanovitch \cite{Kaimanovich} démontra qu'il existe un isomorphisme convexe entre le cône des mesures de Radon sur $\partial^2 \M  := \M(\infty) \times \M(\infty) \smallsetminus\{ \textrm{diag} \} $ invariantes par $\pi_1(M)$ et le cône des mesures de Radon sur $HM$ invariante par le flot géodésique. Via ce résultat on peut obtenir les mesures de Patterson--Sullivan à partir de la mesure de Bowen--Margulis.

Il est aussi connu que chaque classe de mesures ci-dessus est ergodique par rapport à l'action du groupe fondamental.

Si la courbure est constante, ces trois classes de mesures coïncident. Des résultats de Katok \cite{Katok:4_appl} et Ledrappier \cite{Ledrappier:Harm_measures} montrent que, en dimension 2, dès que deux mesures coïncident, la métrique doit être à courbure constante. Ledrappier obtient ce résultat de rigidité comme corollaire du résultat suivant: en dimension quelconque, $\nu_x = \mu_x$ si et seulement si $\lambda_1 = h^2/4$, où $\lambda_1$ est le bas du spectre du laplacien et $h$ est l'entropie topologique du flot géodésique. Dans ce cas, G.\ Besson, G.\ Courtois et S.\ Gallot \cite{BessonCourtoisGallot} prouvèrent que la variété est un espace symétrique.

Il est tentant de chercher à adapter ces résultats à notre cas. En effet, si $F$ est une métrique à courbure négative sur une variété compacte $M$, alors le flot géodésique est encore Anosov (\cite{Fou:EESLSPC}) et si, de plus, $F$ est réversible, le revêtement universel est Gromov-hyperbolique (\cite{Egloff:UFHM}). $\M$ admet donc aussi une frontière visuelle et les mesures de Liouville et de Patterson-Sullivan sont définies de la même manière. Je me suis donc intéressé à l'existence des mesures harmoniques pour l'opérateur de Finsler--Laplace.

Une méthode générale pour obtenir des mesures harmoniques est via la théorie du potentiel. On construit la frontière de Martin pour le couple $(\M, \Delta^{\wt{F}})$ puis on montre que cette frontière est homéomorphe à la frontière visuelle. Ancona \cite{Ancona:theorie_potentiel} prouva un théorème très général identifiant la frontière de Martin à la frontière visuelle, pour une large classe d'opérateurs elliptiques sur une variété Gromov-hyperbolique. Nous montrons que son théorème s'applique à notre opérateur, et en déduisons l'existence de mesures harmoniques. En fait, nous montrons que l'homéomorphisme entre la frontière de Martin et la frontière visuelle est H\"older régulier (en généralisant \cite{AndersonSchoen}). Ceci nous permet d'utiliser les travaux de Ledrappier \cite{Ledrappier:Ergodic_properties} pour montrer que les mesures harmoniques ont toujours les mêmes propriétés ergodiques:
\begin{thmintrofr}
Soient $(M,F)$ une vari\'et\'e de Finsler fermée à courbure négative et $\mu_x$ la famille de mesures harmoniques sur $\M(\infty)$. Les propriétés suivantes sont vérifiées:
\begin{itemize}
 \item[(i)] la classe des mesures harmoniques $\lbrace \mu_x \rbrace$ est ergodique pour l'action de $\pi_1(M)$ sur $\M(\infty)$;
 \item[(ii)] pour tout $x\in \M$, il existe un poids $f \colon \partial^2 \M \rightarrow \R$ tel que la mesure $f \mu_x \otimes \mu_x$ est ergodique pour l'action de $\pi_1(M)$ sur $\partial^2\M$;
 \item[(iii)] il existe une unique mesure $\mu$ sur $HM$, ergodique pour le flot géodésique, telle que son image par la correspondence de Kaimanovitch soit $f \mu_x \otimes \mu_x$.
\end{itemize}
\end{thmintrofr}
Notons que dans le {\it (i)}, les mesures $\mu_x$ sont seulement quasi-invariantes, mais que l'on peut définir ``ergodique'' de la même manière.

Les résultats de rigidité, eux, restent encore hors de portée.

Finalement, notons aussi que nous obtenons une borne supérieure pour le bas du spectre en courbure négative:
\begin{thmintrofr} \label{thmintrofr:borne_dynamique}
 Soient $(M,F)$ une vari\'et\'e de Finsler fermée à courbure négative, $\Delta^{\wt{F}}$ l'opérateur de Finsler--Laplace sur le revêtement universel de $M$ et $\lambda_1$ le bas du spectre de $-\Delta^{\wt{F}}$. On a
\begin{equation*}
 \lambda_1 \leq \frac{nh^2}{4},
\end{equation*}
où $n$ est la dimension de $M$ et $h$ l'entropie topologique du flot g\'eod\'esique de $(M,F)$.
\end{thmintrofr} 
Nous obtenons aussi une borne inférieure topologique, et sans condition de courbure, grâce à une généralisation d'un résultat de Brooks \cite{Brooks:pi1_and_spectrum}:
\begin{thmintrofr} \label{thmintrofr:Brooks}
 Soient $(M,F)$ est une vari\'et\'e de Finsler compacte et $\lambda_1$ le bas du spectre de $-\Delta^{\wt{F}}$. Alors,
\begin{equation*}
 \lambda_1 = 0 \quad \text{si et seulement si} \quad \pi_1(M) \: \text{ est moyennable}
\end{equation*}

\end{thmintrofr}

\subsection*{Angle et co-angle: la géométrie de Finsler et son pendant hamiltonien}

Notre généralisation de l'opérateur de Laplace repose entièrement sur la définition de l'angle $\alpha^F$, nous nous sommes donc attachés à l'étudier un minimum. Par exemple, dans le cas des surfaces, nous montrons comment l'on peut reconnaître qu'une métrique est riemannienne à son angle.

Rappelons que l'angle $\alpha^F$ est obtenu en décomposant le volume canonique $\ada$ sur $HM$ entre une partie verticale et une partie provenant de la base. Mais en considérant le côté symplectique, on voit qu'il existe un autre volume canoniquement associé à $F$.

En effet, il est maintenant bien connu qu'une métrique de Finsler ${F\colon TM \rightarrow \R}$ détermine uniquement un hamiltonien $F^{\ast} \colon T^{\ast}M \rightarrow \R$ et que l'on passe de l'un à l'autre via la transformée de Legendre. La variété $T^{\ast}M$ est symplectique et admet donc un volume canonique $d\lambda^n$, où $\lambda$ est la $1$-forme de Liouville. De plus, de la même manière que l'on obtient la forme de Hilbert $A$, on peut obtenir une $1$-forme $B$ sur $H^{\ast}M$ associée à $F^{\ast}$, où $H^{\ast}M$ est l'homogénéisé de $T^{\ast}M$. La forme $B$ est aussi de contact et nous avons donc une forme volume $\bdb$ sur $H^{\ast}M$.

Heureusement, tous ces volumes sont liés assez simplement. Si l'on note $\hat{r}\colon \mathring{T}^{\ast} M \rightarrow H^{\ast}M$ la projection canonique, alors $\hat{r}^{\ast}B = \lambda/F^{\ast}$ et donc \\$\hat{r}^{\ast}\bdb = \lambda \wedge d\lambda^{n-1} / (F^{\ast})^n$. Enfin, la transformée de Legendre permet de passer de $\ada$ à $\bdb$ (voir section \ref{sec:cotangent}).

De la même manière que l'on a obtenu $\alpha^F$, on obtient un co-angle $\beta^{F^{\ast}}$ en décomposant le volume $\bdb$, et la transformée de Legendre permet de passer de $\alpha^F$ à $\beta^{F^{\ast}}$.
	  
Considérer le côté symplectique est souvent intéressant pour l'étude de la géométrie de Finsler. Cela nous permet de remarquer que le volume $\Omega^F$ apparaissant dans la décomposition de $\ada$ est le volume de Holmes-Thompson. Plus surprenant, nous montrons aussi qu'une métrique de Finsler est uniquement déterminée par son co-angle et son volume.

\section*{Flots d'Anosov alignables en biais}

Anosov \cite{Anosov} parvint à extraire des flots géodésiques en courbure négative les conditions minimales nécessaires à leur comportement hyperbolique. Depuis, les flots d'Anosov forment une source virtuellement inépuisable de questions en théorie des systèmes dynamiques.

J'étudie, dans la seconde partie de ma thèse, des questions topologiques liées aux orbites périodiques de certains flots d'Anosov en dimension trois.

Mes travaux se basent sur l'étude, initiée par Thierry Barbot et Sergio Fenley, de la géométrie transverse des flots d'Anosov. Les principaux objets d'intérêt sont alors l'espace des orbites ainsi que les espaces des feuilles: si $\flot$ est un flot d'Anosov sur une $3$-variété compacte $M$, l'espace des orbites est défini comme le revêtement universel $\M$ quotienté par la relation ``être sur la même orbite'' et l'espace des feuilles (in)stables comme $\M$ quotienté par la relation ``être sur la même feuille (in)stable''.

Si l'espace des orbites est toujours homéomorphe à $\R^2$, les espaces des feuilles ne sont en général pas Hausdorff. Un flot d'Anosov est dit \emph{alignable} si ses espaces des feuilles sont homéomorphes à $\R$, ce qui en fait des flots topologiquement sympathiques. Barbot prouva que, s'il existe un relevé d'une feuille stable intersectant toutes les feuilles instables, alors un flot alignable est topologiquement conjugué à une suspension d'un difféomorphisme d'Anosov. Les autres flots sont appelés flots alignables \emph{en biais} et sont l'objet de notre étude.

Il est assez facile de remarquer que le flot géodésique d'une surface à courbure négative est un flot alignable en biais, et c'est d'ailleurs topologiquement le seul sur les fibrés de Seifert (à rev\^etement fini près, voir \cite{BarbotFenley}). Mais, lorsque l'on considère des $3$-variétés plus générales, de nombreux exemples non algébriques existent. (Les flots géodésiques ainsi que les suspensions de difféomorphisme d'Anosov sont appelés flots d'Anosov algébriques car ils sont topologiquement conjugués à l'action d'un groupe à un paramètre sur un quotient du type $\Gamma \backslash G \slash K$, où $G$ est un groupe de Lie, $K$ un sous-groupe de Lie compact et $\Gamma$ un sous-groupe discret agissant de manière cocompact sur $G \slash K$, voir \cite{Tomter}).

Fenley \cite{Fen:AFM} fut le premier à construire de tels exemples, mais remarquons que la construction de Foulon et Hasselblatt \cite{FouHassel:contact_anosov} en produit aussi.

\subsection*{Cas des variétés hyperboliques}
Lorsque la variété est atoro\"idale et n'est pas un fibré de Seifert, ces flots ont une particularité remarquable: chaque orbite périodique est librement homotope à une infinité d'autres orbites. Ceci est aux antipodes du flot géodésique sur une surface hyperbolique où il n'y a au plus qu'une orbite périodique dans une classe d'homotopie libre. (Une telle géodésique existe pour toute classe d'homotopie libre dans $\pi_1(\Sigma)$, mais pas pour toute classe dans $\pi_1(H\Sigma)$).	

Nous nous sommes intéressés à la question suivante: {\it étant donnée une classe d'homotopie libre d'orbites périodiques, que peut-on dire de leurs classes d'isotopie?} Pour souligner l'intérêt de cette question, nous nous permettons de la paraphraser: une classe d'homotopie libre donne une collection infinie de plongements homotopiquement équivalents de $S^1$, c'est-à-dire une collection de noeuds, dans une $3$-variété; {\it ces noeuds sont-il différents au sens de la théorie des noeuds classique?}

Notons que les questions sur les types de noeuds formés par des orbites périodiques de flots ne sont pas nouvelles, mais très intéressantes (voir par exemple les travaux de Ghys \cite{Ghys:knots_and_dynamics}).

En collaboration avec Sergio Fenley, nous prouvons que toutes les orbites homotopes sont en fait aussi isotopes. Même si la réponse peut être un peu décevante, la manière d'obtenir cette isotopie est intéressante: on pousse chaque orbite par un flot pseudo-Anosov obtenu, grâce aux travaux de Thurston \cite{Thurston:3MFC}, par la géométrie des feuilletages stables et instables.

\subsection*{Cylindres plongés et orbites périodiques}

Une isotopie entre deux orbites créant un cylindre immergé, on peut s'intéresser à la question de savoir quand apparaît un cylindre \emph{plongé} entre deux orbites. En adaptant des résultats de Barbot \cite{Bar:MPOT,BarbotFenley}, on peut montrer que cette question dépend essentiellement de l'action du groupe fondamental sur l'espace des orbites, et on peut à nouveau se ramener à l'utilisation des objets définis par Thurston sur les feuilletages.

Nous montrons qu'il existe des orbites périodiques ne pouvant être reliées à aucune autre par un cylindre plongé. Dans une prochaine publication avec S. Fenley, nous montrerons qu'en général, il n'y a qu'un nombre fini d'orbites pouvant être reliées de cette manière dans une classe d'homotopie libre.

\section*{Structure de la thèse}

Le chapitre 1 introduit d'abord les notions nécessaires à l'étude de la géométrie de Finsler via sa dynamique, et rappelle un certain nombre de résultats. Nous développons aussi le point de vue symplectique et la transformée de Legendre. Le lecteur averti peut commencer la lecture directement à la section \ref{sec:angle_and_volume_in_Finsler_geometry}, dans laquelle nous introduisons l'angle $\alpha^F$ (proposition \ref{prop:construction}) ainsi que son pendant dans l'espace cotangent $\beta^{F^{\ast}}$ (proposition \ref{prop:coangle}). Puis nous étudions quelques propriétés de ces angles et remarquons que le volume $\Omega^F$ est le volume de Holmes-Thompson (section \ref{sec:holmes_thompson}).\\

Le chapitre 2 est le c\oe ur de ce travail. Nous y introduisons l'opérateur de Finsler--Laplace (définition \ref{def:delta}) et y prouvons le théorème \ref{thmintrofr:Laplacian}. Nous introduisons ensuite l'énergie associée (définition \ref{def:energy}) et montrons que le spectre peut être obtenu grâce à l'énergie par la méthode du Min-Max (théorèmes \ref{thm:Min-Max} et \ref{thm:Min-Max_bis}). Enfin, nous utilisons l'énergie pour généraliser la preuve que l'opérateur de Laplace est, à un poids près, un invariant conforme en dimension $2$ (théorème \ref{thm:inv_conforme}).\\

Le chapitre 3 se concentre sur des exemples. Nous montrons d'abord que l'on peut obtenir relativement aisément une expression en coordonnées locales de l'opérateur de Finsler--Laplace pour les surfaces de Randers (proposition \ref{prop:symbol_laplacian_randers}). Cela nous permet d'ailleurs de montrer que tous les laplaciens à poids sur les surfaces peuvent être obtenus comme le Finsler--laplacien d'une métrique de Randers (proposition \ref{prop:Randers_donne_couple}). Nous étudions ensuite le cas des métriques de Katok--Ziller sur le tore et la sphère, exhibant leur spectre dans le premier cas (théorème \ref{thm:KZ_torus}) et une approximation de celui-ci dans le second ainsi que le théorème \ref{thmintrofr:lambda1} (théorème \ref{thm:spectre_S2} et corollaire \ref{cor:lambda1}).\\

Le chapitre 4 s'attaque au lien entre le laplacien et le flot géodésique en courbure négative. Nous prouvons les théorèmes \ref{thmintrofr:borne_dynamique}  (proposition \ref{prop:upper_bound}) et \ref{thmintrofr:Brooks} (théorème \ref{thm:Brooks}). Nous introduisons ensuite les outils nécessaires à l'énoncé du théorème d'Ancona sur l'homéomorphisme entre la frontière de Martin et la frontière visuelle et prouvons que cette homéomorphisme est H\"older régulier (théorème \ref{thm:C_alpha_extension}). Nous montrons ensuite que nous pouvons appliquer le théorème d'Ancona à notre opérateur (théorème \ref{thm:Finsler-Laplace_is_Ancona}) et en déduisons l'existence de solutions au problème de Dirichlet à l'infini (corollaire \ref{cor:harmonic_exists}). Enfin, nous généralisons la preuve de Ledrappier \cite{Ledrappier:Ergodic_properties} sur les propriétés ergodiques des mesures harmoniques (théorème \ref{thm:harmonic_measures_are_ergodic}).\\

Le chapitre 5 constitue la seconde partie de cette thèse. Nous commençons par rappeler les résultats de Barbot et Fenley dont nous avons besoin sur les flots d'Anosov, puis les travaux de Thurston, Calegari et Fenley sur les feuilletages. Nous montrons ensuite que toutes les orbites homotopes sont isotopes (théorème \ref{thm:homo_implies_isotope}). Nous terminons par l'étude des classes ``co-cylindriques'' et leurs liens avec l'action du groupe fondamental sur le cercle universel associé.

}

\selectlanguage{american}
\mainmatter

\part{A dynamical Laplace operator in Finsler Geometry}

\chapter[Dynamical Formalism]{Foulon's dynamical formalism for Finsler geometry and some consequences}

\label{chap:dynamical_formalism}

\section{Definitions} \label{sec:Definitions}
Here we explain the formalism that we will use in this dissertation regarding Finsler geometry. In that respect, we follow the work introduced by Patrick Foulon in \cite{Fou:EquaDiff}.\\

\subsection{Notations}

In this dissertation, we only consider \emph{orientable} manifolds.
For the readers' convenience, we will start with a list of notations that will be used throughout this text but defined later:\\

If $M$ is a manifold, $TM$ is its tangent bundle and $T^{\ast}M$ its cotangent bundle. $\mathring{T}M$ (resp. $\mathring{T}^{\ast}M$) is the tangent (resp. cotangent) bundle minus the zero section. The bundles $HM$ and $H^{\ast}M$ are the (co)-homogenized bundles.\\
$p_{|M} \colon TM \rightarrow M$\\
$\hat{p}_{|M} \colon T^{\ast}M \rightarrow M$\\
$r \colon \mathring{T}M \rightarrow HM$\\
$\hat{r} \colon \mathring{T}^{\ast}M \rightarrow H^{\ast}M$\\
$\pi \colon HM \rightarrow M$\\
$\hat{\pi} \colon H^{\ast}M \rightarrow M$\\
$\L_F \colon \mathring{T}M \rightarrow \mathring{T^{\ast}}M$ the Legendre transform.\\
$\ell_F \colon HM \rightarrow H^{\ast}M$ the Legendre transform on the homogeneous bundles.

If $\alpha$ is a $p$-form on a manifold (where $p$ can be equal to $0$), then $d\alpha$ is the exterior derivative of $\alpha$. Otherwise, if $f\colon M \rightarrow N$ is a map between manifolds, $df \colon TM \rightarrow TN$ is the differential of the map.\\
$VHM = \ker d\pi$.

$L_Z$ stands for the Lie derivative of the vector field $Z$.

In all of this dissertation, a $\ast$ in superscript (resp. subscript) of a map will mean the pull-back (resp. the push-forward) of the following object.

\subsection{Finsler metric and geodesic flow} \label{subsec:Finsler_metric_and_geodesic_flow}

Let $M$ be a smooth manifold of dimension $n$, there are several definitions of a Finsler metric in the literature. We will use what is probably the most common and, unfortunately, the most restrictive of smooth, strongly convex Finsler structure:
\begin{defin}\label{defin:Finsler_metric}
A smooth Finsler metric on $M$\ is a continuous function $F \colon TM \rightarrow \R^+$ that is:

\begin{enumerate}
  \item $C^{\infty}$\ except on the zero section,
  \item positively homogeneous, i.e., $F(x,\lambda v)=\lambda F(x,v)$\ for any $\lambda>0$,
  \item positive-definite, i.e., $F(x,v)\geq0$\ with equality iff $v=0$,
  \item strongly convex, i.e., $ \left(\dfrac{\partial^2 F^2}{\partial v_i \partial v_j}\right)_{i,j}$ is positive-definite.
 \end{enumerate}
If, for any $(x,v)\in TM$, $F(x, -v ) = F(x,v)$, then we say that $F$ is \emph{reversible}.
\end{defin}

A Finsler metric can be thought of as a smooth family $\lbrace I_x \rbrace$ of strongly convex sets containing $0$ in each tangent space $T_xM$, or equivalently as a smooth family of Minkowski norms. If the convex sets are always ellipsoids centered at $0$ then the Finsler metric is called Riemannian.

Some examples of Finsler metrics include for instance smooth and strongly convex deformations of Riemannian metrics. In \cite{Randers}, G. Randers introduced the following: take $g$ a Riemannian metric on $M$, $\theta$ a $1$-form on $M$ and define $F := \sqrt{g} + \theta$. If $\theta$ has norm less than $1$, then $F$ is a Finsler metric (see \cite{BaoChernShen}). Randers metrics are an important particular example of Finsler structures: they are the simplest kind of non-Riemannian metrics, arise a lot in physics and have been widely studied. Note that Randers metrics are never reversible.

Under our conditions, it can be shown that $F$ defines a distance on $M$: for $x,y \in M$,
\begin{equation*}
 d(x,y) = \inf_c \int_0^1 F(c(t), \dot{c}(t) ) dt ,
\end{equation*}
where $c$ runs over all $C^1$-by-part paths such that $c(0) = x$ and $c(1)=y$. Note that, for a non-reversible Finsler metric, the distance function will not be symmetric.

In the definition, the positive homogeneity condition is necessary to ensure that changing the parametrization of a curve does not change its length and the third point assures us that a constant path does not have positive length. The last requirement is not as self-explanatory as the others, but asking for $F$ to be convex in the second variable implies that the length structure is lower semi-continuous which, in turns, implies that the length of a rectifiable path computed with $d$ is the same as its ``integral length'' computed with the formula above (see \cite{BuragoBuragoIvanov}).  Using this distance, we can define geodesics as curves that locally minimize the distance and hence a geodesic flow on the unit tangent bundle. Note that our stronger assumption for convexity is necessary for our purpose, for instance to obtain geodesics as solutions of a second-order differential equation, or equivalently, so that the Hilbert form $A$, defined below, is a contact form.

The right place to study dynamical objects (by this, we mean objects linked to the geodesic flow) seems to be the unit tangent bundle. However, in order to study flows associated to different Finsler metrics without having to change the space, we study everything on the \emph{homogenized tangent bundle}
$$
HM =   \mathring{T}M / \R^+_{\ast}.
$$
We write $r \colon \mathring{T}M \rightarrow HM$ and $\pi \colon HM \rightarrow M$ for the canonical projections.\\
Remark that the fibers of $\pi$ defines a canonical distribution on $THM$, called the \emph{vertical distribution} $VHM$. It is the set of vectors in $THM$ that are tangent to a fiber of $\pi$, or equivalently, $VHM:= \text{Ker} \; d\pi$.

\subsubsection{Hilbert form}

To a Finsler metric $F$, we can canonically associate a $1$-form $A$ on $HM$, called the \emph{Hilbert form}, in the following manner: for $(x,\xi) \in HM$ and $Z \in T_{(x,\xi)}HM$, choose $v \in T_xM$ such that $r(x,v) = (x,\xi)$ and set
\begin{equation}
 A_{(x,\xi)}(Z) := \lim_{\eps \rightarrow 0} \frac{F\left(x, v + \eps d\pi(Z) \right) - F\left( x,v \right)}{\eps}.
\end{equation}
The homogeneity of $F$ implies that the definition of $A$ does not depend on the choice of $v$. Note that we can also define the \emph{vertical derivative} of $F$: it is the $1$-form $d_vF \colon TM \rightarrow T^{\ast}TM$ such that, for any $(x,v) \in TM$ and $Z \in T_{(x,v)}TM$,
\begin{equation*}
 d_vF_{(x,v)}(Z) := \lim_{\eps \rightarrow 0} \frac{F\left(x, v + \eps dp(Z) \right) - F\left( x,v \right)}{\eps}.
\end{equation*}
And we clearly have 
$$
d_vF = r^{\ast} A.
$$
 
\begin{thm} [Hilbert, \dots, Foulon \cite{Fou:EquaDiff}]
 The form $A$ is a contact form, i.e., $\ada$ is a volume form on $M$. Furthermore, if $X$ denotes its Reeb field, then $X$ generates the geodesic flow for $F$.
\end{thm}

Recall that a Reeb field is uniquely determined by the following two equations:
\begin{equation}
\label{eq:Reeb_field}
\left\{ 
\begin{aligned}
  A(X) &= 1 \\
 i_X dA &= 0  \, .
 \end{aligned}
\right.
\end{equation}

\begin{proof}[Sketch of proof]
 Proving that $A$ is contact can be done in local coordinates, using that $F$ is strongly convex and $1$-homogeneous, via the Euler formula. To prove the affirmation about the Reeb field, we just have to remark that Equation \ref{eq:Reeb_field} is a very nice way to write the Euler-Lagrange equations.
\end{proof}

\begin{rem}
This implies that the canonical volume $\ada$ is invariant by the flow, i.e.,
\begin{equation}
\label{eq:lxada}
 L_X\left(\ada\right) = 0.
\end{equation}
\end{rem}

The Reeb field $X$ is a second order differential equation, as defined in \cite{Fou:EquaDiff}, i.e.,
\begin{defin}
 A vector field $X$ on $HM$ is called a \emph{second order differential equation} if the following diagram commutes
$$
\xymatrix{
 THM \ar[r]^{d\pi}  & TM \ar[r]^{r} &   HM \ar[lld]^{\id} & \\
 HM \ar[u]^{X} 
}
$$
\end{defin}

There is an easy and very useful lemma about second order differential equations:
\begin{lem} \label{lem:second_order_ode}
 If $X_1,X_2$ are two second order differential equations on $HM$, then there exists a function $m \colon HM \rightarrow \R$ and a vertical vector field $Y \in VHM$ such that 
\begin{equation}
 X_2 = m X_1 + Y.
\end{equation}
\end{lem}

\begin{proof}
 By definition, $r \circ d\pi \circ X_1=r \circ d\pi \circ X_2$, so there exists $m \colon HM \rightarrow \R$ such that $d\pi \circ X_1=m d\pi \circ X_2$, so $X_1 - m X_2 \in VHM$. 
\end{proof}

\subsection{Dynamical derivative and Jacobi endomorphism}

Foulon defined many objects associated with a geodesic flow, hence generalizing some Riemannian concepts, like curvature, in a purely intrinsic way, without requiring any of the connections in Finsler geometry. However, we do not wish to introduce those objects at length (or at all), because it is already done (both in French and English) in the following articles: \cite{Fou:EquaDiff,Fou:EESLSPC,Fou:LSFSNC,Egloff:thesis,Egloff:DynamicsUFM,Egloff:UFHM,Crampon:these}. Moreover we will not use that much of this machinery, at least directly. In \cite{Egloff:thesis} in particular, the reader can find much about the link between Foulon's definitions and their Riemannian (or connection-obtained Finsler) equivalent. In \cite{Crampon:these} or \cite{Crampon1}, Micka\"el Crampon even generalizes these methods to Hilbert geometry (so to some Finsler metrics with very low regularity). However, as we do need some results further on, we introduce the bare minimum.\\

There exists a $C^{\infty}$-linear operator $H_{X} \colon VHM \rightarrow THM$, called the \emph{horizontal endomorphism} such that, if we set $h_{X}HM:= H_{X}(VHM)$, we have:
$$
THM = VHM \oplus h_{X}HM \oplus \R\cdot X.
$$
 We write $\id = p_{v} + p_{h} + p_X$ for the associated projections.

Note that this decomposition generalizes the classical vertical/horizontal decomposition in Riemannian geometry. Let us also state the following easy fact:
\begin{lem}\label{lem:utilisation_unique2}
 Let $Z \in T_{(x,\xi)}HM$ such that $d\pi (Z) = (x,v) \in TM$. Let $\lambda \in \R$ such that $Z = \lambda X + h + Y$, then 
\begin{equation*}
  \lambda \leq F(x,v), \quad \text{ with equality iff } r(x,\xi) = (x,v).
\end{equation*}
\end{lem}

\begin{proof}
 As $A$ is zero on $H_X HM \oplus VHM$, we have $A(Z) = \lambda$. Now, let $u \in T_xM$ such that $r(x,u) = (x ,\xi)$, then
\begin{align*}
 A(Z) &= \lim_{\eps \rightarrow 0} \frac{1}{\eps} \left( F(x, u + \eps d\pi (Z)) - F(x,u) \right) \\
      &\leq \lim_{\eps \rightarrow 0} \frac{1}{\eps} \left( F(x,u) + \eps F(x,v) - F(x,u) \right),
\end{align*}
where the last line is obtained by convexity of $F$, giving us that $\lambda \leq F(x,v)$. The equality condition comes from the fact that $F$ is strongly, hence strictly, convex.
\end{proof}

There exists a first-order differential operator $D_{X}$ defined on the space of $C^1$ vector fields on $HM$, called the \emph{dynamical derivative} associated with $X$. The splitting of $THM$, namely $THM = VHM \oplus h_{X}HM \oplus \R\cdot X$ is invariant under $D_{X}$ and ,for any vector field $Y \colon HM \rightarrow VHM$,
\begin{equation} \label{eq:horizontal_endomorphism}
 H_{X}(Y) = -[X,Y] +D_{X}(Y) \, .
\end{equation}

In the following, uppercase letters will refer to vector fields.

The space $HM$ comes naturally equipped with a Riemannian metric $g$ such that $g(X,X) = 1$, the above splitting of $THM$ is orthogonal and, for any $y_1,y_2 \in V_{(x,\xi)}HM$, if we choose $Y_1,Y_2$ any extensions of $y_1$ and $y_2$ to vertical vector fields, then
\begin{equation}
 g(y_1,y_2) = dA ([X,Y_2],Y_1)\, .
\end{equation}
Furthermore, $g$ is compatible with $D_{X}$, in the following sense, $L_{X} \left( g(Y_1,Y_2)\right) = g(D_X Y_1, Y_2) + g(Y_1,D_X Y_2)$. The metric $g$ is called the \emph{vertical metric}.

There exists a $C^{\infty}$-linear operator $R^{X} \colon THM \rightarrow THM$, called the \emph{Jacobi endomorphism} or curvature endomorphism. It is defined by
\begin{equation} \label{eq:def_curvature}
 R^{X}(X)=0, \quad R^{X}(Y) = p_v ([X,H_X (Y) ]) , \quad R^{X}(h) = H_{X} (p_v [X,h]).
\end{equation}

\subsection{Cartan's structure equations}

Let $\Sigma$ be a surface and $F$ a Finsler metric on it. There is a way to generalize Cartan's structure equations to this Finsler setting using Foulon's formalism. Foulon never published the proof, A. Reissmann did it in a not easily available preprint \cite{Reissmann}, so we provide his proof. Note that Bryant \cite{Bryant:FS2S,Bryant:PFF2S} gives the same result using a different presentation.

\begin{prop}[Cartan's structure equations] \label{prop:Cartan_structure_equation}
 Let $(\Sigma, F)$ be a Finsler surface. Let $Y'$ be the unique vertical vector field such that $g(Y',Y') =1$ and $h = H_{X}(Y') \in h_{X}HM$. If we write $k$ for the function on $H\Sigma$ such that $R^X (Y') = k Y'$, then we have the following relations:
\begin{equation*}
\left\{
\begin{aligned}
\left[X,Y'\right] &= -h \\
\left[X,h\right] &= k Y' \\
\left[Y',h\right] &= -X + a Y' +b h \, ,
\end{aligned}
\right. 
\end{equation*}
where $a,b \colon HM \rightarrow \R$ satisfy $L_X b =a$ and $L_X a + b k  - L_{Y'} k  =0 $.
\end{prop}

Note that $Y'$ is unique because we consider only orientable manifold.

\begin{proof}
Let start out by computing $\left[X,Y'\right] $. As $g(Y',Y') = 1$ and $D_{X}$ is compatible with $g$, we have
$$
 0 = L_{X} \left(g(Y',Y') \right) = 2g (D_{X}Y', Y') \, . 
$$
Hence $D_{X}Y'$ and $Y'$ are orthogonal, but $D_{X}$ leaves $VH\Sigma$ invariant, so ${D_{X}Y'=0}$ which yields, by Equation \eqref{eq:horizontal_endomorphism}:
$$
h = -\left[X,Y'\right].
$$
Now let us compute the three projections of $\left[X,h\right]$.

First recall that $\ker A$ coincides with $h_{X}H\Sigma \oplus VH\Sigma$ and that $A$ is invariant by $X$. So the equality $ L_{X} (A(h) ) =  (L_{X} A ) (h) + A(\left[X,h\right] ) $ yields $A(\left[X,h\right] )=0$, hence the projection of $\left[X,h\right]$ along $X$ is null.

By definition (see \cite{Fou:EquaDiff}), $D_{X} h = p_h ([X,h])$ and applying the same argument used on $Y'$ above shows that $D_{X} h=0$. And finally, again by definition, $ R^{X}(Y) = p_v ([X,H_X (Y) ])$.

We are left with $\left[Y',h\right]$. By choice of $Y'$, we have that $dA(\left[X,Y'\right],Y')=1$, so $dA(Y',h)=1$. Now,
\begin{equation*}
 dA(Y',h) = L_{Y'} \left(A(h) \right) - L_{h} \left(A(Y') \right) - A\left( \left[ Y',h \right] \right) = - A\left( \left[ Y',h \right] \right), 
\end{equation*}
so the projection of $\left[ Y',h \right]$ along $X$ is $-1$ and projecting on the horizontal and vertical distributions shows that there exist two real-valued functions $a$ and $b$ on $H\Sigma$ such that $\left[ Y',h \right]= -X + a Y' + b h $.

We finish by proving the assertions about the Lie derivatives of $a$ and $b$. They follow from the Bianchi identity. Indeed, we have
\begin{align*}
\left[X,\left[Y',h\right]\right] &= \left[X, -X +aY' + b h\right] \\
           &= a\left[X,Y'\right] + \left( L_{X} a\right) Y' + b \left[X,h\right] + \left( L_{X} b\right) h \\
           &= -a h + \left( L_{X} a\right) Y' + b R^{X}(Y') + \left( L_{X} b\right) h ,\\
\left[h, \left[X,Y'\right] \right] &= 0 , \\
\left[Y', \left[h,X\right]\right] &= \left[Y', -R^X (Y')\right] \\
           &= -\left(L_{Y'} k \right) (Y') \, ,
\end{align*}

So writing that $\left[X,\left[Y',h\right]\right] +\left[h, \left[X,Y'\right] \right] +\left[Y', \left[h,X\right]\right]  =0$ yields
\begin{equation*}
 \left( -a  +  L_{X} b \right) h +  \left( L_X a + b k  - L_{Y'} k \right) Y' = 0.
\end{equation*}
This implies that $L_{X} b = a$ and $L_X a + b k  - L_{Y'} k  =0 $.

\end{proof}

\begin{prop}\label{prop:b=0_alors_F_riem}
 The function $b$ is identically zero if and only if $F$ is Riemannian.
\end{prop}

\begin{proof}
Suppose that $b = 0$. Remark that, as $L_Xb = a$, we immediately get that $a=0$.

Denote by $\chi_s$ the flow generated by $Y'$, take $x\in \Sigma$ and choose a point $u_0 \in H_x \Sigma$. Let $v(s):= d\pi \left( X_{\chi_s (u_0)} \right)$. By construction, we have $F(v(s))=1$ and so $v$ is a parametrization of the unit circle $F^{-1}(1)$ at $x$. The goal is to show that $v$ parametrizes an Euclidean circle.

 We start by computing $\dot{v} = \frac{dv}{ds}$. By definition, we have
\begin{equation*}
 \left[Y',X \right]_{\chi_s(u_0)} = \lim_{\eps \rightarrow 0} \frac{1}{\eps} \left( d \chi_{-\eps} \left( X_{\chi_{s + \eps} (u_0)} \right) -X_{\chi_{s} (u_0)} \right),
\end{equation*}
therefore, using that $\pi \circ \chi_{-\eps} = \pi$,
\begin{align*}
 d\pi \left( \left[Y',X \right]_{\chi_s(u_0)}  \right) &=  \lim_{\eps \rightarrow 0} \frac{1}{\eps} \left( d\pi \circ d \chi_{-\eps} \left( X_{\chi_{s + \eps} (u_0)} \right) - d\pi \left( X_{\chi_{s} (u_0)} \right)\right) \\
         &= \lim_{\eps \rightarrow 0} \frac{1}{\eps} \left( v(s+\eps) -v(s) \right) \\
         &= \dot{v}\, .
\end{align*}
So, by Cartan's structure equations, we get that $\dot{v} = d\pi (h)$. Similar computations show that $\ddot{v} = d\pi \left( \left[Y',h\right] \right)$. Applying $d\pi$ to the last of the structure equations shows that
\begin{equation} \label{eq:utilisation_unique1}
 \ddot{v} = -v + b \dot{v}\, .
\end{equation}
So assuming that $b$ is identically null gives us a differential equation for which every integral curves are Euclidean circles. Hence, $F$ comes from a quadratic form.

Now assume that $F$ is Riemannian. Recall (see \cite{Fou:EquaDiff}) that in that case, for any $u$ in $H_{x} \Sigma$, $d\pi_u$ is an isometry from $\left(\R \cdot X \oplus h_X H\Sigma \right)_u$ equipped with the vertical metric $g_u$ onto $T_x \Sigma$. Thus, by the above computations, $\dot{v}$ is orthogonal to $v$ and $F(\dot{v}) = 1$. Taking the derivative shows that $\dot{v}$ is orthogonal to $\ddot{v}$, which, by Equation \eqref{eq:utilisation_unique1} yields that $b=0$.
\end{proof}

\section{Cotangent space and Legendre transform} \label{sec:cotangent}

In this section, we will recall the construction of the Legendre transform associated to a Finsler metric and the dual Finsler metric. The Legendre transform gives a natural diffeomorphism between the tangent and the cotangent bundles. The construction is well-known in Finsler geometry and even better known in the Hamiltonian context as is the dual Finsler metric, but some of the results proved below might be less well-known. We also try to give intrinsic proofs whenever we can.

\begin{defin}
  We define the dual Finsler metric $F^{\ast} \colon T^{\ast}M \rightarrow \R$ by setting, for $(x,p) \in T^{\ast}M$,
\begin{equation} \label{eq:dual_finsler}
F^{\ast}(x,p) = \sup \lbrace p(v) \mid v\in T_xM \; \text{\rm such that } F(x,v)=1 \rbrace.
\end{equation}

\end{defin}

\begin{rem}
 The function $F^{\ast}$ is a Finsler co-metric, that is, it verifies the same conditions as in Definition \ref{defin:Finsler_metric} but on the cotangent bundle.
\end{rem}

\begin{defin}
 The \emph{Legendre transform} $\L_F : TM \rightarrow T^{\ast}M$ associated with $F$ is defined by $L_F(x,0) = (x,0)$ and, for $(x,v) \in \mathring{T}M$ and $u \in T_x^{\ast}M$, 
\begin{equation}
 \L_F ( x,v) (u) := \frac{1}{2} \left. \frac{d}{dt} F^2(x,v + tu)\right|_{t=0}.
\end{equation}
\end{defin}

As $F^2$ is $2$-homogeneous, we have that $\L_F$ is $1$-homogeneous, so we can project $\L_F$ to the homogenized bundles. Set $H^{\ast}M := \mathring{T}^{\ast}M / \R^+_{\ast}$ and write $\ell_F\colon HM \rightarrow H^{\ast}M$ for the projection.\\
 We can also construct $\ell_F$ via the Hilbert form $A$; As $A$ is zero on ${VHM = \ker d\pi}$, for any $(x,\xi) \in HM$, the linear function $(d\pi_{(x,\xi)})_{\ast} \left( A_{(x,\xi)} \right) \colon TM \rightarrow \R$ is well defined and taking its class in $H^{\ast}M$ gives $\ell_F$.  Remark that $\L_F (x,v) = F(x,v) (d\pi_{r(x,v)})_{\ast} \left( A_{r(x,v)} \right) $. Considering directly $\ell_F$, instead of $\L_F$, can be quite helpful sometimes.

\begin{prop}\label{prop:legendre_transform}
We have the following properties:
\begin{enumerate}
 \item $F = F^{\ast} \circ \L_F$;
 \item $\ell_F$ is a diffeomorphism, $\L_F$ is a bijection from $TM$ to $T^{\ast}M$ and a diffeomorphism outside the zero section;
 \item The following diagram commutes
$$
\xymatrix{
    & \Tzero^{\ast}M \ar[r]^{\hat{r}} \ar[ld]_{\hat{p}}  &   H^{\ast}M \ar[rd]^{\hat{\pi}} & \\
M   &       &       &   M  \\
    & \Tzero M \ar[uu]^{\L_F} \ar[r]_r  \ar[ul]^{p} & HM \ar[uu]_{\ell_F} \ar[ur]_{\pi} & 
}
$$
\end{enumerate}
\end{prop}

\begin{proof}
We start with {\it1.} Let $(x,v) \in TM$, then $F^{\ast} \circ \L_F (x,v) = \sup \lbrace \L_F(x,v) (u) \mid  F(x,u)=1 \rbrace$. Now, for any $u$ such that $F(x,u)=1$, we choose $Z \in T_{r(x,v)}HM$ such that $d\pi (Z) = (x,u)$ and we have 
\begin{equation*}
 \L_F(x,v) (u) = F(x,v) A_{r(x,v)}(Z) \leq F(x,v) F(x,u) = F(x,v),
\end{equation*}
with equality iff $r(x,v) = r(x,u)$ (by Lemma \ref{lem:utilisation_unique2}). This implies that $F^{\ast} \circ \L_F (x,v) = F(x,v)$.\\

For {\it 2.}, several things are already clear; The unique pre-image of $0$ by $\L_F$ is $0$, $\ell_F$ and $\L_F$ are as smooth as $F$ outside the zero section and finally, $\ell_F$ is bijective iff $\L_F$ is, so we will prove injectivity or surjectivity for one of them and it will imply it for the other.

The injectivity is given by Lemma \ref{lem:utilisation_unique2}; Indeed, suppose that there exist $\xi_1,\xi_2 \in H_xM$ distinct and $\mu\in \R^+$ such that $(d\pi)_{\ast} A_{(x,\xi_1)} = \mu (d\pi)_{\ast} A_{(x,\xi_2)}$. Let $u_1,u_2 \in T_xM$ such that $r(x,u_i) = (x,\xi_i)$ and $F(x,u_i) =1$, then, by Lemma \ref{lem:utilisation_unique2}, 
$$
 1= (d\pi)_{\ast} A_{(x,\xi_1)} (u_1) = \mu (d\pi)_{\ast} A_{(x,\xi_2)} (u_1) < \mu,
$$
and switching $u_1$ and $u_2$ gives $\mu <1$.

To prove surjectivity; take $p \in \mathring{T}_x^{\ast}M$ and set $\mu = F^{\ast}(x,p)$. Let $v_p \in T_xM$ such that $F(x,v_p)=1$ and $p(v_p) = \mu$, i.e., $v_p$ is the vector realizing the supremum in Equation \eqref{eq:dual_finsler} and write $(x,\xi_p) :=r(x,v_p)$.

Let $h \in \left(h_{X}HM \right)_{(x,\xi_p)}$ and $w = d\pi(h)$, we claim that $p(w) =0$.

Assume the claim for a moment and recall that $d\pi_{(x,\xi_p)}$ is a bijection from $\left(h_{X}HM \oplus \R \cdot X \right)_{(x,\xi_p)}$ to $T_x M$. So, for any $u \in T_xM$, there are $\lambda$ and $h$ such that $u = \lambda d\pi(X(x,\xi_p)) + d\pi(h)$, that is, $u = \lambda v_p + w$. Therefore 
\begin{equation*}
 p(u) = \lambda p(v_p) = \lambda \mu = \mu (d\pi)_{\ast} A_{(x,\xi_p)}( \lambda v_p  + w) = \mu (d\pi)_{\ast} A_{(x,\xi_p)}(u)
\end{equation*}
and the surjectivity would be proven. So we just need to prove the claim.

Let $h \in \left(h_{X}HM \right)_{(x,\xi_p)}$ and set $c(t) := \left(F(v_p + t d\pi(h)) \right)^{-1}\left(v_p + t d\pi(h)\right)$. By definition of $v_p$, $\left. \frac{d}{dt} p\left( c(t) \right) \right|_{t=0} = 0$ so
\begin{equation*}
 0 = \left. \frac{d}{dt} p\left( c(t) \right) \right|_{t=0} = p \left( \frac{d\pi (h)}{F(x,v_p + t d\pi(h))} - \frac{A_{(x,\xi_p)}(h) v_p }{F^2(x,v_p + t d\pi(h))} \right).
\end{equation*}
As $A_{(x,\xi_p)}(h) =0$, we get that $p(d\pi (h) ) =0$, which proves our claim.\\
Now, writing $\L_F$ in local coordinates, we can easily see that its Jacobian is given by $\left(\dfrac{\partial^2 F^2}{\partial v_i \partial v_j}\right)_{i,j}$ and, as this matrix is non-degenerate, $\L_F$ is a local diffeomorphism, hence a global one because it is bijective.\\

 Finally, the fact that the diagram commutes is trivial.
\end{proof}

\begin{defin}
 We set $B:=\left(\ell_F^{-1}\right)^{\ast} A$, it is a contact form and we write $X^{\ast}$ for its Reeb field.
\end{defin}

\begin{lem}
The following diagram commutes
$$
\xymatrix{
    TH^{\ast}M \ar[r]^{d\hat{\pi}} &   TM \ar[r]^{r} & HM \\
H^{\ast}M  \ar[u]_{X^{\ast}}  \ar[urr]_{\ell_{F}^{-1} } & &
}
$$
Note also that the projection by $\hat{p}$ of the integral curves of $X^{\ast}$ are geodesics of the metric on $M$.
\end{lem}

\begin{proof}
 As $B = \left( \ell_F^{-1}\right)^{\ast} A$ and $X^{\ast}$ is the Reeb field of $B$, we have ${X^{\ast} = \left( \ell_F^{-1}\right)_{\ast} X}$ (just verify that $B \left( \left(  \ell_F^{-1}\right)_{\ast} X \right)=1$ and $i_{\left( \ell_F^{-1}\right)_{\ast} X} dB =0$), and since $\hat{\pi}\circ \ell_F = \pi$, we get that the projections to $M$ of the integral curves of $X^{\ast}$ are the same as those of $X$, which proves the second claim.\\
 As $X^{\ast} = \left( \ell_F^{-1}\right)_{\ast} X$, we can rewrite it as $X \circ \ell_F^{-1} = d\ell_F^{-1} \circ X^{\ast}$, and $X$ being a second order differential equation, we get that 
\begin{equation*}
 \ell_F^{-1} = r\circ d\pi \circ X \circ \ell_F^{-1} = r\circ d\pi \circ d\ell_F^{-1} \circ X^{\ast} = r\circ d\hat{\pi} \circ X^{\ast}.
\end{equation*}

\end{proof}

\begin{lem} \label{lem:doubledual_norm}
 Let $H \colon T^{\ast}M \rightarrow \R$ be a Finsler co-metric, then we can define $H^{\ast}\colon TM \rightarrow \R$ by
 \begin{equation*}
  H^{\ast}(x,v) = \sup \lbrace p(v) \mid p\in T_x^{\ast}M \; \text{\rm such that } H(x,p)=1 \rbrace .
 \end{equation*}
The function $H^{\ast}$ satisfies the following properties:
\begin{enumerate}
 \item $H^{\ast}$ is a Finsler metric;
 \item If $F$ is a Finsler metric, then $F^{\ast\ast} = F$;
 \item We define the Legendre transform associated with $H$, $\L_H \colon T^{\ast}M \rightarrow TM$ by 
\begin{equation*}
 \L_H ( x,p_1) (p_2) := \frac{1}{2} \left. \frac{d}{dt} H^2(x,p_1 + tp_2)\right|_{t=0},
\end{equation*}
 where we identified $TM$ and $T^{\ast\ast}M$. It is a diffeomorphism outside the zero section. Furthermore, if $F$ is a Finsler metric then $\L_{F^{\ast}} \circ \L_F = \id_{TM}$ and $\L_F \circ \L_{F^{\ast}} = \id_{T^{\ast}M}$.
 \item $H = H^{\ast} \circ \L_H$.
\end{enumerate}
\end{lem}

\begin{proof}
We do not prove the first point nor the third one, as we will not use them in this dissertation. Note also that the proof of {\it 4.} is the same as in Proposition \ref{prop:legendre_transform}. The only thing we will use later is {\it 2.}, so we prove it.

Set $F'(x,v) := \sup \lbrace p(v) \mid p\in T_x^{\ast}M \; \text{such that } F^{\ast}(x,p)=1 \rbrace$, we want to show that $F' = F$, by homogeneity, it suffices to show it on the $F$-unit circle. Let $v\in T_xM$ such that $F(x,v)=1$. 

First, if we set $p:= \L_F(x,v) \in T_x^{\ast}M$ then, by definition, we have 
$$p(v) = F(x,v) = F^{\ast}(x,p) = 1,$$
so $F'(x,v) \geq F(x,v)$.

Now, if we take $p_0$ such that $F^{\ast}(x,p_0)= 1$ and $p_0(v) = F'(x,v)$, again by definition
\begin{equation*}
 F'(x,v) = p_0(v) \leq F^{\ast}(x,p_0) = 1 = F(x,v).
\end{equation*}
\end{proof}

The following result was communicated to us by P. Foulon and might be already known to others. However, we are not aware of the existence of a published proof and hence give one:
\begin{thm}[Foulon \cite{Fou:perso}] \label{thm:uniquely_contact}
Any Finsler metric on $M$ defines the same contact structure on $H^{\ast}M$, i.e., if $F$ is a Finsler metric on $M$ and ${B = \left(\ell_F^{-1} \right)^{\ast} A}$, the distribution $\ker B \subset TH^{\ast}M$ is independent of $F$.
Furthermore, if we denote by $\lambda$ the Liouville form on $T^{\ast}M$, we have 
\begin{equation} \label{eq:B_equal_louiville_over_F_star}
 \hat{r}^{\ast} B = \frac{\lambda}{F^{\ast}},
\end{equation}
and
\begin{equation} \label{eq:bdb_and_louiville}
 \hat{r}^{\ast} B\wedge dB^{n-1} = \frac{\lambda \wedge d\lambda^{n-1} }{(F^{\ast})^n}. 
\end{equation}
\end{thm}

\begin{proof} 
We will start by showing Equation \eqref{eq:B_equal_louiville_over_F_star}. First recall the definition of the Liouville form: for any $\in T^{\ast}M$, $\lambda_{p} = p \circ d\hat{p}_{|M}$, where $\hat{p}_{|M} \colon T^{\ast}M \rightarrow M$ is the base point projection. In order to show that $\hat{r}^{\ast} B = \frac{\lambda}{F^{\ast}}$, we will prove that their pull-back by $\L_F$ coincides.\\

On one hand, as $\hat{r} \circ \L_F = r \circ \ell_F$, we have
\begin{equation*}
\L_F^{\ast} \hat{r}^{\ast} B = r^{\ast}  \ell_F^{\ast} B = r^{\ast} A = d_vF \, ,
\end{equation*}
and on the other hand,
\begin{equation*}
 \L_F^{\ast} \left( \frac{\lambda}{F^{\ast}} \right) = \frac{\L_F^{\ast} \lambda }{ F^{\ast} \circ \L_F} = \frac{\L_F^{\ast} \lambda }{F} \, .
\end{equation*}
Now, let us compute $\L_F^{\ast} \lambda $: for $(x,v) \in TM$ and $Z \in T_{(x,v)} TM$,
\begin{align*}
\left( \L_F^{\ast} \lambda\right)_{(x,v)} (Z) &= \lambda_{\L_F(x,v)} \left( d\L_F (Z) \right) \\
	&= \L_F(x,v) \circ d\hat{p}_{|M} \circ d\L_F (Z) \\
	&= \L_F(x,v) \circ dp_{|M} (Z) \\
	&= \frac{1}{2} \frac{d}{dt} F^2\left(x, v + t dp_{|M} (Z) \right) \\
	&= F(x,v) d_vF_{(x,v)}(Z)\, .
\end{align*}
And we proved Equation \eqref{eq:B_equal_louiville_over_F_star}. Once we have that, the uniqueness of the contact structure is trivial.
 
For the last equality, we have
\begin{equation*}
 \hat{r}^{\ast} dB = d \hat{r}^{\ast} B =  \frac{d\lambda}{F^{\ast}} - \frac{\lambda \wedge dF^{\ast} }{(F^{\ast})^2}\, .
\end{equation*}
Therefore $\hat{r}^{\ast} dB^{n-1} =  \left(\frac{d\lambda}{F^{\ast}}\right)^{n-1} + \lambda\wedge\left( \text{Something} \right)$, so
\begin{equation*}
 \hat{r}^{\ast} B\wedge dB^{n-1}=  \frac{\lambda \wedge d\lambda^{n-1}}{(F^{\ast})^{n}} + \lambda \wedge \lambda\wedge\left( \text{Something} \right) = \frac{\lambda \wedge d\lambda^{n-1}}{(F^{\ast})^{n}}\, .
\end{equation*}

\end{proof}

\begin{cor} \label{cor:uniquely_contact}
 If $B$ and $B_1$ correspond to two Finsler metrics $F$ and $F_1$, then there exists a function $f \colon H^{\ast}M \rightarrow \R$ such that $B_1 = f B$.
\end{cor}

\begin{proof}
By Equation \eqref{eq:B_equal_louiville_over_F_star},
\begin{equation*}
 \hat{r}^{\ast} B_1 = \frac{F^{\ast}}{F^{\ast}_1} \hat{r}^{\ast} B,
\end{equation*}
hence the result.
\end{proof}

\section{Angle and volume in Finsler geometry} 
\label{sec:angle_and_volume_in_Finsler_geometry}

We will give here a definition of a solid angle in Finsler geometry, i.e., a volume form on each $H_{x}M$ naturally associated with the Finsler metric. Even if the construction seems to be known, this angle does not appear as such in the literature, at least to the best of the knowledge of the author. We will also construct a similar solid ``co-angle'', i.e.,  a volume form on each $H_{x}^{\ast}M$.\\
Simultaneously to the angle, the construction gives a volume form on $M$. We will show that this volume is the Holmes-Thompson volume for Finsler geometry \cite{HolmesThompson}.

There are already several different angles defined in Finsler geometry. We did not however try to compare known Finsler angles to this new one, making just a quick address to that problem in Section \ref{sec:angles_dim2}.

\subsection{Construction} \label{subsec:construction}

As we have seen in Section \ref{sec:Definitions}, the only truly canonical volume associated with $F$ is given by $\ada$. We are going to split this volume on $HM$ into a volume on the base manifold $M$ and a volume on each fiber $H_xM$ in a canonical way. From now on, we will assume that $M$ is oriented and will only consider volumes preserving the orientation.

The construction is done in two steps. First,
\begin{lem}
\label{lem:existence_angle}
Let $\omega$ be a volume form on $M$, then there exists an $(n-1)$-form $\alpha^{\omega}$\ on $HM$\ such that
 $$
 \alpha^{\omega} \wedge \pi^{\ast}\omega = A\wedge dA^{n-1}.
 $$
 Moreover, the value taken by $\alpha^{\omega}$\ on $VHM$ is uniquely determined.
\end{lem}

\begin{proof}
 The existence of some $\alpha^{\omega}$\ is straightforward, we just complete a $n$-form $\pi^{\ast}\omega$\ into a $(2n-1)$-form $\ada$.\\
The uniqueness is given by the following. Let $Y_1, \dots, Y_{n-1} \in VHM$\  be $(n-1)$ linearly independent vertical vector fields. Then $ Y_1, \dots, Y_{n-1}$, $ X$, $\left[X, Y_1 \right], \dots,\\ \left[X, Y_{n-1}\right]$\ are linearly independent (see \cite[Theorem II.1]{Fou:EquaDiff}). Therefore, we must have
\begin{equation*}
 \alpha^{\omega} \left(Y_1, \dots, Y_{n-1} \right) = \frac{A\wedge dA^{n-1} \left( Y_1, \dots, Y_{n-1}, X, \left[X, Y_1 \right], \dots, \left[X, Y_{n-1}\right] \right) } {\pi^{\ast}\omega\left( X, \left[X, Y_1 \right], \dots, \left[X, Y_{n-1}\right] \right) }.
\end{equation*}
\end{proof}

Note that $\alpha^{\omega}$ does give us a notion of solid angle, despite its non-uniqueness:
\begin{lem}
 For any $\omega$, the integral
\begin{equation}
 l^{\omega}(x)= \int_{H_x M} \alpha^{\omega}
\end{equation}
does not depend on the choice of $\alpha^{\omega}$.
\end{lem}

\begin{proof}
 Follows from the fact that the forms $\alpha^{\omega}$\ are the same on $VHM$.
\end{proof}

Secondly, in order to have a reasonable notion of angle, we wish the volume of the fibers to be constant. To coincide with the Riemannian case we take this constant to be the volume of a Euclidean sphere. It turns out that this condition is realized for a unique volume form on $M$, hence giving a really natural way to associate a pair angle/volume with a Finsler metric:
 \begin{lem}
\label{lem:volume}
 There exists a unique volume form $\Omega^F$\ on $M$\ such that, for all $ x \in M$,
 $$
  l^{\Omega^F}(x)= \text{vol}_{\text{Eucl}} \left( \S^{n-1} \right).
 $$
 Moreover, if $\omega$ is any volume on $M$, then $\Omega^F$ is given by
\begin{equation} \label{eq:definition_Omega_F}
 \Omega^F = \frac{l^{\omega}}{\text{vol}_{\text{Eucl}} \left( \S^{n-1} \right)} \omega \, .
\end{equation}

\end{lem}

Before getting on to the proof, let us state the following remark which has some interest of its own and will be needed afterwards:
\begin{rem}
 If $\omega'$\ is another volume form on $M$, preserving the orientation, and $f : M \rightarrow \R_+^{\ast}$\ such that $\omega'= f \omega$. Then, 
 $$
 \alpha^{f\omega} \wedge \pi^{\ast} \left( f\omega \right) =  \alpha^{f\omega} \wedge f \pi^{\ast} \omega = A\wedge dA^{n-1}.
 $$
 And so, for any  $Y_1, \dots, Y_n \in VHM$,
 \begin{equation}
  i_{Y_1}\dots i_{Y_n} \alpha^{f\omega}= \frac{1}{f} i_{Y_1}\dots i_{Y_n} \alpha^{\omega}.
 \end{equation}

\end{rem}

\begin{proof}[Proof of Lemma \ref{lem:volume}] 
 Let $c_n =\voleucl \left( \S^{n-1} \right)$ and $ \Omega = \frac{l^{\omega}}{  c_n } \omega$, then the remark shows that, on $VHM$, $\alpha^{\Omega} = \frac{c_n}{l^{\omega}} \alpha^{\omega}$, which in turns yields
 \begin{equation*}
  l^{\Omega}(x) = \int_{H_x M} \alpha^{\Omega} = \int_{H_x M} \frac{c_n}{l^{\omega}(x)} \alpha^{\omega} = \frac{c_n}{l^{\omega}(x)} \int_{H_x M} \alpha^{\omega} = c_n.
 \end{equation*}

 Uniqueness is also straightforward.
\end{proof}
 
We can summarize the construction in the following:
\begin{prop}
\label{prop:construction}
 There exists a unique volume form $\Omega^F$\ on $M$\ and an $(n-1)$-form $\alpha^F$\ on $HM$, never zero on $VHM$, such that
\begin{equation}
\label{eq:alpha_wedge_omega}
  \alpha^{F} \wedge \pi^{\ast}\Omega^F =  A\wedge dA^{n-1}, 
\end{equation}
and, for all $x\in M$, 
\begin{equation}
\label{eq:longueur_fibre}
 \int_{H_xM} \alpha^F =  \voleucl(\S^{n-1})\, .
\end{equation}

\end{prop}

\begin{rem}
 Let us emphasize again that $\alpha^F$ is \emph{not} unique, only its restriction to $VHM$, which is what we need in order to have a solid angle.\\
 We will see later (Section \ref{sec:holmes_thompson}) that $\Omega^F$ is in fact the well-known Holmes-Thompson volume.
\end{rem}

\subsection{Coangle}

Taking the Hamiltonian point of view, we can see that the volume $\bdb$ on $H^{\ast}M$ is as canonical as $\ada$ is, so we could have carried out the above construction on the homogenized cotangent bundle $H^{\ast}M$. The exact same steps gives the following:
\begin{prop}
\label{prop:coangle}
 There exists a unique volume form $\Omega^{F^{\ast}}$\ on $M$\ and an $(n-1)$-form $\beta^{F^{\ast}}$\ on $H^{\ast}M$, never zero on $VH^{\ast}M:= \ker d\hat{\pi}$, such that
\begin{equation}
\label{eq:beta_wedge_omega}
  \beta^{F^{\ast}} \wedge \hat{\pi}^{\ast}\Omega^{F^{\ast}} =  \bdb, 
\end{equation}
and, for all $x\in M$, 
\begin{equation}
 \int_{H^{\ast}_xM} \beta^{F^{\ast}} =  \voleucl(\S^{n-1})\, .
\end{equation}

\end{prop}

Fortunately for our claim of ``natural'' angle and volume associated with $F$, there is a relationship between our two constructions:
\begin{prop}
Recall that $\ell_{F} \colon HM \rightarrow H^{\ast}M$ denotes the Legendre transform. Then
\begin{align*}
\Omega^F &= \Omega^{F^{\ast}}, \\
\alpha^F &= \ell_{F}^{\ast} \beta^{F^{\ast}},
\end{align*}
where the second equality holds on $VHM$.
\end{prop}

In the sequel, when the metric is clear from the context, we will often forget the superscript when writing $\alpha^F$, $\beta^{F^{\ast}}$ or $\Omega^F$.

\begin{proof}
First, note that
\begin{equation*}
 \ada = \ell_F^{\ast}\left(\bdb\right) = \ell_F^{\ast}\left( \beta^{F^{\ast}} \wedge \hat{\pi}^{\ast}\Omega^{F^{\ast}} \right) = \ell_F^{\ast}\left( \beta^{F^{\ast}} \right) \wedge \ell_F^{\ast} \hat{\pi}^{\ast} \left( \Omega^{F^{\ast}}  \right).
\end{equation*}
 As $\hat{\pi} \circ \ell_F = \pi$ we have that $\ell_F^{\ast} \hat{\pi}^{\ast} \left( \Omega^{F^{\ast}}  \right) = \pi^{\ast}\Omega^{F^{\ast}}$, which yields 
\begin{equation*}
  \ada = \ell_F^{\ast} \beta^{F^{\ast}}  \wedge \pi^{\ast}\Omega^{F^{\ast}}.
\end{equation*}
It remains to show that the length of the fibers for $\ell_F^{\ast} \beta^{F^{\ast}}$ are equal to $\voleucl\left(\S^{n-1} \right)$ as the uniqueness part in Proposition \ref{prop:construction} would then prove the claim. By the change of variables formula and the definition of $\beta^{F^{\ast}}$, we have for any $x\in M$
\begin{equation*}
 \int_{H_x M} \ell_F^{\ast} \beta^{F^{\ast}}  = \int_{\ell_F(H_xM)} \beta^{F^{\ast}} = \int_{H^{\ast}_xM} \beta^{F^{\ast}} = \voleucl\left(\S^{n-1} \right). \qedhere
\end{equation*}

\end{proof}

\subsection{Angle and conformal change} \label{subsec:angle_and_conformal_change}

Here we will show two properties relating our angle and conformal change of Finsler metrics. The first says that the angle and coangle are invariant under conformal change, which is natural. The second is more surprising; the coangle determines the conformal class of a Finsler metric, therefore a Finsler metric is uniquely determined by a coangle and a volume form on the manifold.

\begin{prop}
Let $(M,F)$\ be a Finsler manifold, $f \colon M \xrightarrow{C^{\infty}} \R$, $F_f= e^{f} F$, $\alpha_f$, $\beta_f$ and $\Omega_f$\ the angle, co-angle and volume form of $F_f$. Then $\alpha_f = \alpha$\ on $VHM$, $\beta_f = \beta$ on $VH^{\ast}M$ and $\Omega_f = e^{nf} \Omega$.
\end{prop}

 \begin{proof}
  Using the definition of the Hilbert form, we immediately have $A_f = e^{f} A$, so 
  $$
  A_f \wedge dA_f^{n-1} = e^{nf} A\wedge dA^{n-1}.
  $$
  Let $\omega$\ be a volume form on $M$. Let $\alpha_F^{\omega}$\ and $\alpha_{F_f}^{\omega}$\ be the two $(n-1)$-forms defined by $\alpha_F^{\omega}\wedge \pi^{\ast} \omega =  A\wedge dA^{n-1}$\ and $\alpha_{F_f}^{\omega}\wedge \pi^{\ast} \omega  =  A_f \wedge dA_f^{n-1}$. We have
$$
\alpha_{F_f}^{\omega}\wedge \pi^{\ast} \omega = e^{nf} \alpha_F^{\omega}\wedge \pi^{\ast} \omega.
$$
From there we get that, for any $Y_1, \dots , Y_{n-1} \in VHM$,
  \begin{align*}
   i_{Y_1}\dots i_{Y_{n-1}} \left( \alpha_{F}^{\omega}\wedge \pi^{\ast} \omega \right) &= \alpha_{F}^{\omega}\left(Y_1, \dots , Y_{n-1} \right)\pi^{\ast} \omega \, , \\
   i_{Y_1}\dots i_{Y_{n-1}} \left( \alpha_{F_f}^{\omega}\wedge \pi^{\ast} \omega \right) &= \alpha_{F_f}^{\omega}\left(Y_1, \dots , Y_{n-1} \right)\pi^{\ast} \omega \, . 
  \end{align*}
  Therefore
  $$
  \alpha_{F_f}^{\omega}\left(Y_1, \dots , Y_{n-1} \right)\pi^{\ast} \omega =  e^{nf} \alpha_{F}^{\omega}\left(Y_1, \dots , Y_{n-1} \right)\pi^{\ast} \omega\, ,
  $$
which leads to
\begin{equation*}
\alpha_{F_f}^{\omega}\left(Y_1, \dots , Y_{n-1} \right) = e^{nf} \alpha_{F}^{\omega}\left(Y_1, \dots , Y_{n-1} \right).
\end{equation*}
And we deduce that, for any $x\in M$,
\begin{equation*}
 \int_{H_xM} \hspace{-2mm} \alpha_{F_f}^{\omega} = e^{nf(x)} \int_{H_xM} \hspace{-2mm} \alpha_{F}^{\omega}\, , 
\end{equation*}
The two volume forms $\Omega$\ and $\Omega_f$\ on $M$\ associated with $F$\ and $F_f$\ are given by (see Equation \eqref{eq:definition_Omega_F})

\begin{equation*} 
\Omega_f = \frac{\int_{\scriptscriptstyle{H_xM}} \hspace{-1mm} \alpha_{F_f}^{\omega}}{c_n} \omega,\; \text{and} \; \, \Omega = \frac{\int_{\scriptscriptstyle{H_xM}} \hspace{-1mm} \alpha_{F}^{\omega}}{c_n} \omega \, ,
\end{equation*}
which yields
\begin{equation}
\label{eq:omega1}
 \Omega_f = e^{nf} \Omega\, .
\end{equation}
Using the definition of $\alpha_f$\ and Equation \eqref{eq:omega1}, we obtain
\begin{equation*}
 e^{nf} \alpha_f \wedge\pi^{\ast} \Omega  = \alpha_f \wedge \pi^{\ast} \Omega_f = A_f\wedge dA_f^{n-1} = e^{nf} A\wedge dA^{n-1} =e^{nf} \alpha \wedge \pi^{\ast}\Omega \, ,
\end{equation*}
which in turns implies that, for any $Y_1, \dots , Y_{n-1} \in VHM$, we have
\begin{equation*}
\alpha\left(Y_1, \dots , Y_{n-1}\right) = \alpha_f\left(Y_1, \dots , Y_{n-1}\right).
\end{equation*}
Applying the following lemma will prove the claim about the co-angles.
\end{proof}

\begin{lem}
If $F_f = e^{f} F$, then $\ell_{F} = \ell_{F_f}$.
\end{lem}

\begin{proof}
Recall (see Section \ref{sec:cotangent}) that for $(x,\xi) \in HM$, $\ell_F(x,\xi)$ is given by the class in $H^{\ast}_x M$ of $ A_{(x,\xi)}$ seen as an element of $T^{\ast}_xM$. Now $A_f = e^f A$, therefore,  $ A_{(x,\xi)}$ and $ \left.A_f \right._{(x,\xi)}$ are in the same class. Hence $\ell_{F} = \ell_{F_f}$.
\end{proof}

\begin{prop} \label{prop:coangles_equals_implies_conformal}
 Let $\beta$ and $\beta_1$ be coangles associated with two Finsler metrics $F$\ and $F_1$. If $\beta$ and $\beta_1$ are equal on $VH^{\ast}M$, then there exists a positive function $f$ on $M$ such that $F_1 = f F$.
\end{prop}

\begin{proof}

By Equation \eqref{eq:bdb_and_louiville}, we have that $ \bdb = \left(\frac{F_1^{\ast}}{F^{\ast}}\right)^n B_1 \wedge dB_1^{n-1}$, where $F_1^{\ast}/ F^{\ast}$ is seen as a function on $H^{\ast}M$, we call it $f$ for the moment (note that $f$ is positive).\\
Using the definition of coangle, we have that
\begin{equation*}
 \beta \wedge \hat{\pi}^{\ast}\Omega = f^n \beta_1 \wedge \hat{\pi}^{\ast}\Omega_1 = f^n \beta \wedge \hat{\pi}^{\ast}\Omega_1\, .
\end{equation*}
Now, applying $i_{Y_1} \dots i_{Y_{n-1}}$ to both sides, for every $Y_1, \dots, Y_{n-1} \in VH^{\ast}M$, shows that we must have $\hat{\pi}^{\ast}\Omega = f^n \hat{\pi}^{\ast}\Omega_1$, i.e., $f$ must be constant on the fibers. Hence, we can see $f$ as a function on $M$, and, as $f = \frac{F_1^{\ast}}{F^{\ast}}$, we get, $F_1^{\ast}=f F^{\ast}$, with $f \colon M \rightarrow \R$. So $F_1^{\ast}$ and $F^{\ast}$ differs by a conformal change.\\
From there, it is easy to deduce it for $F_1$ and $F$. Indeed, recall that $F_1^{\ast\ast} = F_1$ (see Lemma \ref{lem:doubledual_norm}), therefore
\begin{align*}
 F_1(x,v) &=  \sup \lbrace p(v) \mid p\in T_x^{\ast}M \; \text{such that } F_1^{\ast}(x,p)=1 \rbrace \\
	  &=  \sup \lbrace p(v) \mid p\in T_x^{\ast}M \; \text{such that } f(x) F^{\ast}(x,p)=1 \rbrace\\
	  &= f(x) F(x,v)\, .
\end{align*}
\end{proof}

\begin{cor} \label{cor:Finsler_uniquely_determined}
A Finsler metric $F$ is uniquely determined by the two forms $\beta^F$ and $\Omega^F$.
\end{cor}

Note that this corollary gives an interesting characterization of a Finsler metric. However, given two such forms, we did not try to give conditions implying that they come from a Finsler metric even so it is an interesting question.

Remark that it would seem natural to have the above proposition also true for angles (instead of coangles). Unfortunately, we do not know if that is the case. One thing is sure: it is not, as it might seem, a direct consequence of Proposition \ref{prop:coangles_equals_implies_conformal}: If $\alpha = \alpha_1$, we just have $\beta_1 = \left( \ell_F \circ \ell_{F_1}^{-1} \right)^{\ast}(\beta)$, but we have not yet found a reason to believe that having the same angle would imply that the Legendre transforms are the same.

\section{Angles in dimension two} \label{sec:angles_dim2}

In dimension $2$, a solid angle is the same thing as a traditional angle. In this section, we quickly go over a few properties of our angle. It is not aimed to be anything like a thorough study of the angle, more of a quick overlook of some questions that at once came to our mind. First, we study the relationship between this angle and another that is traditionally used in Finsler geometry. Then we state some easy properties of the rotation generated by this angle.

\subsection{A characterization of Riemannian surfaces}

Recall that there is a canonical Riemannian metric $g$ on $VHM$ given by:
\begin{equation*}
 g (y_1,y_2) = dA\left( \left[X,Y_2\right] , Y_1 \right),
\end{equation*}
 where $Y_1$\ and $Y_2$\ are vertical vector fields such that $Y_i(x,v) = y_i$. This defines a distance function on each $H_x\Sigma$, i.e., an angle. Note that this gives an angle in any dimension, not just a solid angle as our $\alpha^F$.

Even though the presentation of this metric was given by P. Foulon, the angle that it defines was known well before him, because it turns out that this vertical metric is the same as the one obtained by considering $\left(\dfrac{\partial^2 F^2}{\partial v_i \partial v_j}\right)_{i,j}$.

Our goal in this section is to show that $\alpha^F$ is in general very different from the angle obtained via the vertical metric. We will show even more: in dimension $2$, if those angles coincides at a point, then the Finsler metric is Riemannian at that point.
\begin{prop}
 Let $Y$ be the vertical vector field such that $\alpha^F(Y)=1$ and $Y'\in VH\Sigma$ such that $g(Y',Y') =1$. Let $c \colon HM \rightarrow \R$, such that $Y' = c Y$. \\
 If $c$ is constant on the fibers, then $F$\ is Riemannian.
\end{prop}

\begin{proof}
 If $c$\ is constant on the fibers, then
\begin{equation*}
 A\wedge i_{Y'} dA =i_{Y'} (A\wedge dA) = \alpha(Y') \pi^{\ast} \Omega  = c \pi^{\ast} \Omega 
\end{equation*}
can be projected to $M$. Therefore $L_{Y'} \left(A\wedge i_{Y'} dA \right)= 0$, and a direct computation gives $L_{Y'} \left(A\wedge i_{Y'} dA \right) = A\wedge L_{Y'} \left( i_{Y'} dA \right)$. We deduce that $L_{Y'} \left( i_{Y'} dA \right)$\ is null on $VHM \oplus h_{X} HM$. In particular, if $h = H^X (Y')$, then $L_{Y'} \left( i_{Y'} dA \right) (h) = 0$.

Using Cartan's structure equations (Proposition \ref{prop:Cartan_structure_equation}), we have
\begin{align*}
 L_{Y'} \left( i_{Y'} dA \right) (h) &= i_{Y'}\left(d i_{Y'} dA \right) (h) \\
   &= d\left( i_{Y'} dA \right) (h, Y') \\
   &= L_h \left( i_{Y'} dA (Y')\right) - L_{Y'} \left( i_{Y'} dA (h)\right)  - i_{Y'}dA\left( \left[h,Y'\right]\right) \\
   &= - L_{Y'} \left(  L_{Y'} A(h) - L_h A(Y') - A([H,Y']) \right) - dA\left( \left[h,Y'\right], Y'\right)\\
   &= - L_{Y'} \left(  - A(X - a Y'-b h) \right) - dA\left( X - a Y'-b h, Y' \right) \\
   &= - L_{Y'} (1) - dA\left(-b h, Y'\right)\\
   &= - b dA\left(\left[X, Y'\right], Y'\right) \\
   &= -b\, .
\end{align*}
So $b=0$, which is equivalent to $F$ being Riemannian (by Proposition \ref{prop:b=0_alors_F_riem}).
\end{proof}

\subsection{Rotations and reversibility}

\begin{prop}
 Let $(\Sigma,F)$\ be a Finsler $2$-manifold. There exists a unique vertical vector field $Y \colon H\Sigma \rightarrow VH\Sigma$\ such that $\alpha(Y)=1$. The one-parameter group $\theta^t$\ generated by $Y$\ is such that $\forall (x,\xi) \in H\Sigma$, $t\in \R$,
\begin{equation}
\label{eq:rotation}
\left\{
\begin{aligned}
 \pi\left(\theta^t(x,\xi)\right) &= x \\
 \theta^{2\pi}(x,\xi) &= (x,\xi) \, .
\end{aligned}
\right.
\end{equation}

\end{prop}

\begin{proof}
 Take a non-degenerate vertical vector field $Y_1$\ and set $Y := Y_1/\alpha\left(Y_1\right)$. Uniqueness is due to the fact that $VH\Sigma$\ is $1$-dimensional.

As $Y$\ is a non-degenerate smooth vector field, it generates a one-parameter group $\theta^t$, and the first part of \eqref{eq:rotation} follows from the fact that $Y$\ is vertical.

The second claim follows from the fact that the length of the fibers is taken to be $2\pi$.\\
Indeed, if we set $x \in \Sigma$, for any $\xi\in H_x\Sigma$, $t \mapsto \theta^{t} (x,\xi)$\ gives a parametrization of $H_x\Sigma$. Let $T_{(x,\xi)} \in \R^+$\ be the period of $\theta^{t}(x,\xi)$. It is easy to see that this period does not depend on $\xi$. Indeed, if $\xi_1,\xi_2 \in H_x\Sigma$, there exists $t<T_{(x,\xi_1)}$\ such that $\theta^{t}(x,\xi_1) = (x,\xi_2)$. Therefore
\begin{equation*}
\theta^{T_{(x,\xi_1)}}(x,\xi_2) = \theta^{T_{(x,\xi_1)}}\left(\theta^{t}(x,\xi_1) \right) = \theta^{t+T_{(x,\xi_1)}}(x,\xi_1) = \theta^{t}(x,\xi_1) = (x,\xi_2),
\end{equation*}
so $T_{(x,\xi_1)}$\ is also a period for $\xi_2$. Now, as the length of the fibers is $2\pi$, we can see that the period must be $2\pi$. Indeed
\begin{equation*}
 2\pi = \int_{H_x\Sigma} \hspace{-2mm} \alpha = \int_{0}^{T_x} \!\alpha \left( \dot{\theta}(t) \right) dt = \int_{0}^{T_x} \alpha \left(Y \right) dt = \int_{0}^{T_x} dt = T_x \, .
\end{equation*}
Therefore, for any $(x,\xi)\in H\Sigma$, $\theta^{2\pi}\left(x,\xi\right)= \left(x,\xi\right)$.
\end{proof}

The first question that sprang to our mind about this rotation is its link with the reversibility of the metric. We prove here something that seems very natural: if the metric is reversible, then rotating a vector by $\pi$ gives its opposite. Note that this is not a characterization of reversible metrics. Indeed, as we will see in the proof, if we write $s \colon HM \rightarrow HM$ such that $s(x,\xi) = (x,-\xi)$, the only thing we need is that $s^{\ast}\alpha = \alpha$.

\begin{prop}
 Let $(\Sigma,F)$\ be a Finsler $2$-manifold. If $F$\ is reversible and $(x,\xi) \in H\Sigma$, then
\begin{equation}
 \theta^{\pi}\left(x,\xi\right) = \left(x,-\xi\right).
\end{equation}

\end{prop}

\begin{proof}
Let   $ s \colon {T\Sigma  \rightarrow  T\Sigma } $, $(x,v)  \mapsto  (x,-v)$ be the local symmetry.

If $F$\ is reversible, then $s^{\ast} d_v F = - d_v F$: for $(x,v)\in T\Sigma$\ and $Z\in T_{(x,v)}T\Sigma$,
\begin{align*}
 s^{\ast} \left(d_v F\right)_{(x,v)}(Z) &= d_v F_{(x,-v)}(ds\, Z) \\
  &= \lim_{\epsilon \rightarrow 0} \frac{1}{\epsilon}\left( F(x, -v + \epsilon d\pi \circ ds \, Z) - F(x,-v)\right) \\
  &= \lim_{\epsilon \rightarrow 0} \frac{1}{\epsilon}\left( F(x, v - \epsilon d\pi \, Z) - F(x,v)\right) \\
  &= - d_v F_{(x,v)}(Z).
\end{align*}
Now, if we also denote by $s$\ the symmetry on $H\Sigma$, we have shown that $s^{\ast} A = -A $. Therefore $s^{\ast} \left( A\wedge dA\right) =  A\wedge dA$ and, using the definition of the angle form, we obtain $s^{\ast} \alpha = \alpha$.\\
The map $s$\ is a diffeomorphism of $H\Sigma$, so for $x\in \Sigma$\ and $U \subset H_x \Sigma $\ measurable, the change of variable formula gives
\begin{equation}
 \int_{U} s^{\ast}\alpha = \int_{s(U)} \alpha,
\end{equation}
if $\xi\in H_x\Sigma$\ and $U$\ is one interval from $\xi$\ to $-\xi$, then $H_x\Sigma = U \cup s\left(U\right)$\ and
\begin{equation*}
 2\pi = \int_{H_x\Sigma} \alpha = \int_{U} \alpha + \int_{s\left(U\right)} \alpha = 2 \int_{U} \alpha.
\end{equation*}
Recall that there exists $t_0$\ such that $U = \lbrace \theta^t(\xi) \mid t\in \left[0,t_0 \right] \rbrace$, so
\begin{equation*}
\pi = \int_{U} \alpha = \int_0^{t_0} \alpha_{\theta^t(\xi)}\left(\dot{\theta}^t(\xi) \right) dt = \int_0^{t_0} dt = t_0,
\end{equation*}
therefore $-\xi = \theta^{\pi}\left(\xi\right)$. \qedhere

\end{proof}

\section{$\Omega^F $ is the Holmes-Thompson volume} \label{sec:holmes_thompson}

Contrarily to what happens in Riemannian geometry, there is no canonical volume in Finsler geometry. We do not go into the reason for that as it is very well explained in \cite{BuragoBuragoIvanov}. There is however a certain number of ``natural'' volumes that have come up in convex geometry (see \cite{BuragoBuragoIvanov,AlvarezThompson}). Among them the most studied ones are the Busemann-Hausdorff and the Holmes-Thompson volumes.

The Busemann-Hausdorff volume (\cite{Busemann:intrinsic_area}) is obtained by considering a Finsler manifold as a metric space and taking its Hausdorff measure, but Alvarez-Paiva and Berck \cite{AlvarezBerck} showed that, despite its naturality, it might not be the ``good'' one.

The Holmes-Thompson volume (\cite{HolmesThompson}) comes from the symplectic structure of the cotangent bundle of a manifold. Recall that, if $\lambda$ denotes the Liouville form on $T^{\ast}M$, then $d\lambda$ is the canonical symplectic form on $T^{\ast}M$ and $\left(d\lambda \right)^n / n!$ is the Liouville volume. Now the Holmes-Thompson volume $\text{Vol}_{HT}$ associated with a Finsler metric $F$ is defined, for $U$ a Borel set on $M$ and $B^{\ast}U := \{ (x,p) \in T^{\ast}M \mid \hat{p}_{|M} (x,p) \in U, \;\; F^{\ast}(x,p) <1 \} \subset T^{\ast}M$, by
\begin{equation}
 \text{Vol}_{HT}(U) = \frac{1}{\epsilon_n} \int_{B^{\ast}U} \frac{\left(d\lambda \right)^n}{n!},
\end{equation}
where $\epsilon_n$ is the volume of the unit ball in the Euclidean space $\mathbb{E}^n$.\\
Remark that, if we denote by $S^{\ast}M$ the cotangent \emph{sphere} bundle, i.e., the subset of $T^{\ast}M$ given by $(F^{\ast})^{-1}(1)$, then an application of Stokes Theorem shows that
\begin{equation*}
 \text{Vol}_{HT}(U) = \frac{1}{\epsilon_n} \int_{S^{\ast}U} \frac{\lambda \wedge \left(d\lambda \right)^{n-1}}{n!}.
\end{equation*}
 
We can now prove that $\Omega^F$ corresponds to $\text{Vol}_{HT}$: For $U$ a Borel set in $M$, an application of Fubini Theorem (see Lemma \ref{lem:Fubini}) gives
\begin{equation*}
\int_{U} \Omega^F = \frac{1}{\voleucl \left(\mathbb{S}^{n-1}\right) } \int_{H^{\ast}U} \beta^F \wedge \hat{\pi}^{\ast} \Omega^F = \frac{1}{\voleucl \left(\mathbb{S}^{n-1}\right) } \int_{H^{\ast}U} \bdb.
\end{equation*}
Now, by Theorem \ref{thm:uniquely_contact}, we have that $\hat{r}^{\ast} \bdb = \frac{1}{(F^{\ast})^{n} }\lambda \wedge \left(d\lambda \right)^{n-1}$, and ${\hat{r} \colon \mathring{T}^{\ast}M \rightarrow H^{\ast}M }$ is a diffeomorphism when restricted to $S^{\ast}M$. So, by the change of variable formula, we obtain
\begin{equation*}
 \int_{H^{\ast}U} \bdb = \int_{\hat{r} \left(S^{\ast}U \right)} \bdb = \int_{S^{\ast}U} \frac{\lambda \wedge \left(d\lambda \right)^{n-1}}{(F^{\ast})^{n} } = \int_{S^{\ast}U} \lambda \wedge \left(d\lambda \right)^{n-1}.
\end{equation*}
Therefore,
\begin{equation}
 \int_{U} \Omega^F = \frac{1}{\voleucl \left(\mathbb{S}^{n-1}\right) } \int_{S^{\ast}U} \lambda \wedge \left(d\lambda \right)^{n-1} =  \frac{n! \epsilon_n}{\voleucl \left(\mathbb{S}^{n-1}\right) } \text{Vol}_{HT}(U) = (n-1)! \text{Vol}_{HT}(U). 
\end{equation}

So we proved the following:

\begin{prop}
 Let $F$ be a Finsler metric on an $n$-manifold, then 
\begin{equation*}
 \Omega^F = (n-1)! \text{Vol}_{HT}\,.
\end{equation*}
\end{prop}

 Before moving on, we wish to make a few remarks.

 First, we could get rid of the constant by considering $\ada/ (n-1)!$ instead of just $\ada$. However, our aim was never to study this volume and so we felt that adding this constant would unnecessarily complicate matters. Indeed, almost everything in this dissertation is invariant under a change of volume by a constant, the only exception being the actual computation of the volume of a Finsler manifold. However, our only actual computation of volume is in dimension two.

 Second, we must admit that, until very recently, we wrongly believed that $\Omega^F$ was different from the Holmes-Thompson volume. Hence in this dissertation, we prove every claim we make about this volume even so they are probably classical results.

 Finally, and this is a side note, it is interesting to point out that the Holmes-Thompson volume is very well-known, but that the naturally associated angle does not seem to have been studied, so we hope that we repair at least a little that injustice.

\chapter{A natural Finsler--Laplace operator}

\section[Finsler--Laplacian]{A natural Finsler--Laplace operator} \label{sec:Finsler-Laplacian}

\subsection{Definition}

In this section we will introduce our generalization of the operator of Laplace--Beltrami, and begin by recalling the different equivalent definitions in Riemannian geometry.

The ``historic'' Laplacian on $\R^n$ is defined as 
$$
\Delta  = \sum_i \frac{\partial^2}{\partial x_i^2}.
$$
So, given a Riemannian metric $g$, at any point $p$, we can choose normal coordinates and use the above expression. For generic local coordinates $(x_1, \dots, x_n)$, where the metric reads $g = [g_{ij}]$, the local expression becomes
$$
 \Delta^{LB} := \frac{1}{\sqrt{\text{det} g}} \frac{\partial}{\partial x_i}\left( \sqrt{\text{det} g} g^{ij} \frac{\partial}{\partial x_j}\right).
$$
A possibly more convenient, coordinate-free expression for the Laplacian is
$$
\Delta^{LB}f = \text{div} \left( \nabla f \right),
$$
where $\nabla f$ is the gradient of $f$ with respect to $g$ and $\text{div}$ is the divergence operator (see, for instance, \cite{GHL}).\\
Finally, the Hodge--Laplace operator gives an expression for the Laplacian on differential forms (see \cite{GHL} or \cite{Warner}).\\

The study of the Laplace operator is of paramount importance as it exhibit deep links between its spectral data and the geometry of the manifold carrying it (for a proof of this claim, the reader can consult, for a start, \cite{Ber:SpectralGeo,BergerGauduchonMazet,Chavel:eigenvalues} and then move on to the thousands of articles on this subject). Hence, giving a generalization of this operator to the Finslerian context is of great  importance, especially if we manage to construct one having the same kind of behavior.

There have already been several generalizations of the Laplace operator to Finslerian geometry \cite{BaoLackey,Centore:FinslerLaplacians,Shen:non-linear_Laplacian}, each starting from a different definition of the Laplace--Beltrami and obtaining different operators. This can be seen either as a drawback or a new source of interest in Finsler geometry. Indeed, it is not uncommon that, when we try to extend definitions that were equivalent, we often end up with different notions, yielding some new insight in the process. And this is particularly true of Finsler geometry, as the history of Finslerian connections, for instance, proves.

Bao and Lackey \cite{BaoLackey} gave a generalization of the Hodge-star operator, allowing them to define a Finslerian Hodge--Laplace operator. Shen \cite{Shen:non-linear_Laplacian} gave the very natural generalization of $\text{div} \left( \nabla  \right)$ to Finsler metrics. Indeed, the gradient of a function is just its derivative seen as an element of the tangent space, i.e., its pullback by the Legendre transform. So clearly this is not just Riemannian. But note however that, for Finsler metrics, the Legendre transform is in general \emph{not} linear, so Shen's Laplacian is not linear. Remark also that, to define a divergence, one needs to choose a volume form on the manifold and, as we have already mentioned, there is no canonical volume form in Finsler geometry. Finally, Centore \cite{Centore:FinslerLaplacians} did not use directly one of the above definitions, but the fact that harmonic functions satisfy the mean-value property and designed an operator in order to keep that property. Here, once again, his definition relies on the choice of a volume form.\\

Our approach for a generalization relies on the first definition, as the sum of the second derivatives in orthonormal directions. As there is no good notion of orthogonality in Finsler geometry, we consider instead the \emph{average} of the second derivatives in every direction. The average being taken with respect to the angle we introduced in Section \ref{sec:angle_and_volume_in_Finsler_geometry}. More precisely, we introduce:
\begin{defin}
\label{def:delta}
 Let $F$ be a Finsler metric on an $n$-manifold $M$. We define the Finsler--Laplace operator, denoted $\Delta^F$, as
 \begin{equation*}
 \Delta^F f (x) = \frac{n}{\voleucl \left(\mathbb{S}^{n-1}\right) }\int_{H_xM} L_X ^2 (\pi^{\ast} f ) \alpha^F,
 \end{equation*}
 for every $x\in M$\ and every $f \colon M \rightarrow \R$\ (or $\C$) such that the integral exists.
\end{defin}

As we will see in the next section, the constant $n/\voleucl \left(\mathbb{S}^{n-1}\right)$\ is chosen so that $\Delta^F$\ is the Laplace--Beltrami operator when $F$\ is Riemannian.
\begin{rem}
 Note that we can define a Laplace-like operator in this fashion for \emph{any contact form} on the homogenized bundle $HM$, but we have not pursued the study of this more general kind.
\end{rem}

It is already clear from the definition that $\Delta^F$\ is a linear differential operator of order two. It also verifes the following:
\begin{thm} \label{thm:finsler_laplace_basics}
 Let $F$\ be a Finsler metric on $M$, then $\Delta^F$\ is a second-order differential operator, furthermore:
\begin{itemize}
 \item[(i)] $\Delta^F$\ is elliptic;
 \item[(ii)] $\Delta^F$\ is symmetric, i.e., for any $f,g \in C^{\infty}_0(M)$, 
 \begin{equation*}
  \int_M f\Delta^F g - g\Delta^F f \; \Omega^F = 0\,;
 \end{equation*}
 \item[(iii)] $\Delta^F$\ is unitarily equivalent to a Schr\"odinger operator;
 \item[(iv)] $\Delta^F$\ coincides with the Laplace--Beltrami operator when $F$\ is Riemannian.
\end{itemize}
\end{thm}

Remark that our definition of Laplace operator could be applied with any angle. However, to obtain a symmetric operator, we fundamentally rely on the fact that $\alpha^F$ and $\Omega^F$ come from the volume $\ada$, which is invariant under the geodesic flow. So, in order to get an operator satisfying the above conditions, the only choice we really made was to ask for the constancy with respect to $\alpha^F$ of the volume of each fiber.

We split the proof of the theorem into four parts presented in the next four sections. Note that all the proofs are surprisingly simple, which is, in my opinion, an asset of this operator.

\subsection{The Riemannian case}
We start by proving Theorem \ref{thm:finsler_laplace_basics} {\it (iv)}.
\begin{prop}
 Let $g$\ be a Riemannian metric on $M$, $F = \sqrt{g}$, $\Delta^F$\ the Finsler--Laplace operator and $\Delta^g$\ the usual Laplace--Beltrami operator. Then,
\begin{equation*}
 \Delta^F = \Delta^g.
\end{equation*}

\end{prop}

\begin{proof}
We will compute both operators in normal coordinates for $g$ and show that they coincide.\\
Let $p\in M$\ and $x_1, \dots, x_n$\ be normal coordinates around it. Denote by $v_1, \dots , v_n$\ their canonical lift to $T_x M$.
For $f \colon M \rightarrow \R$, the Laplace--Beltrami operator is
$$ \Delta^{g}f(p)  = \sum_i  \frac{\partial^2 f }{\partial x_i^2} (p)\, . $$

The first step to compute the Finsler--Laplace operator is to compute the Hilbert form $A$\ and the geodesic flow $X$. In order to write $A$, we identify $HM$\ with $T^1M$\ and coordinates on $H_pM$\ are then given by the $v_i$'s with the condition $\sqrt{\sum (v_i)^2 }=1$.

 The vertical derivative of $F$\ at $p$\ is $d_v F_p = \left(v_i dx_i\right) \big/\sqrt{\sum (v_i)^2 }$. So  $ A_p = v_i\;dx^i$\ and $dA_p = dv_i \wedge dx^i $. Hence $X(p, \cdot ) = v_i \partialxi$. Indeed, we just need to check that $A_p\left(X_p\right) = 1$\ and $\left(i_X dA \right)_p = 0 $, but both equalities follow from $\sum (v_i)^2 =1$.

Let $f \colon M \rightarrow \R$. Then
\begin{equation*}
 L_X^2\left(\pi^{\ast}f \right) (p,v) = v_i v_j \frac{\partial^2 f}{\partial x_i\partial x_j} (p,v)\, ,
\end{equation*}
so that the Finsler--Laplace operator is
\begin{equation*}
 \Delta^F f(p) = \frac{ n}{\voleucl \S^{n-1}} \int_{H_pM} v_i v_j \; \alpha \; \frac{\partial^2 f}{\partial x_i\partial x_j}(p)\, ,
\end{equation*}
and the proof follows from the next two claims. \qedhere

\end{proof}

\begin{claim}
 For all $i \neq j$, 
$$
\int_{H_pM} v_i v_j \, \alpha =0 \,.
$$
\end{claim}

\begin{proof}
 $H_pM$\ is parametrized by 
$$ H_pM = \left\{ (v_1, \dots , v_n ) \mid v_i \in [-1, 1], \sum (v_i)^2 =1 \right\}.$$
 A parity argument then yields the desired result. \qedhere

\end{proof}

\begin{claim}
 For any $1\leq i \leq n $,
$$
\int_{H_pM} v_i^2 \, \alpha = \frac{\voleucl \S^{n-1}}{ n} \, .
$$

\end{claim}

\begin{proof}
 As the $v_i$'s are symmetric by construction, we have that, for any $i\neq j$,
$$
\int_{H_pM} v_i^2 \, \alpha = \int_{H_pM} v_j^2 \, \alpha\, .
$$
So
\begin{equation*}
 n \int_{H_pM} \hspace{-1.8mm} v_i^2 \, \alpha =  \sum_j \int_{H_pM} \hspace{-1.8mm} v_j^2 \, \alpha =   \int_{H_pM} \sum_j v_j^2 \, \alpha  =   \int_{H_pM} \hspace{-2.3mm}1 \, \alpha  = \voleucl \S^{n-1}. \qedhere
\end{equation*}

\end{proof}

\subsection{Ellipticity}
We now prove Theorem \ref{thm:finsler_laplace_basics} {\it (i)}.
\begin{prop}
 The operator $\Delta^F \colon  C^{\infty}(M) \rightarrow  C^{\infty}(M) $\ is elliptic. The symbol $\sigma^F$ is given by
\begin{equation*}
 \sigma^F_x(\xi_1,\xi_2) = \frac{n}{\voleucl \left(\mathbb{S}^{n-1}\right) } \int_{H_xM} L_X(\pi^{\ast} \varphi_1) L_X(\pi^{\ast}\varphi_2)\, \alpha^F
\end{equation*}
for $\xi_1,\xi_2 \in T^{\ast}_x M$, where $\varphi_i \in C^{\infty}(M)$ such that $\varphi_i(x)=0$ and $\left.d\varphi_i\right._x = \xi_i$.
\end{prop}

\begin{rem}
 If we identify the unit tangent bundle $T^1 M$ with the homogenized tangent bundle $HM$ and write again $\alpha^F$ for the angle form on $T^1M$, then the symbol is given by
\begin{equation*}
 \sigma^F_x(\xi_1,\xi_2) = \frac{n}{\voleucl \left(\mathbb{S}^{n-1}\right) } \int_{v\in T^1_xM} \xi_1(v) \xi_2(v) \, \alpha^F(v)
\end{equation*}
for $\xi_1,\xi_2 \in T^{\ast}_x M$.

The symbol of an elliptic second-order differential operator is a non-degenerate symmetric $2$-tensor on the cotangent bundle, and therefore defines a Riemannian metric on $M$. This gives one more way to obtain a Riemannian metric from a Finsler one. Let $\Delta^{\sigma}$\ be the Laplace--Beltrami operator associated with the symbol metric, then $\Delta^F -\Delta^{\sigma}$\ is a differential operator of first order, so is given by a vector field $Z$\ on $M$. The Finsler--Laplace operator therefore is a Laplace--Beltrami operator together with some ``drift'' given by $Z$. We will later see that our operator is in fact characterized by its symbol and the symmetry condition (Section \ref{subsec:a_characterization}).
\end{rem}

\begin{proof}
 To show that $\Delta$\ is elliptic at $p\in M$, it suffices to show that for each $\varphi \colon M  \rightarrow \R $\ such that $\varphi(p) = 0$\ and $d\varphi|_p$\ does not vanish, and for $u \colon M \rightarrow \R^+$ we have $\Delta^F(\varphi^2 u) (p) > 0$ unless $u(p)=0$.\\ 
 We first compute $L^{2}_X \left(\pi^{\ast}\varphi^2 u\right)$:
 \begin{align*}
  L^{2}_X \left(\pi^{\ast}\varphi^2 u\right) &= L_X \left( 2 \pi^{\ast}\varphi u L_X\left( \pi^{\ast}\varphi\right) + \pi^{\ast}\varphi^2 L_X \left(\pi^{\ast}u\right) \right), \\
  &=  2 \pi^{\ast}u \left(L_X\left( \pi^{\ast}\varphi\right)\right)^2  + 2 \pi^{\ast} \varphi u L^{2}_X \left(\pi^{\ast}\varphi \right) \\
   & \quad   +  4 \pi^{\ast}\varphi L_X \left(\pi^{\ast}\varphi\right) L_X \left(\pi^{\ast}u\right) + 2\pi^{\ast} \varphi^2 L^{2}_X\left( \pi^{\ast}u\right) .      
 \end{align*}
Evaluating at $\xi \in H_p M$, we obtain
\begin{equation*}
 L^{2}_X \left(\pi^{\ast}\varphi^2 u\right) \left(\xi\right) = 2 u(p) \left(L_X \pi^{\ast}\varphi\right)^2 \left(\xi\right).
\end{equation*}
 Therefore
 \begin{align*}
 \Delta^F(\varphi^2 u) (p) &= \frac{n}{\voleucl \left(\S^{n-1}\right) } \int_{H_p M} 2 u(p) \left(L_X \pi^{\ast}\varphi\right)^2 \alpha ,\\
  &=  \frac{2 u(p)n}{\voleucl \left(\mathbb{S}^{n-1}\right) } \int_{H_p M} \left(L_X \pi^{\ast}\varphi\right)^2 \alpha \; > 0 \, . \qedhere
 \end{align*}

\end{proof}

\subsection{Symmetry}

We have a hermitian product defined on the space of $C^{\infty}$\ complex functions with compact support on $M$\ by 
\begin{equation*}
\langle f,g \rangle = \int_M f(x)\overline{g(x)} \Omega^F.
\end{equation*}
So we can now prove Theorem \ref{thm:finsler_laplace_basics} {\it (ii)}.
\begin{prop}
\label{prop:symmetry}
 The operator $\Delta^F$\ is symmetric for $\langle \cdot , \cdot \rangle $\ on $C^{\infty}_0(M)$, i.e., for any ${f,g\in C^{\infty}_0(M)}$, we have
 \begin{equation*}
 \langle \Delta^F f , g \rangle =  \langle f,\Delta^F g \rangle.
 \end{equation*}
\end{prop}

\begin{rem}
 The proof of this result is remarkably simple due to our choice of angle form and volume. Indeed, as $\alpha\wedge \pi^{\ast} \Omega$ is the canonical volume on $HM$, it is invariant under the geodesic flow (i.e., $L_X(\alpha\wedge \pi^{\ast} \Omega) =0$) which is the key to the computation.
\end{rem}

In order to prove the proposition, we first need a Fubini-like result. It is certainly known, but as it appears to us that it would take less time to do it than try to look for a reference in the literature, we provide the proof below.
\begin{lem} \label{lem:Fubini}
 Let $f \colon HM \rightarrow \C$\ be a continuous, integrable function on $HM$. We have
 \begin{equation}
 \int_M \left( \int_{H_xM} f(x,\cdot ) \, \alpha \right) \Omega = \int_{HM} f \; \alpha \wedge \pi^{\ast}\Omega \, .
 \end{equation}  
\end{lem}

\begin{proof}
 In the following, we will write $f_x \colon H_x M \rightarrow \C$ for $f(x,\cdot)$.\\
 Let $\lbrace U_a \rbrace$ be a trivializing open covering for $\pi: HM \rightarrow M$, i.e., there exists $S$\ such that, for every $a$, there exists a homeomorphism 
$$\varphi_a \colon HU_a= \pi^{-1}\left( U_a \right) \rightarrow U_a \times S \,.$$
 Moreover, for every $x \in U_a$, $\varphi_a(x, \dot) \colon H_xM=H_x U_a \rightarrow S$ is also a homeomorphism.\\
 Let $\lbrace \varphi_a \rbrace$\ be a partition of unity subordinated to $\lbrace U_a \rbrace$. We have
 \begin{equation*}
 \int_M \left( \int_{H_xM} f_x \alpha \right) \Omega = \sum_a \int_{U_a} \varphi_a(x) \left( \int_{H_xM} f_x \alpha \right) \Omega \,.
 \end{equation*}
 Let $x\in U_a$. We set $\alpha^a_x := \left( \varphi_a(x,\cdot)^{-1} \right)^{\ast} \alpha$. It is a $(n-1)$-form on $S$, but, by definition of $\varphi_a$, $\alpha^a_x$ does not depend on $x$, just on $a$, so we can write $\alpha^a:=\alpha^a_x$. We have
 \begin{equation*}
 \int_{U_a} \varphi_a(x) \left( \int_{H_xM} f_x \alpha \right)\Omega = \int_{U_a} \varphi_a(x) \left( \int_S f_x \alpha^a \right)\Omega \, .
 \end{equation*}
 Then, applying Fubini's Theorem (see, for instance, \cite{BerGos}) gives 
 \begin{equation*}
 \int_{U_a} \varphi_a(x) \left( \int_S f_x \alpha \right)\Omega = \int_{U_a \times S} \varphi_a f \alpha^a\wedge\Omega\, .
 \end{equation*}
 Now, we have $\varphi_a^{\ast} \alpha^a\wedge\Omega = \alpha\wedge\pi^{\ast}\Omega$, hence we get
 \begin{equation*}
 \int_{U_a \times S} \varphi_a f \alpha\wedge\Omega = \int_{HU_a} \varphi_a f \alpha\wedge\pi^{\ast}\Omega\,.
 \end{equation*}
 Summing over $a$, we finally obtain
 \begin{align*}
  \int_M \left( \int_{H_xM} f_x \alpha \right) \Omega &= \sum_a \int_{U_a} \varphi_a(x) \left( \int_{H_xM} f_x \alpha \right) \Omega \\ 
    &=  \sum_a \int_{HU_a} \varphi_a f \alpha\wedge\pi^{\ast}\Omega \\
    &=  \int_{HM} \left( \sum_a \varphi_a\right) f \alpha\wedge\pi^{\ast}\Omega  \\ 
    &=  \int_{HM} f \alpha\wedge\pi^{\ast}\Omega\, . 
 \end{align*}
\end{proof}

We can now proceed with the

 \begin{proof}[Proof of Proposition \ref{prop:symmetry}]
  Let $f,g: M \xrightarrow{C^{\infty}} \C$\ and set $c_n := \dfrac{n}{\voleucl \left(\S^{n-1}\right) }$.
  \begin{align*}
   \langle \Delta^F f,g \rangle &= \int_M \overline{g} \Delta^F f  \;  \Omega \\
                  &= c_n \int_M \overline g \left( \int_{H_xM} L_X ^2 (\pi^{\ast} f ) \alpha \right) \Omega  \\
		  &= c_n \int_M \left( \int_{H_xM} \overline{\pi^{\ast}g} L_X ^2 (\pi^{\ast} f ) \alpha\right) \Omega  \\
		  &= c_n \int_{HM} \overline{\pi^{\ast}g} L_X ^2 (\pi^{\ast} f ) \; \alpha\wedge\pi^{\ast}\Omega \, ,
  \end{align*}
  where the last equality follows from the preceding lemma. \\As ${\alpha\wedge\pi^{\ast}\Omega = A\wedge dA^{n-1}}$, we can write 
  $$
  \langle \Delta^F f,g \rangle = c_n \int_{HM} \overline{\pi^{\ast}g} L_X ^2 (\pi^{\ast} f ) \; A\wedge dA^{n-1}.
  $$
Now
  \begin{multline*} 
  \label{eq:1}  
   L_X \left( \overline{\pi^{\ast}g} L_X (\pi^{\ast} f ) A\wedge dA^{n-1} \right)  =  \overline{\pi^{\ast}g}  L_X^2 (\pi^{\ast} f ) A\wedge dA^{n-1}  \\
    + L_X ( \overline{\pi^{\ast}g} ) L_X (\pi^{\ast} f ) A\wedge dA^{n-1}  + \overline{\pi^{\ast}g} L_X (\pi^{\ast} f ) L_X (\ada).
   \end{multline*}
The last part of the above equation vanishes because of \eqref{eq:lxada}. We also have
\begin{equation*} 
   L_X \left( \overline{\pi^{\ast}g} L_X (\pi^{\ast} f ) A\wedge dA^{n-1} \right)  = d\left( i_X \overline{\pi^{\ast}g} L_X (\pi^{\ast} f ) A\wedge dA^{n-1} \right).
\end{equation*}
Hence 
\begin{multline*}
 \langle \Delta^F f,g \rangle = \frac{n}{\voleucl \left(\S^{n-1}\right) }\Biggl[ \int_{HM} d\left( i_X \overline{\pi^{\ast}g} L_X (\pi^{\ast} f ) A\wedge dA^{n-1} \right)  \\
     -  L_X ( \overline{\pi^{\ast}g} ) L_X (\pi^{\ast} f ) A\wedge dA^{n-1} \Biggr].
\end{multline*}
As $M$\ is closed, $HM$\ is closed and applying Stokes Theorem gives \eqref{eq:green_formula}, thus proving the claim. 

 \end{proof}

In the proof we obtained a Finsler version of Green's formulas:
 \begin{prop}
 \label{prop:green_formula}
\begin{enumerate}
 \item For any $f,g \in C^{\infty}(M)$, we have
  \begin{equation} \label{eq:green_formula}
  \langle\Delta^F f,g \rangle = \frac{-n}{\voleucl \left(\S^{n-1}\right) } \int_{HM} L_X ( \overline{\pi^{\ast}g} ) L_X (\pi^{\ast} f ) \, A\wedge dA^{n-1}.
  \end{equation}
 \item Let $U$ be a submanifold of $M$\ of the same dimension and with boundaries. Then, for any $f \in C^{\infty}(U)$, we have
 \begin{equation}
 \int_U \Delta^F f \; \Omega^F =\frac{n}{\voleucl \left(\S^{n-1}\right) }\int_{\partial HU }  L_X (\pi^{\ast} f ) dA^{n-1}.
 \end{equation}
\end{enumerate}
 \end{prop}

\subsection{A characterization of $\Delta^F$} \label{subsec:a_characterization}
 
The following results were explained to us by Yves Colin de Verdi\`ere and are probably well known to many people. However, they might not be known to everyone and are quite interesting, so we provide the proofs.

\begin{lem}\label{lem:existence_unicity}
 Let $(M,g)$\ be a closed Riemannian manifold and $\omega$\ a volume form on $M$. There exists a unique second-order differential operator $\Delta_{g,\omega}$\ on $M$ with real coefficients such that its symbol is the dual metric $g^{\star}$, that is symmetric with respect to $\omega$ and zero on constants.\\
If $a\in C^{\infty}(M)$ is such that $\omega = a^2 v_g$, where $v_g$ is the Riemannian volume, then, for $\varphi \in C^{\infty}(M)$,
\begin{equation*}
 \Delta_{g,\omega} \varphi = \Delta^{g} \varphi - \frac{1}{a^2} \langle \nabla \varphi , \nabla a^2 \rangle.
\end{equation*}
\end{lem}

Before getting on to the proof, this result deserves a few remarks.
\begin{itemize}
  \item We have seen above that, to a Finsler metric, we can associate a volume and a Riemannian metric via the symbol of the Finsler--Laplace operator. This lemma tells us that conversely a volume together with a Riemannian metric give a Laplace-like operator. This shows that we could have a wealth of Finsler--Laplace operators --- just associate a Riemannian metric and a volume to a Finsler metric --- but not all of them are natural.
 \item As there are many more Finsler metrics than pairs (volume/Riemannian metric), this lemma shows that many Finsler metrics will share the same Finsler--Laplacian. 
\item A related question raised by Yves Colin de Verdi\`ere was to determine the range of pairs (volume/Riemannian metric) that can be obtained from a Finsler metric. We prove that we get everything in the case of surfaces (Proposition \ref{prop:Randers_donne_couple}), but we do not know the general answer.
 \item The operators of the type $\Delta_{g,\omega}$ seem to have been introduced by Chavel and Feldman \cite{ChavelFeldman:Isoperimetric_constants} and Davies \cite{Davies:Heat_kernel_bounds}. They are called \emph{weighted Laplace operators} and have been quite widely studied (see, for instance \cite{Grigoryan:heat_kernels_on_weighted_manifolds}).
 \item Up to now we only considered our operators as acting on  $C^{\infty}(M)$. In the next section, for the purpose of spectral theory, we will start considering them as unbounded operators on $L^2(M,\omega)$. By considering the Friedrich extension, the above result stays true replacing ``symmetric'' by ``self-adjoint'' (we recall the definitions in Appendix \ref{app:unbounded_operators}).
\end{itemize}

\begin{proof}
It is evident from the definition of $\Delta_{g,\omega}$ that it is zero on constant functions, that its symbol is $g^{\ast}$ and that for $\varphi, \psi \in C^{\infty}(M)$, 
\begin{equation*}
\int_M \psi \Delta_{g,\omega}\varphi \; \omega = \int_{M} g^{\ast}\left(d\varphi, d\psi \right) \, \omega = \int_M \varphi \Delta_{g,\omega}\psi \; \omega 
\end{equation*}

 Let us now prove uniqueness. Let $\Delta_1$\ and $\Delta_2$ be two second-order differential operators such that they vanish on constant functions and have the same symbol. This implies that there exists a smooth vector field $Z$ on $M$ such that $\Delta_1 - \Delta_2 = L_Z$. 

Now, suppose that both operators are symmetric with respect to $\omega$. We get, $\int_M \varphi L_Z \psi -\psi L_Z \varphi \, \omega =0$ for any $\varphi, \psi \in C^{\infty}(M)$. And taking $\psi = 1$ yields $\int_M L_Z \varphi \, \omega =0$.

But, if $Z$ is not zero, it is easy to construct a function $\varphi \in C^{\infty}(M)$ such that $L_Z \varphi >0$ in any open set that does not contain a singular point of $Z$. So by continuity, $Z$ must vanish.
\end{proof}

An important consequence of this lemma is that any symmetric, elliptic linear second order operator is unitarily equivalent to a Schr\"odinger operator.
\begin{prop}
 Let $\Delta_{g,\omega}$, $v_g$ and $a$ be as above. Define an operator ${U \colon L^2\left(M, \omega\right) \rightarrow L^2\left(M, v_g \right)}$ by $Uf = af$. Then $U \Delta_{g,\omega} U^{-1} = \Delta^{g} + V $ is a Schr\"odinger operator with potential $V = a\Delta_{g,\omega} a^{-1}$.
\end{prop}

\begin{rem}
 This fact shows that the spectral theory of our operator restricts to the theory of Schr\"odinger operators such that the infimum of the spectrum is zero.
\end{rem}

\begin{proof}
 It suffices to show that $U \Delta_{g,\omega} U^{-1} - V$ is symmetric with respect to $\omega$, vanishes on constant functions and has $g^{\ast}$ for symbol, because then Lemma \ref{lem:existence_unicity} proves the claim. It clearly vanishes on constant functions and the symmetry property is obvious by construction. Let $x\in M$ and $\varphi \in L^2\left(M, v_g \right)$ be such that $\varphi (x) = 0$ and $d\varphi_{x} \neq 0$. We have
\begin{equation}
 \left(U \Delta_{g,\omega} U^{-1} - V\right) \varphi^2 (x) = a\Delta_{g,\omega}(\varphi^2 a^{-1})(x) = \Delta_{g,\omega}(\varphi^2)(x)\, .
\end{equation}
Therefore the symbol of $\left(U \Delta_{g,\omega} U^{-1} - V\right)$ is the same as that of $\Delta_{g,\omega}$.
\end{proof}

\subsection{Relation to other Laplacians}

We did not pursue the study of the comparison between this Finsler--Laplace operator and those introduced before by Bao and Lackey \cite{BaoLackey} and Centore \cite{Centore:FinslerLaplacians}. However, we can make the following easy remark.

Suppose that $L$ is a second-order differential operator on a closed manifold, vanishing on constant functions and symmetric with respect to two volumes $\Omega_1$ and $\Omega_2$, then $\Omega_1$ is a constant multiple of $\Omega_2$. Indeed, if we write $\Omega_2 = f \Omega_1$, then for any function $g$,
\begin{equation*}
 0 = \int_M Lg \; \Omega_2 = \int_M (Lg) f \; \Omega_1 = \int_M g Lf \; \Omega_1 \, ,
\end{equation*}
hence $Lf = 0$, so, if $M$ is closed, $f = \text{cst}$.\\

So an easy way to see that our Finsler--Laplace operator is different from Centore's is by remarking that his operator is symmetric with respect to the Busemann-Hausdorff volume and applying the above remark.

\section{Spectral theory}

Most of the results of this section follow from the general theory of elliptic, symmetric operators on compact manifolds. However, we felt that for the convenience of the reader, as well as for the interest of the results, it was worthwhile to give the proofs. Hence we either reproduced or adapted the proofs to our special case.

\subsection{The space $H^1$}

In order to deal with the spectral theory of our operator, we will stop seeing it as acting on $C^{\infty}$ functions but as an unbounded operator on $L^2(M)$. We collected in Appendix \ref{app:unbounded_operators} the basic definitions and the main results that we need.

We start by defining a very useful functional space. Let $C^{\infty}_0(M)$ be the space of smooth functions with compact support on $M$ (so that, if $M$ is boundaryless, then the second condition is empty). Consider the following inner product on $C^{\infty}_0(M)$
\begin{equation*}
 \langle u,v \rangle_{1} = \int_M uv \; \Omega + \int_{HM} L_X\left(\pi^{\ast}u \right) L_X\left(\pi^{\ast}v \right)\; \ada
\end{equation*}
and denote by $\lVert \cdot \rVert_1$ the associated norm.

\begin{defin}
 We denote by $H^1(M)$\ the completion of $C^{\infty}_0 (M)$\ with respect to the norm $\rVert \cdot \lVert_{\empty_{1}}$.
\end{defin}

\begin{rem}
 Using the Riemannian metric given by the symbol of the Laplacian, we have:
\begin{equation*}
 \langle u,v \rangle_{1} = \int_M uv \; \Omega + \int_M \nabla u \nabla v \; \Omega\, .
\end{equation*}
Note also that we do not use the classical notations of $H^1(M)$ and $H^1_0(M)$ for the completion of respectively $C^{\infty}(M)$ and $C^{\infty}_0(M)$. But, as our main focus will later be closed manifolds, we did not feel it worth introducing two notations. 
\end{rem}

The space $H^1(M)$ is a Sobolev space and we have the following embedding result (see \cite[Lemma 3.9.3]{Narasimhan}):
\begin{thm}[Rellich--Kondrachov]
 If $M$\ is compact with smooth boundary, then $H^1(M)$\ is compactly embedded in $L^2(M)$.
\end{thm}

The Finsler--Laplace operator is an unbounded operator on $L^2(M)$ with domain in $H^1(M)$.

\subsection{Energy integral and Rayleigh quotients}

\begin{defin} \label{def:energy}
 For any function $u\in H^1(M)$, we define the \emph{energy} of $u$ by
\begin{equation}
 E(u) := \frac{n}{\voleucl \left(\S^{n-1}\right) } \int_{HM} \left|L_X\left(\pi^{\ast}u \right)\right|^2 \ada
\end{equation}
and the \emph{Rayleigh quotient} by
\begin{equation}
 R(u) := \frac{E(u)}{\int_M u^2\, \Omega}\, .
\end{equation}
\end{defin}

\begin{rem}
 The energy as well as the Rayleigh quotient can also be defined using the cotangent setting, i.e. for any $u\in H^1(M)$, we have 
\begin{equation*}
 E(u) := \frac{n}{\voleucl \left(\S^{n-1}\right) } \int_{H^{\ast}M} \left|L_{X^{\ast}}\left(\hat{\pi}^{\ast}u \right)\right|^2 \bdb,
\end{equation*}
where $B = (\ell_F^{-1})^{\ast} A$ and $X^{\ast}$ is its Reeb field (see Chapter \ref{chap:dynamical_formalism}).
\end{rem}

The Energy we defined is naturally linked to the Finsler--Laplace operator:
\begin{thm}
\label{th:min_energy}
 A function $u \in H^1(M)$\ is a minimum of the energy if and only if $u$\ is harmonic, i.e., $\Delta^F(u) = 0$.
\end{thm}

\begin{rem}
When $M$\ is closed, this just proves that harmonic functions are constant. But this result stays true \emph{without restrictions} on the manifold and is therefore fundamental when used on manifolds with boundary.
\end{rem}

\begin{proof}
 Let $u,v \in H^1(M)$\ we want to compute $\frac{d}{dt} E(v+tu)$. Let $c_n = \dfrac{n}{\voleucl\left(\S^{n-1}\right) } $, we have
\begin{equation}
\label{calcul_E_v_tu}
 E(v+tu) = c_n \int_{HM} \! \left(L_X \pi^{\ast}v\right)^2 +2 t L_X \pi^{\ast}v L_X \pi^{\ast}u +t^2 \left(L_X \pi^{\ast}u\right)^2 \; \ada,
\end{equation}
and therefore
\begin{eqnarray*}
 \frac{d}{dt}\left(E(v+tu) \right)_{|_{t=0}} &=& 2 c_n \int_{HM} L_X \pi^{\ast}v L_X \pi^{\ast}u  \, \ada.
\end{eqnarray*}
Applying the Finsler--Green formula (Proposition \ref{prop:green_formula}, note that $u \in H^1(M)$\ implies that $u|_{\partial M} =0$\;hence the Finsler--Green formula applies without modifications even when $M$\ has a boundary), we obtain
\begin{equation*}
 \frac{d}{dt}\left(E(v+tu) \right)_{|_{t=0}} = 2 \int_{HM} u \Delta^F v  \, \Omega^F.
\end{equation*}
So, if $v$\ is harmonic, then it is a critical point of the energy, and \eqref{calcul_E_v_tu} shows that it must be a minimum. Conversely, if $v$\ is a critical point, then, for any $u \in H^1(M)$, $\langle \Delta^F v, u \rangle = 0$, which yields $\Delta^F v = 0$. \qedhere

\end{proof}

\subsection{Spectrum}

\begin{thm} \label{thm:spectre_discret} 
Let $M$\ be a compact manifold and $F$\ a Finsler metric on $M$.
 \begin{enumerate}
  \item The set of eigenvalues of $-\Delta^F$\ consists of an infinite, unbounded sequence of non-negative real numbers $ \lambda_0 < \lambda_1 < \lambda_2 < \dots$.
  \item Each eigenvalue has finite multiplicity and the eigenspaces corresponding to different eigenvalues are $L^2\left(M, \Omega \right)$-orthogonal.
  \item The direct sum of the eigenspaces is dense in $L^2\left(M, \Omega \right)$\ for the $L^2$-norm and dense in $C^k\left(M\right)$ for the uniform $C^k$-topology.
 \end{enumerate}
\end{thm}

For the convenience of the reader, we will give the adaptation to this special case of the two classical proofs of Riemannian geometry (see, for instance \cite{Ber:SpectralGeo}). The first one uses the whole machinery of the theory of unbounded operators, the second one uses the min-max principle and is, in my opinion, more agreeable.

Before starting the proof, we wish to recall one special characteristic of elliptic operators (see, for instance \cite[Theorem 3.9.1]{Narasimhan}): 
\begin{thm}[Elliptic Regularity Theorem]
Let $L$\ be an elliptic operator on a smooth manifold $M$. If $s \in H^1(M)$ is such that $Ls \in C^{\infty}\left(M\right)$, then $s \in C^{\infty}\left(M\right)$.
\end{thm}
One classical problem of partial differential equations is, given a differential operator $L$ on $M$ and a smooth function $w$, to find a smooth function $u$ such that $Lu = w$ and $u$ verifying some additional conditions on the boundary of $M$ if it exists. Finding solutions to this problem is generally hard, but can sometimes be simplified by weakening the expected regularity of $u$. These solutions are called \emph{weak solutions}. Now, if the operator is elliptic, then finding weak solutions is sufficient due to the above regularity theorem! This technique of finding weak solutions is applied in the proofs to follow. (We did not wish to go into any detail about weak/classical solutions as it is not the main concern here and it can be found in many books on PDEs, see for instance \cite{GilTru,Showalter}.)

\begin{proof}[Proof of Theorem \ref{thm:spectre_discret}: first method]
The operator $-\Delta^F$ is a positive, symmetric operator, so its Friedrich extension $\mathfrak{F}$ exists (see Theorem \ref{thm:Friedrich_extension}). The operator $\mathfrak{F}$ is closed, positive, and self-adjoint, so its spectrum is in $\R^+$. Now take any $\mu$ not in $\R^+$, by definition, the resolvent $R_{\mu}:= \left( \mathfrak{F} -\mu \id \right)$ is a bounded operator from $L^2(M)$ onto the domain of $\mathfrak{F}$ which is $H^1$. By Rellich-Kondrachov Theorem, the embedding $H^1(M)\hookrightarrow L^2(M)$ is compact, so $R_{\mu}$ is a compact operator. Therefore, by the classical result on the spectrum of compact operators (see for instance \cite[Theorem VI.16]{ReedSimon:I}), we deduce that $\mathfrak{F}$ has an infinite, unbounded, discrete spectrum. Finally, as $-\Delta^F$  is elliptic, so is $\mathfrak{F}$, hence all its eigenfunctions are smooth (by the elliptic regularity theorem) and so belong to the domain of $-\Delta^F$.
\end{proof}

The following proof is classical for the Laplace--Beltrami operator (and stays true in a much wider setting, see \cite{ReedSimon:IV}). Here, we adapted the proof given in \cite{Ber:SpectralGeo}.
\begin{proof}[Proof of Theorem \ref{thm:spectre_discret}: The Min-Max method]
Let
 $$
 \mu_1= \inf\lbrace R(u)\; : \;  u \in H^1(M), \; \int_M |u|^2 \Omega \neq 0 \rbrace.
 $$
As $R(u) \geq 0$, $\mu_1$\ exist. Using the Rellich-Kondrachov theorem, we show that from a sequence $(u_n) \in H^1(M)$ such that $R(u_n)$\ tends to $\mu_1$, we can extract a subsequence converging in $\L^2(M)$ to a function in $H^1(M)$. Therefore, $\mu_1$ is realized in $H^1(M)$. Define $E_1$\ as all $v\in H^1(M)$\ such that $R(v)=\mu_1$\ or $v=0$. The second step of the proof is to show that $E_1$\ is an eigenspace for $\Delta^F$. 
\begin{claim}
\label{claim:carac_eigenspace}
 We have the following characterization:
 \begin{equation}
 \label{eq:carac_eigenspace}
 \left( v \in E_1 \right) \quad \Leftrightarrow \quad \left( \forall u \in H^1(M), \;  \langle u , v \rangle_{\empty{1}} =  \left(\frac{\mu_1}{c_n} + 1 \right) \langle u , v \rangle \right).
 \end{equation}
\end{claim}

\begin{proof}[Proof of Claim]
The implication from right to left is trivial, just take $u=v$, so we focus on the other implication.\\
Choose any $v \in E_1$, $u\in H^1(M)$\ and $t\in \R$\ sufficiently small, we have 
$$R\left( v +tu\right) \geq R(v) = \mu_1,$$
so the derivative at $t=0$\ of the function $t \mapsto R\left( v +tu\right)$\ is zero. We denote $c_n := n \left(\voleucl \left(\mathbb{S}^{n-1}\right) \right)^{-1}$, direct computation gives 

\begin{align*}
  R\left(v +tu\right) &= c_n \frac{\int_{HM} \left(L_X v\right)^2 + 2t L_Xv L_Xu + o(t) \; \ada }{\int_M  v^2 + 2t uv +o(t) \; \Omega } \\
   &= \frac{c_n}{ \int_M v^2 \; \Omega} \left( \int_{HM} \left(L_X v\right)^2 + 2t L_Xv L_Xu + o(t) \; \ada\right) \\
   & \times \left(1- \frac{1}{\int_M v^2 \; \Omega} \int_M 2tuv \; \Omega +o(t) \right) \\
   &= R(v) + 2t \frac{c_n}{ \int_M v^2 \; \Omega} \left(- \frac{\int_{HM} \left(L_X v\right)^2 \; \ada }{\int_M v^2 \; \Omega } \int_M uv \; \Omega \right. \\
    & + \left.  \int_{HM} L_Xv L_Xu \; \ada \right) +o(t) \\
    &= R(v) + \frac{2t}{\int_M v^2 \; \Omega} \left(- \mu_1 \langle u,v \rangle +  c_n \left( \langle u,v\rangle_{\empty_{1}} - \langle u,v \rangle \right) \right)  + o(t) \, .
\end{align*}

 So writing that the term in $t$\ is $0$\ yields the characterization (\ref{eq:carac_eigenspace}).
\end{proof}

The Claim \ref{claim:carac_eigenspace} shows first that $E_1$\ is a vector space. Furthermore, as the $\lVert. \rVert_1$-norm and the $\lVert. \rVert$-norm are proportional on $E_1$, using once more the Rellich-Kondrachov theorem, it shows that the balls of $E_1$\ are compact in $L^2(M)$, so $E_1$\ is finite dimensional.

Using the Finsler-Green formula (Proposition \ref{prop:green_formula}), we can rewrite Equation \eqref{eq:carac_eigenspace}\ as 
 \begin{equation}
 \left( v \in E_1 \right) \quad \Leftrightarrow \quad \left( \forall u \in H^1(M), \; \langle \Delta v , u \rangle_0 = -\mu_1 \langle u,v\rangle_0 \right)\, ,
 \end{equation}
that is, an element of $E_1$\ is a weak solution of the closed eigenvalue problem, for the eigenvalue $\mu_1$. So the elliptic regularity theorem yields that $E_1$\ is inside $C^{\infty}(M)$, and that all the weak solutions are in fact classical solutions.\\
To get the next eigenvalue, set $H_1$\ (resp. $L_1$) the orthogonal complement of $E^1$\ with regard to $\lVert. \rVert_1$\ (resp. $\lVert. \rVert$) and we define 
$$
 \mu_2 = \inf \lbrace R(u) : u\in H_1, u\neq 0 \rbrace.
$$
Claim \ref{claim:carac_eigenspace} shows that $H_1 \subset L_1$, and these spaces are closed in $H^1(M)$\ and $L^2(M)$\ respectively. So, the inclusion $H_1 \subset L_1$\ is closed and we can apply the same arguments to show that $\mu_2$\ is attained, and that the space of functions $E_2$\ that realizes this minimum is the eigenspace for $\mu_2$. By definition, $\mu_2 > \mu_1$. Redoing the previous steps, we construct a sequence of eigenvalues $0 \leq \mu_1 < \mu_2 < \mu_3 ...$\ together with a sequence of associated, mutually orthogonal, finite dimensional subspaces of $H^1(M)$, $E_1,E_2, \dots$. The sequence is necessarily infinite because $H^1(M)$\ is infinite dimensional and the $E_i$\ are only finite dimensional.

To finish the proof of the first part of the theorem, we must show that the sequence $\lbrace \mu_i \rbrace$\ is unbounded.  If it was bounded by a real number $\mu$, we could take an infinite sequence $\lbrace \phi_i \rbrace$\ of orthonormal functions in $L^2(M)$\ such that $R(\phi_i) \leq \mu$\ for every $i$\ (just take a $L^2$-orthonormal basis of every $E_i$'s), so $\lVert \phi_i \rVert_1 \leq \mu +1$, and as the inclusion $H^1(M) \hookrightarrow L^2(M)$\ is compact we would obtain a subsequence of orthonormal functions which converges in $L^2(M)$, which is absurd. 
\end{proof}
 
The advantage of the above proof is that it gives an actual expression for eigenvalues. Let us summarize it in the following:
\begin{thm}[Min-Max principle] \label{thm:Min-Max}
If $M$ is a compact manifold, then the first eigenvalue of $-\Delta^F$ is given by
\begin{equation*}
\lambda_0 =  \inf \left\{ R(u) \mid u \in H^1(M) \right\},
\end{equation*}
and its eigenspace $E_0$ is the set of functions realizing the above infimum. The following eigenvalues are given by
\begin{equation*}
 \lambda_k= \inf \left\{ R(u) \; \biggl| \; u \in \bigcap_{i=1}^{k-1} E_i^{\perp} \right\} ,
\end{equation*}
where their eigenspaces $E_k$ are given by the set of functions realizing the above infimum.
\end{thm}

\begin{rem}
 In particular, if $M$ is closed, the first non-zero eigenvalue is 
\begin{equation*}
 \lambda_1 = \inf \lbrace R(u) \mid u \in H^1(M), \; \int_M \!\! u \; \Omega = 0 \rbrace\, .
\end{equation*}
\end{rem}

The Min-Max principle admits another formulation:
\begin{thm}[Min-Max principle (bis)] \label{thm:Min-Max_bis}
 Let $M$ be a compact manifold and $\lambda_k$ the $k^{\text{th}}$-eigenvalue (counted with multiplicity) of $-\Delta^F$, then
\begin{equation}
 \lambda_k = \inf_{V_k} \sup_{u \in V_k} R(u),
\end{equation}
where the infimum is taken over all the $k$-dimensional subspaces of $H^1(M)$ when $M$ is closed, and $k+1$-dimensional when $M$ has a non-empty boundary.
\end{thm}

\begin{proof}
 The proof given in \cite{Ber:SpectralGeo} applies without any modifications.
\end{proof}

When $M$ is non-compact, we do not have such a nice spectra. However, we still have the following:
\begin{prop} \label{prop:spectral_gap}
 The infimum of the essential spectrum of $-\Delta^F$ on a non-compact manifold is given by
\begin{equation}
 \lambda_1 = \inf \lbrace R(u) \mid u \in H^1(M) \rbrace\, .
\end{equation}
\end{prop}

\begin{proof}
 This result is a consequence of the Friedrich extension Theorem (see Theorem \ref{thm:Friedrich_extension}).
\end{proof}

\section{Behavior under conformal change}

As for the Laplace--Beltrami operator, the Energy allows us to give a simple proof that the Laplacian is a conformal invariant only in dimension $2$.

\begin{thm}
\label{thm:inv_conforme}
Let $(\Sigma,F)$\ be a Finsler surface, $f \colon \Sigma \xrightarrow{C^{\infty}} \R$\ and $F_f = e^f F$. Then,
\begin{equation*}
 \Delta^{F_f} = e^{-2f} \Delta^F.
\end{equation*}

\end{thm}

We first prove the following result:
\begin{prop}
 Let $(M,F)$\ be a Finsler manifold of dimension $n$, $f \colon M \xrightarrow{C^{\infty}} \R$\ and $F_f = e^f F$. Set $E_f$\ the Energy associated with $F_f$. Then, for $u\in H^1\left(M\right)$
\begin{equation*}
 E_f(u) = c_n \int_{HM} e^{(n-2)f} \left(L_X \pi^{\ast}u \right)^2 \, \ada,
\end{equation*}
where $c_n = n \left(\voleucl \left(\mathbb{S}^{n-1}\right) \right)^{-1}$.\\
In particular, when $n=2$\ the Energy is a conformal invariant.
\end{prop}

\begin{proof}
 The subscript $f$\ indicates that we refer to the object associated with the Finsler metric $F_f$. The vector field $X_f$\ is a second-order differential equation, so (see Lemma \ref{lem:second_order_ode}) there exists a function $m \colon HM \rightarrow \R$\ and a vertical vector field $Y$\ such that 
$$
X_f = m X + Y.
$$
We have already seen (in Section \ref{subsec:angle_and_conformal_change}) that $A_f = e^f A$\ and that 
$$A_f \wedge dA_f^{n-1} = e^{nf} \ada.$$
 Using $A_f\left(X_f\right)=1$\ and that $VHM$\ is in the kernel of $A$, we have
\begin{equation*}
 1 = e^f A\left(mX +Y\right) = e^f m A\left(X\right) = e^f m\,.
\end{equation*}
Now,
\begin{align*}
 E_f(u) &= c_n \int_{HM} \left(L_{X_f} \pi^{\ast} u \right)^2 \, A_f \wedge dA_f^{n-1}, \\
        &= c_n \int_{HM} \left(L_{m X + Y} \pi^{\ast} u \right)^2   e^{nf} \, \ada, \\
        &= c_n \int_{HM}  e^{nf}\left(mL_{ X } \pi^{\ast} u + L_{Y} \pi^{\ast} u \right)^2   \, \ada.
\end{align*}
As $u$\ is a function on the base and $Y$\ is a vertical vector field, $L_{Y} \pi^{\ast} u =0$. So the preceding equation becomes
\begin{align*}
 E_f(u) &= c_n \int_{HM}  e^{nf} m^2 \left(L_{ X } \pi^{\ast} u  \right)^2   \, \ada, \\
   &= c_n \int_{HM}  e^{(n-2)f}  \left(L_{ X } \pi^{\ast} u  \right)^2   \, \ada. \qedhere
\end{align*}

\end{proof}

\begin{proof}[Proof of Theorem \ref{thm:inv_conforme}]
 Let $u,v \in H^1\left(\Sigma\right)$, we have already shown (Theorem \ref{th:min_energy}) that $\frac{d}{dt}\left(E(v+tu) \right)_{|_{t=0}} = -2 \int_{\Sigma} u \Delta^{F} v \; \Omega $.\\
The conformal invariance of the Energy yields: for $u,v\in H^1 \left(\Sigma\right)$,
\begin{equation*}
 -2 \int_{\Sigma} u \Delta^{F} v \; \Omega = -2 \int_{\Sigma} u \Delta^{F_f} v \; \Omega_f = -2 \int_{\Sigma} e^{2f} u \Delta^{F_f} v \; \Omega\,,
\end{equation*}
where we used $\Omega_f = e^{2f} \Omega$\ (see Equation \eqref{eq:omega1}) to obtain the last equality. We can rewrite this last equality as: for $u,v\in H^1\left(\Sigma\right)$,
\begin{equation}
 \langle \left(\Delta^{F} - e^{2f} \Delta^{F_f} \right)v , u \rangle = 0,
\end{equation}
which yields the desired result.
\end{proof}

\begin{subappendices}
 \section{Unbounded operators and the Friedrich extension} \label{app:unbounded_operators}

We give here a quick presentation of unbounded operators and the results that we used above. We refer to \cite{Kato:Perturbation_theory} or to the series of books \cite{ReedSimon:I,ReedSimon:II,ReedSimon:III,ReedSimon:IV} for the details and the proofs.

\subsection{Some basic definitions}

Let $H$ be a Hilbert space and $\langle \cdot, \cdot \rangle$ its scalar product.
\begin{defin}
  An unbounded operator $L$ is a linear map from a dense linear subspace of $H$, called the \emph{domain} of $L$ and denoted by $D(L)$, into $H$. The operator $L_1$ is an \emph{extension} of $L$ if $D(L_1) \supset D(L)$ and if $L_1 \varphi = L \varphi$ for any $\varphi \in D(L)$.
\end{defin}

\begin{defin}
 An unbounded operator $L$ is called \emph{closed} if its graph $\Gamma(L)$ is closed in $H \times H$. The graph of $L$ is the set of pairs
\begin{equation*}
 \Gamma(L):= \lbrace \left( \psi , L \psi \right) \mid \psi \in D(L) ^rbrace \subset H \times H \,.
\end{equation*}
The operator $L$ is called \emph{closable} if there exists a closed extension of $L$.
\end{defin}

\begin{defin}
 Let $L$ be a closed unbounded operator on $H$ and $D(L)$ its domain. A complex number $\lambda$ is said to be in the \emph{resolvent set} for $L$ if $L-\lambda \id$ is a bijection from $D(L)$ to $H$ such that its inverse is bounded. We denote the resolvent set of $L$ by $\rho(L)$, and $R_{\lambda}(L) := \left(L-\lambda \id\right)^{-1}$ is called the \emph{resolvent} of $L$. The set $\sigma(L):= \C \smallsetminus \rho(L)$ is the \emph{spectrum}.
\end{defin}

The \emph{point spectrum} consists of those elements in the spectrum which are eigenvalues.

\subsection{The Friedrich extension}

\begin{defin}
 Let $L$ be an unbounded operator on $H$. Let $D(L^{\ast})$ be the set of $\varphi \in H$ for which there exists a $\psi \in H$ such that
\begin{equation*}
 \langle L \rho , \varphi \rangle =\langle \rho , \psi \rangle, \quad \text{for } \rho \in D(L).
\end{equation*}
For any such $\varphi \in D(L^{\ast})$, we set $L^{\ast}\varphi := \psi $. $L^{\ast}$ is called the \emph{adjoint} of $L$.
\end{defin}

Note that we always have $D(L) \subset D(L^{\ast})$, this will provide the difference between being symmetric and self-adjoint:
\begin{defin}
The operator $L$ is called \emph{symmetric} if $L^{\ast}$ is an extension of $L$, i.e., if for $\varphi, \psi \in D(L)$
\begin{equation*}
\langle L \psi , \varphi \rangle =\langle \psi , L\varphi \rangle\, .
\end{equation*}
The operator $L$ is called \emph{self-adjoint} if $L = L^{\ast}$, i.e., if $L$ is symmetric and $D(L)=D(L^{\ast})$.
\end{defin}

There are many criteria to determine whether a symmetric operator is self-adjoint (see \cite{ReedSimon:I}) however we do not recall them as we will not use them directly. Indeed, for the operator studied in this thesis there exists a self-adjoint extension called the \emph{Friedrich extension} (see below Theorem \ref{thm:Friedrich_extension}). But, before stating that result, we will need some more definitions.

\begin{defin}
 The operator $L$ is called \emph{semi-bounded} if there exists a positive constant $C$ such that $\langle L \psi , \psi \rangle \geq -C \lVert \psi \rVert^2$ for all $\psi \in D(L)$.\\
When $C$ can be taken to be $0$, then we say that $L$ is positive.
\end{defin}

\begin{defin}
 An unbounded quadratic form $q$ on $H$ is a bilinear form $q \colon D(q) \times D(q) \rightarrow \R$, where $D(q)$ is a dense linear subspace of $H$. It is called \emph{symmetric} if $q(\varphi,\psi) = q(\psi,\varphi)$ for any $\varphi,\psi \in D(q)$ and \emph{semi-bounded} if there exists a positive constant $C$ such that $q(\psi,\psi) \geq -C \lVert \psi \rVert^2$ for all $\psi \in D(L)$.\\
\end{defin}

As we defined an extension of an unbounded operator, we define in the same manner an extension of a quadratic form.
To a symmetric operator $L$, we can naturally associate a symmetric quadratic form, just by letting $q_L (\varphi, \psi ) = \langle \varphi , L \psi \rangle$ for any $\varphi,\psi \in D(L)$. For bounded operators, the Riesz Lemma gives a one-to-one correspondence, but this is not always true for unbounded quadratic forms.

\begin{defin}
 Let $q$ be a semi-bounded quadratic form, $q(\psi,\psi) \geq -C \lVert \psi \rVert^2$ for all $\psi \in D(L)$. It is called \emph{closed} if $D(q)$ is complete under the norm:
\begin{equation*}
 \lVert \psi \rVert_{C+1} = \sqrt{q(\psi,\psi) + (C+1) \lVert \psi \rVert^2}.
\end{equation*}
 A semi-bounded quadratic form $q$ is called \emph{closable} if there exists a closed extension $\overline{q}$ of $q$
\end{defin}

\begin{thm}[Friedrich extension] \label{thm:Friedrich_extension}
 Let $L$ be a symmetric semi-bounded operator and let $q(\varphi, \psi ) := \langle L \psi , \varphi \rangle$ for $\psi, \varphi \in D(L)$. Then $q$ is a closable quadratic form, its closure $\hat{q}$ is the quadratic form associated with a self-adjoint operator $\hat{L}$. The operator $\hat{L}$ is the unique self-adjoint extension of $L$ with domain in $D(\hat{q})$. Furthermore, the lower bound of its spectrum is the lower bound of $q$.
\end{thm}

\begin{proof}
 See, for instance, \cite[Theorem X.23]{ReedSimon:II} or \cite[p. 325ff.]{Kato:Perturbation_theory}
\end{proof}

\end{subappendices}

\chapter[Explicit representations and spectra]{Explicit representations and computations of spectra}

In this chapter, we give explicit representations for our Finsler--Laplace operator. We start with general Randers metrics on surfaces and then give explicit spectra for some Katok-Ziller metrics.

\section{The Finsler--Laplace operator for Randers metrics}

\subsection{Some generalities on Randers metrics}

Among the classical examples of non-Riemannian Finsler metrics, the Randers metrics play an important role, arise naturally in physics (\cite{Randers}) and have been widely studied. Let us give the definition:

Let $g$ be a Riemannian metric on $M$ and $\theta$ a $1$-form on $M$, define, for $(x,v) \in TM$
\begin{equation*}
 F\left(x,v\right) = \sqrt{g_x(v,v)} + \theta_x(v).
\end{equation*}
If the norm of $\theta$ with respect to $g$ is strictly less than one, than $F$ is a Finsler metric (see \cite{BaoChernShen}) and it is called a \emph{Randers metric}.

Note that Randers metrics are never reversible if they are not Riemannian.\\

One advantage of Randers metrics in our case is that it is particularly easy to compute the volume and angle forms.
\begin{prop} \label{prop:volume_angle_Randers}
 Let denote by $\Omega_0$, $\alpha_0$ and $X_0$ the volume form, angle form and geodesic flow associated with the Riemannian metric $g$. We have
\begin{align*}
 \Omega^F &= \Omega_0\,, \\
 \alpha^F &= \left(1 + \pi^{\ast}\theta(X_0) \right) \alpha_0 \,.
\end{align*}
\end{prop}

\begin{rem}
 The above result holds in general whenever a Finsler metric $F$ differs from a \emph{reversible} Finsler metric $F_0$ by a $1$-form $\theta$.\\
The fact that the Holmes-Thompson volume for Randers metrics is equal to the Riemannian volume is already known (see for instance \cite{ChengShen}) but maybe not very widely.
\end{rem}

 \begin{proof}
 By definition of $A$ (see Chapter \ref{chap:dynamical_formalism}), for any $\xi \in H_{x}M$ and ${Z \in T_{\xi}HM}$, we have
\begin{align*}
 A_{\xi}(Z) &= \lim_{\eps \rightarrow 0} \frac{F_0\left(x,v + \eps d\pi (Z)\right) -F_0\left(x,v \right) + \eps \theta_x\left( d\pi (Z) \right)}{\eps} \\ 
  &= \left. A_0 \right._{\xi}(Z) + \pi^{\ast}\theta(Z).
\end{align*}
From now on we will write $\theta$ instead of $\pi^{\ast}\theta$ as it will simplify notations and hopefully not lead to any confusion. Using this notation, we have: $A = A_0 + \theta$ and therefore $dA = dA_0 + d\theta$.\\
Note that $dA^{n-1} = dA_0^{n-1} + T$ where $T$ is a $(2n-2)$-form. So, as $d\theta$ is a $2$-form vanishing on $VHM$, and for $Y_1, Y_2 \in VHM$, $i_{Y_1} i_{Y_2} dA_0 = 0$, $T$ can be given at most $n-2$ vertical vectors, i.e., if $Y_1, \dots, Y_{n-1} \in VHM$, then $i_{Y_1} \dots i_{Y_{n-1}} T = 0$. Now this implies that the top-form $A\wedge T$ vanishes, hence $\ada = (A_0 + \theta)\wedge dA_0^{n-1}$. As $\ada$ and $A_0 \wedge dA_0^{n-1}$ are both volume forms, there exists a function $\lambda$ such that $\ada = \lambda A_0 \wedge dA_0^{n-1}$. We have 
$$
i_{X_0}(\ada) = (1+\theta(X_0)) dA_0^{n-1} = \lambda dA_0^{n-1},
$$
therefore $\lambda = 1+\theta(X_0)$.

Let $\alpha^{\Omega_0}$ be defined by $\alpha^{\Omega_0} \wedge \pi^{\ast}\Omega_0 = \ada$ (it exists by Lemma \ref{lem:existence_angle}). We have $\alpha^{\Omega_0} \wedge \pi^{\ast}\Omega_0 = \lambda A_0\wedge dA_0^{n-1} = \lambda \alpha_0 \wedge \pi^{\ast}\Omega_0 $, hence $\alpha^{\Omega_0}$ and $\lambda \alpha_0$ coincide on $VHM$. It is then immediate (see Section \ref{subsec:construction}) that 
\begin{equation*}
 \Omega = \frac{\int_{H_xM} (1+\theta(X_0)) \alpha_0}{\voleucl\left(\S^{n-1}\right)}  \Omega_0\, .
\end{equation*}
As the metric $g$ is Riemannian, it is reversible, therefore $\int_{H_xM} \theta(X_0) \, \alpha_0$ must be zero. Hence $\Omega = \Omega_0$ and $\alpha = (1+\theta(X_0)) \alpha_0$.
 \end{proof}

With our knowledge of the angle and volume forms for Randers metrics, we can give a more explicit expression of our Finsler--Laplace operator. Indeed, recall that there exists a function $m:HM \rightarrow \R$ and a vertical vector field $Y_0$ such that $X =mX_0 +Y_0$ (because $X$ and $X_0$ are second-order differential equations, see Section \ref{subsec:Finsler_metric_and_geodesic_flow}). 

As $1 =A(X) = m A_0(X_0) + m\theta(X_0)$, we get that $m = (1 + \theta(X_0))^{-1}$. Hence, we can rewrite the Finsler--Laplace operator as
\begin{multline} \label{eq:laplacien_pour_randers_general}
\Delta^F f(p) = \frac{n}{\voleucl(\S^{n-1})}\left( \int_{H_p M} \frac{1}{1+\theta(X_0)} L_{X_0}^2(\pi^{\ast}f) \: \alpha_0  \right. \\
 \left.  + \int_{H_p M}  \left(\frac{-L_{X_0}\left(\theta(X_0)\right)}{(1+\theta(X_0))^2} + \frac{-L_{Y_0}\left( \theta(X_0) \right) }{(1+\theta(X_0))^3}  \right) L_{X_0}(\pi^{\ast}f) \: \alpha_0 \right).
\end{multline}

We agree that this formula does not represent a great improvement compared to the definition. However, note that the symbol of $\Delta^F$ is determined only by the first integral, so to obtain a coordinate expression, we can just compute the symbol and then use the formula in Lemma \ref{lem:existence_unicity} and the traditional Riemannian local expressions.\\
 A particularly easy case arises from Randers metrics that are also Minkowski metrics (see Remark \ref{rem:Minkowsky_randers}).

\subsection{The symbol of the Laplacian on Randers surfaces}

This part is unfortunately a quite heavy computational section, but our justification for doing so is two-fold. 

First, we want to compute explicitly the symbol for any Randers metric on a surface to prove that it can be done by hand. Furthermore, this expression indicates what kind of Riemannian metric we can expect to come from a Randers metric via the symbol of its Finsler-Laplacian. 

Which brings us to our second goal. Namely, to give an answer, in the case of $2$-dimensional manifolds, to the question of which pairs (Riemannian metric, volume) can be obtained from a Finsler metric via the Finsler-Laplacian (see Section \ref{subsec:a_characterization}).

\begin{prop}\label{prop:symbol_laplacian_randers}
 Let $F= \sqrt{g} + \theta$ be a Randers metric on a $2$-dimensional manifold, $(x,y)$ be normal coordinates for $g$ at $p$, $\left(\sigma_{ij}(p)\right)$ be the symbol at $p$ of $\Delta^F$. We write $\theta= \theta_x dx + \theta_y dy$ and  $T:= \frac{\theta_x -i \theta_y}{2} =|T|e^{i\varphi}$. We have:
\begin{align}
 \sigma_{11}(p) &= \frac{1}{\sqrt{1-\lVert\theta\rVert^2}}\left( 1 + \cos\left(2\varphi\right) \left( \frac{1-\sqrt{1-\lVert\theta\rVert^2}}{1+\sqrt{1-\lVert\theta\rVert^2}}\right) \right) \label{eq:sigma11} ,\\
\sigma_{22}(p) &= \frac{1}{\sqrt{1-\lVert\theta\rVert^2}} \left( 1 - \cos\left(2\varphi\right) \left(\frac{1-\sqrt{1-\lVert\theta\rVert^2}}{1+\sqrt{1-\lVert\theta\rVert^2} } \right) \right)\label{eq:sigma22} , \\
 \sigma_{12}(p) &=  \frac{ \sin \left(2\varphi\right)}{\sqrt{1-\lVert\theta\rVert^2}} \left( \frac{1-\sqrt{1-\lVert\theta\rVert^2}}{1+\sqrt{1-\lVert\theta\rVert^2}}\right) \label{eq:sigma12},
\end{align}
where $\lVert \theta \rVert$ is the norm of $\theta$ with respect to the Riemannian metric $g$.
\end{prop}

\begin{rem}
 One interest of this result is that it shows that for a given Riemannian metric $g$, no two Randers metrics $F= \sqrt{g} + \theta$ gives the same symbol, hence the same Laplacian. We can also obtain a full coordinate expression of the Laplacian for Randers surfaces by using the formula given in Lemma \ref{lem:existence_unicity}. However, I doubt that it would yield much information in the general case, so I have not done it.
\end{rem}

So now, given a couple $(g_{\text{goal}},\Omega_{\text{goal}} )$ our aim is to give a condition for the existence of a Randers metric $F = \sqrt{g} + \theta$ such that $\Omega^F = \Omega_{\text{goal}}$ and, if we denote by $g_{\sigma}$ the dual of the symbol of $\Delta^F$, then   $g_{\sigma} = g_{\text{goal}}$.\\
We remark that the above local expression of the symbol already gives a condition on the volumes.
\begin{cor} \label{cor:utilisation_unique3}
 Let $F= \sqrt{g} + \theta$ be a Randers metric on a $2$-manifold, and $g_{\sigma}$ the dual of the symbol of $\Delta^F$. Then
\begin{equation*}
 \Omega^{g_{\sigma}} = \sqrt{\frac{\sqrt{1-\lVert \theta \rVert^2}\left(1 + \sqrt{1-\lVert \theta \rVert^2} \right)^2 }{4} } \Omega^F.
\end{equation*}
\end{cor}

Using this result, we see that the norm of $\theta$ is uniquely determined by the quotient $\Omega^{g_{\sigma}}/\Omega^F$, and so, a first condition to get a positive answer to the above question is that the quotient $\Omega^{g_{\text{goal}}} /\Omega_{\text{goal}} $ \emph{can} be realized by the norm of a $1$-form.

A trivial counter-example would be to take any $g_{\text{goal}}$ on the $2$-sphere and $\Omega_{\text{goal}}$ to be a constant multiple of $\Omega^{g_{\text{goal}}}$, as we cannot have a $1$-form on the sphere with constant norm we get a negative answer to our question.

However, this question was spurred by the fact that $\Delta^F = \Delta_{g,\Omega}$ (where $\Delta_{g,\Omega}$ is defined in Lemma \ref{lem:existence_unicity}) and the question was really ``can we obtain every second-order, elliptic, symmetric operators that vanish on the constants from a Finsler--Laplace operator?''.

But, again from Lemma \ref{lem:existence_unicity}, it is clear that $\Delta_{g,\Omega} = \Delta_{g,K \Omega}$ for any constant $K >0$. So we should not be too upset about not being able to get all the possible $\Omega_{\text{goal}}$ but just one in each constant multiple class.\\

So, in order to answer our question, we will have to reconstruct a Randers metric just from the information given by a Riemannian metric of the form obtained in Proposition \ref{prop:symbol_laplacian_randers}. 

We described the $1$-form $\theta$ by its norm and a certain angle $\varphi$, which depends on the normal coordinates we chose. Now, suppose that on a contractible set $U$, we are given two smooth functions $k\colon U \rightarrow \R^+$ and $\varphi \colon U \rightarrow \left[0,2 \pi \right]$, a Riemannian metric $g$ \emph{and} a preferred choice of coordinates $(x,y)$. With this information, we can construct a $1$-form $\theta$ such that its norm for $g$ is $k$ and such that $\varphi$ is the angle computed in the (unique) normal coordinates $(x^n,y^n)$ for $g$ such that $x^n$ and $x$ are collinear. In other words, to reconstruct $\theta$ from its norm $k$ and angle $\varphi$ what we really need is a Riemannian metric and a vector field that is not zero when $k$ is not zero. This remark is all we need to prove the following:

\begin{prop} \label{prop:Randers_donne_couple}
Let $g_1$ be a Riemannian metric on a $2$-manifold $M$ and $\Omega$ a volume form on $M$. Denote by $\mu \colon M\rightarrow \R_+^{\ast}$ the function such that $\Omega^{g_1} = \mu \Omega$. Let $K:= \sup \mu$, suppose that either $M$ is contractible and $K < \infty$ or that $M$ is compact. Then there exists a Randers metric $F$ such that $\Omega = K \Omega^F$ and $g_1$ is the dual of the symbol of $\Delta^F$.
\end{prop}

\begin{rem}
 The proof we give is entirely based on the coordinate expression we obtained before, but it would be much more interesting to have a coordinate-free proof in order to generalize it to any dimension.
\end{rem}

We now get on to the proofs.

\begin{proof}[Proof of Proposition \ref{prop:symbol_laplacian_randers}]
 Equation \eqref{eq:laplacien_pour_randers_general} shows that the symbol of $F$ at $p$ is obtained by computing
\begin{equation*}
 \frac{1}{\pi} \int_{H_pM} \frac{1}{1+\theta\left(X_0\right)} L_{X_0}^2\left(\pi^{\ast}f\right) \alpha_0\, .
\end{equation*}
 In normal coordinates $(x,y, \theta)$ at $p$ on $HM$, it is easy to check that 
\begin{align*}
 (X_0)_p &= \cos \theta \dx + \sin \theta \dy \, , \\
 (A_0)_p &= \cos \theta dx + \sin \theta dy \, , \\
 (A_0 \wedge dA_0)_p &= -d\theta \wedge dx\wedge dy\, .
\end{align*}
Therefore we can write, $\alpha_0 = d\theta$, $\theta\left(X_0\right) = \theta_x \cos \theta + \theta_y \sin \theta$ and $$L_{X_0}^2\left(\pi^{\ast}f\right) = \cos^2 \theta \dfxx + \sin^2 \theta \dfyy +2\cos\theta\sin\theta \dfxy\, .$$
 Hence the symbol at $p$ is
\begin{align*}
 \sigma_{11}(p) &= \frac{1}{\pi} \int_0^{2\pi} \frac{\cos^2 \theta}{1+\theta_x \cos \theta + \theta_y \sin \theta} d\theta \,, \\
 \sigma_{22}(p) &= \frac{1}{\pi} \int_0^{2\pi} \frac{\sin^2 \theta}{1+\theta_x \cos \theta + \theta_y \sin \theta} d\theta\,, \\
 \sigma_{12}(p) &= \frac{1}{\pi} \int_0^{2\pi} \frac{\cos \theta \sin \theta}{1+\theta_x \cos \theta + \theta_y \sin \theta} d\theta\,.
\end{align*}
Unfortunately, a lack of knowledge on my part associated with Maple's unhelpfullness prevented me to give a computerized computation. Therefore I give below the computations by-hand proving the proposition (and would strongly advise the reader to skip the next two pages).\\

{\bf Computation of $\sigma_{22}$:} Recall that $T = \frac{\theta_x -i \theta_y}{2}$\ and let $R :=\frac{\theta_x + i \theta_y}{2}$ and $z=e^{i\theta}$, we have
\begin{equation*}
 \sigma_{22} =   \int_{S^1} \frac{\left( \frac{1}{2i}\left(z-z^{-1} \right) \right)^2}{1 + T z + R z^{-1} } \frac{dz}{iz} =  \frac{i}{4} \int_{S^1} \frac{z^4 -2 z^2 + 1 }{T z^4 + z^3 + R z^2 } dz \,.
\end{equation*}
We are going to apply the Residue Theorem, hence we must found the zeros of the polynomial $T z^4 + z^3 + R z^2$\ and compute the residues.
\begin{equation}
\label{eq:decomp_poly}
T z^4 + z^3 + R z^2 = Tz^2 \left(z-z^-\right) \left(z-z^+\right),
\end{equation}
where
\begin{equation*}
 z^- =  \frac{-1 -\sqrt{1- 4|T|^2} } {2T},\quad z^+ =  \frac{-1 +\sqrt{1- 4|T|^2} } {2T}\,.
\end{equation*}
As we had chosen the $1$-form $\theta$\ with a norm strictly less than $1$, we have $4|T|^2 = \theta_x^2 +\theta_y^2 <1$, so $z^-$\ and $z^+$\ are well-defined reals. Furthermore, as $\theta$\ is non-null, $\left|z^-\right| > 1$, so $z^-$\ is not inside the unit disc. The poles of $\frac{z^4 -2 z^2 + 1 }{T z^4 + z^3 + R z^2 }$\ inside the unit disc are then $0$\ and $z^+$. As
\begin{multline*}
 \frac{z^4 -2 z^2 + 1 }{T z^4 + z^3 + R z^2 }  = \frac{z^2}{T \left(z-z^-\right) \left(z-z^+\right)} - \frac{2}{T \left(z-z^-\right) \left(z-z^+\right)} \\ + \frac{1}{Tz^2 \left(z-z^-\right) \left(z-z^+\right)}\,,
\end{multline*}
we get
\begin{align}
 \res_{z^+}\left( \frac{z^2}{ \left(z-z^-\right) \left(z-z^+\right)} \right) &= \frac{\left(z^+\right)^2}{z^+ - z^-} \,,\label{eq:residu_z+}\\
 \res_{z^+}\left(\frac{2}{ \left(z-z^-\right) \left(z-z^+\right)} \right) &= \frac{2}{z^+ - z^-}, \\
 \res_{z^+}\left( \frac{1}{z^2 \left(z-z^-\right) \left(z-z^+\right)} \right) &= \frac{1}{\left(z^+\right)^2\left(z^+ - z^-\right)}\, , \label{eq:residu_z+2}\\                                                                                                                                                           \res_{0}\left( \frac{1}{z^2 \left(z-z^-\right) \left(z-z^+\right)} \right) &= \frac{z^+ + z^-}{\left(z^+ - z^-\right)^2}\,. \label{eq:residu0}
\end{align}
Then, using the Residue Theorem, we obtain

\begin{align*}
  \int_{S^1} \frac{z^2}{T \left(z-z^-\right) \left(z-z^+\right)} dz &= \frac{2i\pi}{T} \frac{\left(z^+\right)^2}{z^+ - z^-}\,, \\
  \int_{S^1} -\frac{2}{T \left(z-z^-\right) \left(z-z^+\right)} dz &= -\frac{2i\pi}{T} \frac{\left(z^+\right)^2}{z^+ - z^-}\,, \\
  \int_{S^1} \frac{1}{Tz^2 \left(z-z^-\right) \left(z-z^+\right)} dz &= \frac{2i\pi}{T} \left( \frac{1}{\left(z^+\right)^2\left(z^+ - z^-\right)} + \frac{z^+ + z^-}{\left(z^+ - z^-\right)^2} \right)\,, 
\end{align*}
and the sum of these three integrals gives $\sigma_{22}$
\begin{equation}
 \sigma_{22} = \frac{2i^2 \pi}{4 T} \left( \frac{\left(z^+\right)^2}{z^+ - z^-} - \frac{2}{z^+ - z^-} + \frac{1}{\left(z^+\right)^2\left(z^+ - z^-\right)} + \frac{z^+ + z^-}{\left(z^+ - z^-\right)^2} \right).
\end{equation}
To simplify a bit the above equation, note that
\begin{equation*}
 z^+ + z^- = -\frac{1}{T}, \;\; z^+  z^- = \frac{R}{T}, \;\;  z^+ - z^- = \frac{\sqrt{1-|T|^2}}{T} \;\; \text{and} \left(z^+\right)^2 = -\left( \frac{z^+ + R }{T} \right).
\end{equation*}
So
\begin{align*}
\sigma_{22} &= -\frac{ 1}{2 T}\frac{1}{\left(z^+  z^- \right)^2} \left( \left( \left( \left(z^+\right)^2-2 \right)\left(z^+  z^- \right)^2 + \left(\frac{z^+  z^- }{z^+}\right)^2 \right) \frac{1}{z^+ - z^-} +z^+ +z^- \right)  \\
   &= -\frac{ 1}{2 T}\frac{1}{\left(z^+  z^- \right)^2}\frac{1}{z^+ - z^-} \left( \left(z^+\right)^4\left(z^-\right)^2 - 2\left(z^+  z^- \right)^2 + \left(z^+\right)^2 \right) \\
   &=  -\frac{ 1}{2 T}\frac{1}{z^+ - z^-} \left( \left(z^+\right)^2\left(1+\frac{1}{\left(z^+ z^- \right)^2} \right) -2 \right)\\
   &= -\frac{ 1}{2 T} \frac{T}{\sqrt{1-|T|^2}} \left( - \left( \frac{z^+ + R }{T} \right) \left(1+ \frac{T^2}{R^2} \right) -2 \right) \\
  &= \frac{1}{2\sqrt{1-|T|^2}} \left(2 + \left(\frac{z^+}{T} + \frac{R}{T}\right)\left(1+ \frac{T^2}{R^2} \right) \right).
\end{align*}
We define $\varphi$\ as the argument of $T$ so that
\begin{equation*}
 \frac{z^+}{T} = \frac{-1 + \sqrt{1-|T|^2}}{2|T|^2e^{2i\varphi}} \quad \text{and } \frac{R}{T} = e^{-2i\varphi}.
\end{equation*}
Hence
\begin{align*}
 \sigma_{22} &= \frac{1}{2\sqrt{1-|T|^2}} \left(2 + e^{-2i\varphi} \left(\frac{-1 + \sqrt{1-|T|^2} }{2|A|^2} +1 \right)\left( 1 + e^{4i\varphi} \right) \right) \\
 &=  \frac{1}{\sqrt{1-|T|^2}} + \frac{1}{\sqrt{1-|T|^2}} \cos\left(2\varphi\right) \left( 1 + \frac{-1+\sqrt{1-4|T|^2}}{2|T|^2} \right).
\end{align*}
Using that $\lVert\theta\rVert^2 =4|T|^2 =  \left( 1-\sqrt{1-4|T|^2}\right)\left(1+\sqrt{1-4|T|^2}\right)$ gives Formula \eqref{eq:sigma22}.

{\bf Computation of $\sigma_{11}$:} Using the above notations, we have
\begin{equation*}
 \sigma_{11} = -\frac{i}{4\pi} \int_{S^1} \frac{z^4 + 2z^2 +1}{Tz^4 + z^3 + R z^2} dz \, .
\end{equation*}
We need to compute the residues at $0$\ and $z^+$\ of $\frac{z^2}{ \left(z-z^-\right) \left(z-z^+\right)}$, $\frac{2}{ \left(z-z^-\right) \left(z-z^+\right)}$\ and $\frac{1}{z^2 \left(z-z^-\right) \left(z-z^+\right)}$. The values of these residues are given by Equations (\ref{eq:residu_z+}) to (\ref{eq:residu0}). So, applying once again the Residue Theorem, we get
\begin{equation}
 \sigma_{11} = \frac{-2i^2 }{4 T} \left( \frac{\left(z^+\right)^2}{z^+ - z^-} + \frac{2}{z^+ - z^-} + \frac{1}{\left(z^+\right)^2\left(z^+ - z^-\right)} + \frac{z^+ + z^-}{\left(z^+ - z^-\right)^2} \right). \nonumber
\end{equation}
Then, as above, we rewrite this formula to have something that behaves well when $|T|$\ tends to $0$. By doing the same transformations as in the case of $\sigma_{22}$, we obtain

$$
\sigma_{11} = \frac{ 1}{2 T}\frac{1}{z^+ - z^-} \left( \left(z^+\right)^2\left(1+\frac{1}{\left(z^+ z^- \right)^2} \right) + 2 \right).
$$
Simplifying once again, we get
\begin{align*}
 \sigma_{11} &= \frac{ 1}{2 T} \frac{T}{\sqrt{1-|T|^2}} \left( - \left( \frac{z^+ + R }{T} \right) \left(1+ \frac{T^2}{R^2} \right) + 2 \right) \\
  &= \frac{1}{2\sqrt{1-|T|^2}} \left(2 - \left(\frac{z^+}{T} + \frac{R}{T}\right)\left(1+ \frac{T^2}{R^2} \right) \right)\\
 &= \frac{1}{2\sqrt{1-|T|^2}} \left(2 - e^{-2i\varphi} \left(\frac{-1 + \sqrt{1-|T|^2} }{2|A|^2} +1 \right)\left( 1 + e^{4i\varphi} \right) \right) \\
 &=  \frac{1}{\sqrt{1-|T|^2}} - \frac{1}{2\sqrt{1-|T|^2}} \cos\left(2\varphi\right) \left( 2 + \frac{-1+\sqrt{1-4|T|^2}}{|T|^2} \right),
\end{align*}
that we can rewrite as (\ref{eq:sigma11}).\\

{\bf Computation of $\sigma_{12}$:} The now usual transformations give
\begin{equation*}
 2\sigma_{12} = -\frac{1}{2 \pi T} \int_{S^1} \frac{z^2}{\left(z - z^+\right)\left(z - z^-\right)} + \frac{1}{z^2\left(z - z^+\right)\left(z - z^-\right)} dz \,.
\end{equation*}
The residues we need are given by Equations (\ref{eq:residu_z+}), (\ref{eq:residu_z+2}) and (\ref{eq:residu0}), so we obtain
\begin{multline*}
 2\sigma_{12} =  -\frac{i}{ T} \left(\frac{\left(z^+\right)^2}{z^+ - z^-} - \frac{1}{\left(z^+\right)^2\left(z^+ - z^-\right)} - \frac{z^+ + z^-}{\left(z^+ - z^-\right)^2} \right) \\ = -\frac{i\pi}{ T} \frac{\left(z^+\right)^2 }{\left(z^+ - z^-\right)} \left( 1 - \frac{1}{\left(z^+z^-\right)^2} \right).
\end{multline*}
Simplifying in the same fashion as above yields
\begin{align*}
  2\sigma_{12} &= -\frac{i}{ T} \frac{T}{\sqrt{1-4|T|^2}} \frac{z^+ +R}{T}\left(1- \left(\frac{T}{R} \right)^2 \right) \\
    &= -\frac{i}{\sqrt{1-4|T|^2}} \left( \frac{-1 + \sqrt{1-4|T|^2} }{2 |T|^2 e^{2i\varphi}} + e^{-2i\varphi} \right) \left(1-e^{4i\varphi}\right) \\
    &= -\frac{ \sin \left(2 \varphi\right)} {\sqrt{1-4|T|^2}} \left( 2 + \frac{-1 + \sqrt{1-4|T|^2} }{|T|^2} \right).
\end{align*}
\end{proof}

We gave a local expression of $\sigma$ in normal coordinate for $g$. Now we generalize it to any local coordinates, and in the same stroke, we prove Corollary \ref{cor:utilisation_unique3}:

\begin{lem} \label{lem:Randers_coord_general}
 Let $F = \sqrt{g} + \theta$ and $g_{\sigma}$ be the dual of the symbol of $\Delta^F$. Let $|g| := \det g$. Then we have
\begin{equation*}\label{eq:volume_sigma_en_randers}
 |g_{\sigma}| = \frac{\sqrt{1-\lVert \theta \rVert^2}\left(1 + \sqrt{1-\lVert \theta \rVert^2} \right)^2 }{4} |g|\, ,
\end{equation*}
Now let $(x,y)$ be local coordinates on $M$ and write $g = \left[ g_{ij} \right]$ and $g_{\sigma} = \left[ g_{\sigma;ij} \right]$ in this basis. Then
\begin{align*}
 g_{\sigma;11} &= \frac{g_{11}}{4} \left( \left(1 + \sqrt{1-\lVert \theta \rVert^2} \right)^2 + \lVert \theta \rVert^2 \cos(2\varphi) \right) \,,\\
 g_{\sigma;12} &= \frac{g_{12} }{4} \left(1 + \sqrt{1-\lVert \theta \rVert^2} \right)^2 + \frac{1}{4} \lVert \theta \rVert^2 \left( -g_{12}\cos(2\varphi) + \sqrt{|g|} \sin(2\varphi) \right) \,,\\
 g_{\sigma;22} &= \frac{g_{22}}{4} \left(1 + \sqrt{1-\lVert \theta \rVert^2} \right)^2 + \frac{1}{4} \lVert \theta \rVert^2 \left(\frac{ g_{11}g_{22} - 2 g_{12}^{2} }{g_{11}} \cos(2\varphi) + 2\frac{ g_{12} \sqrt{|g|}}{g_{11}} \sin(2\varphi) \right)\,,
\end{align*}
 where $\varphi$ is given in Proposition \ref{prop:symbol_laplacian_randers}, computed in the normal coordinates for $g$ given by $\left( \dfrac{x}{\sqrt{g_{11}}} , -\dfrac{g_{12}}{\sqrt{g_{11}} \sqrt{|g|}} x + \dfrac{\sqrt{g_{11}} }{ \sqrt{|g|}} y \right)$.
\end{lem}

\begin{proof}
 If $F = \sqrt{g}+\theta$, then, according to Proposition \ref{prop:symbol_laplacian_randers}, in normal coordinates at $p$ for $g$, we have
\begin{equation*}
 |\sigma| = \frac{4}{b(1 + b)^2}\,,
\end{equation*}
where $b := \sqrt{1-\lVert \theta\rVert^2}$. Hence, switching from normal coordinates to any coordinates gives the claim and therefore Corollary \ref{cor:utilisation_unique3} because it is well known that the volume form of a Riemannian metric is given by the square-root of its determinant.\\

 Note that $\left(x^n, y^n\right) := \left( \dfrac{x}{\sqrt{g_{11}}} , -\dfrac{g_{12}}{\sqrt{g_{11}} \sqrt{|g|}} x + \dfrac{\sqrt{g_{11}} }{ \sqrt{|g|}} y \right)$ are normal coordinates for $g$. Let 
\begin{equation*}
 T := \left[ \begin{array}{cc}
                   \sqrt{g_{11}} & \frac{g_{12}}{\sqrt{g_{11}}} \\
		    0  & \frac{\sqrt{|g|}}{\sqrt{g_{11}}}
                  \end{array} \right].
\end{equation*}
We have that $T\left(x^n, y^n\right) = (x,y)$ and $ \leftexp{t}{T} T = [g_{ij}]$.\\
 Now, as $g_{\sigma} = \sigma^{\ast}$, in the normal coordinate $\left(x^n, y^n\right)$ at $p$, we have, using Proposition \ref{prop:symbol_laplacian_randers}:
\begin{equation*}
 g_{\sigma}(p) = \frac{1}{4} \left[ \begin{array}{cc}
                     (1+b)^2 - (1-b^2)\cos(2\varphi) & (1-b^2)\sin(2\varphi) \\
		      (1-b^2)\sin(2\varphi)    &   (1+b)^2 + (1-b^2)\cos(2\varphi)
                   \end{array} \right]
\end{equation*}
 and $\leftexp{t}{T} g_{\sigma}(p) T$ gives $g_{\sigma}$ in the $(x,y)$ coordinates.
\end{proof}

\begin{proof}[Proof of Proposition \ref{prop:Randers_donne_couple}]
Let $g_1$ and $\Omega$ be respectively a Riemannian metric and a volume form on a $2$-manifold $M$. Let $\mu \colon M\rightarrow \R_{+}^{\ast}$ be the function such that $\Omega^{g_1} = \mu \Omega$, so that we can write $K:= \sup_{x\in M} \mu(x)$, by hypothesis $K$ is finite, and set $\mu':= \frac{\mu}{K}$.\\
We choose a smooth vector field $Z$ on $M$ such that it is non-zero on $\mu'^{-1}\lbrace 1 \rbrace$. Such a vector field exists by hypothesis, indeed, either $M$ is contractible and we can take $Z$ to be never zero, or $M$ is supposed to be compact, which force $\mu'^{-1}\lbrace 1 \rbrace$ to be non-empty.\\
Now our goal is to construct a Riemannian metric $g$ and a $1$-form $\theta$ such that $g_1$ is obtained from the Randers metric $F = \sqrt{g} + \theta$ via the symbol of the Finsler--Laplace operator and $\Omega = \frac{1}{K} \Omega^F$. \\

We set $b \colon M \rightarrow \left] 0, 1 \right]$ to be the unique function such that
\begin{equation*}
 \frac{b(1+b)^2}{4}= \left.\mu'\right.^2 .
\end{equation*}
Such a $b$ exists, and is smooth, by definition of $\mu'$.\\
At any point $z \in \mu'^{-1}\lbrace 1 \rbrace$, we set $g(z) := g_1(z)$ and $\theta_z = 0$.\\
Now, take $z \notin \mu'^{-1}\lbrace 1 \rbrace$ and choose local coordinates $(x,y)$ around $z$ such that $Z(z)$ and $\left.\frac{\partial}{\partial x}\right|_z$ are collinear (such coordinates exist as $\mu'^{-1}\lbrace 1 \rbrace$ is closed).\\
In order to avoid too much subscript clutter, we write $g_1 = \left[ \begin{array}{cc}
                                                                                                u& v\\
												v & w 
                                                                                              \end{array} \right]$ in the $(x,y)$ coordinates.

In those coordinates, we define $g(z)$ by
\begin{equation*}
\left\{
\begin{aligned}
  g_{11}(z) &:= \frac{4u}{\left(1 + b\right)^2} \,, \\
  g_{12}(z) &:= \frac{4v -(1-b^2) \frac{\sqrt{uw-v^2}}{\mu'}  }{\left(1 + b\right)^2} \,, \\
  g_{22}(z) &:= \frac{4w -2 (1-b^2) \frac{\sqrt{uw-v^2}}{\mu'} \left(\frac{4v -(1-b^2) \frac{\sqrt{uw-v^2}}{\mu'}  }{4u} \right) } {\left(1 + b\right)^2} \,.
\end{aligned} \right. 
\end{equation*}
Remark that it implies that 
\begin{equation*}
\left\{
\begin{aligned}
 \sqrt{uw-v^2} &= \mu' \sqrt{|g|} \,, \\
  4u &= g_{11}(z) \left(1 + b\right)^2 \,,\\
  4v &= g_{12}(z) \left(1 + b \right)^2 + (1-b^2) \sqrt{|g(z)|} \,,\\
  4w &= g_{22}(z)\left(1 + b \right)^2 + 2 (1-b^2)\frac{ g_{12}(z) \sqrt{|g(z)|}}{g_{11}(z)} \,,
\end{aligned} \right.
\end{equation*}
where again we wrote $|g| = g_{11}(z) g_{22}(z) - g_{12}(z)^2$. Only the first equation is not evident, but the computation can be done by hand or by plugging $g(z)$ in your favorite formal computation program.\\
By Lemma \ref{lem:Randers_coord_general}, the second system of equations shows that $g_1(z) = g_{\sigma}(z)$ where $g_{\sigma}$ is the dual symbol coming from the Randers metric $F_z = \sqrt{g(z)} + \theta_z$, where $\theta_z$ is defined by $\lVert \theta_z \rVert^2_{g} = 1- b(z)^2$ and $\varphi(z) = \pi/4$, i.e., $\theta_z$ is $\L_g \circ Z(z)$ rotated (for the metric $g$) by $\pi/2$. Moreover, $\Omega^F = \Omega^{g}$ because $F$ is Randers, and $\Omega^{g_{\sigma}} = \sqrt{\dfrac{b(1-b)^2}{4}} \Omega^g$. By our choice of $b$, we have
\begin{equation*}
 \mu' \Omega^F = \Omega^{g_1} = K \mu' \Omega,
\end{equation*}
so $\Omega^F = K \Omega$ as wanted.\\

For the moment, given any point $z$ and local coordinates $(x,y)$ around $z$ (satisfying the above condition with respect to $Z$) we constructed a scalar product $g(z)$ on $T_zM$ and an element $\theta_z$ of $T_z^{\ast}M$ which verifies our conclusion. But in order to be done, we still need to show two things: First, the definitions of $g(z)$ and $\theta_z$ must be independent of the local coordinates we choose, and once we have that, it remains to be shown that everything is smooth.\\

Let us prove the independence of $g(z)$ from the coordinates $(x,y)$:\\
For any $z\in \mu'^{-1}\lbrace 1 \rbrace$ there exists only one local coordinate system $(x^n, y^n)$ that is normal for $g_1$ at $z$ and such that $Z(z)$ is collinear to $\left.\frac{\partial}{\partial x^n}\right|_z$. In these coordinates, $g(z) := [g_{ij}^n(z)]$ is given by
\begin{equation*}
\left\{
\begin{aligned}
  g_{11}^{n}(z) &:= \frac{4}{\left(1 + b\right)^2} \,, \\
  g_{12}^{n}(z) &:= -\frac{ (1-b^2)  }{\lambda\left(1 + b\right)^2} \,, \\
  g_{22}^{n}(z) &:= \frac{4 + 2 (1-b^2)^2 \frac{1}{4\lambda^2}} {\left(1 + b\right)^2}  \,.
\end{aligned} \right. .
\end{equation*}
A transformation from $(x^n,y^n)$ to $(x,y)$ is given by 
\begin{equation*}
 T_z := \left[ \begin{array}{cc}
                   \sqrt{u(z)} & \frac{v(z)}{\sqrt{u(z)}} \\
		    0  & \frac{\sqrt{u(z)w(z) - v(z)^2}}{\sqrt{u(z)}}
                  \end{array} \right]. .
\end{equation*}
and to get the independence, we just need to verify that 
$$
\leftexp{t}{T_z} \left[ \begin{array}{cc} 
                       g_{11}^{n}(z) & g_{12}^{n}(z) \\ g_{12}^{n}(z) & g_{22}^{n}(z)
                      \end{array} \right]
  T_z = \left[ \begin{array}{cc} 
                       g_{11}(z) & g_{12}(z) \\ g_{12}(z) & g_{22}(z)
                      \end{array} \right].
$$
This is easily done, even by hand.\\
So $g$ is well defined and as $\theta$ just depends on $g$, $b$ and $Z$, it is also well-defined.\\

We are left with the smoothness issue, which is easy. Indeed, $b$ is a smooth function because $\mu'$ is. So $g$ is also smooth because of how we defined it, and finally as $\theta$ depends just on $g$, $b$ and $Z$, all of which are smooth, it is smooth.

\end{proof}

\section[Spectral data for KZ metrics]{Finsler--Laplacian and spectral data for Katok--Ziller metrics}

When we started studying this Finsler--Laplace operator, one of our first goals was to show that it was ``usable'', that is, that we could take purely Finslerian examples and compute the spectrum and the eigenfunctions.

 However, computing spectral data is a daunting task even in the Riemannian case. Indeed, past the three model spaces $\R^n$, $\S^n$ and $\mathbb{H}^n$ and some of their quotients, we do not know any full spectra of a Laplace--Beltrami operator. So, in order to have any chance of success, we wanted to start with some equivalent of a Finsler model space, but as this does not exist, we settled for just constant flag curvature, preferably on surfaces.\\

Akbar-Zadeh \cite{AkZ} showed that any closed surface endowed with a Finsler metric of constant negative flag curvature is in fact Riemannian. On the $2$-sphere however, the Finsler case is richer, Bryant \cite{Bryant:FS2S,Bryant:PFF2S} constructed a $2$-parameter family of such metrics. Katok previously \cite{Katok:KZ_metric} had constructed a $1$-parameter family of Finsler metrics on the sphere which later turned out to be of constant flag curvature $1$ (see \cite{Rademacher:Sphere_theorem}).\\

We chose to study Katok's examples for several reasons, the main being its dynamical interest. Indeed, they were constructed to give examples of metrics on the $2$-sphere with only a finite number of closed geodesics at a time when it was thought impossible. Furthermore, the construction method was generalized by Ziller \cite{Ziller:GKE} and so we use these metrics to give examples of spectral data in the torus case. Lastly, these metrics admits a quite agreeable explicit expression compared to the above mentioned examples due to Bryant (see Proposition \ref{prop:expression_KZ} and compare to Equation (12.7.4) on p.346 of \cite{BaoChernShen}). We recall the construction, in a slightly more general context than in \cite{Ziller:GKE}, as well as some other properties below.

\subsection{Construction}

Let $M$\ be a closed manifold and $F_0$\ a smooth Finsler metric on $M$. We suppose furthermore that $M$ admits a Killing field $V$, i.e., $V$\ is a vector field on $M$\ such that the one-parameter group of diffeomorphisms that it generates are composed of isometries for $F_0$. Katok-Ziller examples are constructed in the Hamiltonian setting.

Recall that the Legendre transform $\L_{F_0} \colon TM \rightarrow T^{\ast}M$ associated with $F_0$ is a diffeomorphism outside the zero section and that $H_0 := F_0 \circ \L_0 ^{-1} : T^{\ast}M \rightarrow \R$ is a Finsler co-metric (see Section \ref{sec:cotangent}).

 We consider $T^{\ast}M$\ as a symplectic manifold with canonical form $\omega$. Any function $H \colon T^{\ast}M \rightarrow \R$\ gives rise to a Hamiltonian vector field $X_H$\ defined by 
\begin{equation*}
 dH(y) = \omega\left(X_H , y \right), \; \text{for all} \; y\in TT^{\ast}M\, .
\end{equation*}

We define $H_1: T^{\ast}M \rightarrow \R$\ by $ H_1(x) := x(V)$, and, for sufficiently small $\eps$, we set 
\begin{equation*}
 H_{\eps} := H_0 + \eps H_1 \, .
\end{equation*}
The function $H_{\eps}$\ is smooth off the zero-section, homogeneous of degree one, and also strongly convex for sufficiently small $\eps$. Therefore, the Legendre transform ${\L_{\eps} : T^{\ast}M \rightarrow TM}$\ associated to $\frac{1}{2} H_{\eps}^2$\ is a global diffeomorphism, so we can state the
\begin{defin}
 The family of generalized Katok-Ziller metrics on $M$\ associated to $F_0$\ and $V$\ is given by
\begin{equation*}
 F_{\eps} = H_{\eps} \circ \L_{\eps}^{-1}.
\end{equation*}

\end{defin}

Katok, in his first example, took $F_0$\ to be the standard Riemannian metric on $\S^n$ and showed that some of these metrics had only a finite number of closed geodesics. Ziller \cite{Ziller:GKE} then showed that, for $\eps/\pi$\ irrational, there was in fact $n$\ closed geodesics for $n$\ even and $n-1$\ for $n$\ odd. In general, Bangert and Long \cite{MR2563691} showed that every non-reversible Finsler metric on $\S^2$\ has at least two closed geodesics. The minimal number of closed geodesics is still unknown in higher dimension.

As we will use local coordinates formulas for the Katok-Ziller metrics on the torus and the sphere, we can state the general formula in local coordinates when $F_0$\ is Riemannian. The computation of these metrics in local coordinates is probably not new (see Rademacher \cite{Rademacher:Sphere_theorem} for the expression on the sphere) and was communicated to us in this more general form by Foulon.
\begin{prop}[Foulon \cite{Fou:perso}] \label{prop:expression_KZ}
 Let $F_0 = \sqrt{g}$\ be a Riemannian metric on $M$, $V$\ a Killing field on $M$, and $F_{\eps}$\ the Katok-Ziller metric associated. Then
\begin{equation*}
 F_{\eps}(x,\xi) = \frac{1}{1- \eps^2 g\left(V,V\right)} \left[ \sqrt{g\left(\xi,\xi \right) \left(1-\eps^2 g\left(V,V\right) \right) + \eps^2 g\left(V,\xi\right)^2} -\eps g\left(V,\xi\right) \right].
\end{equation*}

\end{prop}

\begin{proof}
Let $x\in M$, we choose the normal coordinates $\left(\xi_i\right)$\ on $T_xM$, so that we have $F_0^2\left(x,\xi\right) = \sum \xi_i^2$. We will write $p$\ for an element of $T_x^{\ast}M$, and $p^i$\ will be the associated coordinates. As $F_0$\ is Riemannian, we have 
$$
H_0 \left(x,p\right) = ||p|| = \sqrt{\sum (p^i)^2}.
$$
The function $H_{\eps}$\ is then given by
\begin{align*}
 H_{\eps}\left(x,p\right) &= H_0 \left(x,p\right) + \eps H_1 \left(x,p\right) \\
     &= ||p|| + \eps \left<p | V\right>,
\end{align*}
and, if we set 
\begin{equation*}
 \Le= d_v\left(\frac{1}{2} H_{\eps}^2 \right): T^{\ast}M \rightarrow TM,
\end{equation*}
we have
\begin{equation*}
 F_{\eps}\left(x, \xi \right) = H_{\eps} \circ \Le^{-1} \left(x, \xi \right) .
\end{equation*}
In order to compute $ F_{\eps}$, we will first compute $\Le$. Recall that $\left(\frac{\partial}{\partial x_i}\right)$\ represents a basis of $T_xM$
\begin{align*}
 \Le \left(x,p\right) &= \frac{\partial}{\partial p^i}\left(\frac{1}{2} H_{\eps}^2 \right) \frac{\partial}{\partial x_i} \\
      &= \frac{\partial}{\partial p^i}\left[ \frac{1}{2} ||p||^2 + \eps ||p|| \left<p | V\right> + \frac{\eps^2}{2} \left<p | V\right>^2 \right] \frac{\partial}{\partial x_i} \\
      &= \left[ p^i + \eps \left(\frac{p^i}{||p||} \left<p | V\right> + \eps ||p|| V_i \right) + \eps^2 V_i \left<p | V\right> \right] \frac{\partial}{\partial x_i} \\ 
     &=  \left( p^i + \eps ||p|| V_i \right) \left( 1 + \frac{\eps}{||p||} \left< p | V\right>\right) \frac{\partial}{\partial x_i} \, . 
\end{align*}
Using that $H_{\eps} \left(x,p\right) = F_{\eps}\left( \Le \left(x,p\right) \right)$\ and setting $\frac{p}{||p||} = u $, we obtain
\begin{align*}
  ||p|| + \eps \left<p | V\right> &= F_{\eps}\left(||p|| \left( u^i + \eps  V_i \right) \left( 1 + \frac{\eps}{||p||} \left< p | V\right>\right) \frac{\partial}{\partial x_i} \right) \\
    &= \left(||p|| + \eps \left<p | V\right> \right) F_{\eps}\left( \left( u^i + \eps  V_i \right) \frac{\partial}{\partial x_i} \right) ,
\end{align*}
hence
$$
 1 = F_{\eps}\left( \left( u^i + \eps  V_i \right) \frac{\partial}{\partial x_i} \right).
$$
Let $ \xi = \Le \left(x,p\right)$, we have obtained that:
\begin{align*}
 \xi &=  F_{\eps}\left(x,\xi\right) \left( u + \eps V \right),
\end{align*}
where this equality must be understood coordinate by coordinate. Therefore,
\begin{align*}
  \left< u |V \right> &= \frac{1}{F_{\eps}\left(x,\xi\right)} \left< \xi |V \right> - \eps ||V||^2,
\end{align*}
and
\begin{align*}
 ||\xi||^2 &=  F_{\eps}^2\left(x,\xi\right) \left[ ||u||^2 + 2 \eps \left< u | V\right> + \eps^2 ||V||^2 \right] \\
          &= F_{\eps}^2\left(x,\xi\right)\left[ 1  + 2 \eps  \frac{ \left< \xi |V \right> }{F_{\eps}\left(x,\xi\right)}  -2\eps^2 ||V||^2+ \eps^2 ||V||^2 \right] .
\end{align*}
And we are led to solve
\begin{equation*}
  F_{\eps}^2\left(x,\xi\right) \left(1 - \eps^2 ||V||^2 \right)  + 2 \eps   \left< \xi |V \right>  F_{\eps}\left(x,\xi\right) - ||\xi||^2 = 0,
\end{equation*}
which yields
\begin{align*}
  F_{\eps}\left(x,\xi\right) = \frac{ -  \eps   \left< \xi |V \right> + \sqrt{\eps^2  \left< \xi |V \right>^2 + \left(1 - \eps^2 ||V||^2 \right)||\xi||^2 }}{\left(1 - \eps^2 ||V||^2 \right)}\,  .
\end{align*}

\end{proof}

Before getting on to the examples, we want to point out the following property of the Katok-Ziller examples:
\begin{thm}[Foulon \cite{Fou:perso}]
 The flag curvatures of the family of Katok-Ziller metrics are constant.
\end{thm}

\subsection{On the Torus}

Let $\T$ be an $n$-dimensional torus, $g$ a flat metric on $\T$ and $V$ a Killing field for $g$. Let $F_{\eps}$ be the associated Katok--Ziller metric, we have:

\begin{prop}
 There exists a (unique) Riemannian metric $\sigma_{\eps}$ such that $\Delta^{F_{\eps}} = \Delta^{\sigma_{\eps}}$.
\end{prop}

\begin{proof}
 This result follows from the fact that all objects involved are invariant by translations, hence independent of the point on the torus. Indeed, $V$ is a Killing field on $\T$ and so is translation invariant. From Proposition \ref{prop:expression_KZ}, we deduce that $F_{\eps}$ is independent of the base point on the torus, which yields, via Equation \eqref{eq:laplacien_pour_randers_general} and Proposition \ref{prop:volume_angle_Randers} that both the symbol $\sigma_{\eps}$ and the volume $\Omega^{F_{\eps}}$ are constant on $\T$. Now, recall (see Lemma \ref{lem:existence_unicity}) that ${\Delta^{F_{\eps}} = \Delta^{\sigma} - \frac{1}{a^2} \langle \nabla^{\sigma} \varphi , \nabla^{\sigma} a^2 \rangle}$, where $a$ is the function such that $\Omega^{F_{\eps}} = a^2 \Omega^{\sigma}$. So as $a$ is constant, $\nabla^{\sigma} a^2 = 0$ and we get the result.
\end{proof}

\begin{rem} \label{rem:Minkowsky_randers}
The above result of course holds for any ``Minkowski-Randers space'' (i.e., a Randers metric depending only on the tangent vector and not on the base-point on the manifold) and we could therefore use what is known on the Riemannian spectrum on subsets of $\R^n$ to obtain the Finsler--Laplace spectrum of these spaces.
\end{rem}
We give below the actual computation of the Finsler--Laplace operator for Katok--Ziller metrics on the $2$-torus to get an idea of how the Katok--Ziller transformation actually acts on the spectrum. This could be obtained by doing the right change of variables in the formula of Proposition \ref{prop:symbol_laplacian_randers}. However, direct computations are not much longer and were already typed. Therefore that is how we proceed.
 
 \subsubsection{An example in dimension two}

We set $\T  = \R^2 /\Z^2$, $(x,y)$\ (global) coordinates on $\T$\ and $\left( \xi_x, \xi_y \right)$\ local coordinates on $T_p\T$. Let $\eps<1$, the Katok-Ziller metric on $\T$\ associated with the standard metric and the Killing field $V = \frac{\partial}{\partial x}$, is given by
$$
F_{\eps}(x,y; \xi_x, \xi_y ) = \frac{1}{1-\eps^2} \left( \sqrt{\xi_x^2 + (1-\eps^2) \xi_y^2} - \eps \xi_x \right),
$$

\begin{thm}
\label{thm:KZ_torus}
The Laplace operator, in local coordinates, is given by
$$
\Delta^{F_{\eps}} = \frac{2 \left(1-\eps^2\right)}{1 + \sqrt{1- \eps^2}} \left( \sqrt{1- \eps^2} \parxx + \paryy \right)
$$
and the spectrum is the set of $\lambda_{(p,q)}$, $(p,q)\in \Z^2$, given by
\begin{equation*}
\lambda_{(p,q)} = 4\pi^2 \frac{2 \left(1-\eps^2\right)}{1 + \sqrt{1- \eps^2}} \left( \sqrt{1-\eps^2} p^2 + q^2 \right).
\end{equation*}

\end{thm}

Recall that for any flat Riemannian torus, the Poisson formula gives a link between the eigenvalues of the Laplacian and the length of the periodic orbits. In the case at hand we lose this relationship, as there is no a priori link between the length of the periodic geodesics for the Finsler metric and the length of the closed geodesics in the isospectral torus.

\begin{proof}

 \textbf{Vertical derivative and coordinate change.}\\
 In the local coordinates $\left(x,y, \xi_x, \xi_y\right)$\ on $T\T$ we have
\begin{equation*}
 d_v F_{\eps}= \frac{1}{1-\eps^2} \left( f_x dx + f_y dy \right),
\end{equation*}
where
\begin{equation*}
 f_x := \frac{\xi_x}{\sqrt{\xi_x^2 + (1-\eps^2) \xi_y^2}} - \eps  \quad \text{and} \quad  f_y := \frac{\xi_y}{\sqrt{\xi_x^2 + (1-\eps^2) \xi_y^2}}\, .
\end{equation*}
We choose a local coordinate system $\left(x,y,\theta\right)$\ on $H\T$ where $\theta$\ is determined by 
\begin{equation*}
\left\{
\begin{aligned} 
 \cos \theta & =  f_x + \eps \\
 \sin \theta & =  \frac{f_y}{\sqrt{1-\eps^2}}\, .
\end{aligned}
\right.
\end{equation*}
As the Hilbert form $A$\ is the projection on $H\T$\ of the vertical derivative of $F$, we have
\begin{equation*}
 A = \frac{1}{1-\eps^2} \left( \left(\cos \theta -\eps \right) dx + \sqrt{1-\eps^2}\sin \theta dy \right).
\end{equation*}

\textbf{Liouville volume and angle form.}\\
 We have
\begin{align*}
 dA & = \frac{1}{1-\eps^2} \left( -\sin \theta d\theta \wedge dx + \sqrt{1-\eps^2} \cost d\theta \wedge dy \right) \\
A \wedge dA & = \left(\frac{1}{1-\eps^2}\right)^{\frac{3}{2}} \left(-1 + \eps \cos \theta \right) d\theta \wedge dx \wedge dy \, .
\end{align*}
Therefore $ \alpha = \left(1 - \eps \cost\right) d\theta$.

\textbf{Geodesic flow.}\\
Let $X = X_x \parx + X_y \pary + X_{\theta} \partheta$ be the geodesic flow, Equation \eqref{eq:Reeb_field} is equivalent to
\begin{equation*}
\left\{
\begin{aligned}
X_{\theta} &= 0  \\
\sint X_x -  \sqrt{1-\eps^2} \cost X_y &= 0 \\
\left(\cost -\eps \right) X_x +  \sqrt{1-\eps^2} \sint X_y &= 1-\eps^2  \,.
\end{aligned}
\right.
\end{equation*}
Hence $X_x = \frac{1-\eps^2}{1- \eps \cost} \cost$, $X_y = \frac{\sqrt{1-\eps^2}}{1- \eps \cost} \sint$ and $X_{\theta} = 0$.

\textbf{The Laplacian.}\\
The second Lie derivative of $X$\ is $L_X^2  = X_x^2 \parxx + X_y^2 \paryy + X_x X_y \parxy$. So, for $p\in S$, 
\begin{equation*}
 \Delta^{\eps} = \frac{1}{\pi}\left( \int_{H_pS} X_x^2 \alpha \parxx + \int_{H_pS} X_y^2 \alpha \paryy + \int_{H_pS} X_x X_y \alpha \parxy \right).
\end{equation*}
As $X_x$\ and $X_y$\ are of different parity (in $\theta$), we have $\int_{H_pS} X_x X_y \alpha =0 $. Hence
\begin{equation*}
 \Delta^{\eps} = \frac{1}{\pi}\left( \int_{H_pS} X_x^2 \alpha \parxx + \int_{H_pS} X_y^2 \alpha \paryy \right).
\end{equation*}
Direct computations give
\begin{align*}
 \int_{H_pS} X_x^2 \alpha &= 2 \pi \frac{\left(1-\eps^2\right)^{\frac{3}{2}}}{1 + \sqrt{1- \eps^2}}\, ,\\
 \int_{H_pS} X_y^2 \alpha &= 2 \pi \frac{1-\eps^2}{1 + \sqrt{1- \eps^2}} \, .
\end{align*}
Therefore, the Finsler--Laplace operator is given by 
\begin{equation*}
 \Delta^{\eps} = \frac{2 \left(1-\eps^2\right)^{\frac{3}{2}}}{1 + \sqrt{1- \eps^2}} \parxx + \frac{2 \left(1-\eps^2 \right)}{1 + \sqrt{1- \eps^2}} \paryy \, .
\end{equation*}

\textbf{The spectrum.}\\
To compute the spectrum we consider Fourier series of functions on $\T$.\\
Any function $f \in C^{\infty }(\T)$\ can be written as 
$$
f(x,y) = \sum_{(p,q)\in \Z^2} c_{(p,q)} e^{2i\pi (px+qy)}
$$
and we are led to solve
\begin{equation}
\label{eq:spectre_fourier}
 \Delta^{F_{\eps}} f + \lambda f = \sum_{(p,q)\in \Z^2} c_{(p,q)} \left[ -4\pi^2 \left(a p^2 + b q^2\right) +\lambda \right] e^{2i\pi (px+qy)} = 0 \,,
\end{equation}
where
\begin{equation*}
 a = \frac{2 \left(1-\eps^2\right)^{3/2}}{1 + \sqrt{1- \eps^2}} \quad \text{and}\quad b = \frac{2 \left(1-\eps^2\right)}{1 + \sqrt{1- \eps^2}}\, .
\end{equation*}

Now, for any $(p,q) \in \Z^2$, $\lambda_{(p,q)} = 4\pi^2 \frac{2 \left(1-\eps^2\right)}{1 + \sqrt{1- \eps^2}} \left( \sqrt{1-\eps^2} p^2 + q^2 \right)$ is a solution to \eqref{eq:spectre_fourier}.
\end{proof}

\subsection{On the $2$-Sphere}

Let $\S^2 \smallsetminus\{\text{N,S}\} = \lbrace (\phi,\theta) \mid \phi \in \left]0,\pi\right[, \; \theta \in \left[0,2 \pi\right] \rbrace$ be polar coordinates on the sphere minus the poles, and take $\left(\phi,\theta; \xi_{\phi}, \xi_{\theta} \right)$\ the associated local coordinates on $T\S^2\smallsetminus\{\text{N,S}\}$. All the formulas afterwards can be extended by taking $\phi=0$ and $\phi=\pi$ for the North and South poles.

The Katok--Ziller metrics associated with the standard metric and the Killing field $V = \sinp \frac{\partial}{\partial \theta}$\ are given by
\begin{equation*}
 F_{\eps} \left(\phi,\theta; \xi_{\phi}, \xi_{\theta} \right) = \frac{1}{1-\ee} \left( \sqrt{\left(1-\ee\right)\xi_{\phi}^2 + \sinpp \xi_{\theta}^2} - \eps\sin^2(\phi) \xi_{\theta}\right),
\end{equation*}

\begin{thm}
\label{th:laplacien_S2}
 The Finsler--Laplace operator on $(\mathbb{S}^2, F_{\eps})$ is given by
\begin{multline} \label{eq:laplacien_S2}
    \Delta^{F_{\eps}} = \frac{2}{1+\sqrt{1-\ee}} \Biggl[ \frac{1}{\sinpp} \left(1-\ee \right)^{\frac{3}{2}} \parttheta  \\
              + \left(1-\ee \right) \parpphi  +\frac{\cosp}{\sinp} \left(\ee + \sqrt{1-\ee}  \right)\parphi \Biggr].
\end{multline} 
\end{thm}

By computing the Laplacian associated with the symbol metric, we can remark that this Laplacian is \emph{not} a Riemannian Laplacian, contrarily to the flat torus case above. Hence the question of it being isospectral to a Riemannian Laplacian becomes non-trivial, but we do not know the answer.

Recall that the spectrum of the Laplace--Beltrami operator on $\S^2$ is the set ${\lbrace -l(l+1) \mid l \in \N \rbrace}$ and that an eigenspace is spanned by functions $\ylm$ with $m\in \Z$ such that $-l\leq m \leq l$. These functions are called \emph{spherical harmonics} and are defined by
\begin{equation*}
\ylm\left(\phi, \theta\right) := e^{i m \theta} P_l^m\left(\cosp\right),
\end{equation*}
where $P_l^m$\ is the associated Legendre polynomial.\\
 We can see clearly from Formula \eqref{eq:laplacien_S2} that, when $\eps$\ tends to $0$, we obtain the usual Laplace--Beltrami operator on $\S^2$, we will therefore look for eigenfunctions close to the spherical harmonics. It turns out that the $Y_1^m$\ are eigenfunctions of $\Delta^{F_{\eps}}$\ for any  $\eps$, which yields:
\begin{cor}
\label{cor:lambda1}
 The smallest non-zero eigenvalue of $-\Delta^{F_{\eps}}$\ is 
\begin{equation}
 \lambda_1 = 2 -2 \eps^2 = \frac{8 \pi}{\text{vol}_{\Omega^{F_{\eps}}}\left(\S^2\right)} \,.
\end{equation}
It is of multiplicity two and the eigenspace is generated by $Y_1^1$\ and $Y_1^{-1}$.
\end{cor}

The fact that we have the above formula for $\lambda_1$ is quite interesting; first, it shows us that there does exist relationships between some geometrical data associated with a Finsler metric (here the volume) and the spectrum of the Finsler--Laplace operator. Secondly, recall the following result:
\begin{thm*}[Hersch \cite{Hersch}]
 For any Riemannian metric $g$\ on $\S^2$,
$$
\lambda_1 \leq \frac{8\pi}{\text{vol}_{g}\left(\S^2\right)}\, .
$$
Furthermore, the equality is realized only in the constant curvature case.
\end{thm*}
So the Katok--Ziller metrics on $\S^2$ give us a continuous family of metrics realizing that equality! We do not know however whether this is a Finslerian maximum or not.\\

Note that we also have $\Delta^{F_{\eps}} Y_1^0 = -2 Y_1^0$. However, the $\ylm$\ with $l\geq 2$\ are no longer eigenfunctions of $\Delta^{F_{\eps}}$. This is probably related to the breaking of the symmetries that the Katok--Ziller metrics induce.\\
 In the following, if $m$\ happens to be greater than $l$, we set $\ylm =0$. We denote by $\langle \cdot , \cdot \rangle$\ the inner product on $L^2\left(\S^2\right)$ defined by:
\begin{equation*}
 \langle f,g \rangle = \int_0^{2 \pi} \int_0^{\pi} f \bar{g} \sinp d\phi d\theta \, .
\end{equation*}

\begin{thm}
\label{thm:spectre_S2}
 Let $f$\ be an eigenfunction for $\Delta^{F_{\eps}}$ and $\lambda$\ its eigenvalue. There exist unique numbers $l$\ and $m$\ in $\N$, $0\leq m \leq l$, such that $f = a \ylm + b Y_l^{-m} + g$, where $g$\ uniformly tends to $0$\ with $\eps$, and 
\begin{multline}
\label{eq:spectre_S2}
  \lambda =  -l(l+1) + \eps^2 \Biggl[ \frac{m^2}{2\left(2l-1 \right)} \left( 2\left(l+1\right) + \frac{3 l\left(l-1\right)}{ \left(2l+3 \right)} \right)   \\
             +\frac{3l\left(l-1\right)}{2\left(2l-1 \right)}\left( 1 +  \frac{l^2+l-1}{ \left(2l+3 \right)\left(2l-1 \right)} \right) \Biggr] +o\left(\eps^2 \right).
\end{multline}

\end{thm}

Note that the Katok--Ziller transformation gets rid of most of the degeneracy of the spectrum. If $\eps\neq 0$, the eigenvalues are at most of multiplicity two, and are of multiplicity $2l+1$\ if $\eps$\ is zero.\\
We can state even more on the multiplicity of eigenvalues. Define $\Psi  \colon  \S^2 \rightarrow \S^2 $ by
\begin{equation*}
        \Psi(\phi, \theta ) := (\pi - \phi, -\theta) \, .
\end{equation*}
Theorem \ref{th:laplacien_S2} implies that $\Delta^{F_{\eps}}$\ is stable by $\Psi$, i.e., for any $g$, $\left(\Delta^{F_{\eps}} g \right) \circ \Psi = \Delta^{F_{\eps}} \left( g \circ \Psi \right)$.
So, if $f$\ is an eigenfunction for $\lambda$, then $f \circ \Psi$\ also. Therefore, either the subspace generated by $f$\ is stable by $\Psi$\ or $\lambda$\ is of multiplicity at least (and hence exactly) two.

\begin{rem}
When $\eps >0$, $F_{\eps}$\ is not preserved by $\Psi$.
\end{rem}

\subsubsection{Proof of Theorem \ref{th:laplacien_S2}}
This proof follows the same lines as that of Theorem \ref{thm:KZ_torus}, the computations being more involved and a bit lengthy. We just give the main steps.

\textbf{Vertical derivative and change of coordinates.}\\
Set $g_{\eps}\left(\phi,\theta; \xi_{\phi}, \xi_{\theta} \right) = \left(1-\ee\right)\xi_{\phi}^2 + \sinpp \xi_{\theta}^2$. We have 
$$
d_v F_{\eps} = \frac{\partial F_{\eps}}{\partial \xi_{\phi}} d\phi + \frac{\partial F_{\eps}}{\partial \xi_{\theta}} d\theta
$$
 where $  \dfrac{\partial F_{\eps}}{\partial \xi_{\phi}} = \dfrac{\xi_{\phi}}{\sqrt{g_{\eps}}}$ and $\dfrac{\partial F_{\eps}}{\partial \xi_{\theta}} =  \dfrac{1}{1-\ee}\left(\dfrac{\xi_{\theta} \sinpp}{\sqrt{g_{\eps}}} - \eps \sinpp \right)$.\\
From now on we consider the local coordinate $\psi \in \left[0,2\pi\right]$\ on $H_{(\phi,\theta)}\S^2$ defined by,
\begin{equation*}
\left\{
\begin{aligned} 
\cospsi& =  \frac{\xi_{\theta} \sinp}{\sqrt{g_{\eps}}} \\
 \sinpsi& = \sqrt{1-\ee} \frac{\xi_{\phi}}{\sqrt{g_{\eps}}} \,. 
\end{aligned}
\right.
\end{equation*}

\textbf{Hilbert form and Liouville volume.}\\
As in the above coordinates, we have
$$ \frac{\partial F_{\eps}}{\partial \xi_{\phi}} = \frac{\sinpsi}{\sqrt{1-\ee}} \quad \text{and} \quad \frac{\partial F_{\eps}}{\partial \xi_{\theta}} = \frac{1}{1-\ee}\left(\sinp\cospsi - \eps \sinpp \right),$$
 we deduce that the Hilbert form $A$ associated to $F_{\eps}$\ is given by
\begin{equation*}
 A = \frac{1}{1- \ee}\left( f_1 d\phi + f_2 d\theta \right),
\end{equation*}
with $f_1 = \sqrt{1-\ee} \sinpsi$ and $f_2 = \sinp\cospsi - \eps \sinpp$. In order to simplify the computations, note that $f_1$\ is odd in $\psi$, $f_2$\ is even and they do not depend on $\theta$. The exterior derivative of $A$\ is given by
\begin{equation*}
 dA = \frac{1}{1- \ee} \left(\parfunpsi d\psi\wedge d\phi + \parfdepsi d\psi\wedge d\theta + f_3 d\phi\wedge d\theta \right).
\end{equation*}
where
\begin{align*}
 \parfunpsi &= \sqrt{1-\ee} \cospsi, \\
 \parfdepsi &= -\sinp \sinpsi, \\
 f_3 &= \cosp \frac{\cospsi -2 \e +\ee \cospsi}{1-\ee} \, .
\end{align*}
Now, we have $A \wedge dA = \frac{1}{\left(1-\ee \right)^2} \left( -f_1 \parfdepsi + f_2 \parfunpsi \right) d\psi \wedge d\phi            	\wedge d\theta$ and $-f_1 \parfdepsi + f_2 \parfunpsi = \sinp \sqrt{1-\ee} \left(1 - \e \cospsi\right)$. Therefore
\begin{equation}
\label{eq:ada_sphere}
 A \wedge dA = \frac{\sinp}{\left(1-\ee \right)^{3/2}} \left( 1- \e \cospsi \right) d\psi \wedge d\phi \wedge d\theta \, .
\end{equation}

We can now use the construction of the angle form (see Section \ref{subsec:construction}). Let $\alpha'$\ be the $1$-form associated to the volume $d\phi \wedge d\theta$\ on $\S^2$, on $VH\S^2$\ we have
\begin{equation*}
 \alpha' = -\frac{1}{\left(1-\ee\right)^2} \left(\fundeux \right) d\psi\, .
\end{equation*}
If we denote by $l(p)$\ the length for $\alpha'$ of the fiber above a point $p\in \S^2$ we obtain
\begin{align*}
 l(p) &= \int_{H_p \S^2} \alpha' \\
      &= -\frac{1}{\left(1-\ee\right)^2} \int_0^{2\pi} \left(\fundeux \right) d\psi\\
      &= \frac{2\pi}{\left(1-\ee\right)^{\frac{3}{2}}} \sinp \,, 
\end{align*}
where we used again that 
$$\fundeux = \sinp \sqrt{1-\ee} \left(-1 + \e \cospsi\right).$$
 As $ \alpha = \frac{2\pi}{l(p)} \alpha'$ we obtain
\begin{equation}
\label{eq:angle_sphere}
 \alpha = \left(1- \e \cospsi\right) d\psi.
\end{equation}

\textbf{Geodesic flow.}\\
Let $X = \xpsi \parpsi + \xtheta \partheta + \xphi \parphi$\ be the geodesic flow of $F_{\eps}$. As $X$ is the Reeb field of $A$, we can use Equations \eqref{eq:Reeb_field} to determine $X$. We have
\begin{multline*}
 0 = i_X dA = \frac{1}{1-\ee} \Biggl( \left(- \parfunpsi \xphi - \parfdepsi \xtheta \right) d\psi +\left( \xphi f_3+ \xpsi \parfdepsi \right) d\theta  \\ 
+ \left( - \xtheta f_3 + \xpsi \parfunpsi\right) d\phi \Biggr),
\end{multline*}
and
\begin{equation*}
 1= A(X) = \frac{1}{1-\ee} \left( f_1 \xphi + f_2 \xtheta \right).
\end{equation*}
The above equations give the system
\begin{equation*}
\left\{
 \begin{aligned}
  \parfunpsi \xphi + \parfdepsi \xtheta &= 0 \\
 \xphi f_3+ \xpsi \parfdepsi &= 0 \\
 - \xtheta f_3 + \xpsi \parfunpsi &= 0 \\
 f_1 \xphi + f_2 \xtheta &= 1-\ee
 \end{aligned}
\right. 
\end{equation*}
which yields
\begin{align*}
\xtheta &= \frac{1-\ee}{\sinp} \; \frac{\cospsi}{1-\e \cospsi} ,\\
\xphi &= \sqrt{1-\ee} \; \frac{\sinpsi}{1-\e \cospsi} ,\\
\xpsi &= \frac{1}{\sqrt{1-\ee}} \; \frac{\cosp}{\sinp} \; \frac{\cospsi -2\e + \ee \cospsi}{1-\e\cospsi}.
\end{align*}

\textbf{The Finsler--Laplace operator.}\\
Let $f \colon \S^2 \rightarrow \R$. We start by computing $L_X^2 \pi^{\ast}f$.\\
As $\parpsi\left(\pi^{\ast}f\right) = 0 $\ and $X$\ does not depend on $\theta$, we get

\begin{multline*}
 L_X^2 \pi^{\ast}f = \xtheta^2 \frac{\partial^2 f}{\partial \theta^2} + \xtheta\xphi \frac{\partial^2 f}{\partial \phi \partial \theta}  + \xphi \xtheta \frac{\partial^2 f}{\partial \theta \partial \phi } + \xphi \frac{\partial \xtheta }{\partial \phi} \frac{\partial f}{\partial \theta}  \\
+ \xphi \frac{\partial \xphi }{\partial \phi} \frac{\partial f}{\partial \phi} + \xphi^2 \frac{\partial^2 f}{\partial \phi^2} + \xpsi \frac{\partial \xtheta }{\partial \psi} \frac{\partial f}{\partial \theta}  + \xpsi \frac{\partial \xphi }{\partial \psi} \frac{\partial f}{\partial \phi} \, .
\end{multline*}
Since we are only interested in $\int_{H_x\S^2} L_X^2 \pi^{\ast}f \alpha$, we can use the parity properties (with respect to $\psi$) of the functions $\xtheta, \; \xphi$\ and $\xpsi$\ (which are respectively even, odd and even) to get rid of half of the above terms. We obtain
\begin{multline*}
 \pi \Delta^{F_{\eps}} f(p) =  \int_{H_p\S^2} \hspace{-1,5mm}\xtheta^2 \; \alpha \; \frac{\partial^2 f}{\partial \theta^2} + \int_{H_p\S^2} \hspace{-1,5mm} \xphi^2 \; \alpha \; \frac{\partial^2 f}{\partial \phi^2} \\
   + \int_{H_p\S^2} \left( \xpsi \frac{\partial \xphi }{\partial \psi} + \xphi \frac{\partial \xphi }{\partial \phi} \right) \alpha \; \frac{\partial f}{\partial \phi} \, .
\end{multline*}

Direct computation (with a little help from Maple) yields
\begin{multline*}
   \Delta^{F_{\eps}} = \frac{2 \left(1-\ee \right)^{\frac{3}{2}}}{\sinpp \left(1+\sqrt{1-\ee} \right)} \; \frac{\partial^2 }{\partial \theta^2}  + 2  \frac{1-\ee}{1+\sqrt{1-\ee}} \; \frac{\partial^2 }{\partial \phi^2}  \\
 + \frac{2 \cosp}{\sinp} \left( 2 - \frac{1}{1+ \sqrt{1-\ee}} -\sqrt{1-\ee} \right) \frac{\partial }{\partial \phi} \, .
\end{multline*}
This concludes the proof of Theorem \ref{th:laplacien_S2}.

\subsubsection{Proof of Theorem \ref{thm:spectre_S2}}

We state the following property of spherical harmonics that will be useful in later computations:
\begin{prop}
\label{prop:ylm}
 Let $l\in \N$, and $m\in \Z$, such that $|m|\leq l$, then the associated Legendre polynomial $P_l^m\left(\cosp\right)$, denoted here by $\plm$, is a solution to the equation
\begin{equation}
 \frac{\partial^2 \plm}{\partial \phi^2} + \frac{\cosp}{\sinp}\frac{\partial \plm}{\partial \phi} +\left(l(l+1) - \frac{m^2}{\sinpp} \right) \plm = 0 \label{eq:legendre},
\end{equation}
They verify (see \cite[Formulas 8.5.3 to 8.5.5]{MR1225604})
\begin{subequations}
\begin{align}
 \left(2l-1\right) \cosp \tilde{P}^{m}_{l-1} &= (l-m)\plm + \left(l+m-1\right) \tilde{P}^{m}_{l-2} \, ,\label{eq:rec_legendre1} \\
\sinp \frac{\partial \plm}{\partial \phi} & = l \cosp \plm - (l+m) \tilde{P}^{m}_{l-1} \label{eq:derivative_legendre} \, , \\
 \sinp \plm &= \frac{1}{2l+1}\left( \tilde{P}_{l-1}^{m+1} - \tilde{P}_{l+1}^{m+1}\right) \label{eq:rec_legendre2}.
\end{align}
\end{subequations}
 The spherical harmonics form an orthogonal Hilbert basis of $L^2\left(\S^2\right)$\ and their norm is given by
\begin{equation}
 ||\ylm || = \sqrt{\frac{4\pi}{2l+1} \frac{\left(l+m\right)!}{\left(l-m\right)!}} \label{eq:norm_ylm}\, .
\end{equation}

\end{prop}

We can now proceed with the proof. Take $f$\ an eigenfunction of $\Delta^{F_{\eps}}$ and $\lambda$\ the associated eigenvalue. As the $\ylm$\ form an Hilbert basis of $L^2\left(\S^2\right)$, there exist $a_l^m$\ such that
$$
f = \sum_{l=0}^{+\infty} \sum_{|m|\leq l} a_l^m \ylm,
$$
where the convergence is a priori in the $L^2$-norm. The elliptic regularity theorem implies that $f \in C^{\infty}\left(\S^2 \right)$. Therefore the convergence above is uniform. So $\Delta^{F_{\eps}} f = \sum_{l=0}^{+\infty} \sum_{|m|\leq l} a_l^m \Delta^{F_{\eps}} \ylm$.\\ Let $l,m$\ be fixed. The equation $\langle \Delta^{F_{\eps}} f, \ylm \rangle = \lambda \langle f, \ylm \rangle$ yields
\begin{equation}
\label{eq:lambda_sphere}
 \lambda a_l^m \lVert \ylm \rVert^2 = \sum_{k=0}^{+\infty} \sum_{|n|\leq k} a_k^n \langle \ylm , \Delta^{F_{\eps}} Y_k^n \rangle.
\end{equation}

\begin{claim}
For any $l,m$\ we have
\begin{multline}
\label{eq:Delta_ylm}
 \Delta^{F_{\eps}} \ylm =   -l(l+1) \ylm \\
 + \frac{\eps^2}{\left(1+\sqrt{1-\ee}\right)^2}  \Biggl[  \left(1+2\sqrt{1-\ee}\right)l\left( l-1\right) \sinpp \ylm  \\
  +  \left(2 m^2 \left(1-\ee\right)\right)\ylm  + 2 \frac{l^2 + m^2 +l}{2l+1} \left(1+2\sqrt{1-\ee}\right) \ylm \\
 - 2(l+m)(l+m-1)\left(1+2\sqrt{1-\ee}\right) Y_{l-2}^m \Biggr].
\end{multline}
\end{claim}

\begin{proof}
 By the formula \eqref{eq:laplacien_S2} for $\Delta^{F_{\eps}}$, we have
\begin{align*}
 \Delta^{F_{\eps}} \ylm &= \frac{2}{1+\sqrt{1-\ee}} \left[ \frac{1}{\sinpp} \left(1-\ee \right)^{\frac{3}{2}} (-m^2) \ylm \right. \\
            &  \left. + e^{im\theta} \Biggl( \left(1-\ee \right) \parpphi\plm \right. \\
            & \left. +\frac{\cosp}{\sinp} \left(\ee + \sqrt{1-\ee}  \right)\parphi\plm \Biggr) \right],\\
\end{align*}
Applying first Equation \eqref{eq:legendre} and then Equation \eqref{eq:derivative_legendre}, we get
\begin{align*}
 \Delta^{F_{\eps}} \ylm  &= \frac{2}{1+\sqrt{1-\ee}} \left[ -m^2 \frac{1}{\sinpp} \left(1-\ee \right)^{\frac{3}{2}}  \ylm \right. \\
            &   \left. + e^{im\theta} \Biggl( \frac{\cosp}{\sinp} \left(-1+2\ee + \sqrt{1- \ee} \right) \parphi\plm \right.\\
            & \left. -\left(1-\ee \right)\left( l(l+1) - \frac{m^2}{\sinpp} \right) \plm \Biggr) \right] \\
   &= \frac{2}{1+\sqrt{1-\ee}} \Biggl[ \left( 1- \ee \right)\left(-l(l+1) + \frac{m^2}{ \sinpp} \left(1-\sqrt{1-\ee}\right) \right)\ylm \\
            & + \left(-1+2\ee + \sqrt{1- \ee} \right) l \frac{\cospp}{\sinpp} \ylm \\
            & -\frac{l+m}{\sinpp} \left( \frac{l-m}{2l-1} \ylm + \frac{l+m-1}{2l-1} Y_{l-2}^m\right) \Biggr].
\end{align*}
After a bit of rearranging, we get 
\begin{align*}
\Delta^{F_{\eps}} \ylm   &= -l(l+1) \ylm + \eps^2 \Biggl[ \frac{1+2\sqrt{1-\ee}}{\left(1+\sqrt{1-\ee}\right)^2} l\left( l-1\right) \sinpp \ylm \\
            & + \frac{2 m^2 \left(1-\ee\right)}{\left(1+\sqrt{1-\ee}\right)^2} \ylm + 2 \frac{l^2 + m^2 +l}{2l+1} \frac{1+2\sqrt{1-\ee}}{\left(1+\sqrt{1-\ee}\right)^2} \ylm \\
            & -2(l+m)(l+m-1)\frac{1+2\sqrt{1-\ee}}{\left(1+\sqrt{1-\ee}\right)^2} Y_{l-2}^m \Biggr],
\end{align*}
which gives our claim after some more simplifications.
\end{proof}
Using the claim, Equation \eqref{eq:lambda_sphere} becomes
\begin{equation*}
 \lambda a_l^m \lVert \ylm \rVert^2 = \sum_{k=0}^{+\infty} a_k^m \langle \ylm , \Delta^{F_{\eps}} Y_k^m \rangle.
\end{equation*}
Now, we can use an expansion of $\Delta^{F_{\eps}} Y_k^m $\ in powers of $\eps$.

\begin{claim}
 For any $l,m$, we have
 \begin{multline}
 \Delta^{F_{\eps}} \ylm  =  -l(l+1) \ylm + \eps^2 \Biggl[\frac{3 l (l+1)}{4} \sinpp\ylm \\
       + \left( \frac{m^2}{2} + \frac{3\left(l(l+1) +m^2\right)}{ 2l+1} \right) \ylm  + \frac{3}{2} (l+m)(l+m-1) Y_{l-2}^m \Biggr] + O\left(\eps^4\right).
\end{multline}
\end{claim}
The claim follows once again from a straightforward computation.\\
Using this second claim and the orthogonality of the spherical harmonics, Equation \eqref{eq:lambda_sphere} now reads
\begin{multline}
\label{eq:expansion_de_lambda}
 \lambda a_l^m \lVert \ylm \rVert^2 =  -l(l+1) a_l^m \lVert \ylm \rVert^2 + a_l^m \eps^2 \Biggl[\frac{3 l (l+1)}{4} \langle \sinpp\ylm , \ylm \rangle \\
       + \left( \frac{m^2}{2} + \frac{3\left(l(l+1) +m^2\right)}{ 2l+1} \right) \lVert \ylm \rVert^2 \Biggr]  \\
 + \sum_{k \neq l} a_k^m \eps^2  \Biggl[\frac{3 k (k+1)}{4} \langle \sinpp Y_k^m, \ylm \rangle    + \frac{3}{2} (k+m)(k+m-1) \langle Y_{k-2}^m, \ylm \rangle \Biggr] + O\left(\eps^4\right).
\end{multline}

\begin{claim}
 There is at most one $l$\ such that $\frac{1}{a_l^m}$\ is bounded independently of $\eps$.
\end{claim}

\begin{proof}
 Equation \eqref{eq:expansion_de_lambda} shows that, if $\frac{1}{a_l^m}$\ is bounded as $\eps$\ tends to $0$, then $\lambda$\ tends to $-l(l+1)$. Therefore we can have only one such $l$.
\end{proof}

Let $l$\ be given by the previous claim, \eqref{eq:expansion_de_lambda} reduces to
\begin{multline*}
 \lambda  = -l(l+1)  + \frac{\eps^2}{\lVert \ylm \rVert^2 } \Biggl[\frac{3 l (l+1)}{4} \langle \sinpp\ylm , \ylm \rangle \\
        + \left( \frac{m^2}{2} + \frac{3\left(l(l+1) +m^2\right)}{ 2l+1} \right) \lVert \ylm \rVert^2 \Biggr] + o\left(\eps^2\right).
\end{multline*}
Some more computations (using Equations \eqref{eq:rec_legendre2}, \eqref{eq:norm_ylm} and the orthogonality of the spherical harmonics) give
\begin{equation*}
 \frac{\langle \sinpp\ylm , \ylm \rangle}{\lVert \ylm \rVert^2 }  = 2\frac{l^2+l-1+ m^2}{ \left(2l+3 \right)\left(2l-1 \right)} \, ,
\end{equation*}
so that
\begin{multline} \label{eq:lambda_derniere}
 \lambda  =  -l(l+1)  + \eps^2 \Biggl[\frac{3 l (l+1)}{2} \frac{l^2+l-1+ m^2}{ \left(2l+3 \right)\left(2l-1 \right)} \\
        + \left( \frac{m^2}{2} + \frac{3\left(l(l+1) +m^2\right)}{ 2l+1} \right) \Biggr] + o\left(\eps^2\right).
\end{multline}
From this equation, we deduce
\begin{claim}
 There can only be one $m$\ such that $\frac{1}{a_l^m}$\ or $\frac{1}{a_l^{-m}}$\ is bounded independently of $\eps$.
\end{claim}

\begin{proof}
 Otherwise, we would find two different coefficients in $\eps^2$\ for $\lambda$.
\end{proof}

We sum up what we proved, namely that there exist unique $l,m\in \N $, $a,b \in \C$ and $g \colon \S^2 \rightarrow \C$\ such that
\begin{equation*}
f = a \ylm + b Y_l^{-m} + g \,.
\end{equation*}
Furthermore, for any $p\in \S^2$, $|g(p)|$\ tends to $0$\ with $\eps$ and the associated eigenvalue verifies Equation \eqref{eq:lambda_derniere}. That is, we proved Theorem \ref{thm:spectre_S2}.

\subsubsection{First eigenvalue and volume}

We finish by proving Corollary \ref{cor:lambda1}. Recall:

\begin{corspecial}
 The smallest non-zero eigenvalue of $-\Delta^{F_{\eps}}$\ is 
\begin{equation*}
 \lambda_1 = 2 - 2 \eps^2 = \frac{8\pi}{\text{vol}_{\Omega}\left(\S^2\right) }.
\end{equation*}
It is of multiplicity two and the eigenspace is generated by $Y_1^1$\ and $Y_1^{-1}$.
\end{corspecial}

\begin{proof}
Computation using either \eqref{eq:Delta_ylm} or directly Theorem \ref{th:laplacien_S2} gives $\Delta^{F_{\eps}} Y_1^1 = (-2 +2 \eps^2) Y_1^1$\ and $\Delta^{F_{\eps}} Y_1^{-1} = (-2 +2 \eps^2) Y_1^{-1}$. It also yields $\Delta^{F_{\eps}} Y_1^0 = -2 Y_1^0 $, now Theorem \ref{thm:spectre_S2} shows that the eigenfunctions for the first (non-zero) eigenvalue must live in the vicinity of the space generated by $Y_1^1, \; Y_1^0$\ and $Y_1^{-1}$, therefore $\lambda_1 = 2 - 2 \eps^2$.

Now using Equation \eqref{eq:ada_sphere} we get that the Finsler volume form for $(\S^2,F_{\eps})$\ is 
$$
\Omega^{F_{\eps}} = \frac{\sinp}{\left(1-\ee\right)^{3/2}} d\theta \wedge d\phi\, .
$$
So
$$
\text{vol}_{\Omega} \left(\S^2\right) = \frac{4\pi}{1-\eps^2} \, ,
$$
and hence,
\begin{equation*}
 \lambda_1 = \frac{8\pi}{\text{vol}_{\Omega}\left(\S^2\right) } \, .
\end{equation*}

\end{proof}

\chapter{Spectrum and geometry at infinity}

Our focus here will be the study of the links between the Finsler--Laplace operator and the dynamics or geometry for Finsler metrics with negative curvature (in the sense of equation \eqref{eq:def_curvature}).

In Riemannian manifolds of negative curvature, there are (at least) three natural classes of measures on the boundary at infinity: the Liouville (or visual) measure class, which is obtained by pushing the Lebesgue measure on unit spheres to the boundary via the geodesic flow; the Patterson-Sullivan measure class, which can be obtained from the Bowen-Margulis measure via the Kaimanovich correspondence (see \cite{Kaimanovich}); and the Harmonic measure class which is linked to the Laplace--Beltrami operator in a way that we will explicit later.

 In the case of surfaces we have a famous rigidity phenomenon: when two of those classes are equivalent, it forces the Riemannian metric to be of constant curvature (this is due to Katok \cite{Katok:4_appl,Katok:entropy_closed_geod} and Ledrappier \cite{Ledrappier:Harm_measures}).

 In higher dimensions, Ledrappier \cite{Ledrappier:Harm_measures} showed that equality between the Harmonic and Patterson-Sullivan classes is equivalent to $\lambda_1 = \frac{h^2}{4}$, where $\lambda_1$ is the bottom of the spectrum of $\Delta$ and $h$ is the topological entropy. In \cite{BessonCourtoisGallot}, G.\ Besson, G.\ Courtois and S.\ Gallot proved that $\lambda_1 = \frac{h^2}{4}$ implies that the manifold is a symmetric space.

When we started studying the Finsler--Laplace operator in negative curvature, our goal was to generalize some of (or get counter-examples to) the above results. Unfortunately, this is still out of reach. The first difficulty we stumbled upon was the existence of harmonic measures associated with our Finsler--Laplace operator. Many papers prove their existence in the Riemannian case, or even for weighted Laplace operators (see the remark after Lemma \ref{lem:existence_unicity}) when the symbol is of negative curvature, but none, to my knowledge, was made for our more general case. However, Ancona gives in \cite{Ancona:theorie_potentiel} a very general theorem that implies existence of harmonic measures.

Sections \ref{sec:harmonic_measures} and \ref{sec:existence_of_harmonic} are devoted to stating Ancona's theorem and proof that it applies to our case. But beforehand, we start by recalling some geometrical and dynamical properties of negatively curved Finsler manifolds and use them to give an upper bound for the first eigenvalue of the Finsler--Laplace operator in terms of topological entropy.

If not stated otherwise, in this chapter, $M$ is a closed manifold of dimension $n$ endowed with a Finsler metric of negative curvature $F$ and $\M$ is a fixed universal cover of $M$ endowed with the lifted Finsler metric $\F$.

\section{Negatively curved Finsler manifolds}

Manifolds of negative Finsler curvature enjoy many of the same dynamical and geometrical features of Riemannian ones. We will recall here two of those.

\subsection{Gromov-hyperbolicity} \label{subsec:gromov-hyperbolicity}

Egloff, in his Ph.D Thesis \cite{Egloff:thesis} (the reader can also refer to \cite{Egloff:UFHM} as it is available on-line when the dissertation is not), studied the Finsler equivalent of Cartan Hadamard manifolds that he called \emph{uniform Finsler Hadamard manifolds}.

Note that in Egloff's definition, uniform refers to a control of the quadratic forms $\left( \frac{\partial^2 F^2}{\partial v_i v_j} \right)$, not to a control of the curvature. We do not enter into more details as, for us, uniform Finsler Hadamard manifolds will just be the universal cover of a closed manifold of non-positive curvature. Such manifolds are in particular homeomorphic to $\R^n$ and Egloff studied the property of the Finsler distance and the existence of a visual boundary. He proved:

\begin{thm}[Egloff \cite{Egloff:UFHM}]
 Let $\M$ be a uniform Finsler Hadamard manifold of strictly negative curvature. Then $\M$ is Gromov-hyperbolic.
\end{thm}

\begin{rem}
 Note that Egloff only studied \emph{reversible} metrics, as is normally the case in metric geometry. However Fang and Foulon \cite{FangFoulon} proved that the same theorem holds for non-reversible metrics (with an appropriate definition of Gromov-hyperbolicity).
\end{rem}

We very briefly recall some facts about Gromov-hyperbolic spaces. Proofs, better explanations and much more can be found in \cite{CooDelPap} or \cite{GhysDelaHarpe}.\\

 Let $(V,d)$ be a complete, locally compact, geodesic (i.e., there exists at least one distance-minimizing curve between two points), simply connected metric space.

Let $x,a, b\in V$. Then the \emph{Gromov product} at $x$ of $a$ and $b$ is defined as
\begin{equation*}
 \lbrace a, b\rbrace_x = \frac{1}{2}\left( d(a,x) +d(x,b) -d(a,b) \right),
\end{equation*}
The metric space $V$ is called \emph{Gromov-hyperbolic} if there exists $\delta >0$ such that, for any $x, a,  b, c\in V$,
\begin{equation*}
 \min\left( \lbrace a, b\rbrace_x , \lbrace a, c\rbrace_x \right) \leq \delta \, .
\end{equation*}
If we want to make explicit the constant $\delta$, we say that $V$ is $\delta$-hyperbolic.

The Gromov product is very useful, but unfortunately it is hard (at least for me) to get an insight of what being Gromov-hyperbolic represents using the above definition. An equivalent definition uses geodesic triangles:

The space $V$ is Gromov-hyperbolic if there exists $\delta>0$ such that for any geodesic triangle $(a,b,c) \subset V$, any side is contained in the $\delta$-neighborhood of the union of the remaining sides.

A Gromov-hyperbolic space admits a boundary at infinity $V(\infty)$. One way to define it is to take equivalence classes of geodesic rays; if $O\in V$ is a base point, two geodesic rays $\gamma_1, \gamma_2 \colon \R^+ \rightarrow V$ issuing from $O$ are equivalent if $d(\gamma_1(t),\gamma_2(t) )$ stays bounded for any $t\in \R^+$.

Consider the elements in $V$ as endpoints of geodesic rays starting at $O$ and endow the set of all rays with the uniform convergence on compact topology. Then $\overline{V}:= V \cup V(\infty)$ with the quotient topology is compact, $V$ is a dense open set in $\overline{V}$ and $\partial \overline{V} = V(\infty)$. This boundary is traditionally called the \emph{visual boundary} of $V$ and is independent of the base point $O$.

Using only the Gromov product, we can also define a boundary, which turns out to be homeomorphic to the visual boundary. The advantage of this presentation is that it comes naturally equipped with a metric.

Fix a base point $O\in V$. A sequence $(x_n)$ in $V$ is a Gromov-sequence if $\lbrace x_i, x_j \rbrace_O \rightarrow +\infty$ when $i,j \rightarrow +\infty$. Two Gromov-sequences $(x_n)$ and $(y_n)$ are equivalent if $\lbrace x_i, y_i \rbrace_O \rightarrow +\infty$ when $i \rightarrow +\infty$. Then the set of equivalence classes of Gromov-sequences is the Gromov boundary of $V$.\\
For $\xi, \eta \in V(\infty)$, we define the Gromov product of $\xi$ and $\eta$ by
\begin{equation*}
 \lbrace \xi , \eta \rbrace_O := \inf \liminf_{n \rightarrow +\infty} \lbrace a_n , b_n \rbrace_O \, ,
\end{equation*}
where the infimum is taken over all sequences $a_n$ converging to $\xi$ and $b_n$ converging to $\eta$.

We can now describe the metric on the boundary: let $\epsilon >0$ and set, for any $\xi, \eta \in V(\infty)$,
\begin{equation*}
\rho_{\epsilon}(\xi, \eta):=e^{-\epsilon \lbrace \xi , \eta \rbrace_O}. 
\end{equation*}
Unfortunately, $\rho_{\epsilon}$ does not yet verify the triangle inequality, but can be slightly altered in order to do so. A \emph{chain} between $\xi, \eta \in V(\infty)$ is a finite sequence $\xi = \xi_0, \xi_1,\dots , \xi_n = \eta$ in $V(\infty)$ and we write $\mathcal{C}_{\xi,\eta}$ for the set of chains between $\xi$ and $\eta$. Let $c= (\xi_0, \xi_1, \dots , \xi_n) \in \mathcal{C}_{\xi,\eta}$, define
\begin{align*}
 \rho_{\epsilon}(c) &:= \sum_{i=0}^{n-1} \rho_{\epsilon}(\xi_i, \xi_{i+1}), \\
 d_{G,\epsilon}(\xi,\eta) &:= \inf \lbrace \rho_{\epsilon}(c) \mid c \in \mathcal{C}_{\xi,\eta} \rbrace.
\end{align*}

\begin{prop}  \label{prop:Gromov_metric}
 If $\epsilon >0$ is chosen such that $ e^{\epsilon \delta} < \sqrt{2}$, then $d_{G,\epsilon}$ is a distance on $V(\infty)$, compatible with the above topology.\\
Furthermore,
\begin{equation*}
 (3 - 2 e^{\epsilon \delta}) \rho_{\epsilon}(\xi,\eta) \leq d_{G,\epsilon}(\xi,\eta) \leq \rho_{\epsilon}(\xi,\eta).
\end{equation*}
We call $d_{G,\epsilon}$ a \emph{Gromov metric} on $V(\infty)$.
\end{prop}

The proof of the above proposition is given in Chapter 7 of \cite{GhysDelaHarpe}.

Remark that the boundary of a Gromov-hyperbolic space admits a H\"older structure (see \cite[Chapitre 11]{CooDelPap}).

\subsection{Geodesic flow and entropy}

As is the case in Riemannian geometry, negatively curved Finsler metrics gives hyperbolic dynamics:
\begin{thm}[Foulon \cite{Fou:EESLSPC}]
 The geodesic flow of a negatively curved Finsler manifold is a (contact) Anosov flow.
\end{thm}

We give the precise definition of an Anosov flow and some of its basic properties in the second part of this dissertation (see Section \ref{subsec:basics_AF}). For the moment, we just recall that it is uniformly hyperbolic, i.e., in the tangent space there exist one direction of (uniform) exponential expansion and one direction of (uniform) exponential shrinking.\\

Let $h$ denote the topological entropy associated with the geodesic flow of $F$ (see, for instance \cite{KatokHassel} for equivalent definitions). Manning \cite{Manning} proved that the topological entropy for the geodesic flow of a Riemannian metric of non-positive curvature is the same as the volume entropy, i.e., the exponential growth of the volume of balls. It turns out that this is still true for Finsler metrics

\begin{thm}[Egloff \cite{Egloff:DynamicsUFM}] \label{thm:topo_entropy=volume}
 Let $(M,F)$ be a compact Finsler manifold and $h$ the topological entropy of the geodesic flow. Let $x\in \M$ and $B(R)$ be the ball of radius $R$ centered at $x$, and set 
\begin{equation*}
 h_{\text{vol}} := \lim_{R \rightarrow +\infty}\frac{1}{R}\log\left( \int_{B(R)} \Omega^{\F} \right).
\end{equation*}
Then $h \geq h_{\text{vol}}$.
Furthermore, if $F$ is of non-positive curvature, then 
$$
h = h_{\text{vol}} \, .
$$
\end{thm}
Note that Egloff proved the above result with another volume form, however as $M$ is compact, there exists a constant controlling the ratio of two different volume forms, and this constant disappears when we consider exponential growth.

Note that Egloff \cite{Egloff:DynamicsUFM} also showed, using the Anosov property, that the visual boundary of $\M$ admits a H\"older structure with a constant depending on the Lyapunov exponents.

\section{Bounds for the first eigenvalue}

We denote by $\lambda_1$ the infimum of the essential spectrum of $-\Delta^{\F}$. Recall that it is given by the infimum of the Rayleigh quotients (see Proposition \ref{prop:spectral_gap}).\\

\subsection{A dynamical upper bound}
As in the Riemannian case, we have an upper bound for $\lambda_1$ depending only on the dimension of the manifold and the topological entropy of the geodesic flow.

\begin{prop} \label{prop:upper_bound}
If $M$\ is of dimension $n$ and $h$ is the topological entropy, then
\begin{equation*}
 \lambda_1 \leq n \frac{h^2}{4} \,. 
\end{equation*}

\end{prop}

\begin{rem}
 This bound is far less sharp than in the Riemannian case, where we have $\lambda_1^{\textrm{Riem}} \leq h^2/4$. The additional $n$ appears in the proof because we don't know how to control locally the Finsler metric. It would be interesting to decide whether we could improve this bound to the Riemannian one or if there exist Finsler metrics with $h^2/4 < \lambda_1 \leq nh^2/4$.
\end{rem}

The proof follows the Riemannian one and is based on the following:
\begin{claim}
 Let $x_0 \in \M$ and $\rho(x):= d(x_0,x)$, then $\exp(-s \rho(x))$\ is in $L^2(\M)$\ for any $s> \frac{h}{2}$.
\end{claim}

\begin{proof}
By Theorem \ref{thm:topo_entropy=volume}, we have
\begin{equation*}
 h = \lim_{R \rightarrow +\infty}\frac{1}{R}\log\left( \int_{B(R)} \Omega^{\F} \right).
\end{equation*}
Therefore, if $s > \frac{h}{2}$, $exp(-s \rho(x))$\ is in $L^2(\M)$.
\end{proof}

\begin{proof}[Proof of Proposition \ref{prop:upper_bound}]
Our goal is to find an upper bound for the Rayleigh quotient of $e^{-s\rho(x)}$. We have  
\begin{equation*}
L_X \pi^{\ast} \exp(-s \rho) (x,\xi) = -s \left(L_X\pi^{\ast} \rho\right)(x,\xi) \exp(-s \rho(x)).
\end{equation*}
So, using Fubini theorem,
\begin{equation*}
 \int_{H\M} \left(L_X \pi^{\ast} \exp(-s \rho) \right)^2 \ada = \int_{x\in \M} s^2 \left(\int_{\xi \in H_x\M} \left(L_X\pi^{\ast} \rho (x,\xi)\right)^2 \alpha \right) \exp(- 2 s \rho(x)) \Omega.
\end{equation*}
To deduce the proposition, we just have to bound $\int_{\xi \in H_x\M} \left(L_X\pi^{\ast} \rho (x,\xi)\right)^2 \alpha$ because $\lambda_1$ is the infimum of the Rayleigh quotient. Since $\left|L_X\pi^{\ast} \rho (x,\xi) \right|\leq 1$, we have
\begin{align*}
  \lambda_1 &\leq \frac{n}{\voleucl\left(\S^{n-1}\right)} \frac{ s^2 \int_{\M} \exp(- 2 s \rho)  \int_{H_{x}\M} \alpha\, \Omega}{\lVert \exp(-s \rho) \rVert^2} \\
 &\leq n s^2 \, . 
\end{align*}
\end{proof}

\subsection{A topological lower bound}

In this section we do not need the negative curvature assumption.\\

Bounds on $\lambda_1$ can be obtained through bounds for the Laplace--Beltrami operator associated with the symbol of $\Delta^{\F}$:
\begin{prop}
Let $g_{\sigma}$ be the Riemannian metric on $\M$ given by the dual of the symbol of $\Delta^{\F}$, $\Delta^{\sigma}$ its associated Laplace--Beltrami operator and $\lambda_1(\sigma)$ the infimum of the spectrum of $-\Delta^{\sigma}$. Then, if we let $a\in C^{\infty}(\M)$ such that $\Omega^{\F}= a \Omega^{g_{\sigma}}$, we have
\begin{equation}
\lambda_1(\sigma) \frac{\sup_{x\in \M} a(x)}{\inf_{x\in \M} a(x)} \geq \lambda_1 \geq \lambda_1(\sigma) \frac{\inf_{x\in \M} a(x)}{\sup_{x\in \M} a(x)} \, .
\end{equation}
\end{prop}

\begin{rem}
 As $a$ is a $\pi_1(M)$-invariant function and $M$ is compact, $a$ is bounded, so the above makes sense.
\end{rem}

\begin{proof}
Recall once again that
\begin{equation*}
 \lambda_1 = \inf \frac{\int_{H\M} \left( L_X \pi^{\ast} f \right)^2 \ada}{\int_{\M} f^2 \Omega^F} \,,
\end{equation*}
where the infimum is taken over all functions in $H^1(\M)$. We will prove the lower bound. The proof for the upper bound follows along the same lines.
\begin{align*}
 \int_{H\M} \left( L_X \pi^{\ast} f \right)^2 \ada &= \int_{x\in \M} \left(\int_{H_x \M} \left( L_X \pi^{\ast} f \right)^2 \alpha^F\right) \Omega^F \\
	&= \int_{\M} \lVert\nabla f \rVert_{\sigma}^2 \Omega^F \\
	&\geq \int_{\M} \lVert\nabla f \rVert_{\sigma}^2 \Omega^{g_{\sigma}} \inf_{x\in \M} a(x) \, .
\end{align*}
Therefore, for any $f\in H^1(\M)$,
\begin{equation*}
 \frac{\int_{H\M} \left( L_X \pi^{\ast} f \right)^2 \ada}{\int_{\M} f^2 \Omega^F} \geq \frac{\int_{\M} \lVert\nabla f \rVert_{\sigma}^2 \Omega^{g_{\sigma}} }{\int_{\M} f^2 \Omega^{g_{\sigma}} } \frac{\inf_{x\in \M} a(x)}{\sup_{x\in \M} a(x)} \, .
\end{equation*}
Hence the result.
\end{proof}

We do not know much about lower bounds for $\lambda_1(\sigma)$ because we do not know anything on the metric $g_{\sigma}$, but Brooks \cite{Brooks:pi1_and_spectrum} gave a purely topological condition for it to be strictly positive. Using the above proposition leads to a generalization of Brooks theorem to the Finsler--Laplace operator:
\begin{thm}  \label{thm:Brooks}
One has $\lambda_1 = 0 $ if and only if $\pi_1(M)$ is amenable.
\end{thm}

\begin{proof}
 This is a direct application of the proposition together with the main result of \cite{Brooks:pi1_and_spectrum}.
\end{proof}

\section{Harmonic measures and the Martin Boundary} \label{sec:harmonic_measures}

We are going to leave Finsler geometry for a bit, to give some basics about potential theory. As an introduction, we recall how harmonic measures are obtained in Riemannian geometry.

 Given a Riemannian metric $g$ on $\M$, one way to construct the harmonic measure associated with a point $x\in \M$ is by defining the measure of a Borel set $U \subset \M(\infty)$ as the probability for a Brownian motion for $g$ leaving $x$ to end in $U$. In Riemannian geometry, Brownian motion can be thought of in two ways: one is as the limit of a random walk; the other as the diffusion associated with the Laplacian. We don't know how to generalize the first approach to our Finsler setting, without losing all the Finslerian information, but the second approach comes in fact from the more general theory of elliptic equations and hence will apply in our context.\\
 A related way of viewing harmonic measures is given via solutions to the Dirichlet problem at infinity. If $\M$ is a Cartan--Hadamard manifold of bounded, strictly negative scalar curvature, we have the following result:

\begin{thm}[Anderson \cite{And:Dirichlet_at_infinity}, Sullivan \cite{Sul:Dirichlet_at_infinity}]
\label{thm:Dirichlet_infiny_Riem}
 Let $f \in C^0(\M(\infty))$. There exists a unique function $u_f \in C^{\infty}(\M)$ such that
\begin{equation}
 \left\{
\begin{aligned}
 \Delta u_f &= 0 \quad \textrm{on } \M \\
 u_f(x)  &\rightarrow f(\xi) \quad \textrm{when } x\rightarrow \xi, \; \xi \in \M(\infty) \,.
\end{aligned}
\right. 
\end{equation}

\end{thm}
 Now, to define the harmonic measures, take $x$ in $\M$. There exists a positive linear functional on $C^0(\M(\infty))$ given by $f \mapsto u_f(x)$. This defines a probability measure $\mu_x$ on $\M(\infty)$ which is the harmonic measure at $x$. 

Furthermore, the solution to the Dirichlet problem at infinity is given by
\begin{equation}\label{eq:poisson_representation}
u_f(x) = \int_{\xi \in \M(\infty)} f(\xi) \,d\mu_x (\xi) \, .
\end{equation}

We wish to extend this construction to our Finsler--Laplace operator. One way to prove Theorem \ref{thm:Dirichlet_infiny_Riem} (see \cite{And:Dirichlet_at_infinity,AndersonSchoen,Ancona:elliptic_operators_Martin_boundary}) is to study the Martin Boundary associated with $\Delta$ and show that the harmonic functions are given by Equation \eqref{eq:poisson_representation} where $\mu_x$ is a measure on the Martin boundary. Ancona \cite{Ancona:theorie_potentiel,Ancona:elliptic_operators_Martin_boundary} showed that this method works for a very general class of elliptic operators, requiring no Riemannian setting, but just assumptions on the operator and Gromov-hyperbolicity.\\

The redaction of the next section has been a (possibly failed) challenge for me. My main goal was to present Ancona's theorem (Theorem \ref{thm:Ancona_martin_boundary} below), but, in order to make it somewhat understandable, we need a good deal of background which is, if my personal example is significant, not in the usual toolbox of the average student in Finsler geometry or dynamical systems. Therefore, I tried to include some of the basics I learned on potential theory but not too much for space and time issues. I also tried to give the proofs that seem to me to be important for the understanding of this theory, while using only the tools I introduce. So, all in all, the resulting redaction is probably disappointing for many reasons but I hope that it gives a reasonable idea of the theory used to prove Ancona's result and to obtain harmonic measures.

\subsection{Some potential theory and the Martin Boundary}

In this section, we will recall the construction of the Martin compactification of $\M$. Our main reference is \cite{Ancona:theorie_potentiel}, but if you want to get back to the roots, see \cite{Martin}.

\subsubsection{The Green function}

We start out by recalling some definitions. In the following, $L$ will be a second-order uniformly elliptic differential operator. We will furthermore always assume that $L$ is self-adjoint with respect to some volume form $\Omega$ on $\M$, even if this is not necessary for the general theory.

Recall that an operator is \emph{elliptic} if, for any local coordinates, the symbol of $L$ is a positive-definite matrix. It is \emph{uniformly elliptic} if the positivity of the symbol is uniform, i.e., there exists a constant $c$ depending only on $L$, such that, if $(\sigma_{ij})$ is the symbol of $L$ in a local chart, then $\sum \sigma_{ij}x_i x_j \geq c \sum x_i^2$.
\begin{defin}
Let $U \subset \M$.
\begin{itemize}
 \item A $L$-\emph{harmonic} function on $U$ is a $C^2$ function such that $Lu=0$ on $U$. 
 \item A relatively compact open set $V$ is called \emph{(Dirichlet-)regular} if for any $f\in C(\partial V)$ there exists a unique function $u\in C(\overline{V})$ harmonic on $V$ such that $u=f$ on $\partial V$.
 \item If $V$ is a regular open set and $x\in V$, we denote by $\mu_{x}^{V}$ the harmonic measure at $x$ relatively to $V$. That is, the 	only Borel measure such that, for any $f \in C(\partial V)$, 
 \begin{equation}
  H_f^{V}(x) = \int_{\partial V} f(\xi) d\mu_x^{V}(\xi)
 \end{equation}
 is harmonic on $V$.
 \item A function $u: U \rightarrow ]-\infty , +\infty]$ is $L$-\emph{superharmonic} on $U$ if $u$ is lower semi-continuous and if, for any regular set $V \subset\subset U$, $u \geq H_{u}^V$ on $V$.
 \item A superharmonic function $s$ is \emph{non-degenerate} on $U$ if for any $V \subset\subset U$ regular, $s_{|V}$ is harmonic.
 \item A function $u$ is a $L$-\emph{potential} on $U$ if $u \geq 0$, superharmonic on $U$ and such that any $L$-harmonic function on $U$ smaller than $u$ is non-positive.
\end{itemize}
 
\end{defin}

\begin{rem}
 A $C^2$ function is superharmonic if and only if $L u \leq 0$.
\end{rem}
Harmonic functions enjoy the following fundamental property:
\begin{prop}[Harnack's principle] \label{prop:Harnack}
 Let $U$ be a domain of $\M$, then for any $p\in U$, the set of functions $\lbrace u \mid u \; L-\textrm{harmonic, positive on U}, \; u(p)=1 \rbrace$ is compact for the topology of uniform convergence on compact subsets of $U$.
\end{prop}

\begin{prop}[Harnack's inequality]
 Let $u$ be a harmonic function on a bounded domain $U$ of $\M$. Then there exists a constant $c$ such that 
\begin{equation}
 \sup_{x\in U} u(x) \leq c \inf_{x\in U} u(x)\, .
\end{equation}
\end{prop}
Another fundamental piece of the theory is the maximum (or minimum) principle:
\begin{prop}[Minimum Principle]
  Let $U$ be a bounded domain in $\M$ and $u\in C^2(U)\cap C^0(\overline{U})$ such that $Lu\leq 0$. If $u(y) \geq 0$ for every $y\in \partial U$, then $u(x) \geq 0$ on $U$.
\end{prop}

Note that there exists also a global version:
\begin{prop} \label{prop:global_min_principle}
 Suppose that $u$ is a potential on $\M$, harmonic outside a closed set $F$ and continuous on $\partial F$. If $s$ is a non-negative superharmonic function on $\M$ such that $s \geq u$ on $\partial F$, then $s \geq u$ on $\M \smallsetminus F$.
\end{prop}

To get on with potential theory, we need to assume that the operator $P = L -\frac{\partial}{\partial t}$ admits a fundamental solution, the \emph{heat kernel} of $L$.
\begin{defin}
 The heat kernel of $L$ is a positive function $p(x,y,t)$ defined on $M\times M \times \R$, vanishing identically when $t\leq 0$, continuous for $t>0$ and $x \neq y$, $C^2$ with respect to $y$, $C^1$ in $t$ and such that:
\begin{enumerate}
 \item For any fixed $x\in M $,
      \begin{equation}
	P p(x, \cdot, \cdot) = 0 \quad \text{on} \; M\times \R \smallsetminus\{(x,0)\} ;
      \end{equation}
 \item For any bounded continuous function $f$ on $M$, 
      \begin{equation}
       \lim_{t\rightarrow 0} \int_{y\in M} p(x,y,t) f(y) \, \Omega_y = f(x)\,.
      \end{equation}
\end{enumerate}
\end{defin}

Such a function always exists for a uniformly elliptic operator with uniform H\"older continuous coefficients. See, for instance, \cite{Friedman:PDE_of_parabolic_type} for the construction using the parametrix method in $\R^n$, or \cite{Azencott:diffusion_semi-groups} for a diffusion approach on manifolds. For an approach more specifically adapted to the case we will be interested in in the next sections, the reader can consult \cite{Grigoryan:heat_kernels_metric_measure_space,Grigoryan:heat_kernels_on_weighted_manifolds}.

\begin{defin}
 The Green function of $L$ is defined, for $(x,y) \in M\times M$, by
\begin{equation}
 G(x,y) = \int_0 ^{+\infty} p(x,y,t) dt\, .
\end{equation}
\end{defin}

\begin{rem}
 When $G$ is not identically infinite, then, for any fixed $x$, $G(x,\cdot)$ is a $L$-potential.\\
 Note also that, if $L$ is self-adjoint, then $G(x,y) = G(y,x)$.
\end{rem}

The following result gives us a characterization of when the Green function is not identically infinite:
\begin{thm}[\cite{Ancona:theorie_potentiel}, Th\'eor\`emes 1 and 13] \label{thm:green_function}
 The following propositions are equivalent:
\begin{enumerate}
 \item There exist $x_0,y_0$ in $M$ such that $G(x_0,y_0) <\infty$;
 \item The function $G$ is finite and continuous on $M\times M \smallsetminus\{x=y\}$ and, for $x \in M$, $G(x, \cdot)$ is a $C^2$ $L$-harmonic function on $M\smallsetminus \{x\}$;
 \item There exists a strictly positive $L$-potential on $M$;
 \item There exists a superharmonic, non-degenerate, non-negative function on $M$ which is not $L$-harmonic.
\end{enumerate}

\end{thm}

In the following, $L$ will be assumed to satisfy one of the above equivalent properties.

\subsubsection{Martin Compactification}

Let $O\in \M$ be fixed, for $x\in \M$, let 
\begin{equation}
K_x(y):= \frac{G(x,y)}{G(x,O)} \, .
\end{equation}
We say that a sequence $(x_n)$ converging to infinity in $M$ converges to a Martin point if $(K_{x_n})$ is pointwise convergent. By the Harnack principle (see Proposition \ref{prop:Harnack}) any sequence converging to infinity admits a subsequence converging to a Martin point and, if we denote by $K_{\xi}$ the function associated to a Martin point $\xi$, then $K_{\xi}$ is a non-negative harmonic function such that $K_{\xi}(O)=1$. 

By definition, we say that two sequences define the same Martin point if and only if the limit functions are the same. We will therefore often think of $\xi$ and $K_{\xi}$ as the same thing. We write $\MM$ for the set of Martin points and $\hat{M} =  \M \cup \MM$. We define on $\hat{M}$ the following metric: for $x,x'\in \hat{M}$
\begin{equation}
 \rho(x,x') = \sup_{y\in B(O,1)} |K_{x}(y) - K_{x'}(y)|,
\end{equation}
where $B(O,1)\subset M$ is the ball of center $O$ and radius $1$.
\begin{prop}[Martin \cite{Martin}]
 The space $\hat{M}$ equipped with the metric $\rho$ is a complete, compact space for which $\MM$ is the boundary and $M$ the interior. Furthermore, the topology induced on $M$ coincides with its natural topology.
 The space $\hat{M}$ is called the (L-)\emph{Martin compactification} of $M$.
\end{prop}

\begin{proof}
The fact that $\rho$ is a metric is straightforward from the definition.\\
\emph{Both topologies coincide on $M$}: If $x\in M$ and $x_n$ is a sequence that converges to $x$ (for the distance on $M$), then, by continuity of the Green function, $\rho(x,x_n)$ also tends to zero.\\
\emph{$\MM$ is compact}: Let $\left( \xi_n \right) $ be a sequence of points in $\MM$ and $\left(x_n^i\right)$ a sequence of points converging to $\xi_n$. Choose $\left( K_n \right)$ a sequence of compact sets such that $K_{n+1} \supset K_n$ and $M = \cup K_n$. Write $y_n$ for a point in the sequence $x_n^i$ outside of $K_n$ and such that $\rho(y_n, \xi_n) \leq \frac{1}{n}$. Then $\left(y_n\right)$ is a sequence that leaves all compact sets of $M$, therefore admits a subsequence converging to a Martin point $\xi$ and by construction, $\rho(\xi_n, \xi)$ tends to zero.\\
\emph{$\hat{M}$ is compact}: If $\left(x_n\right)$ is a sequence in $\hat{M}$, then either there is a subsequence that stays in a compact set of $M$, or there is a subsequence that leaves every compact set of $M$ and there is a subsequence converging to a Martin point, or, finally, there is a subsequence staying in $\MM$, and we apply the previous fact.\\
\emph{$M$ is open inside $\hat{M}$ with boundary $\MM$}: this is obvious.
\end{proof}

Let $\mathcal{H}_+$ be the convex cone of positive $L$-harmonic functions and write ${\mathcal{K}:= \lbrace u \in \mathcal{H}_+ \mid u(O)=1 \rbrace}$ for the subset of normalized $L$-harmonic functions. The space $\mathcal{K}$ is a base of $\mathcal{H}_+$, it is a convex set and we denote by $E$ its extremal points. A harmonic function $u$ such that $u/u(O)$ is in $E$ is called \emph{minimal}. A function $u$ is minimal iff $u$ does not dominate any other harmonic function apart from multiples of itself (see \cite{Martin}).

\begin{thm}[Martin \cite{Martin}] \label{thm:Martin_contient_partie_minimal}
For any $u\in E$, there exists a Martin point $\xi$ such that $K_{\xi} = u$.
\end{thm}
 The proof can be found, for instance, p. 31 of \cite{Ancona:theorie_potentiel}.\\

The Martin boundary does not in general coincide with $E$, but the following result shows the importance of the case when $\MM$ is reduced to its minimal part.
\begin{prop} \label{prop:representation_integrale_fonctions_harmonique}
 For any $u \in \mathcal{H}_+$, there exists a unique positive and finite Borel measure $\mu_u$ on $E$ such that, for $x \in M$, 
\begin{equation}
 u(x) = \int_{\xi \in E} K_{\xi}(x) d\mu_u(\xi) \, .
\end{equation}
\end{prop}

\begin{rem}
 If the Martin boundary is reduced to $E$, then any harmonic function is obtained as the integral of the $K_{\xi}$ over $\MM$. It is not quite Equation \eqref{eq:poisson_representation} that we are ultimately looking for, but it is getting closer.
\end{rem}

\begin{proof}
The existence is given by Choquet's theorem and the uniqueness follows from the fact that $\mathcal{K}$ is a base of a cone $\mathcal{H}_+$ which is a lattice (see \cite[Chapter 10]{Phelps}).
\end{proof}

\subsection{Ancona's Theorem}

In \cite{Ancona:elliptic_operators_Martin_boundary}, Ancona shows that, for an operator of the form $L= \Delta + \langle B, \nabla \rangle$, where $\Delta$ and $\nabla$ come from a Riemannian metric of bounded negative curvature with $B$ satisfying certain conditions, then the Martin boundary is homeomorphic to the visual one.

In our case, we do not know whether the Riemannian metric we obtain from the symbol is of negative curvature, and hence cannot apply directly this result. However, the proof Ancona gives remains true for a generic Gromov-hyperbolic space with a suitable elliptic operator and he presents it in the general setting in \cite{Ancona:theorie_potentiel}.

\begin{thm}[Ancona] \label{thm:Ancona_martin_boundary}
 Let $\M$\ be a Gromov-hyperbolic space, $L$\ a self-adjoint elliptic operator, and denote by $d(\cdot ,\cdot)$\ the distance on $\M$. Suppose that $L$\ verifies:\\
There exist strictly positive constants $r_0, \, \tau, \, c_1$, and $c_2$ such that 
\begin{enumerate}[(H1)]
\item For all $ m \in \M$, there exists a function $\theta: B(m,r_0) \rightarrow \R^n $, such that for all $ x,y \in B(m,r_0) $, we have
\begin{equation*}
 c_2^{-1} d(x,y) \leq \lVert \theta(x) -\theta(y) \rVert \leq c_2 d(x,y)\,;
\end{equation*}
 \item For all $ x_0 \in \M$ and $\forall 0\leq t \leq \tau $\ the Green function $g_t$\ relative to $B(x_0,r_0)$\ of the operator $L + t \id$ satisfies: for all $x,y \in B(x_0,r_0/2)$
\begin{equation*}
 \left\{
\begin{aligned}
  c_1 &\leq g_t(x,y) \\
  g_t(x,y) &\leq c_2, \quad \text{if }  d(x,y) \geq \frac{r_0}{4} \,;
 \end{aligned}
\right.
\end{equation*}
 \item There exists $\eps$, $0<\eps<\frac{1}{2}$, such that for all $ t < 2\eps$, $L + t\id$\ admits a Green function $G_t$\ on $\M$.

\end{enumerate}

Then the Martin compactification of $\M$ is homeomorphic to $\M \cup \M(\infty)$ and the Martin boundary $\MM$ is reduced to its minimal part. Furthermore, if we have chosen a base point $O$ and denote by $K_{\xi} \in \MM$ the harmonic function corresponding to $\xi \in \M(\infty)$, then the application $(\xi,x) \mapsto K_{\xi}(x)$ is continuous on $\M(\infty) \times \M$.

\end{thm}

\begin{rem}
 Ancona proves also that, given the above conditions on $L$, the Green function tends to zero on the boundary: there exist constants $c>0$ and $\alpha >0$, depending on $L$, such that, for any $x,y \in \M$, if $d(x,y) \geq 2$, 
\begin{equation*}
 G(x,y) \leq c e^{-\alpha d(x,y)}.
\end{equation*}
\end{rem}

The main point in the proof is to give an estimate for the Green function of $L$:
\begin{thm}[Ancona, Theorem 6.1 of \cite{Ancona:theorie_potentiel}] \label{thm:utilisation_unique}
 If $\M$ is a Gromov-hyperbolic manifold and $L$ verifies Conditions (H1), (H2) and (H3), then, for any geodesic segment $[x,z]$ in $\M$ and any point $y$ on it satisfying ${\min \left( d(x,y) , d(y,z) \right) \geq 1}$, there exists a constant $c>0$ depending only on $L$ and the $\delta$ given by the Gromov-hyperbolicity such that the Green function $G$ of $L$ verifies
\begin{equation*}
 c^{-1} G(y,x)G(z,y) \leq G(z,x) \leq c G(y,x) G(z,y)\, .
\end{equation*}
\end{thm}

Note that this theorem contains in fact two results: one, hard and general, giving estimates of the Green function along what Ancona calls $\phi$-chain (Theorem 5.2 of \cite{Ancona:theorie_potentiel}) and one, easier, showing that, when the space is Gromov-hyperbolic, you obtain $\phi$-chain by following geodesics.\\

Assuming this result, we can start the:
\begin{proof}[Sketch of proof of Theorem \ref{thm:Ancona_martin_boundary}]
 Ancona splits his proof into three steps:
\begin{itemize}
 \item Choose a base point $O \in \M$, let $\xi \in \M(\infty)$ and $\gamma \colon \R^+ \rightarrow \M$ be a geodesic ray issuing from $O$ and ending in $\xi$. Set $x_j := \gamma\left((j+1) \delta \right)$ and $V_j := \lbrace x\in \M \mid \lbrace x,x_j \rbrace_O \geq j\delta \rbrace$ where $\delta$ is a constant coming from the Gromov-hyperbolicity assumption. For $x\in \M$, recall that $K_x(y) := G(x,y)/G(x,O)$. Then, using Theorem \ref{thm:utilisation_unique}, we can deduce that there exists a constant $c$ such that, for any $x\in V_{j+1}$ and $y\in \M\smallsetminus V_j$,
\begin{equation}
 c^{-2} K_{x_j}(y) \leq K_{x}(y) \leq c^2 K_{x_j}(y)\, . \label{eq:utilisation_unique4}
\end{equation}
Now take a sequence $(z_n)$ in $\M$ converging to $\xi \in \M(\infty)$ and such that $K_{z_n}$ converges to a harmonic function $h$, i.e., the sequence $(z_n)$ is a Martin point. By Gromov-hyperbolicity, $V_j$ is a base of neighborhood of $\xi$ (we can deduce that from the characterization of the visual boundary via the Gromov-compactification, see Section \ref{subsec:gromov-hyperbolicity}). So, we can deduce that for $ j \geq 2$ and $y\in  \M\smallsetminus V_j$,
\begin{equation*}
 c^{-2} K_{x_j}(y) \leq h(y) \leq c^2 K_{x_j}(y)\, .
\end{equation*}
Note that, as a conclusion to that first step, we proved the following: If $h$ is a harmonic function corresponding to a Martin point $(z_n)$ which, as a sequence in $\M$, converges to $\xi \in \M(\infty)$, then, for any $\xi' \neq \xi \in \M(\infty)$, there exists a neighborhood $V$ of $\xi'$ such that $h$ is bounded above by a multiple of $K_{x_j}$ on $\M \cap V$. Using the Harnack principle, we can deduce that $h$ is bounded by a multiple of $G(\cdot, O)$, with a constant depending on the point $\xi'$.

\item Now define $\mathcal{H}_{\xi}$ the cone of non-negative $L$-harmonic functions on $\M$ that are bounded above by a multiple of $G(\cdot, O)$ in a neighborhood of any $\xi' \in \M(\infty)$ different from $\xi$. Ancona proves that, for any $h \in \mathcal{H}_{\xi}$, the following holds: for any $j\geq2$ and $y\in \M \smallsetminus V_j $,
\begin{equation*}
 c^{-2} h(O) K_{x_j}(y) \leq h(y) \leq c^2 h(O) K_{x_j}(y)\, .
\end{equation*}
From this we deduce that, for any $h,h' \in \mathcal{H}_{\xi}$ and any $y\in \M$, we have
\begin{equation*}
 c^{-4} h(O) h'(y) \leq h'(0)h(y) \leq c^2 h(O) h'(y)\, ,
\end{equation*}
that is, any element of $\mathcal{H}_{\xi}$ is bounded by a multiple of any other (non-zero) elements.

 \item Recall that, in the first step, we showed that $\mathcal{H}_{\xi}$ is not reduced to zero. Now the second step induces that $\mathcal{H}_{\xi}$ is one-dimensional. Indeed, take two strictly positive elements $h,h' \in \mathcal{H}_{\xi}$ and set $\lambda:= \sup\left\{ \frac{h(y)}{h'(y)} : y\in \M \right\}$. By definition, $\lambda h' - h \in \mathcal{H}_{\xi}$, but, by our choice of $\lambda$, $\lambda h' - h$ cannot be bounded below by a multiple of $h'$. We set $K_{\xi}$ the unique element of $\mathcal{H}_{\xi}$ such that $K_{\xi}(O)= 1$.
\end{itemize}
Given a point $\xi \in \M(\infty)$ and any sequence $(z_n)$ in $\M$ converging to $\xi$, by Equation \eqref{eq:utilisation_unique4} and the Harnack principle, we know that $(K_{z_n})$ must converge simply and the third step shows that the limit must be $K_{\xi}$. Moreover $K_{\xi}$ is minimal because if $u$ is a positive harmonic function bounded above by $K_{\xi}$, then $u$ is in $\mathcal{H}_{\xi}$, so if $u(O)=1$ then $u= K_{\xi}$. So we constructed an application from $\M(\infty)$ to the minimal part of $\MM$.\\
 
 Let $\xi,\xi'$ be two distinct points in $\M(\infty)$, then $K_{\xi} \neq K_{\xi'}$. Indeed, otherwise we would have shown that $K_{\xi}$ is bounded above by a multiple of $G(\cdot, O)$ on all of $\M(\infty)$ and therefore on all of $\M$. Hence $G(\cdot, O)$ would be bounded below by a positive harmonic function, which is impossible because $y \mapsto G(y,O)$ is a potential.\\
 
 So our application $\M(\infty) \rightarrow \MM$ is injective, and we have everything to see that it is also surjective: take a Martin sequence $(z_n)$ such that $\left(K_{z_n}\right)$ converges to $h$, there exists a subsequence that converges to an element $\xi \in \M(\infty)$ so $h = K_{\xi}$ and therefore $(z_n)$ must converge to $\xi$.
\end{proof}

\begin{rem} \label{rem:characterization_of_Poisson_kernel}
 Note that the proof gives the following characterization of $K_{\xi}$: it is the only harmonic function such that $K_{\xi}(O)=1$ and that is bounded by a multiple of $ G(\cdot, O) $ in a neighborhood of every point in $\M(\infty) \smallsetminus \{\xi \}$. The second condition really gives a meaning to ``$K_{\xi}$ is zero on $\M(\infty) \smallsetminus \{\xi \}$''.
\end{rem}

Some corollary of this result is that the Dirichlet problem at infinity for an operator satisfying (H1) to (H3) has a unique solution. The proof is once again copied from \cite{Ancona:theorie_potentiel}. Note that it is also a direct transcription of the classical proof of the Dirichlet Problem for the Euclidean disc via the Poisson integral formula (see for instance \cite{GilTru}).

\begin{thm}[Dirichlet problem at infinity] \label{thm:solution_dirichlet_a_infini}
Let $(\M,L)$ be as above and $L(\mathds{1}) =0$. Then, for any $f\in C(\M(\infty))$, there exists a unique $u \in C(\M\cup \M(\infty))$ such that $u = f$ on $\M(\infty)$ and $Lu= 0$ on $\M$.\\
Furthermore, if we choose a base point $O$ on $\M$, we can write
\begin{equation*}
 u(x):= \int_{\xi \in \M(\infty)} K_{\xi}(x) f(\xi) d\mu(\xi)\,.
\end{equation*}
\end{thm}

\begin{proof}
Note that uniqueness is a direct consequence of the maximum principle.

With Ancona's theorem, applying Proposition \ref{prop:representation_integrale_fonctions_harmonique} to $\mathds{1}$ shows that we have a measure $\mu$ on $\M(\infty)$ such that 
\begin{equation*}
 \mathds{1}(x) = \int_{\xi \in \M(\infty)} K_{\xi}(x) d\mu(\xi) \, ,
\end{equation*}
where we chose a base point $O \in \M$ for the normalization of $K_{\xi}$. Now let 
\begin{equation*}
 u(x):= \int_{\xi \in \M(\infty)} K_{\xi}(x) f(\xi) d\mu(\xi)\, ,
\end{equation*}
$u$ is $L$-harmonic by definition, so we just need to prove that $u(x)$ tends to $f(\xi)$ when $x \rightarrow \xi$.\\
Let $\xi_0 \in \M(\infty)$ and $V$ a neighborhood of $\xi_0$, we write
\begin{equation*}
 u(x)- f(\xi_0):= \int_{\M(\infty) \smallsetminus V}  K_{\xi}(x) (f(\xi)-f(\xi_0) ) d\mu(\xi) + \int_{V} K_{\xi}(x) (f(\xi)-f(\xi_0) ) d\mu(\xi)\, .
\end{equation*}
Now, for $x$ in a smaller neighborhood of $\xi_0$, $K_{\xi}(x)$ is bounded above by a multiple of $G(x, O)$ (by the first step in the proof of Ancona's Theorem), so the first part in the RHS of the above equation is small. The second part is easily seen to be small because $ \int_{\M(\infty)} K_{\xi}(x) d\mu(\xi) = 1$, which proves the theorem.
\end{proof}

In fact, we can say more about the regularity of the identification between the Martin and the visual boundary. Recall (Proposition \ref{prop:Gromov_metric}) that there is a natural metric on the boundary of a Gromov-hyperbolic space $d_{G, \epsilon}$, where $\epsilon$ is a small real depending only on $\delta$, and that we have, for $\xi, \eta \in \M(\infty)$,
\begin{equation} \label{eq:gromov_metric}
 (3 - 2 e^{\epsilon \delta}) e^{-\epsilon \lbrace \xi, \eta \rbrace_O}\leq d_{G,\epsilon}(\xi,\eta) \leq e^{-\epsilon \lbrace \xi, \eta \rbrace_O}.
\end{equation}
Moreover, the boundary admits a H\"older structure (see \cite[Chapitre 11]{CooDelPap}).

\begin{thm} \label{thm:Martin_is_holder}
 Let $(\M, L)$ satisfying the conditions of Theorem \ref{thm:Ancona_martin_boundary}. There exists a constant $\alpha>0$ depending on $\M$ and $L$ such that the identification $\MM \rightarrow \M(\infty)$ is $\alpha$-H\"older. 

Moreover, if we denote by $K(O,x,\xi)$ the Poisson kernel normalized at $O$, for any compact $K \subset \M$, the application $(O,x,\xi) \in K\times K \times \M(\infty) \mapsto K(O,x,\xi)$ is H\"older-continuous.
\end{thm}

The above result is a direct consequence of the following, more technical result:

\begin{thm} \label{thm:C_alpha_extension}
 Let $(\M, L)$ satisfying the conditions of Theorem \ref{thm:Ancona_martin_boundary} and $U$ an unbounded domain in $\M$. Suppose that $u$ and $v$ are two harmonic functions on $U$, continuous on $\overline{U}$, and such that $u_{|\overline{U}\cap \M(\infty)}= v_{|\overline{U}\cap \M(\infty)}=0$. Then the quotient $u/v$ has a $C^{\alpha}$ extension to $\overline{U}\cap \M(\infty)$, where $\alpha$ depends on $M$ and $L$.
\end{thm}

As far as I know there is no published proof of that result in this generality. However, the hard part is due to Ancona \cite{Ancona:elliptic_operators_Martin_boundary} and Anderson and Schoen \cite{AndersonSchoen}, the only personal contribution being the following lemma. Note also that there are some closely related results of this generality: in \cite{IzuNeshOka}, M. Izumi, S. Neshveyev and R. Okayasu prove that the Martin kernel for a random walk on a hyperbolic group is H\"older continuous.
\begin{lem} \label{lem:minoration_V_i} %
 Let $\gamma$ be a geodesic ray from $O$ to a point $\xi_0 \in \M(\infty)$. We define $A_i := \gamma(4i\delta)$ and $A_i' :=\gamma((4i+2)\delta)$.  Let $V_i :=  \lbrace x\in M \mid \lbrace x, A'_i \rbrace_O > 4 i \delta \rbrace$.\\
 Let $\xi$ and $\eta$ be the two points in $\partial \left(\M(\infty) \cap \overline{V_i} \right)$, then 
\begin{equation} \label{eq:majoration_produit_gromov}
 \lbrace \xi, \eta \rbrace_O \leq \left( 4(i + 1) + 2\right) \delta + 17 \delta\, .
\end{equation}
Therefore, for $i \geq 12 + \dfrac{-\log(3-2\exp(\epsilon \delta))}{2\epsilon \delta}$,
\begin{equation*}
 \sup_{a,b \in \overline{V_i}} d_{G,\epsilon}(a,b) \geq e^{-8 i \epsilon \delta } \,.
\end{equation*}
\end{lem}

\begin{proof}
First note that $\sup_{a,b \in \overline{V_i}} d_{G,\epsilon}(a,b) \geq (3 - 2 e^{\epsilon \delta})e^{-\epsilon \lbrace \xi,\eta\rbrace_O}$ (see Equation \eqref{eq:gromov_metric}), so proving Equation \eqref{eq:majoration_produit_gromov} will indeed give us the result.\\

Now remark that, for any point $p_0$ on the geodesic $(\xi,\eta)$, $\lbrace \xi,\eta\rbrace_O \leq d(O, p_0)$. Indeed, if we take two sequences $(a_n)$ and $(b_n)$ on $(\xi,\eta)$ converging respectively to $\xi$ and $\eta$, we have
\begin{align*}
 \lbrace \xi,\eta\rbrace_O &= \lim_{n \rightarrow +\infty} \frac{1}{2}\left( d(a_n, O) + d(b_n, O) - d(a_n, b_n) \right) \\
	  &\leq \lim_{n \rightarrow +\infty} \frac{1}{2}\left( d(a_n,p_0 ) + d(p_0, O) + d(b_n, p_0) + d(p_0, O)  - d(a_n, b_n) \right) \\
	  &\leq d(p_0,O)\, .
\end{align*}

Let  $p:= \gamma \cap (\xi,\eta)$ (we take $i$ big enough so that this intersection exists). Depending on the position of $p$, we have two possibilities:
\begin{itemize}
 \item Either $d(O,p) \leq d(O, A_i') = (4i +2) \delta $ which implies Equation \eqref{eq:majoration_produit_gromov},
 \item Or $d(O,p) \geq d(O, A_i')$; then, $p$ is in $V_i$ (because the part of $\gamma$ after $A_i$ is in $V_i$), so we again have two possible cases; 
 \begin{itemize}
  \item Either $p \notin V_{i+1}$ and therefore $d(O,p) \leq d(O, A_{i+1}) = 4(i+1) \delta $ which again implies Equation \eqref{eq:majoration_produit_gromov};
  \item Or  $p \in V_{i+1}$; in that case, we know (using Scholie 3.1 of \cite{Ancona:theorie_potentiel}) that the geodesic ray $[p,\xi)$ must pass no further than $17\delta$ from $A_{i+1}'$; if we denote by $p'$ the point realizing that distance, we have $d(O,p') \leq 17\delta + d(O, A_{i+1}') = \left( 4(i + 1) + 2\right) \delta + 17 \delta$, which ends the proof.
 \end{itemize}
\end{itemize}
\end{proof}

\begin{proof}[Proof of Theorem \ref{thm:C_alpha_extension}]
Let $\xi_0 $ be an interior point of $ \overline{U} \cap \M(\infty)$, $O \in U$ a base point and $\gamma$ a geodesic ray from $O$ to $\xi_0$. Set $A_i := \gamma(4 i \delta)$ and $A'_i := \gamma(4(i+2)\delta ) $, where $\delta>0$ is given by the Gromov-hyperbolicity of $\M$. Now we define $V_i:= \lbrace x\in \M \mid \lbrace x, A'_i \rbrace_O >4i\delta \rbrace$, the $V_i$'s form a basis of neighborhoods of $\xi_0$ and are such that $\overline{V_{i+1}} \subset V_i$. So for $i$ big enough, $\overline{V_i} \subset U $.\\ 

Replacing their $C_i$'s by $V_i$, we can copy verbatim the proof of Theorem 6.2 of \cite{AndersonSchoen} and obtain that $u/v$ admits a radial extension $\varphi$ to $\overline{U}$ and that there exists a constant $c_1>0$, depending on $\delta$ and $L$, such that
\begin{equation} \label{eq:utilisation_unique_10}
 \sup_{x\in \overline{V_i}} \varphi(x) - \inf_{x\in \overline{V_i}} \varphi(x) \leq c_1^i \varphi(O)\, .
\end{equation}
The only change that needs to be done is to use Ancona's Harnack inequality at infinity given by Theorem 5' of \cite{Ancona:elliptic_operators_Martin_boundary}, instead of the Harnack inequality Anderson and Schoen use. Theorem 5' applies because the $(A_i, V_i)$ form a $\phi$-chain (see the proof of Theorem 6.1 in \cite{Ancona:theorie_potentiel}).\\

Now, by Lemma \ref{lem:minoration_V_i} above, we have, for $i \geq c_2$, where $c_2$ depends only on $\delta$ (as $\epsilon$ depends only on $\delta$),
\begin{equation} \label{eq:utilisation_unique_11}
 \sup_{a,b \in \overline{V_i}} e^{-\lbrace a,b\rbrace_O} \geq e^{-8 i \epsilon \delta }.
\end{equation}

So putting together Equation \eqref{eq:utilisation_unique_10} and Equation \eqref{eq:utilisation_unique_11}, we obtain that, for $y$ and $y'$ sufficiently far from $O$, setting $\alpha := (1/8 \epsilon \delta) \log \left(1/c_1  \right) >0$
\begin{equation*}
 \left| \varphi(y) - \varphi(y') \right| \leq \left[e^{-\lbrace y,y'\rbrace_O} \right]^{\alpha} \varphi(O) \, .
\end{equation*}
This proves that the extension $\varphi$ is in fact $C^{\alpha}$.
\end{proof}

Finally we can deduce Theorem \ref{thm:Martin_is_holder} from Theorem \ref{thm:C_alpha_extension} (see \cite[Theorem 6.3]{AndersonSchoen}):
\begin{proof}[Proof of Theorem \ref{thm:Martin_is_holder}]
 Recall that, for $\xi \in \M(\infty)$, the Poisson Kernel normalized at $O$, $K(O, \cdot, \xi )$ is such that
\begin{equation*}
 K(O, x, \xi ) = \lim_{n \rightarrow \infty} \frac{G(x,y_n)}{G(O,y_n)}\,,
\end{equation*}
 where $(y_n)$ is a sequence converging to $\xi$. Applying Theorem \ref{thm:C_alpha_extension} to the $G(x,y_n)$ gives that there exists a constant $C$ such that, for any $x\in B(O,1)$ and any distinct $\xi, \xi' \in \M(\infty)$
\begin{equation*}
 \left| K(O, x, \xi ) - K(O, x, \xi' ) \right| \leq C \left[ d_G (\xi, \xi' ) \right]^{\alpha},
\end{equation*}
which proves the theorem.
\end{proof}

\section{Existence of Finsler--Laplace harmonic measures} \label{sec:existence_of_harmonic}

We finally get back to Finsler geometry. Recall that $\M$ is the universal cover of a closed manifold $M$ and $\F$ is a Finsler metric of negative flag curvature lifted from $M$.

 From now on, we assume that $\F$ is \emph{reversible}. Indeed Ancona's Theorem is proved only for symmetric distance. Note however, that Fang and Foulon \cite{FangFoulon} showed that, for irreversible Finsler metric of negative curvature, there exist two boundaries at infinity: one is given by following the geodesics into the future and the other into the past. It seems very probable that if we consider a non-symmetric distance and redo the steps of Ancona's proof, we should obtain identifications of the past and forwards boundaries with the Martin boundary by taking the Poisson kernels $K_{\xi}$ along forward, respectively backwards, geodesics. But, if this is true, proving it will remain a project for later (or for any interested reader).

\begin{thm}\label{thm:Finsler-Laplace_is_Ancona}
 If $(M,F)$ is a closed reversible Finsler manifold of strictly negative curvature and $(\M,\F)$ its universal cover, then $\Delta^{\F}$ verifies Conditions (H1) to (H3) of Theorem \ref{thm:Ancona_martin_boundary}.
\end{thm}

Before starting the proof, let us state the main corollary:
\begin{cor} \label{cor:harmonic_exists}
 There is a $C^{\alpha}$ identification between the Martin and the visual boundaries of $(\M,\F)$ and the Dirichlet problem at infinity for $\Delta^{\F}$ admits a unique solution.\\
 That is, for any $f\in C(\M(\infty))$, there exists a unique $u \in C(\M\cup \M(\infty))$ such that $u = f$ on $\M(\infty)$ and $\Delta^{\F}u= 0$ on $\M$.\\
Furthermore, for any $x\in \M$, there exists a measure $\mu_x$, called the \emph{harmonic measure} for $\Delta^{\F}$ such that
\begin{equation*}
 u(x):= \int_{\xi \in \M(\infty)} f(\xi) d\mu_x(\xi)\, .
\end{equation*}
\end{cor}

\begin{proof}
 The first part of the corollary is Theorem \ref{thm:Martin_is_holder} and the second part is Theorem \ref{thm:solution_dirichlet_a_infini} with $d\mu_x := K_{\xi}(x) d\mu$.
\end{proof}

Condition (H1) just means that $\M$ is of ``bounded geometry'' and it would stay true for any uniform Finsler Hadamard manifold in Egloff's sense. Conditions (H2) and (H3) come from the following:

\begin{prop}
 The operator $\Delta^{\F}$ is coercive, i.e., there exists $c>0$ such that, for ${f \in C^{\infty}_0}$,
\begin{equation*}
 \int_{H\M} \left( L_X \pi^{\ast} f \right)^2 \ada \geq c \int_{\M} f^2 \Omega\, .
\end{equation*}
\end{prop}

\begin{proof}
 By Theorem \ref{thm:Brooks}, we know that $\lambda_1$ is strictly positive because $\pi_1(M)$ is Gromov-hyperbolic, hence not amenable, and the characterization of $\lambda_1$ as the infimum of the Rayleigh quotient (see Proposition \ref{prop:spectral_gap}) shows that we can take $c$ to be $\lambda_1$.
\end{proof}

We can deduce (H3) from there (see \cite[Lemma 2]{Ancona:elliptic_operators_Martin_boundary}, we recall the proof below):
\begin{cor}[Weak coercivity]
 There exists $\eps$, $0<\eps<\frac{1}{2}$, such that for all $ t < 2\eps$, $\Delta^{\F} + t\id$\ admits a Green function $G_t$\ on $\M$.
\end{cor}

\begin{rem}
 Recall that, as a weighted Laplace operator, there always exists a heat kernel for $\Delta^{\F}$ (see for instance \cite{Grigoryan:heat_kernels_on_weighted_manifolds}). Therefore we can apply Theorem \ref{thm:green_function} to decide whether a Green function exists.
\end{rem}

\begin{proof}
 Take $\eps$ smaller than $c/2$. Then, for any $t <2 \eps$, $L:= \Delta^{\F} + t\id$ is still coercive and therefore we have a coercive bilinear form $q_L$ associated with $L$ and continuous on $H^1_0(\M)$. Indeed, just set $q_L(u,v):= - \int_{\M} u\left( \Delta^{\F} v + tv\right) \; \Omega$. We will construct a $L$-superharmonic positive function $s$ and use Theorem \ref{thm:green_function} to conclude.\\
 Take a positive test function $f \in C^{\infty}_0(\M)$. There exists (by Lax--Milgram Theorem) an $s \in H_0^1(\M)$ such that, for any $\varphi\in C_0^{\infty}(\M)$,
\begin{equation*}
q_L(s, \varphi) = \int_{\M} f \varphi \, \Omega \, .
\end{equation*}
Now $s$ is positive, because if we let $s^-:= \max \{ 0, - s \}$, then $q_L (s^-, s^-) = - q_L (s, s^-) \leq 0$ (because of the above equation, using a suitable $C^{\infty}$ approximation of $s^-$). As $s$ is a weak solution of $Ls = - f$, $s$ is superharmonic and we can choose $f$ so that $s$ is strictly positive and Theorem \ref{thm:green_function} proves the claim.
\end{proof}

\begin{lem}
There exist strictly positive constants $r_0, \, \tau_0, \, c_1$, and $c_2$ such that 
 \begin{enumerate}[(H1)]
\item For all $ m \in \M$, there exists a function $\theta: B(m,r_0) \rightarrow \R^n $, such that, for all $ x,y \in B(m,r_0) $ we have
\begin{equation*}
 c_2^{-1} d(x,y) \leq || \theta(x) -\theta(y) || \leq c_2 d(x,y) \,;
\end{equation*}
 \item For all $x_0 \in \M$ and $\forall 0\leq t \leq \tau_0 $\ the Green function $g_t$\ relative to $B(x_0,r_0)$\ of the operator $\Delta^{\F} + t \id$ satisfies, for all $ x,y \in B(x_0,r_0/2)$
\begin{equation*}
 \left\{
\begin{aligned}
  c_1 &\leq g_t(x,y) \\
  g_t(x,y) &\leq c_2, \quad \text{if }  d(x,y) \geq \frac{r_0}{4} \,.
 \end{aligned}
\right.
\end{equation*}
\end{enumerate}
\end{lem}

\begin{rem}
 These two conditions follow from the fact that $\M$ and $\Delta^{\F}$ are \emph{well-adapted} (in the terminology of Ancona \cite{Ancona:theorie_potentiel}), i.e., that $\M$ verifies (H1) and that the push-forwards of $\Delta^{\F}$ by $\theta$ have coefficients with H\"older norms bounded by uniform constants. Condition (H2) just means that the Green function of $\Delta^{\F}$ should behave like a Green function of a uniformly elliptic operator on $\R^n$.
\end{rem}

\begin{proof}

For (H1), first remark that, for any $m \in \M$, the exponential map at $m$ is a diffeomorphism from $T_m\M$ to $\M$, and we can take $\theta = \exp_m^{-1}$. For any real number $r_0$, there exists a constant $c= c(r_0) $ depending only on $r_0$ such that (H1) is satisfied. The constant $c$ exists by compactness of $B(m,r_0)$ and does not depend on $m$ by compactness of $M$.\\
The constant $c_2$ will be determined by Condition (H2), as long as we take $c_2 \geq c(r_0)$.\\
Let $\tau_0 < \lambda_1$ and fix $0 \leq t_0 \leq \tau_0$, let $L:= \theta_{\ast}\left(\Delta^{\F} + t_0\id\right)$. The operator $L$ is a uniformly elliptic operator on $U= \theta(B(x_0,r_0))$ with smooth bounded coefficients (by compactness of $M$). Therefore (see, for instance \cite[Th\'eor\`eme 9.6]{Stampacchia:probleme_dirichlet} or Chapter 1 of \cite{Friedman:PDE_of_parabolic_type}, Equations (6.12) and (6.13) in particular), the operator $L$ admits a Green function $g^L$ and there exist two constants $c_1$ and $c_2$ (depending on the H\"older norm of the coefficients of $L$) such that, for $u,v$ a bounded distance apart and a bounded distance from $\partial U$, we have $g^L(u,v) \geq c_1$ and $g^L(u,v) \leq c_2$. As $g^t$ is the pullback of $g^L$ by $\theta$ and again using the compactness of $M$, we have a uniform control on all those bounds and hence have proven our lemma. 
\end{proof}

\section{Ergodic property of harmonic measures}

In this section, we will adapt ergodic results on the harmonic measures to our case. The result and proof are based on Ledrappier's work on harmonic measures for negatively curved Riemannian manifolds in \cite{Ledrappier:Ergodic_properties}.

A \emph{measure class} on a space $V$ is a set $\{ \mu_x \}_{x\in V}$ indexed by $V$ such that, for any $x\in V$, $\mu_x$ is a measure on $V$ and, for $y\in V$, $\mu_x$ and $\mu_y$ are equivalent. All the measures we consider are Radon measures. If a group $\Gamma$ acts on $V$, we say that a measure class is \emph{invariant} by $\Gamma$ if, for $x\in V$ and $\gamma \in \Gamma$, 
\begin{equation*}
 \mu_{\gamma \cdot x} = \gamma_{\ast} \mu_x\, ,
\end{equation*}
where $\gamma_{\ast}\mu_x$ is defined by, for $U \subset V$ measurable, $\gamma_{\ast}\mu_x(U) := \mu_{x} \left(\gamma^{-1} \cdot U \right)$.
Remark that, if $\{ \mu_x\}$ is an invariant measure class, then the measures $\mu_x$ are \emph{quasi-invariant}, that is, for any $\gamma \in \Gamma$, $\mu_{x}$ and $\gamma_{\ast} \mu_x$ are equivalent.

For an invariant measure class (or a quasi-invariant measure), we can use the traditional definition of ergodicity: an invariant measure class is \emph{ergodic} if, for any measurable set $U$ invariant under $\Gamma$, then $U$ is either of measure (for any $\mu_x$) full or null.

In this section, $(M,F)$ is still a closed reversible Finsler manifold of negative curvature, $(\M,\F)$ its universal cover with the lifted metric. For any $x \in \M$, we denote by $\hmu_x$ the harmonic measure on $\M(\infty)$ associated to $\Delta^{\F}$.

\begin{lem}
 The harmonic measures $\{\hmu_x \}$ form a measure class invariant by the action of $\pi_1(M)$ on $\M(\infty)$. Moreover, for $x,y\in \M$, the Radon-Nykodim derivative of $\hmu_x$ and $\hmu_y$ is
\begin{equation*}
 \frac{d\hmu_x}{d\hmu_y}(\xi) = K(y,x,\xi)\, ,
\end{equation*}
where $x\mapsto K(y,x,\xi)$ is the Poisson Kernel normalized at $y$.
\end{lem}

\begin{proof}
We can obtain the harmonic measure $\mu_x$ in the following way: let $O$ be a point in $\M$ and $K(O,x,\xi)$ the Poisson kernel normalized at $O$. If $\hmu_O$ is the measure on $\M(\infty)$ such that 
\begin{equation*}
 \int_{\xi \in \M(\infty)} K(O,x,\xi) d\hmu_O(\xi) =1\, ,
\end{equation*}
then $\hmu_x$ is such that, for any Borel set $U \subset \M(\infty)$, $ \hmu_x(U) = \int_U K(O,x,\xi) d\hmu(\xi)$. From this we see that all the measures are equivalent and that their Radon-Nykodim derivative is given by 
\begin{equation*}
 \frac{d\hmu_x}{d\hmu_y}(\xi) = \frac{K(O,x,\xi)}{K(O,y,\xi)} \, .
\end{equation*}
Now recalling the characterization of $K(O,\cdot,\xi)$ given by the proof of Theorem \ref{thm:Ancona_martin_boundary} (see Remark \ref{rem:characterization_of_Poisson_kernel}), we see that $\frac{K(O,x,\xi)}{K(O,y,\xi)} = K(y,x,\xi)$. Indeed, the map ${ x \mapsto K(O,x,\xi)/K(O,y,\xi)}$ is harmonic, normalized at $y$ and such that it tends to zero when $x$ tends to $\xi' \neq \xi$. So we have 
\begin{equation*}
 \frac{d\hmu_x}{d\hmu_y}(\xi) = K(y,x,\xi) \, .
\end{equation*}

Using the fact that the $\mu_x$ solves the Dirichlet problem at infinity, we get that, for any $x\in \M$ and $f \in C^0(\M(\infty))$, 
\begin{equation*}
 \int_{\xi \in \M(\infty)} f(\xi) d\mu_{\gamma\cdot x} = \int_{\xi \in \M(\infty)} f\circ \gamma (\xi) d\mu_{ x} = \int_{\eta \in \M(\infty)} f(\eta) d\left(\gamma_{\ast} \mu_{x} \right) (\eta)\, ,
\end{equation*}
so we do have $\mu_{\gamma\cdot x} = \gamma_{\ast} \mu_{x}$.
\end{proof}

Let $\tau : H\M \rightarrow \M(\infty)$ be the application sending an element $(x,v) \in H\M$ to the point at infinity obtained as the limit of the geodesic ray leaving $x$ in the direction $v$. For any fixed $x\in \M$, $\tau_x$ is a (H\"older)-homeomorphism.\\
As the harmonic measures are $\pi_1(M)$-invariant, we can define the \emph{spherical harmonic measures} $\lbrace \bar{\mu}_y, y\in M \rbrace$ as measures on $H_yM$ by setting
\begin{equation*}
 \bar{\mu}_y := \mu_{\widetilde{y}} \circ \tau_{\widetilde{y}}\, ,
\end{equation*}
where $\widetilde{y}\in \M$ is any lift of $y\in M$.

\begin{thm} \label{thm:harmonic_measures_are_ergodic}
Let $(M,F)$ and $\hmu_x$ be as before. We have the following properties:
\begin{itemize}
 \item[(i)] The harmonic measure class $\lbrace \hmu_x \rbrace$ is ergodic for the action of $\pi_1(M)$ on $\M(\infty)$;\\
 \item[(ii)] For any $x\in \M$, the product measure $\hmu_x \otimes \hmu_x$ is ergodic for the action of $\pi_1(M)$ on $\partial^2\M := \M(\infty) \times \M(\infty) \smallsetminus \text{diagonal}$;\\
 \item[(iii)] There exists a unique $\flot$-invariant measure $\mu$ on $HM$ such that the family of spherical harmonics $\bar{\mu}_x$ is a family of transverse measures. Moreover, $(HM, \flot, \mu)$ is ergodic.
\end{itemize}
\end{thm}

Note that, by the following result of Kaimanovich, which still holds in this context, proving {\it (iii)} gives the theorem.
\begin{thm}[Kaimanovich \cite{Kaimanovich}]
There exists a convex isomorphism between the cone of Radon measures on $\partial^2 \M$ and the cone of Radon measures on $H\M$ invariant by $\hflot$.\\
Similarly, there exists a convex isomorphism between the cone of Radon measures on $\partial^2 \M$ invariant by $\pi_1(M)$ and the cone of Radon measures on $HM$ invariant by $\flot$.\\
\end{thm}
So if we construct $\mu$, the Kaimanovich correspondence shows that there exists a weight $f \colon \partial^2 \M \rightarrow \R$ such that the measure $f \hmu_x \otimes \hmu_x$ is invariant by $\pi_1(M)$ (see \cite{Kaimanovich}). And proving that $\mu$ is ergodic for the flow proves that $f \hmu_x \otimes \hmu_x$ is ergodic.

To prove {\it (iii)} of Theorem \ref{thm:harmonic_measures_are_ergodic}, it suffice to copy verbatim the proof of Proposition 3 in \cite{Ledrappier:Ergodic_properties} (using that our Poisson kernel is H\"older continuous by Theorem \ref{thm:Martin_is_holder}). We get:
\begin{prop}
 There exists a H\"older continuous function $F_0$ on $M$ such that the spherical harmonic measures $\bar{\mu}_x$ can be chosen as a family of transverse measures for the equilibrium state $\mu$ of $F_0$.
\end{prop}
But we know (see \cite{Bowen:Equilibrium_states}) that an equilibrium state is ergodic, and so we have {\it (iii)}.

\part{Periodic orbits of Skewed $\R$-covered Anosov flows}

\chapter{Skewed $\R$-covered Anosov flows}

\section{Definitions}

\subsection{Basics on Anosov flows} \label{subsec:basics_AF}

In all this part, we will be interested in Anosov flows on closed $3$-manifolds, but we can state the definition in any dimension:

\begin{defin} \label{def:Anosov}
 Let $M$ be a compact manifold and $\flot \colon M \rightarrow M$ a $C^1$ flow on $M$. The flow $\flot$ is called \emph{Anosov} if there exists a splitting of the tangent bundle ${TM =  \R\cdot X \oplus E^{ss} \oplus E^{uu}}$ preserved by $D\flot$ and two constants $a,b >0$ such that:
\begin{enumerate}
 \item $X$ is the generating vector field of $\flot$;
 \item For any $v\in E^{ss}$ and $t>0$,
    \begin{equation*}
     \lVert D\flot(v)\rVert \leq be^{-at}\lVert v \rVert \, ;
    \end{equation*}
  \item For any $v\in E^{uu}$ and $t>0$,
    \begin{equation*}
     \lVert D\phi^{-t}(v)\rVert \leq be^{-at}\lVert v \rVert\, .
    \end{equation*}
\end{enumerate}
In the above, $\lVert \cdot \rVert$ is any Riemannian or Finsler metric on $M$.
\end{defin}

The subbundle $E^{ss}$ (resp. $E^{uu}$) is called the \emph{strong stable distribution} (resp. \emph{strong unstable distribution}). It is a classical result (\cite{Anosov}) that $E^{ss}$, $E^{uu}$, $\R\cdot X \oplus E^{ss}$ and $\R\cdot X \oplus E^{uu}$ are integrable. We denote by $\mathcal{F}^{ss}$, $\mathcal{F}^{uu}$, $\fs$ and $\fu$ the respective foliations and we call them the strong stable, strong unstable, stable and unstable foliations.

In all the following, if $x \in M$, then $\fs(x)$ (resp. $\fu(x)$) is the leaf of the foliation $\fs$ (resp. $\fu$) containing $x$.\\

Another kind of flow that will appear is a \emph{pseudo-Anosov flow}. This type of flows is the generalization of suspensions of pseudo-Anosov diffeomorphisms. They should be thought of as Anosov flows everywhere apart from a finite number of periodic orbits where the stable and unstable foliations are singular. For foundational works on pseudo-Anosov flows, see \cite{Mosher,Mosher:DS,Mosher:DS_II}. 
\begin{defin} \label{def:pseudo-Anosov}
 A flow $\psi^t$ on a closed $3$-manifold $M$ is called \emph{pseudo-Anosov} if it satisfies the following conditions:
\begin{itemize}
 \item For each $x\in M$, the flow line $t\mapsto \psi^t(x)$ is $C^1$, not a single point, and the tangent vector field is $C^0$;
 \item There is a finite number of periodic orbits, called \emph{singular orbits}, such that the flow is smooth off of the singular orbits;
 \item The flow lines of $\psi^t$ are contained in two possibly singular $2$-dimensional foliations $\Lambda^s$ and $\Lambda^u$ satisfying:
outside of the singular orbits, the foliations are not singular, are transverse to each other and their leaves intersect exactly along the flow lines of $\psi^t$. A leaf containing a singularity is homeomorphic to $P\times [0,1] / f$ where $P$ is a $p$-prong in the plane and $f$ is a homeomorphism from $P\times\{1\}$ to $P\times\{0\}$. \emph{We will always assume that $p\geq 3$};
 \item In a stable leaf, all orbits are forward asymptotic; in an unstable leaf, they are all backward asymptotic.
\end{itemize}
\end{defin}

In the definition of Anosov flow, we asked for $\flot$ to be at least $C^1$ but we will only care about smooth (i.e., $C^{\infty}$) flows. Note that the foliations however might not be very regular.

We further assume that $\flot$ is \emph{transversally oriented}, i.e., there exists an orientation on $M$ given by an orientation on each leaf of $\fs$ together with an orientation on each leaf of $\mathcal{F}^{uu}$. Note that this hypothesis will be essential for the description we give of skewed Anosov flows, for instance, in order to have an orientation on the leaf spaces (to be defined below). However, it can be achieved by taking the lift of the flow to a two-fold cover (four-fold if the manifold is not orientable).

Both Sergio Fenley and Thierry Barbot --- at the same time and independently --- started studying Anosov flow via their transversal geometry, that is via the study of the space of orbits. We will follow their lead and use their works throughout this part. So some of the main objects of study here will be the orbit and the leaf spaces that we define as follow.

Let $\M$ be the universal cover of $M$ and $\pi \colon \M \rightarrow M$ the canonical projection. The flow $\flot$ and all the foliations lift to $\M$ and we denote them respectively by $\hflot$, $\widetilde{\mathcal{F}}^{ss} $, $\hfs$, $\widetilde{\mathcal{F}}^{uu}$ and $\hfu$. Now we can define
\begin{itemize}
 \item The \emph{orbit space} of $\flot$ as $\M$ quotiented out by the relation ``being on the same orbit of $\hflot$''. We denote it by $\orb$.
 \item The \emph{stable} (resp. \emph{unstable}) \emph{leaf space} of $\flot$ as $\M$ quotiented out by the relation ``being on the same leaf of $\hfs$ (resp. $\hfu$)''. We denote them by $\leafs$ and $\leafu$ respectively.
\end{itemize}
Note that the foliations $\hfs$ and $\hfu$ obviously project to two transverse foliations of $\orb$. We will keep the same notations for the projected foliations, hoping that it will not lead to any confusion.\\ 

For (pseudo)-Anosov flows in $3$-manifolds, the orbit space is always homeomorphic to $\R^2$ (see \cite{Bar:CFA} and \cite{Fen:AFM} for the Anosov case and \cite{FenleyMosher} for the pseudo-Anosov case). The leaf spaces $\leafs$ and $\leafu$ however are in general \emph{non-Hausdorff} $1$-manifolds, but still connected and simply-connected. In this work, we are specially interested in one particular case:
\begin{defin}
 An Anosov flow is called $\R$-covered if $\leafs$ and $\leafu$ are homeomorphic to $\R$.
\end{defin}

Remark that to prove that a flow is $\R$-covered, we just need to show that one of the leaf spaces is homeomorphic to $\R$:
\begin{thm}[Barbot \cite{Bar:these}, Fenley \cite{Fen:AFM}]
 If $\leafs$ is Hausdorff, then $\leafu$ is Hausdorff and vice versa.
\end{thm}

Let us also recall the following result of A.\ Verjovsky which is fundamental for the following (and already used in the proof of the above mentioned results):
\begin{prop}[Verjovsky \cite{Ver:codim1}]
 Let $\flot$ be an Anosov flow on a $3$-manifold $M$. Then:
\begin{enumerate}
 \item Periodic orbits of $\flot$ are not null-homotopic;
 \item Leaves of $\hfs$ (resp. $\hfu$) are homeomorphic to $\R^{2}$;
 \item $\M$ is homeomorphic to $\R^3$.
\end{enumerate}
\end{prop}
Verjovsky's result is in fact more general; the above stays true (with obvious modifications) for \emph{codimension one} Anosov flows, i.e., Anosov flows such that (say) the strong unstable foliation $\mathcal{F}^{uu}$ is one-dimensional. In the same article, Verjovsky also proved the following result:
\begin{prop}
 If $\flot$ is a codimension $1$ Anosov flow, then any leaf of $\widetilde{\mathcal{F}}^{uu}$ intersects at most once a leaf of $\hfs$.
\end{prop}
Note that this result does \emph{not} tell you that there always is an intersection. Indeed, we say that a $\R$-covered flow is \emph{skewed} if, for every leaf $L^{u} \in \hfu$, there exists a leaf $L^{s}$ of $\hfs$ such that $L^u \cap L^s = \emptyset$ and vice-versa. We have the following:
\begin{thm}[Barbot \cite{Bar:CFA}]
 If $\flot$ is a $\R$-covered flow on a $3$-manifold $M$, then either $\flot$ is skewed, or it is a topologically conjugated to a suspension of an Anosov diffeomorphism.
\end{thm}

An $\R$-covered Anosov flow is always transitive (see \cite[Theorem 2.5]{Bar:CFA}), i.e., admits a dense orbit. If it is skewed, it is even more:
\begin{prop}[Barbot]
 A skewed $\R$-covered Anosov flow is topologically mixing.
\end{prop}
The proof is given in Remark 2.2 of \cite{Bar:PAG}. Topologically mixing means that given two open sets, there exists a time after which the image by the flow of one set always intersects the other (see \cite{KatokHassel}).

In all the rest, we will be considering skewed $\R$-covered Anosov flows. Note that a geodesic flow of a negatively curved surface is a skewed $\R$-covered Anosov flow. More generally, any contact Anosov flow is skewed $\R$-covered (\cite{Bar:PAG})

\subsubsection{Orbit space and fundamental group}

It is easy to see that the action of the fundamental group of $M$ on $\M$ projects to the orbit and leaf spaces. We can even say a bit more about this action:
\begin{prop}
 Let $\flot$ be an Anosov flow on $M$.
\begin{enumerate}
 \item The stabilizer by $\pi_1(M)$ of a point in $\orb$, $\leafs$ or $\leafu$ is either trivial or cyclic.
 \item If $\gamma\in \pi_1(M) $ fixes a point $O \in \orb$, then $O$ is a hyperbolic fixed point of $\gamma$.
 \item If $\gamma\in \pi_1(M) $ fixes a point $l \in \leafs$ (or $\leafu$), then $l$ is either an attractor or a repeller for the action of $\gamma$.
\end{enumerate}
\end{prop}
The proof can be found in \cite{Bar:these} and holds once again for codimension $1$ Anosov flows.\\

One fundamental remark of Fenley in \cite{Fen:AFM} is the following:
\begin{prop}[Fenley] \label{prop:eta_s_eta_u}
 Let $\flot$ be a skewed, $\R$-covered Anosov flow in a $3$-manifold $M$. Then, there exist two functions $\eta^s \colon \leafs \rightarrow \leafu$ and $\eta^u \colon \leafu \rightarrow \leafs$ that are monotonous, $\pi_1(M)$-equivariant and $C^{\alpha}$. Furthermore, $\eta^u \circ \eta^s$ and $\eta^s \circ \eta^u$ are strictly increasing and we can define $\eta \colon \orb  \rightarrow \orb$ by
\begin{equation*}
\eta(o):= \eta^u \left( \hfu(o)\right) \cap \eta^s\left(\hfs(o) \right) .
\end{equation*}
\end{prop}

\begin{proof}
 Let $L^s \in \leafs$. Define $I : = \lbrace L^u \in \leafu \mid L^u \cap L^s \neq \emptyset\rbrace $. The set $I$ is an open, connected subset in $\leafs \simeq \R$. Hence $\partial I$ consists of $2$ elements and, as $\flot$ is transversally oriented, $\leafs$ as a natural orientation. So we can set $\eta^s(L^s)$ to be the largest element of $\partial I$. The function $\eta^u$ can be defined in exactly the same fashion.\\
 Monotonicity is trivial to check using the definition, as is the equivariance under the fundamental group. H\"older continuity is done in \cite{Bar:PAG}.
\end{proof}

Using this result, we can get a better picture of the space of orbits: Let 
$$\Gamma(\eta^s) := \lbrace \left( \lambda^s , \eta^s(\lambda^s) \right) , \; \lambda^s \in \leafs \rbrace \subset \leafs \times \leafu$$
 and 
$${\Gamma(\eta^u) := \lbrace \left( \eta^u(\lambda^u) , \lambda^u \right) , \; \lambda^u \in \leafu \rbrace} \subset \leafs \times \leafu$$
 be the graphs of $\eta^s$ and $\eta^u$ respectively. Then $\orb$ is the subset of $\leafs \times \leafu$ in between $\Gamma(\eta^s)$ and $\Gamma(\eta^u)$, and the foliations $\hfs$ and $\hfu$ in $\orb$ are just given by vertical and horizontal lines (see Figure \ref{fig:orbit_space}).
\begin{figure}[h]
 \begin{center}
\scalebox{0.8} %
{

\begin{pspicture}(0,-4.3903127)(14.97,4.3903127)
\psline[linewidth=0.04cm,arrowsize=0.05291667cm 2.0,arrowlength=1.4,arrowinset=0.4]{->}(2.7,-3.976875)(11.72,-3.956875)
\psline[linewidth=0.04cm,arrowsize=0.05291667cm 2.0,arrowlength=1.4,arrowinset=0.4]{->}(1.12,-2.896875)(1.14,4.223125)
\psbezier[linewidth=0.04,linestyle=dashed,dash=0.16cm 0.16cm](4.66,-2.916875)(5.0,-2.256875)(5.978522,-2.5493124)(6.76,-1.916875)(7.541478,-1.2844378)(7.22,-0.956875)(7.96,-0.456875)(8.7,0.043125)(8.143241,0.3636842)(9.12,0.743125)(10.096759,1.1225657)(10.036516,0.8545029)(10.94,1.243125)(11.843484,1.631747)(11.56,2.043125)(11.94,2.503125)(12.32,2.963125)(12.52,2.663125)(12.96,3.543125)
\psbezier[linewidth=0.04,linestyle=dashed,dash=0.16cm 0.16cm](1.62,-2.276875)(1.96,-1.616875)(2.7385218,-2.0293124)(3.5,-1.436875)(4.261478,-0.8444377)(4.18,-0.336875)(4.9,0.143125)(5.62,0.623125)(5.323241,0.7636842)(6.08,1.383125)(6.836759,2.0025659)(7.116516,1.914503)(8.02,2.283125)(8.923484,2.651747)(8.68,3.043125)(9.02,3.483125)(9.36,3.923125)(9.62,3.563125)(10.5,4.083125)
\psline[linewidth=0.04cm](1.78,-2.196875)(6.26,-2.176875)
\psline[linewidth=0.04cm](3.36,-1.616875)(7.08,-1.596875)
\psline[linewidth=0.04cm](4.0,-1.016875)(7.4,-1.016875)
\psline[linewidth=0.04cm](4.34,-0.576875)(7.7,-0.576875)
\psline[linewidth=0.04cm](4.56,-0.156875)(8.26,-0.156875)
\psline[linewidth=0.04cm](4.98,0.183125)(8.48,0.183125)
\psline[linewidth=0.04cm](5.48,0.583125)(8.76,0.623125)
\psline[linewidth=0.04cm](5.74,0.983125)(9.86,1.003125)
\psline[linewidth=0.04cm](6.14,1.383125)(11.28,1.383125)
\psline[linewidth=0.04cm](6.82,1.803125)(11.58,1.823125)
\psline[linewidth=0.04cm](7.94,2.223125)(11.74,2.223125)
\psline[linewidth=0.04cm](8.62,2.563125)(11.98,2.563125)
\psline[linewidth=0.04cm](8.96,3.163125)(12.72,3.163125)
\psline[linewidth=0.04cm](9.28,3.563125)(12.9,3.563125)
\psline[linewidth=0.04cm](3.38,-1.656875)(3.38,-2.836875)
\psline[linewidth=0.04cm](3.94,-1.116875)(3.94,-2.816875)
\psline[linewidth=0.04cm](4.52,-0.176875)(4.52,-2.756875)
\psline[linewidth=0.04cm](5.18,0.243125)(5.18,-2.476875)
\psline[linewidth=0.04cm](5.74,0.943125)(5.74,-2.376875)
\psline[linewidth=0.04cm](6.3,1.503125)(6.3,-2.056875)
\psline[linewidth=0.04cm](6.8,1.763125)(6.8,-1.816875)
\psline[linewidth=0.04cm](7.3,1.963125)(7.3,-1.116875)
\psline[linewidth=0.04cm](7.92,2.183125)(7.92,-0.396875)
\psline[linewidth=0.04cm](8.52,2.523125)(8.52,0.343125)
\psline[linewidth=0.04cm](9.18,3.503125)(9.18,0.823125)
\psline[linewidth=0.04cm](9.76,3.723125)(9.76,1.003125)
\psline[linewidth=0.04cm](10.38,3.903125)(10.38,1.203125)
\psline[linewidth=0.04cm](10.96,3.963125)(10.96,1.463125)
\psline[linewidth=0.04cm](11.74,4.003125)(11.74,2.363125)
\usefont{T1}{ptm}{m}{n}
\rput(12.025,-4.206875){$\leafs$}
\usefont{T1}{ptm}{m}{n}
\rput(0.8,4.2){$\leafu$}
\usefont{T1}{ptm}{m}{n}
\rput(13.2,2.873125){$\Gamma(\eta^s)$}
\usefont{T1}{ptm}{m}{n}
\rput(9.927813,4.193125){$\Gamma(\eta^u)$}
\end{pspicture} 
} 
\end{center}
\caption{The space $\orb$ seen in $\leafs \times \leafu$} \label{fig:orbit_space}
\end{figure}

\subsubsection{Free homotopy class of periodic orbits} 

In \cite{Fen:AFM}, Fenley constructed examples of skewed $\R$-covered Anosov flows on atoroidal, not Seifert-fibered spaces. So in particular, these flows are not geodesic flows. It turns out that, if you consider the free homotopy class of periodic orbits, Fenley's examples behave in a very different way:
\begin{thm}[Fenley \cite{Fen:AFM}] \label{thm:infinite_homotopy_class}
 If $M$ is atoroidal and not Seifert-fibered, then the free homotopy class of a periodic orbit of $\flot$ contains infinitely many distinct periodic orbits.
\end{thm}
For the geodesic flow of a negatively curved (Riemannian or Finsler) manifold, it is a classical result (\cite{Klingenberg}) that there is at most one periodic orbit in a free homotopy class (and exactly one geodesic for each element in the fundamental group of the manifold).\\

We give a sketch of proof of this result as it is useful for the understanding of such flows.
\begin{proof}[Sketch of proof]
 Let $\alpha$ be a periodic orbit of $M$, and $\widetilde{\alpha}$ a lift to the universal cover. There exists an element $\gamma \in \pi_1(M)$ such that $\gamma$ leaves $\widetilde{\alpha}$ invariant. For any $i \in \Z$, $\eta^i(\widetilde{\alpha})$ is also left invariant by $\gamma$ (by the previous proposition) and hence its projection on $M$ is a periodic orbit. Just by looking at the action of $\gamma$ on $\orb$ we can deduce that $\eta^i(\widetilde{\alpha})$ and $\eta^{i+1}(\widetilde{\alpha})$ have reverse directions (see Figure \ref{fig:gamma_stabilizing_chain}).\\
 Then, there is some work to show that, if $\gamma$ was the generator of the stabilizer of $\wt{\alpha}$ for the action of $\pi_1(M)$, then it is the generator for any $\eta^i(\widetilde{\alpha})$. This proves that $\eta^{2i}(\widetilde{\alpha})$ are all freely homotopic.\\
Finally, using the topological assumptions, Fenley shows that the projections of $\eta^{2i}(\widetilde{\alpha})$ to $M$ are all distinct (otherwise, there would be a $\Z^2$ in $\pi_1(M)$).
\end{proof}

In the following, we will often abuse terminology and say that an orbit $\wt{\alpha}$ of $\hflot$ is periodic if its projection to $M$ is periodic, or equivalently, if $\wt{\alpha}$ is stabilized by an element of the fundamental group.

\subsubsection{Lozenges}

In \cite{Fen:AFM}, Fenley introduced the notion of lozenges, which is a kind of basic block in the orbit space and is fundamental to the study of (not only $\R$-covered) (pseudo)-Anosov flow.

\begin{figure}[h]
\begin{center}
\scalebox{1}{

\begin{pspicture}(0,-1.97)(3.92,1.97)


\psbezier[linewidth=0.04](0.02,0.17)(0.88,0.03)(0.74114233,-0.47874346)(1.48,-1.11)(2.2188578,-1.7412565)(2.38,-1.73)(3.26,-1.95)
\psbezier[linewidth=0.04](0.0,0.27)(0.96,0.3105634)(0.75286174,0.37057108)(1.78,0.77)(2.8071382,1.169429)(2.66,1.47)(3.36,1.45)
\psbezier[linewidth=0.04](0.6,1.95)(1.2539726,1.8871263)(1.1265805,1.3646309)(2.0345206,0.7973154)(2.9424605,0.23)(3.2249315,0.2543258)(3.9,0.31)
\psbezier[linewidth=0.04](0.52,-1.33)(1.48,-1.33)(1.3597014,-0.9703507)(2.3,-0.63)(3.2402985,-0.28964934)(3.14,0.05)(3.84,0.23)
\psdots[dotsize=0.16](1.98,0.85)
\psdots[dotsize=0.16](1.46,-1.11)
\usefont{T1}{ptm}{m}{n}
\rput(1.6145313,-0.06){$L$}
\usefont{T1}{ptm}{m}{n}
\rput(1.4,-1.38){$\alpha$}
\rput(2,1.2){$\beta$}

\rput(0.6,-0.6){$A$}
\rput(3,-0.6){$B$}
\rput(0.6,0.6){$D$}
\rput(3,0.6){$C$}
\end{pspicture} 
 }
\end{center}
\caption{A lozenge with corners $\alpha$, $\beta$ and sides $A,B,C,D$} \label{fig:a_lozenge}
\end{figure}
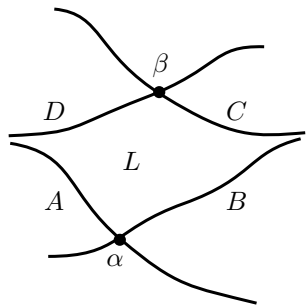

\begin{defin}
 A lozenge $L$ in $\orb$ is a subset of $\orb$ such that (see Figure \ref{fig:a_lozenge}):\\
There exist two points $\alpha,\beta \in L$ and four half leaves $A \subset \hfs(\alpha)$, $B \subset \hfu(\alpha)$, $C \subset \hfs(\beta)$ and $D \subset \hfu(\beta)$ verifying:
\begin{itemize}
 \item For any $\lambda^s \in \leafs$, $\lambda^s \cap B \neq \emptyset$ if and only if $\lambda^s \cap D\neq \emptyset$,
 \item For any $\lambda^u \in \leafu$, $\lambda^u \cap A \neq \emptyset$ if and only if $\lambda^u \cap C \neq \emptyset$,
 \item The half-leaf $A$ does not intersect $D$ and $B$ does not intersect $C$.
\end{itemize}
Then,
\begin{equation*}
 L := \lbrace \alpha,\beta \rbrace \cup \lbrace p \in \orb \mid \hfs(p) \cap B \neq \emptyset, \; \hfu(p) \cap A \neq \emptyset \rbrace.
\end{equation*}
The points $\alpha$ and $\beta$ are called the \emph{corners} of $L$ and $A,B,C$ and $D$ are called the \emph{sides}.
\end{defin}
Note that in our definition, we do not count the sides as part of a lozenge, but we do include the two corners.

\begin{defin}
 A chain of lozenges is a union (finite or infinite) of lozenges $L_i$ such that two consecutive lozenges $L_i$ and $L_{i+1}$ always share a corner.
\end{defin}

There are basically two configurations for consecutive lozenges in a chain: either they share a side, or they don't. The first case is characterized by the fact that there exists a leaf intersecting the interior of both lozenges, while it cannot happen in the second case (see Figure \ref{fig:chain_of_lozenges})
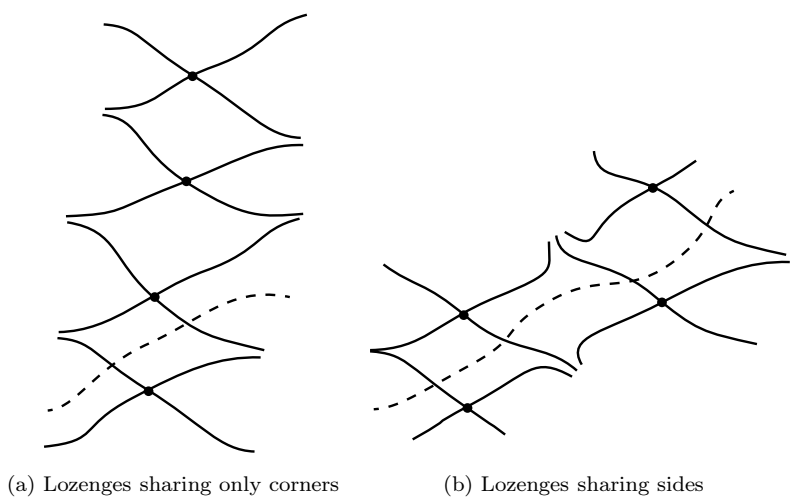
\begin{figure}[h]
 \centering
  \subfloat[Lozenges sharing only corners]{\label{fig:string_lozenges} \scalebox{0.8} { 

\begin{pspicture}(-0.4,-3.8)(4.8,4)


\psdots[dotsize=0.16](2.44,2.6)
\psbezier[linewidth=0.04](0.38,0.18)(1.24,0.04)(1.1011424,-0.46874347)(1.84,-1.1)(2.5788577,-1.7312565)(2.74,-1.72)(3.62,-1.94)
\psbezier[linewidth=0.04](0.36,0.28)(1.32,0.32056338)(1.1128618,0.38057107)(2.14,0.78)(3.1671383,1.1794289)(3.56,1.48)(4.26,1.46)
\psbezier[linewidth=0.04](0.96,1.96)(1.6139725,1.8971263)(1.4865805,1.3746309)(2.3945205,0.80731547)(3.3024607,0.24)(3.5849316,0.26432583)(4.26,0.32)
\psbezier[linewidth=0.04](0.24,-1.64)(1.16,-1.62)(1.7197014,-0.9603507)(2.66,-0.62)(3.6002986,-0.27964935)(3.5,0.06)(4.2,0.24)
\psdots[dotsize=0.16](2.34,0.86)
\psdots[dotsize=0.16](1.82,-1.06)
\psbezier[linewidth=0.04](0.98,3.46)(1.5739726,3.4171262)(1.5465806,3.1746309)(2.4545205,2.6073155)(3.3624606,2.04)(3.64,1.6)(4.22,1.58)
\psbezier[linewidth=0.04](1.0,2.06)(1.96,2.06)(1.8397014,2.4196494)(2.78,2.76)(3.7202985,3.1003506)(3.62,3.44)(4.32,3.62)
\psbezier[linewidth=0.04](0.22,-1.74)(0.8139726,-1.7828737)(0.7865805,-2.0253692)(1.6945206,-2.5926845)(2.6024606,-3.16)(2.88,-3.6)(3.46,-3.62)
\psbezier[linewidth=0.04](0.0,-3.54)(0.94,-3.4394367)(0.54,-3.22)(1.44,-2.74)(2.34,-2.26)(2.86,-2.04)(3.56,-2.06)
\psdots[dotsize=0.16](1.72,-2.62)
\psbezier[linewidth=0.04,linestyle=dashed,dash=0.16cm 0.16cm](0.06,-2.94)(0.54,-2.9)(0.8125543,-2.2401693)(1.7,-1.84)(2.5874457,-1.4398307)(3.08,-0.86)(4.04,-1.06)

\end{pspicture} }}
  \subfloat[Lozenges sharing sides]{\label{fig:adjacent_lozenges} \scalebox{0.8} { 

\begin{pspicture}(6.6,-3.6)(13.8,2)


\psbezier[linewidth=0.04](7.0,-0.32)(7.54,-0.72)(7.6411424,-0.56874347)(8.38,-1.2)(9.118857,-1.8312565)(9.64,-1.6)(10.18,-2.08)
\psbezier[linewidth=0.04](6.78,-1.74)(7.7,-1.72)(7.8397017,-1.2803507)(8.86,-0.86)(9.880299,-0.43964934)(9.72,-0.42)(9.72,0.06)
\psdots[dotsize=0.16](8.32,-1.16)
\psbezier[linewidth=0.04](6.78,-1.76)(7.3739724,-1.8028737)(7.3665805,-1.9853691)(8.25452,-2.6126845)(9.142461,-3.24)(8.6,-2.84)(9.0,-3.14)
\psbezier[linewidth=0.04](7.48,-3.22)(8.32,-2.74)(7.56,-3.12)(8.48,-2.62)(9.4,-2.12)(9.46,-1.82)(10.1,-2.18)
\psdots[dotsize=0.16](8.38,-2.7)
\psbezier[linewidth=0.04](9.84,0.16)(9.92,-0.4)(10.641142,-0.16874346)(11.4,-0.8)(12.158857,-1.4312565)(12.24,-1.44)(13.14,-1.64)
\psbezier[linewidth=0.04](10.26,-1.98)(9.98,-1.5)(10.66,-1.44)(11.56,-0.96)(12.46,-0.48)(12.98,-0.26)(13.68,-0.28)
\psbezier[linewidth=0.04](10.46,1.56)(10.52,0.96)(11.121142,1.3112565)(11.88,0.68)(12.638858,0.048743468)(12.72,0.04)(13.62,-0.16)
\psbezier[linewidth=0.04](9.98,0.22)(10.62,-0.18)(10.26,0.28)(11.12,0.78)(11.98,1.28)(11.5,0.96)(12.14,1.4)
\psdots[dotsize=0.16](11.43,0.95)
\psdots[dotsize=0.16](11.58,-0.95)
\psbezier[linewidth=0.04,linestyle=dashed,dash=0.16cm 0.16cm](6.84,-2.72)(7.24,-2.7)(7.572743,-2.4578607)(8.44,-1.98)(9.307257,-1.5021392)(8.894286,-1.4042312)(9.74,-0.92)(10.585714,-0.4357688)(10.980949,-0.86126333)(11.76,-0.28)(12.539051,0.30126333)(12.3,0.7)(12.76,0.9)

\end{pspicture} }}
 \caption{The two types of consecutive lozenges in a chain} \label{fig:chain_of_lozenges}
\end{figure}

In the case at hand, lozenges and chain of lozenges are pretty nice:
\begin{prop}[Fenley \cite{Fen:AFM}]
 Let $\flot$ be a skewed $\R$-covered Anosov flow, and $C = \bigcup L_i$ a chain of lozenges. Let $p_{i-1}$ and $p_{i}$ be the two corners of $L_i$, then $p_{i} = \eta(p_{i-1})$. Furthermore, if $p_{i}$ is the shared corner with $L_{i+1}$, then the union of the sides through $p_i$ of $L_i$ and $L_{i+1}$ is $\hfs(p_i) \cup \hfu(p_i)$. In other words, consecutive lozenges never share a side. In particular, an (un)stable leaf cannot intersect the interior of more than one lozenge in $C$.
\end{prop}

\begin{cor}
 If $L$ is a lozenge such that one of its corners is fixed by an element $\gamma$ of $\pi_1(M)$, then $\gamma$ stabilizes the whole lozenge and fix the other corner.\\
Now, if we assume furthermore that $M$ is atoroidal and not a Seifert-fibered space, then we have:
\begin{itemize}
 \item If $C$ is a chain of lozenges with corners $p_i$ and if one corner is a periodic orbit, then every corner is a periodic orbit, the loops $\{\pi(p_{2i})\}$ (resp. $\{\pi(p_{2i+1})\}$) are in the free homotopy class of $\pi(p_0)$ (resp. $\pi(p_1)$) and $C$ is stabilized by the deck transformation fixing the $p_i$.
 \item Conversely, if $\alpha$ is a periodic orbit of $\flot$ and $\wt{\alpha}$ a lift to $\M$, then $\wt{\alpha}$ is a corner of an infinite maximal chain of lozenges $C$. The stabilizer of $C$ in $\pi_1(M)$ is generated by one element. We call the projection to $M$ of \emph{all} the corners of $C$ the \emph{double free homotopy class} of $\alpha$.
\end{itemize}
\end{cor}

\begin{rem}
 The difference between the free homotopy class of a periodic orbit and the double free homotopy class, is that in the latter, we forget about the orientation of the curves.
\end{rem}

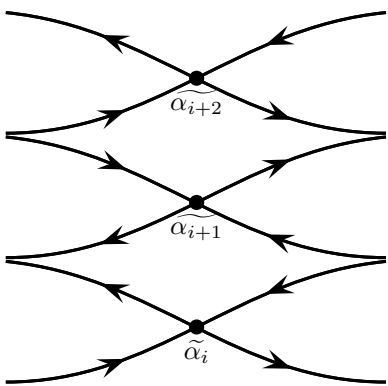
\begin{figure}[h]
 \begin{center}
 \scalebox{1}{

\begin{pspicture}(0,0)(6,6)


\rput(2.51,0.4){$\al i$}
\rput(2.51,2.05){$\al{i+1}$}
\rput(2.51,3.7){$\al{i+2}$}

\psset{arrowscale=2}
\psbezier[linewidth=0.04,ArrowInside=->,ArrowInsidePos=0.30](0,0)(2,0)(3,1.4)(5,1.6)
\psbezier[linewidth=0.04,ArrowInside=-<,ArrowInsidePos=0.75](0,0)(2,0)(3,1.4)(5,1.6)

\psbezier[linewidth=0.04,ArrowInside=->,ArrowInsidePos=0.75](5,0)(3,0)(2,1.4)(0,1.6)
\psbezier[linewidth=0.04,ArrowInside=-<,ArrowInsidePos=0.30](5,0)(3,0)(2,1.4)(0,1.6)

\psdots[dotsize=0.2](2.51,0.73)

\rput(0,1.65){
\psbezier[linewidth=0.04,ArrowInside=-<,ArrowInsidePos=0.75](5,0)(3,0)(2,1.4)(0,1.6)
\psbezier[linewidth=0.04,ArrowInside=->,ArrowInsidePos=0.30](5,0)(3,0)(2,1.4)(0,1.6)

\psdots[dotsize=0.2](2.51,0.73)

\psbezier[linewidth=0.04,ArrowInside=-<,ArrowInsidePos=0.30](0,0)(2,0)(3,1.4)(5,1.6)
\psbezier[linewidth=0.04,ArrowInside=->,ArrowInsidePos=0.75](0,0)(2,0)(3,1.4)(5,1.6)
}

\rput(0,3.3){
\psbezier[linewidth=0.04,ArrowInside=->,ArrowInsidePos=0.75](5,0)(3,0)(2,1.4)(0,1.6)
\psbezier[linewidth=0.04,ArrowInside=-<,ArrowInsidePos=0.30](5,0)(3,0)(2,1.4)(0,1.6)

\psdots[dotsize=0.2](2.51,0.73)

\psbezier[linewidth=0.04,ArrowInside=->,ArrowInsidePos=0.30](0,0)(2,0)(3,1.4)(5,1.6)
\psbezier[linewidth=0.04,ArrowInside=-<,ArrowInsidePos=0.75](0,0)(2,0)(3,1.4)(5,1.6)
}

\end{pspicture}
}
\end{center}
\caption{The action of an element $\gamma \in \pi_1(M)$ stabilizing a chain of lozenges} \label{fig:gamma_stabilizing_chain}
\end{figure}

\begin{proof}
 The first assertion is easy: if $\gamma$ stabilizes some corner, then it stabilizes the other corner because it is an image of the first one by a power of $\eta$ and $\eta$ commutes with deck transformations. Now, this implies that $\gamma$ stabilizes every side of $L$, hence stabilizes $L$.

 The second assertion is a trivial application of the first one and the fact that free homotopy classes are obtained by powers of $\eta$ (see \cite{Fen:AFM} or the sketch of proof of Theorem \ref{thm:infinite_homotopy_class}).

 Finally, the only hard part of the last assertion is that the stabilizer of $C$ is some cyclic group. It amounts to the same thing as we already admitted in the proof of Theorem \ref{thm:infinite_homotopy_class}, i.e., that the projections to $M$ of the corners are distinct. Indeed, the image by a deck transformation of a lozenge is a lozenge and in particular it sends corner to corner.
\end{proof}
Note that Fenley obviously first studied these lozenges and obtained the above results and \emph{then} deduced Theorem \ref{thm:infinite_homotopy_class}. I hope the reader will forgive the liberty I took with the order in which I present the results. My goal was not to give complete proofs but merely an idea of what these flows look like.

\begin{rem} 
Looking at the orientation of the sides of a lozenge, we can see that they come in two different types.\\
Recall that the transverse orientability of $\flot$ gives an orientation on each leaf of $\hfs$ and $\hfu$ seen in $\orb$. So any orbit defines two half stable leafs (positive and negative) and two half unstable leafs. Now, let $p$ be a corner of a lozenge $L$. The sides of $L$ going through $p$ --- call them $A$ for the stable and $B$ for the unstable --- could either be both positive, both negative, or of different signs. It is quite easy to see that the stable (resp. unstable) side of the other corner needs to have switched sign from $B$ (resp. from $A$). So each lozenge could be of two types, either $(+,+,-,-)$ or $(+,-,-,+)$, but evidently, all the lozenges of the same (transversally orientable) flow are of the same type (\cite{Fen:AFM}).
\end{rem}
In the sequel, we will consider flows such that lozenges are of the type $(+,+,-,-)$ (see Figure \ref{fig:lozenge++--}).

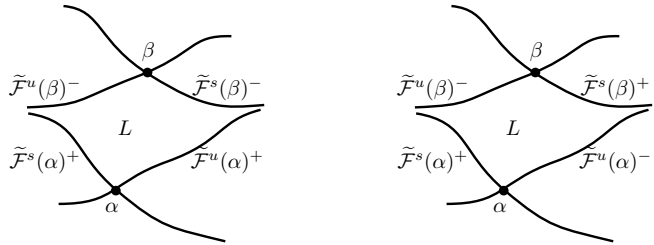
\begin{figure}[h]
 \centering
  \subfloat[A lozenge of type $(+,+,-,-)$ ]{\label{fig:lozenge++--} \scalebox{0.8} { \begin{pspicture}(-1,-2.2)(5,2)


\psbezier[linewidth=0.04](0.02,0.17)(0.88,0.03)(0.74114233,-0.47874346)(1.48,-1.11)(2.2188578,-1.7412565)(2.38,-1.73)(3.26,-1.95)
\psbezier[linewidth=0.04](0.0,0.27)(0.96,0.3105634)(0.75286174,0.37057108)(1.78,0.77)(2.8071382,1.169429)(2.66,1.47)(3.36,1.45)
\psbezier[linewidth=0.04](0.6,1.95)(1.2539726,1.8871263)(1.1265805,1.3646309)(2.0345206,0.7973154)(2.9424605,0.23)(3.2249315,0.2543258)(3.9,0.31)
\psbezier[linewidth=0.04](0.52,-1.33)(1.48,-1.33)(1.3597014,-0.9703507)(2.3,-0.63)(3.2402985,-0.28964934)(3.14,0.05)(3.84,0.23)
\psdots[dotsize=0.16](1.98,0.85)
\psdots[dotsize=0.16](1.46,-1.11)
\usefont{T1}{ptm}{m}{n}
\rput(1.6145313,-0.06){$L$}
\usefont{T1}{ptm}{m}{n}
\rput(1.4,-1.38){$\alpha$}
\rput(2,1.2){$\beta$}

\rput(0.3,-0.6){$\hfs(\alpha)^+$}
\rput(3.3,-0.6){$\hfu(\alpha)^+ $}
\rput(0.3,0.6){$\hfu(\beta)^-$}
\rput(3.3,0.6){$\hfs(\beta)^-$}
\end{pspicture} }}
  \subfloat[A lozenge of type $(+,-,-,+)$]{\label{fig:lozenge+-+-} \scalebox{0.8} { \begin{pspicture}(-1,-2.2)(4.5,2)


\psbezier[linewidth=0.04](0.02,0.17)(0.88,0.03)(0.74114233,-0.47874346)(1.48,-1.11)(2.2188578,-1.7412565)(2.38,-1.73)(3.26,-1.95)
\psbezier[linewidth=0.04](0.0,0.27)(0.96,0.3105634)(0.75286174,0.37057108)(1.78,0.77)(2.8071382,1.169429)(2.66,1.47)(3.36,1.45)
\psbezier[linewidth=0.04](0.6,1.95)(1.2539726,1.8871263)(1.1265805,1.3646309)(2.0345206,0.7973154)(2.9424605,0.23)(3.2249315,0.2543258)(3.9,0.31)
\psbezier[linewidth=0.04](0.52,-1.33)(1.48,-1.33)(1.3597014,-0.9703507)(2.3,-0.63)(3.2402985,-0.28964934)(3.14,0.05)(3.84,0.23)
\psdots[dotsize=0.16](1.98,0.85)
\psdots[dotsize=0.16](1.46,-1.11)
\usefont{T1}{ptm}{m}{n}
\rput(1.6145313,-0.06){$L$}
\usefont{T1}{ptm}{m}{n}
\rput(1.4,-1.38){$\alpha$}
\rput(2,1.2){$\beta$}

\rput(0.3,-0.6){$\hfs(\alpha)^+$}
\rput(3.3,-0.6){$\hfu(\alpha)^- $}
\rput(0.3,0.6){$\hfu(\beta)^-$}
\rput(3.3,0.6){$\hfs(\beta)^+$}
\end{pspicture} }}
 \caption{The two possible orientations of lozenges} \label{fig:type_of_lozenge}
\end{figure}

\subsection{Foliations and slitherings}

We leave for a bit Anosov flows to digress about ($\R$-covered) foliations. Thurston in \cite{Thurston:3MFC} introduced the following notion:
\begin{defin}
 The manifold $M$ slithers around the circle if there exists a fibration $s \colon \widetilde{M} \rightarrow S^1$ such that $\pi_1(M)$ acts by bundle automorphisms of $s$, i.e., an element $\gamma$ of $\pi_1(M)$ sends a fiber $s^{-1}(x)$ to a possibly different fiber. Such a map $s$ is called a slithering.
\end{defin}

A slithering $s$ defines a foliation $\mathcal{F}(s)$ on $M$, just by taking the leaves to be the projections on $M$ of the connected components of $s^{-1}(x)$. Reciprocally, we say that a foliation comes from a slithering if it is obtained in that way. It is immediate that foliations coming from slitherings are $\R$-covered. Here, $\R$-covered means that the leaf space of the foliation is $\R$.\\

 Recall that a \emph{taut} foliation is a foliation that admits a closed transversal. Note that the foliations we are interested in are always taut. A skewed $\R$-covered Anosov flow is transitive, so each strong leaf is dense (see \cite{Plante}). Hence, in order to get a closed path transverse to the weak stable foliation, we can just follow a strong unstable leaf until we are close to where we began and close it up in a transverse way.\\
 Candel managed to apply the classical uniformization theorem to taut foliations:
\begin{thm}[Candel's Uniformization Theorem \cite{Candel:uniformization}]
 Let $\mathcal{F}$ be a taut foliation on an atoroidal $M$. Then there exists a Riemannian metric such that its restriction to every leaf is hyperbolic.
\end{thm}

Using the metric given by Candel's uniformization theorem, we can define a boundary at infinity for any leaf $\lambda$ of $\widetilde{\mathcal{F}}$ that we denote by $S^1_{\infty}(\lambda)$. Thurston explained how to stitch those boundaries together to obtain a ``universal circle'' that takes into account the topological and geometrical information of the foliation.

Let us write $\mathcal{L}$ for the leaf space of the foliation $\mathcal{F}$.
\begin{defin}[Universal circle]
Let $\mathcal{F}$ be a $\R$-covered taut foliation on an atoroidal $3$-manifold $M$. A \emph{universal circle} for $\mathcal{F}$ is a circle $\univ$ together with the following data:
\begin{enumerate}
 \item There is a faithful representation
\begin{equation*}
 \rho_{\text{univ}} \colon \pi_1(M) \rightarrow \text{Homeo}^+\left(\univ \right)\,;
\end{equation*}
 \item For every leaf $\lambda$ of $\wt{\mathcal{F}}$, there is a monotone map 
\begin{equation*}
 \Psi_{\lambda} \colon \univ \rightarrow S^1_{\infty}(\lambda).
\end{equation*}
Moreover, the map
\begin{equation*}
 \Psi \colon \univ \times \mathcal{L} \rightarrow E_{\infty}
\end{equation*}
defined by $\Psi(\cdot , \lambda) := \Psi_{\lambda}(\cdot)$ is continuous.\\
Here, $E_{\infty}:= \bigcup_{\lambda \in \mathcal{L}} S^1_{\infty}(\lambda)$ is given the largest topology such that the endpoint map $e\colon T^1 \wt{\mathcal{F}} \rightarrow E_{\infty}$, which associates the endpoint in $ S^1_{\infty}(\lambda)$ of a geodesic ray defined by an element of the unit tangent bundle of the leaf $\lambda$, is continuous.
 \item For every leaf $\lambda$ of $\wt{\mathcal{F}}$ and any $\gamma \in \pi_1(M)$ the following diagram commutes:
\begin{equation*}
 \xymatrix{
 \univ \ar[r]^{\rho_{\text{univ}}(\gamma)} \ar[d]_{\Psi_{\lambda}} & \univ \ar[d]^{\Psi_{\gamma \cdot \lambda}} \\
  S^1_{\infty}(\lambda) \ar[r]_{\gamma} & S^1_{\infty}(\gamma \cdot \lambda)
}
\end{equation*}
\end{enumerate}

\end{defin}
This definition is taken from \cite{Calegari:book}. Note however that Calegari defines universal circles for any kind of taut foliations and hence needs a last condition that is empty for $\R$-covered foliations.

\begin{thm}[Thurston \cite{Thurston:3MFC}, Fenley \cite{Fen:Foliations_TG3M}, Calegari \cite{Calegari:geometry_of_R_covered}, Calegari and Dunfield \cite{Calegari-Dunfield}]
 A foliation coming from a slithering defines a universal circle.
\end{thm}
The different sources for this result are in fact different generalizations of the original result due to Thurston.

\begin{defin}[Regulating flow]
 A flow $\psi^t$ on $M$ is said to be regulating for $\mathcal{F}$ if $\psi^t$ is transverse to $\mathcal{F}$ and, when lifted to the universal cover, any orbit of $\widetilde{\psi}^t$ intersects every leaf of $\widetilde{\mathcal{F}}$. In other words, we have a homeomorphism between any orbit of $\widetilde{\psi}^t$ and the leaf space of $\mathcal{F}$.
\end{defin}

The main result concerning slitherings is probably the following, but note that Fenley and Calegari obtained that result for a larger class of foliations:
\begin{thm}[Thurston, Fenley, Calegari \cite{Calegari:geometry_of_R_covered}] \label{thm:regulating_pseudo_Anosov}
 If $\mathcal{F}$ is a foliation coming from a slithering on an atoroidal, aspherical closed $3$-manifold $M$, then it admits a pseudo-Anosov regulating flow  $\psi^t \colon M \rightarrow M$.
\end{thm}

Fenley proves even more about these pseudo-Anosov regulating flows:
\begin{thm}[Fenley \cite{Fenley:Ideal_boundaries}, \cite{Fenley:Rigidity_pseudo-Anosov}]
 Let $\mathcal{F}$ be a foliation coming from a slithering on an atoroidal, aspherical closed $3$-manifold $M$. Then, up to topological conjugacy, there is only one regulating pseudo-Anosov flow $\psi^t$.\\
 Furthermore, the orbit space of $\psi^t$ is a disc that admits $S^1_{\text{univ}}$ as a natural boundary.
\end{thm}

In order to have a better picture of the above results, let us describe very roughly how these regulating pseudo-Anosov flows are obtained. The first step is to construct a lamination of $\univ$. Let us recall the definition,
\begin{defin}
 Let $(a,b)$ and $(c,d)$ be two pairs of points in $S^1$. We say that they \emph{intersect} if $(c,d)$ is contained in different components of $S^1\smallsetminus \{a,b\}$.
\end{defin}
Calegari \cite{Calegari:book} says that the two pairs are \emph{linked}. We will justify later (Remark \ref{rem:intersects}) why we use this name.
\begin{defin}
 A \emph{lamination} of $S^1$ is a closed subset of the set of unordered pairs of distinct points in $S^1$ with the property that no two elements of the lamination intersect.
\end{defin}

Now, the first (big) step towards Theorem \ref{thm:regulating_pseudo_Anosov} is:
\begin{thm}[Thurston, Fenley, Calegari]
 If $\mathcal{F}$ is a foliation coming from a slithering on an atoroidal, aspherical closed $3$-manifold $M$, then the associated universal circle $\univ$ admits two laminations $\Lambda^{\pm}_{\text{univ}}$ which are preserved under the natural action of $\pi_1(M)$ on $\univ$.
\end{thm}
While proving that result, they also construct two laminations $\wt{\Lambda}^{\pm}$ of $\M$ such that they are transverse, $\pi_1(M)$-invariant and the intersection of $\wt{\Lambda}^{+}$ (or $\wt{\Lambda}^{+}$) with a leaf of $\mathcal{F}$ is a geodesic (for the leaf-wise hyperbolic metric). The regulating pseudo-Anosov flow is then obtained from $\wt{\Lambda}^{\pm}$ by ``collapsing'' the complementary regions, thus obtaining two transverse singular foliations, and taking the flow to be the line field generated by the intersection.\\
Note also that, in our case, any lamination in $\M$, $\pi_1(M)$-invariant and obtained from the laminations $\Lambda^{\pm}_{\text{univ}}$ will give rise to a regulating pseudo-Anosov flow (see \cite[Chapter 9]{Calegari:book}).

\section{Skewed $\R$-covered Anosov Flows}

One of the motivations of Thurston in \cite{Thurston:3MFC} to study foliations coming from slithering was Fenley's examples. Indeed:
\begin{prop}[Thurston]
 Let $\flot$ be a skewed $\R$-covered Anosov flow. Then the foliations $\fs$ and $\fu$ come from slitherings.
\end{prop}

\begin{proof}
 Let $C:= \leafs / \eta^u \circ \eta^s$ and $\pi_C \colon \leafs \rightarrow C$ the projection. As $\leafs$ is homeomorphic to $\R$ and $\eta^u \circ \eta^s$ is a strictly increasing, continuous, $\pi_1(M)$-equivariant function (by Proposition \ref{prop:eta_s_eta_u}), we have that $C$ is homeomorphic to $S^1$ and that the action of $\pi_1(M)$ on $\leafs$ descends to an action by bundle automorphisms on $C$. Hence we can define $s^s \colon \M \rightarrow C$ as the canonical projection from $\M$ to $\leafs$, then to $C$ and it is clear that $\fs$ comes from the slithering $s^s$.\\
 The same thing applies for $\fu$.
\end{proof}

So a skewed $\R$-covered Anosov flow comes with two slitherings and we can apply the theory developed by Thurston. Note that those foliations are linked:
\begin{prop}
 A regulating flow for $\fs$ transverse to $\fu$ is also regulating for $\fu$ and vice-versa.
\end{prop}

\begin{proof}
 Any lift of a(n) (un)stable leaf separates $\widetilde{M}$ into two connected components, hence the result.
\end{proof}

Note that, when the skewed $\R$-covered Anosov flow we consider is just a geodesic flow on a negatively curved surface $\Sigma$, we have an obvious regulating flow coming to mind: here, $M$ corresponds to $H\Sigma$, so consider the flow that just push vectors along the fibers, without moving the base point (in other words, the flow generating rotations). It is clear that this flow is regulating. Now $\widetilde{H\Sigma}$ is (not surprisingly) homeomorphic to $\R \times \widetilde{\Sigma} $, where the first coordinate is given by how much a vector is turned with respect to a fixed direction (taken as a point on the visual boundary of $\widetilde{\Sigma}$). It turns out that we can always have this kind of identification whenever we consider a skewed $\R$-covered Anosov flow (see also Figure \ref{fig:skewed_R-covered_AF}):

\begin{SCfigure}[50][h]
\centering
 \scalebox{1.5}{ 

\begin{pspicture}(-2,-1)(2,7)
\psset{viewpoint=30 30 30,Decran=60}

\defFunction[algebraic]{helice}(t){cos((2*Pi/3)*t)}{sin((2*Pi/3)*t)}{t}
\psSolid[object=courbe,r=0.001,range=0 6,linecolor=blue,linewidth=0.02,resolution=360,function=helice]

\defFunction[algebraic]{geodesic}(t){(1-t)*(-0.70710) + t*cos((2*Pi/3)*3)}{(1-t)*(-0.70710) + t*sin((2*Pi/3)*3)}{3}

\defFunction[algebraic]{unstable_leaf}(s,t){(1-t)*(-sqrt(2)/2) + t*cos((2*Pi/3)*s)}{(1-t)*(-sqrt(2)/2) + t*sin((2*Pi/3)*s)}{s}
\psSolid[object=surfaceparametree,
    opacity=0.7,
    linecolor={[cmyk]{1,0,1,0.5}},
    base= 1.875 4.875 0 1,
    action = draw,
    function=unstable_leaf,
    linewidth=0.5\pslinewidth,ngrid=30 1]

\defFunction[algebraic]{stable_leaf}(s,t){(1-t)*cos((2*Pi/3)*(1+s)) + t*cos((2*Pi/3)*1)}{(1-t)*sin((2*Pi/3)*(1+s)) + t*sin((2*Pi/3)*1)}{1}
\psSolid[object=surfaceparametree,
    linecolor=red,
    base= 1.875 4.875 0 1,
    action = draw,
    function=stable_leaf,
    linewidth=0.5\pslinewidth,ngrid=20 1]


\psSolid[object=cylindrecreux,h=6,r=1,action=draw,linewidth=0.01,ngrid=3 25,opacity = 0.1,fillcolor=red,incolor=white](0,0,0)

\end{pspicture} }
 \caption{Using Proposition \ref{prop:identification_M_tilda}, we can represent $\hflot$ in the following way: $\M$ is identified with a solid cylinder where each horizontal slice is a stable leaf. On a stable leaf, the orbits of $\hflot$ are lines all pointing towards the same point on the boundary at infinity of the leaf. We represented a stable leaf, with some orbits on it, in red. The blue curve represents the point at infinity where orbits ends. It is a way of seeing $\leafs$ ``slithers''. An unstable leaf now, represented in green, is given by fixing the $(x,y)$-coordinates (i.e., the points that project to the same point on the universal circle) and taking the lines pointing towards the blue curve. Finally, the orbits of the regulating pseudo-Anosov flow $\wt{\psi}^t$ are vertical curves inside the cylinder and stabilize the foliation by vertical straight lines on the boundary.} \label{fig:skewed_R-covered_AF}
\end{SCfigure}
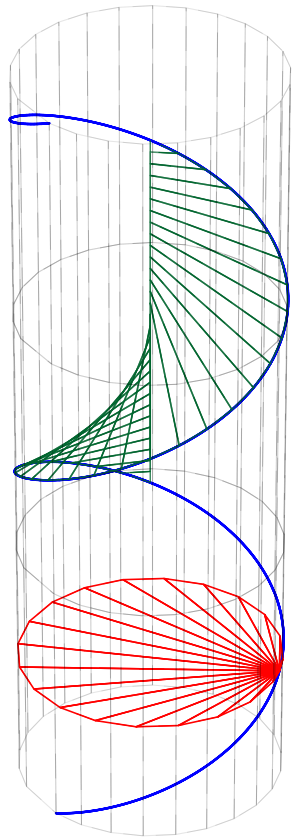

\begin{prop} \label{prop:identification_M_tilda}
Let $\flot$ be a skewed $\R$-covered Anosov flow and $\psi^t$ a regulating flow for both $\fs$ and $\fu$. Then, we can construct two continuous identifications of the universal cover:
\begin{align*}
 I^s \colon &\widetilde{M} \rightarrow \leafs \times \mathcal{O}(\psi)\,,  &    I^u \colon &\widetilde{M} \rightarrow \leafu \times \mathcal{O}(\psi) \,, \\
	    & x \mapsto \left( \hfs(x), \widetilde{\psi}^t (x) \right) & & x \mapsto \left( \hfu(x),\widetilde{\psi}^t (x) \right) 
\end{align*}
where $\orb(\psi)$ is the orbit space of $\psi^t$.
\end{prop}

\begin{proof}
 Injectivity and surjectivity are given by the definition of a regulating flow. Continuity follows from the fact that $\hfs$ is a foliation of $\M$, $\wt{\psi}^t$ is transverse to it and $\orb(\psi)$ is homeomorphic to $\R^2$.
\end{proof}

Note that, for any orbit of $\hflot$, using the slithering given by $\fs$, it is possible to ``project'' this orbit onto the universal circle. Indeed, an orbit determines two leaves $L^s \in \leafs$ and $L^u \in \leafu$. Then, using $\eta^u \colon \leafu \rightarrow \leafs$ (defined in Proposition~\ref{prop:eta_s_eta_u}), we get a pair of points $(L^s, \eta^u(L^u))$ in $\leafs$ and this in turn determines two distinct points in the universal circle. When we consider the reciprocal image of this application, we obtain:

\begin{lem}
 Two distinct points on the universal circle define a (countable) collection of orbits of $\flot$.
\end{lem}

Instead of giving a rigorous proof of that lemma, I would like to refer the reader to Figure \ref{fig:skewed_R-covered_AF}: two points on $\univ$ determine two vertical lines on the outside of the cylinder, and every time one of these lines intersects $\leafs$, we obtain one orbit.

\section{Isotopy and co-cylindrical class}

The question that started our study of these kinds of Anosov flows was the following: 
Suppose you are given a skewed $\R$-covered Anosov flow in a hyperbolic $3$-manifold $M$. Any periodic orbit is freely homotopic to infinitely many other orbits. In other words, we have a family of knots in $M$. Then {\it are these knots different?} Here, we understand ``different'' in the sense of traditional knot theory, i.e., two knots are equivalent if there exists an isotopy between them. And if some of these knots are indeed different, can we say more about them? That is, can we develop a kind of knot theory adapted to Anosov flows?

It turns out that there is no knot theory in that case. Indeed we will show below that any freely homotopic orbits are isotopic.

We will also study a related question: {\it among a free homotopy class, can we say when two orbits are boundaries of an embedded cylinder?} Indeed, an isotopy between two orbits gives an immersed cylinder. So it seems natural to wonder whether this can be made into an embedding. Furthermore, Barbot \cite{Bar:MPOT} (and later together with Fenley \cite{BarbotFenley}) studied embeddings of tori in manifolds equipped with an Anosov flow. This is in some sense the atoroidal equivalent.

The results in the following sections are joint work with Sergio Fenley and will be published with full details later.

\subsection{Isotopy class of periodic orbits}

Let us start by giving the definition of isotopy we will use here:
\begin{defin}
 Two curves $c_1$ and $c_2$ in $M$ are isotopic if there exists a continuous application $H \colon S^1 \times [0,1] \rightarrow M$ such that $H(S^1,0)= c_1$, $H(S^1,0)= c_2$ and, for any $t\in [0,1]$, $H(S^1,t)$ is an embedding of $S^1$ in $M$.
\end{defin}

Among isotopic orbits, we define:
\begin{defin}
 Two curves $c_1$ and $c_2$ in $M$ are \emph{co-cylindrical} if there exists an embedded annulus $A$ in $M$ such that $\partial A = c_1 \cup c_2$.
\end{defin}
Note that this is \emph{not} an equivalence relation as it is clearly non-transitive. However, as we will see, its study is quite interesting.

Let us start by considering geodesic flows for a minute. In that case, the question of isotopy is trivial (because there is, at most, one periodic orbit in a free homotopy class). What is \emph{not} trivial however is answering the following question: given a periodic orbit $\alpha$, is there an embedded torus in $H\Sigma$ containing $\alpha$? If you suppose that $\alpha$ is simple, then the answer is clearly yes. Indeed, just take $\{ (x,v) \in H\Sigma \mid x \in \pi(\alpha) \}$. If the orbit is non-simple however, it turns out that there is no such embedded torus.

To my everlasting surprise, this kind of condition will remain true for any skewed $\R$-covered Anosov flow.\\

But before studying co-cylindrical classes, we can use the work of Thurston, Fenley and Calegari to answer our first question and deduce that the isotopy classes are the same as the double free-homotopy classes:
\begin{thm}[Barthelm\'e, Fenley] \label{thm:homo_implies_isotope} 
 Let $\flot$ be a skewed $\R$-covered Anosov flow on a closed atoroidal, not Seifert fibered, $3$-manifold. If $\alpha_i$ is a double free homotopy class of periodic orbits of $\flot$, then all the $\alpha_i$s are isotopic.
\end{thm}

\begin{proof}[Sketch of proof]
 We are going to construct an isotopy between $\alpha_0$ and $\alpha_1$. As isotopy is an equivalence relation, it will show that all free homotopic orbits are isotopic.

Let $\psi^t$ be a regulating flow for $\hfs$ and $\hfu$, $\al 0$ a lift of $\alpha_0$ to $\M$ and ${\al 1 = \eta(\al 0)}$. For any $x \in \al 0$ there exists a time $T(x)$ such that $\wt{\psi}^{T(x)}(x) \in \hfs(\al 1)$.

Let $C:=\lbrace \psi^t(\pi(x)) \mid x\in \al 0 , 0\leq t \leq T(x) \rbrace$. It is an immersed cylinder with one boundary $\alpha_0$ and the other one a closed curve on $\fs(\alpha_1)$. Let us call the second boundary component $\alpha_1'$.

Up to a $C^1$ modification of $\psi^t$, we can show that there is only a finite number of (transverse) intersections of $\alpha_0$ with $C$. We can therefore find a continuous time change $\Psi^t$ of $\psi^t$ such that, for some $t_1\in \R$, $\alpha_{1}' = \Psi^{t_1}(\alpha_{0})$. As $\Psi^t$ is a flow, for any $t\in [0,t_1]$, $\Psi^t(\alpha_0)$ is an embedded $S^1$ in $M$.

We produced an isotopy from $\alpha_0$ to $\Psi^{t_1}(\alpha_0)$. Now, $\Psi^{t_1}(\alpha_0)$ is freely homotopic to $\alpha_1$ on the surface $\fs(\alpha_1)$, hence is isotopic.
\end{proof}

\subsection{Co-cylindrical class}

We will now show the link between having two co-cylindrical periodic orbits and simple chain of lozenges. This is essentially based on Barbot's work \cite{Bar:MPOT}.

In \cite{Bar:MPOT} (see also \cite{BarbotFenley}) Barbot studied embedded tori in (toroidal) $3$-manifolds supporting skewed $\R$-covered Anosov flows, showing that they could be put in a quasi-transverse position (i.e., transverse to the flow, apart from along some periodic orbits). We will use his work to obtain properties of embedded annuli:
\begin{thm} \label{thm:simple_and_cocylindrical}
 Let $\alpha$ and $\beta$ be two orbits in the same free homotopy class, choose coherent lifts $\wt{\alpha}$ and $\wt{\beta}$, and denote by $B(\wt{\alpha}, \wt{\beta})$ the chain of lozenges between $\wt{\alpha}$ and $\wt{\beta}$.\\
 If $\alpha$ and $\beta$ are co-cylindrical, then $B(\wt{\alpha}, \wt{\beta})$ is simple, i.e., if we denote by $(\al i)_{i=0 \dots n}$ the corners of the lozenges in $B(\wt{\alpha}, \wt{\beta})$, with $\al 0 = \wt{\alpha}$ and $\al n= \wt{\beta}$, then
\begin{equation*}
 \left( \pi_1(M) \cdot \al i \right) \cap B(\wt{\alpha}, \wt{\beta}) = \{ \al i\}.
\end{equation*}
Conversely, if $B(\wt{\alpha}, \wt{\beta})$ is simple, then there exists an embedded annulus, called a Birkhoff annulus, with boundary $\alpha \cup \beta$.
\end{thm}

\begin{proof}
 Construction of an embedded Birkhoff annulus from a simple chain of lozenges is done in \cite{Bar:MPOT}, hence proving the converse part.

To prove that, if $\alpha$ and $\beta$ are co-cylindrical, then $B(\wt{\alpha}, \wt{\beta})$ is simple, we have to re-prove Lemma 7.6 of \cite{Bar:MPOT} (or equivalently step 1 of the proof of Theorem 6.10 of \cite{BarbotFenley}) when, instead of having an embedded torus, we just have an embedded cylinder.\\
 Let $C$ be an embedded cylinder such that $\partial C = \lbrace \alpha, \beta \rbrace$ and $\wt{C}$ the lift of $C$ in $\M$ such that its boundary is on $\wt{\alpha}$ and $\wt{\beta}$. Let us also denote the generator of the stabilizer of $\wt{\alpha}$ by $\gamma \in \pi_1(M)$. Following \cite{Bar:MPOT}, we can construct a embedded plane $\wt{C}_0$ in $\M$ such that
\begin{itemize}
 \item $\wt{C}_0$ is $\gamma$-invariant, 
 \item $\wt{C}_0$ contains all the $\al i$,
 \item $\wt{C}_0$ is transverse to $\hflot$ except along the $\al i$,
 \item the projection of $\wt{C}_0$ to $\orb$ is $B(\wt{\alpha}, \wt{\beta})$.
\end{itemize}
Barbot's trick to obtain such a plan is, for every lozenge in $B(\wt{\alpha}, \wt{\beta})$, to take a simple curve $\bar{c}$ from one corner of the lozenge to the other (for instance $\al i$ and $\al{i+1}$). Then, lift $\bar{c}$ to $\wt{c}\subset \M$ such that $\wt{c}$ is transverse to $\hflot$. Now choose an embedded rectangle $R_i$ in $\M$ such that $R_i$ is bounded by $\wt{c}$, $\gamma \cdot \wt{c}$, and the two pieces of $\al i$ and $\al{i+1}$ between the endpoints of $\wt{c}$ and $\gamma \cdot \wt{c}$. Then define $\wt{C}_0$ as the orbit under $\gamma$ of the unions of the rectangles $R_i$.\\

From now on, we copy the proof of \cite[Theorem 6.10, step 1]{BarbotFenley}.\\
Suppose that $B(\wt{\alpha}, \wt{\beta})$ is not simple. Then there exist $\al i$ and $h \in \pi_1(M)$ such that $\theta:= h \cdot \al i$ intersects the interior of $B(\wt{\alpha}, \wt{\beta})$. Then $\theta$ intersects $\wt{C}_0$ in a single point $p$. Let $\theta^+$ and $\theta^-$ be the two rays in $\theta$ defined by $p$. If we denote by $\wt{V}$ the subset of $\M$ delimited by $\hfs(\wt{\alpha})$ and $\hfs(\wt{\beta})$ and containing $\wt{C}_0$, then $\wt{C}_0$ separates $\wt{V}$ in two components.

\begin{claim}
 Either $\theta^+$ or $\theta^-$ stays a bounded distance away from $\wt{C}_0$.
\end{claim}
\begin{proof}
Assume they don't: for any $R>0$, there exist points $q_R^-,q_R^+$ on $\theta^-, \theta^+$ such that $d(q_R^{\pm}, \wt{C}_0) >R$.\\
As $\pi (\wt{C}_0)$ and $C$ are freely homotopic, there exists $R_0$ such that $\wt{C}$ is contained in the $R_0$-neighborhood of $\wt{C}_0$. Then, for any $R > 2 R_0$, any path in $\wt{V}$ joining two points $q^-$ and $q^+$ such that $d(q^{\pm}, q_R^{\pm})<R$ must intersect $\wt{C}$.\\
Now, $\pi(\theta)$ is freely homotopic to a curve in $C$, and, as $C$ is embedded in an oriented manifold, it must be two-sided. So $\pi(\theta)$ is homotopic to a curve disjoint from $C$. (Note that this is the only point where we use the fact that $C$ is embedded). Lifting it gives a homotopy from $\theta$ to a curve $\theta_1$ disjoint from $\wt{C}$. But homotopies move points a bounded distance away: there exists $r>0$ such that, for any $R>0$, there are two points $m_R^{\pm}$ on $\theta_1$ such that $d(m_R^{\pm}, q_R^{\pm}) <r$.\\
Choose $R > \max\{ 2 R_0, r \}$, according to the above, the segment in $\theta_1$ from $m_R^-$ to $m_R^+$ must intersect $\wt{C}$ hence a contradiction.
\end{proof}

We assume that $\theta^+$ stays at a distance $\leq a_1$ from $\wt{C_0}$.\\
Let $g\in \pi_1(M)$ be the generator of the stabilizer of $\theta$. Choose a sequence $\left(p_i\right)$ with $p_i := g^{n_i}\cdot p \in \theta^+$. Let $\left(q_i\right)$ be a sequence in $\wt{C}_0$ such that $d(q_i,p_i) \leq a_1$, up to a subsequence, we can assume that $\pi(q_i)$ converges and as $\pi(\wt{C_0})$ is compact, we can even assume that $\pi(q_i)$ is constant. Now, up to another subsequence, we can assume that there are segments $u_i$ in $\wt{V}$ from $p_i$ to $q_i$ such that $\pi(u_i)$ converges in $M$. Adjusting once again, we can assume that $\pi(u_i)$ is constant for big enough $i$.\\
We consider the following closed curve in $\wt{V}$: start by a segment in $\theta^+$ from $p_i$ to $p_k$, $k>i$, then follow $u_k$, then choose a segment in $\wt{C}_0$ from $q_k$ to $q_i$ and close up along $u_i$. Since $\pi(u_i) = \pi(u_k)$, this shows that there exists $n\in \Z$ such that $g^n(q_i)= q_k$.\\
Hence, for some $n \neq 0$, $\wt{C}_0$ is left invariant by $g^n$, which implies that $B(\wt{\alpha}, \wt{\beta})$ is also invariant by $g^n$. But $g^n \cdot \theta = \theta$, so $g^n$ leaves invariant a point in the interior of a lozenge as well as the whole lozenge, which is impossible.
\end{proof}

Using the theorem, we can deduce the following property of co-cylindrical class:
\begin{prop}\label{prop:cardinality_isotopy_class}
If the co-cylindrical class of one orbit is finite, then all the co-cylindrical classes in the same double free homotopy class are finite. Moreover, they all have the same cardinality.
\end{prop}

\begin{proof}
This result just relies on the fact that the homeomorphism $\eta$ of $\orb$, defined by applying $\eta^u$ and $\eta^s$ to respectively the unstable and stable leaf commutes with the action of $\pi_1(M)$.

 Let $\alpha_i$ be a double free homotopy class of periodic orbit. Suppose that $\alpha_0$ has $k$ elements in its co-cylindrical class. Then it implies that, for any coherent lift $\al i$ of the $\alpha_i$, the chain of lozenges $B(\al 0, \al k)$ is non-simple. More precisely, we have an element $h\in \pi_1(M)$ such that $h \cdot \al 0$ is in $L(\al{k-1}, \al k)$, the lozenge with corners $\al{k-1}$ and $\al k$. Indeed, recall that $\al i = \eta^i (\al 0)$, so, if $h \cdot \al i \in L(\al{k-1}, \al k)$, then 
\begin{equation*}
h \cdot \al 0 = h \cdot \eta^{-i}(\al i) = \eta^{-i}\left(h \cdot \al i \right) \in \eta^{-i}\left(L(\al{k-1}, \al k) \right) = L(\al{k-1-i}, \al{k-i}).
\end{equation*}
 So $\alpha_0$ and $\alpha_{k-i}$ would not be co-cylindrical.

Then, for any $i$, $h \cdot \al i \in L(\al{k-1 + i}, \al{k+i})$, which proves that the number of orbits co-cylindrical to $\alpha_i$ is at most $k$, and again the same argument as above shows that it is also at least $k$.
\end{proof}

\section[Action on $\univ$ and co-cylindrical orbits]{Action of the fundamental group on $\univ$ and co-cylindrical orbits} \label{sec:action_of_pi1}

Thanks to Thurston's work in \cite{Thurston:3MFC} we know that the fundamental group of a $3$-manifold admitting a $\R$-covered foliation acts on the universal circle implying many results about the type of group it can be, as we can see for instance in \cite{Calegari:geometry_of_R_covered,Calegari-Dunfield,Fen:Foliations_TG3M}. There is a remarkable link between the existence of co-cylindrical orbits and the action of $\pi_1(M)$ on \emph{pairs} of points in $\univ$.\\

\begin{defin}
 Let $(\alpha^+, \alpha^-)$ and $(\beta^+, \beta^-)$ be two pairs of points in $\univ$. We say that $(\alpha^+, \alpha^-)$ and $(\beta^+, \beta^-)$ \emph{intersect} if, for some order on $\univ$, we have
\begin{equation*}
 \alpha^- < \beta^- <\alpha^+ < \beta^+.
\end{equation*}
We will say that $(\alpha^+, \alpha^-)$ \emph{self-intersects} if there exists $h\in \pi_1(M)$ such that $(\alpha^+, \alpha^-)$ and $(h \cdot \alpha^+, h \cdot \alpha^-)$ intersect.
\end{defin}

\begin{prop}\label{prop:intersects_equivalent_simple}
 Let $\alpha$ be a periodic orbit of $\flot$, $\wt{\alpha}$ a lift to $\M$ and $(\alpha^+, \alpha^-)$ the projection of $\wt{\alpha}$ on $\univ$.\\
 The co-cylindrical class of $\alpha$ is finite if and only if $(\alpha^+, \alpha^-)$ self-intersects.
\end{prop}

\begin{proof}
 If the co-cylindrical of $\alpha$ is finite, then (by Theorem \ref{thm:simple_and_cocylindrical}) the chain of lozenges containing $\wt{\alpha}$ is non-simple. So, there exists $h\in \pi_1(M)$ such that $h\cdot \wt{\alpha} \in L(\al i, \al{i+1})$. Projecting that lozenge to $\univ$ shows that $(\alpha^+, \alpha^-)$ and $h \cdot (\alpha^+, \alpha^-)$ intersect.
Reciprocally, if there exists $h \in \pi_1(M)$ such that $(\alpha^+, \alpha^-)$ and $h \cdot (\alpha^+, \alpha^-)$ intersect, then $h\cdot \wt{\alpha} \in L(\al i, \al{i+1})$ for some $i$. Hence, by Theorem \ref{thm:simple_and_cocylindrical}, the co-cylindrical class of $\alpha$ must be finite.
\end{proof}

Let us announce the following result with Sergio Fenley, to be published later: 
\begin{thm}[Barthelm\'e, Fenley] \label{thm:everything_intersects} 
 Let $(\alpha^+, \alpha^-)$ be the projection on $\univ$ of a periodic orbit $\wt{\alpha}$ of $\hflot$. Then $(\alpha^+, \alpha^-)$ self-intersects.
\end{thm}

The theorem is proved by seeing $\univ$ as the boundary at infinity of the orbit space of a regulating pseudo-Anosov flow $\psi^t$ and using the transitivity of such flows (Mosher \cite{Mosher}, proved that any pseudo-Anosov flow on an atoroidal manifold is transitive).

\begin{rem}\label{rem:intersects}
 Suppose that $(\alpha^+, \alpha^-)$ self-intersects and denote by $(\al i)$ the orbits in $\M$ projecting to $(\alpha^+, \alpha^-)$ and $\alpha_i = \pi (\al i)$ their projection to $M$. Then, for any $i$, there exist a $j$ and a $t$ such that $\psi^t(\alpha_i) \cap \alpha_j \neq \emptyset$. By flowing one orbit we get an actual intersection.
\end{rem}

If we consider the geodesic flow case now, there is also a natural circle at infinity. Just take the visual boundary $\wt{\Sigma}(\infty)$ and the fundamental group $\pi_1(H\Sigma)$ naturally acts on it. So, for a pair $(\alpha^+, \alpha^-)$ in $\wt{\Sigma}(\infty)$, to self-intersect in that case means that the only geodesic in $\Sigma$ such that a lift of it has endpoints $(\alpha^+, \alpha^-)$ is non-simple. Hence, in the geodesic flow case, there exist points on $\wt{\Sigma}(\infty)$ representing a periodic orbit that does not self-intersect, in contrast with the atoroidal case we studied here.

As a corollary of Proposition \ref{prop:intersects_equivalent_simple} and Theorem \ref{thm:everything_intersects}, we obtain:
\begin{thm}[Barthelm\'e, Fenley] \label{thm:co_cylindrical_are_finite}
 Every co-cylindrical class is finite.
\end{thm}

Note that it is still an open question whether a co-cylindrical class can be non-trivial. We only know that some are:
\begin{prop}
 There exist periodic orbits of $\flot$ with trivial co-cylindrical class. 
\end{prop}

\begin{rem}
For such an orbit, Proposition \ref{prop:cardinality_isotopy_class} shows that every other orbit in the double free homotopy class must also have a trivial co-cylindrical class.
\end{rem}

\begin{proof}
 Let $V$ be a flow box of $\flot$, as $\flot$ is transitive, we can pick a long segment of a dense orbit that $\eps$-fills $V$. Then, by the Anosov Closing lemma (see \cite{KatokHassel}), we get a periodic orbit $\alpha$ that $2\eps$-fills $V$. Now, choose $x$ on one of the connected components of $\alpha \cap V$. If $\eps$ was chosen small enough, then there must exist $y$ on another connected component of $\alpha \cap V$ such that there is a close path $c$ \emph{staying in $V$}, starting at $x$ going through the positive stable leaf of $x$, then the negative unstable leaf of $y$, then the negative stable leaf of $y$ and finally close up along the positive stable leaf of $x$. If we lift the path $c$ to the universal cover of $M$ and project it to the orbit space $\orb$, as $V$ has no topology, we see that the projection of the lift of $y$ must be inside the lozenge determined by the lift of $x$ (remember that we chose our flow so that the lozenges orientation is $(+,+,-,-)$, otherwise, we would have to modify our path $c$, see Figure \ref{fig:lozenge++--}). Hence the lozenge is non-simple and therefore the co-cylindrical class of $\alpha$ is trivial.
\end{proof}

\subsection{Some open questions} 

I wanted to end this dissertation with a list of questions I have about skewed $\R$-covered Anosov flow, because even if a lot of things are known, thanks mostly to T. Barbot and S. Fenley, the things that are unknown justify, at least in my view, a continuation of their study.\\

Let's start with the ``topological'' questions:

P. Foulon and B. Hasselblatt \cite{FouHassel:contact_anosov} have constructed \emph{contact} Anosov flows (i.e., Anosov flow preserving a contact form) on not Seifert-fibered spaces and it seems very likely that their construction often yields hyperbolic manifolds. Now, contact Anosov flows are skewed and $\R$-covered (see \cite{Bar:PAG}) and are the ``nicest'' flows from a regularity point of view (see \cite{FouHassel:Zygmund_strong_foliations}). In \cite{BarbotFenley}, Barbot and Fenley showed that skewed $\R$-covered Anosov flows in Seifert-fibered spaces are (up to a finite cover) topologically conjugated to a geodesic flow on a closed surface. A natural question is then: \emph{Is a skewed $\R$-covered Anosov flow on an atoroidal manifold always topologically conjugate to a contact Anosov flow?}

 Indeed, it seems that the structure of $\M$ given by the regulating pseudo-Anosov flow (see proposition \ref{prop:identification_M_tilda}) is very rigid, so can we use that to show that $\flot$ is topologically contact? (See \cite{Bar:PAG} for a definition.) And, from there, can we get an actual topological conjugacy?

Given a continuous map $s \colon \R \times \mathbb{H}^2 \rightarrow S^1$ such that, for all $t \in \R$, $s(t, \cdot)$ is constant and for any $x \in \mathbb{H}^2$, $s(\cdot, x)$ is strictly monotone, we can construct, using Figure \ref{fig:skewed_R-covered_AF}, an Anosov flow on $\R \times \mathbb{H}^2$. Now suppose that we are given a discreet group $\Gamma$ acting in a ``good'' way on $\R \times \mathbb{H}^2$, then is the quotient flow a contact Anosov flow? And if that is true, then, can we get all contact Anosov flows on atoroidal $3$-manifolds in this fashion?\\

Finally, there are a lot of ergodic theoretical questions for these flows:

A classical question (initiated by Bowen and Margulis) for Anosov flows is to count the number of closed orbits of length less than $R$ and find an asymptotic equivalent when $R$ gets big. In \cite{ParryPollicott}, Parry and Policott prove that this number is asymptotic to $e^{hR}/hR$ where $h$ is the topological entropy. Following them, Katsuda and Sunada \cite{KatSun} answered the question of counting closed orbits inside an \emph{homology} class. So it seems natural to ask, in the case of skewed $\R$-covered Anosov flows on atoroidal manifolds, whether we can give an equivalent to the number of closed orbits of length less than $R$ inside a free \emph{homotopy} class.\\
A somewhat related question (asked by M. Crampon) is the following: let $\alpha_i$ be the orbits in a free homotopy class of $\flot$, denote by $l_i$ the length of $\alpha_i$ and $\delta_{\alpha_i}$ the Dirac measure on $\alpha_i$. Let 
\begin{equation*}
\mu_n := \sum_{ |i| \leq n } \frac{\delta_{\alpha_i}}{ l_i }\,.
\end{equation*}
The sequence $(\mu_n)$ admits at least one weak limit $\mu$. Can we show that this limit is unique and ergodic? If that is so, then what is the measure-entropy of $\mu$? Can we link that entropy to the previous counting question?

\pagebreak

{\thispagestyle{empty}
\selectlanguage{french}{
\null \vspace{\stretch{1}}
\begin{flushright}
 Je ne sais pas le reste.
\end{flushright}
\vspace{\stretch{2}} \null
}


\backmatter

\selectlanguage{french}{
\chapter{Remerciements}


J'ai souvent du mal à commencer et à finir les choses, surtout quand il s'agit d'écrire. Ici c'est encore plus difficile parce qu'il faut que je commence à remercier des gens et qu'en plus ça terminera ma thèse... 


Ces remerciements sont longs, car il y a une quantité de gens que je veux remercier. C'est aussi une bonne chose: ça permettra aux gens assistant à la soutenance et pas passionnés par la géométrie de Finsler ou les flots d'Anosov d'avoir quelque chose à lire. Ce sera ma première manière de remercier tout ceux qui assistent à ma soutenance. 

Quand on écrit des remerciements, il faut aussi essayer d'oublier personne. C'est impossible d'oublier personne, le travail que j'ai produit ici n'est pas vraiment de moi; il est le résultat de tous les gens que j'ai croisés, qui m'ont orienté, appris des choses, de la (les) société(s) dont je fais parti. Comme impossible est non seulement français mais universel, j'ai juste essayé de me concentrer sur les personnes qui ont eu un impact direct sur ce travail mathématiques et d'en oublier le moins possible. 

Il est traditionnel de commencer par remercier son (ses) directeur(s) de thèse, souvent j'aime pas beaucoup les traditions, mais il y en a des bonnes. En plus, ça me fait vraiment plaisir de respecter celle là, car non seulement je n'aurais jamais pu écrire ses lignes sans eux, mais en plus j'ai toujours apprécié mes discussions avec eux, que ce soit à propos de maths ou d'autres choses.

Donc, j'aimerais remercier Patrick pour plein de raisons. D'abord, c'est lui qui m'a donné le Laplacien en Finsler (entre autres pistes qu'il m'a montrées), donc c'est peu dire qu'il est à la base de cette thèse. Il a toujours bien su me guider et j'ai vraiment énormément appris à son contact. Mais, plus généralement, c'est quelqu'un que j'apprécie et que j'admire beaucoup: il a toujours plein d'idées intéressantes sur des quantités de choses, et, ce qui ne cessera jamais de m'étonner, il arrive à discuter de maths de 16h à 20h après s'être enchaîné 18 réunions administratives dans la journée. Donc, merci Patrick d'avoir pris sur ton temps pour m'apprendre toutes ces choses et m'avoir aidé durant ces années.

\selectlanguage{american}{%
My second advisor is Boris, and I switch to English because it is the language we speak together. Boris works at Tufts University, it is in between Somerville and Medford, but let's just say that it is in Boston, because more people know where that is. Boris is also an amazing advisor, and I'm indebted to him for so many reasons. But to explain the full extent of it, I need to digress a bit. As you might have noticed, my dissertation has two fairly disjoint parts: the second part is about Anosov flows and the first part is about Finsler geometry, with dynamical systems coming up only later. Boris knows everything there is to know about dynamical systems (and this statement is such a slight exageration that it can be considered true) and I am thankful to him for sharing some of that knowledge with me. But, when I started my Ph.D., he was not an expert in Finsler geometry. Despite that, he has always been helpful, going out of his way to try answering a question not in the range of his expertise. I think that gives an idea of how nice he is. In addition to his tremendous help in Mathematics, Boris invited me to Tufts University several times, even offering me twice a lecturing position there. This allowed me to spend long periods of time in Boston. And those who know me know why spending time in Boston is important to me. So, for all that and much more, thanks Boris!
}%


Avoir eu non pas un, mais deux directeurs de thèse sympas et m'ayant énormément aidé constitue déjà une grande chance. Moi, j'ai eu encore plus de chance: j'ai eu un frère jumeau de thèse vachement cool. Il s'appelle Mickaël Crampon et, ça va l'emmerder mais je l'écris quand même, il est super intelligent. Tous les deux ont a discuté de plein de trucs, et il m'a été d'une aide précieuse. Non seulement il a des tas d'idées et connaît plein de trucs en maths, mais, quand on en a marre de faire des maths, on peut aller boire de la bière ou du rhum et parler politique.

Un truc dont je me suis rendu compte au cours de cette thèse, c'est que faire des maths tout seul c'est bien, mais à plusieurs c'est encore mieux. Il y a une personne qui apparaît ici tout naturellement et ça tombe bien car en plus c'est un de mes rapporteurs.

Quand j'ai rencontré Bruno Colbois (à l'occasion d'une conférence organisée en partie par le G.D.R. Platon, j'y reviendrai), il a tout de suite été intéressé par mes travaux et je l'en remercie. Il connaît une quantité impressionnante de choses sur le(s) Laplacien(s) et il m'a fait découvrir et réfléchir à plein de nouvelles questions. Il m'a invité à Neuchâtel en juin dernier, c'est là que j'ai eu pour la première fois l'impression de faire des maths à plusieurs comme un grand. On a discuté tous les trois avec Patrick Verovic qui était aussi là, et j'en profite pour remercier ce Patrick là, et je me suis dis que c'était quand même vachement mieux de faire des maths en collaboration!
Donc Bruno, merci d'avoir accepté de rapporter cette thèse, merci de ton intérêt pour mon travail et merci de m'offrir l'occasion de travailler avec toi après cette thèse.

La seconde personne qui m'a fait l'honneur d'accepter de rapporter ma thèse c'est François Ledrappier. J'ai beaucoup lu ses papiers et je l'ai rencontré quelque fois, grâce à nouveau au G.D.R. Platon. Les fois où je l'ai rencontré, j'ai jamais trop osé lui parler: les gens connus me font un peu peur. Je savais par Mickaël qu'il était très gentil. Les contacts que j'ai eu avec lui depuis qu'il a accepté de rapporter cette thèse l'ont confirmé. Merci François d'avoir encombré tes vacances de Noël avec le rapport de cette thèse et merci d'avoir fait tous ces théorèmes que je peux essayer d'adapter.

Tant que j'y suis à parler de l'honneur qui m'est fait et de gens importants, merci à Jean-Pierre Bourguignon d'avoir accepté d'être membre de ce jury. Merci de prendre de votre temps pour assister à ma soutenance, j'en suis vraiment honoré et reconnaissant. Enfin, merci d'avoir lu et commenté si attentivement mon manuscrit.

Merci à Athanase Papadopoulos, qui est le représentant de l'université de Strasbourg à ma soutenance. J'ai été bien content de le c\^otoyer pendant mes années d'études à Strasbourg, il connaît plein de choses sur ce qui est discuté dans cette thèse, et c'est aussi quelqu'un de vraiment bien. Athanase ça me fait plaisir que tu participes à ma soutenance.

Thierry Barbot a aussi accepté d'être membre de mon jury. J'en suis très content car j'ai passé beaucoup de temps avec ses articles et un petit peu à discuter directement avec lui. Ses articles et lui m'ont bien aidé à comprendre un peu les flots d'Anosov alignables en biais. Il connaît tout de ces flots et en a prouvé la moitié. Donc merci Thierry.

\selectlanguage{american}{%
The second half of what is known about these Anosov flows is due to Sergio Fenley. I met him when he came to give a talk at Tufts in 2009. Since then he has always been helpful with my questions. He also invited me to Tallahassee and we worked together there. It is fun to work with him and I hope we will continue because there is still so much to be done.
}%

Au cours de ma thèse, j'ai aussi rencontré Yves Colin de Verdière, il m'a fait l'honneur de s'intéresser à mes travaux. Il m'a suggéré de nombreuses choses et m'a posé plein de questions auxquelles j'ai essayé de répondre. Le rencontrer fut très important pour moi et cette thèse lui doit beaucoup.

Encore une fois, c'est grâce au G.D.R. Platon que j'ai rencontré Yves Colin de Verdière, et il est grand temps de remercier ceux qui se cachent derrière.
Françoise Dal'bo est une des organisatrices de ce G.D.R. et je suis s\^ur qu'elle est en grande partie responsable du fait que les jeunes chercheurs y sont tellement bien traités. Sans elle et le G.D.R., je n'aurais pas rencontré la moitié des gens qui ont eu un impact direct sur mes travaux.

Il y a, bien s\^ur, plein d'autres mathématicien(ne)s que j'ai rencontrés, avec qui j'ai discuté, et qui m'ont permis d'apprendre des choses. Parmi ces personnes, j'aimerais en distinguer certaines, comme Vincent Berard (et oui, aussi surprenant que cela puisse paraître, j'ai aussi un peu parlé de maths avec lui), Bruce Boghosian, Sofien Souaifi, Camille Tardif et Genevieve Walsh.

J'aimerais aussi remercier mes compagnons de bureau, Anne-Laure et Ambroise (et Mickaël mais c'est déjà fait plus haut). L'ambiance à toujours été très bonne et ça fait du bien.

Si j'ai pu faire des maths pendant toutes ces années sans me soucier de rien ou presque, c'est grâce aux équipes administratives qui m'ont choyé que ce soit à Strasbourg ou à Tufts. Donc un grand merci à elles!

J'ai aussi été membre du Collège Doctoral Européen, promotion Rosa Parks. Le CDE a fait plein de choses pour moi. En particulier, il a financé une partie de mes voyages à Boston et il m'a appris des choses sur l'Europe et sur d'autres sujets, ce qui est bien. Donc merci CDE et puis surtout merci Céline, Christine et Jean-Paul.

Ma vie mathématiques n'est (heureusement!) pas la seule, mais ce serait beaucoup trop long de remercier tous les gens qui ont une importance pour moi dans d'autres domaines, et d'expliquer en quoi ils ont été (et continueront à être) essentiels. Donc je serai rapide et je sais qu'ils me pardonneront.

Pour rester sain d'esprit durant ma thèse j'ai joué aux cartes avec mes compagnons thésards à Strasbourg, fait plein de manifs (merci Nico?), bu de la bière et discuté de trucs avec plein de gens que j'aime beaucoup en France et à Boston et fait de l'escalade (à mains nues, oui madame!). Donc merci à tout mes potes.

J'aimerais enfin remercier mes parents, Jacqueline et François et mon frère Simon, parce que c'est un bon moment pour remercier sa famille qui a eu à vous supporter pendant un bon bout de temps (27 ans pour l'instant et ça n'arrête pas). En plus, ils ont gentiment corrigé une partie des innombrables fautes de grammaire que j'ai cachées dans cette thèse.

I also met someone very dear to me during my thesis. Thanks Christine for being there and thank you for coping with me all the time, especially with the stress of the last few months.

}

\bibliographystyle{amsplain}
\bibliography{/home/thomas/Dropbox/maths/these/tout}

\end{document}